\documentclass{article}
\usepackage[utf8]{inputenc}
\usepackage{fullpage}
\usepackage{amsmath}
\usepackage{amssymb}
\usepackage{amsthm}
\usepackage{epsfig}
\usepackage{epstopdf}
\usepackage{mathtools}
\usepackage{relsize}
\usepackage{hyperref}
\hypersetup{
    colorlinks=true,     
    linkcolor=blue,     
    citecolor=blue,     
}
\usepackage[numbers,sort&compress]{natbib}
\usepackage{enumitem}
\usepackage{algorithm,algpseudocode}
\usepackage{setspace}
\usepackage{caption}
\usepackage{subcaption}
\usepackage{xspace}
\usepackage{authblk}
\usepackage{dsfont}
\usepackage{multicol}
\usepackage[export]{adjustbox}
\usepackage{thmtools,thm-restate}

\title{Data-Driven Sample Average Approximation \\ with Covariate Information\thanks{An earlier version of this article was available on Optimization Online on July 24, 2020. This version also incorporates analysis from our unpublished technical report arXiv preprint arXiv:2101.03139.}}

\date{Version 2 (this document): July 25, 2022 \\[0.1in] \hspace*{-0.85in} Version 1: July 24, 2020}

\author[1]{Rohit Kannan\thanks{I dedicate this work to the memory of my grandparents (Thathi and Thatha).}}
\author[2]{G{\"u}zin Bayraksan}
\author[3]{James R. Luedtke}
\affil[1]{Center for Nonlinear Studies (T-CNLS) and Applied Mathematics \& Plasma Physics (T-5), \protect\\ Los Alamos National Laboratory, Los Alamos, NM, USA. E-mail: rohitk@alum.mit.edu}
\affil[2]{Department of Integrated Systems Engineering, The Ohio State University, Columbus, OH, USA. \protect\\ E-mail: bayraksan.1@osu.edu}
\affil[3]{Department of Industrial \& Systems Engineering and Wisconsin Institute for Discovery, \protect\\ University of Wisconsin-Madison, Madison, WI, USA. E-mail: jim.luedtke@wisc.edu}

\newcounter{mycounter}

\newcommand{\Set}[2]{\left\lbrace #1 : #2 \right\rbrace}

\DeclareMathOperator*{\argmin}{arg\,min}

\DeclarePairedDelimiter\ceil{\lceil}{\rceil}
\DeclarePairedDelimiter\floor{\lfloor}{\rfloor}

\newcommand{\tr}[1]{\ensuremath{{#1}^\text{T}}}

\newcommand{\uset}[2]{\ensuremath{\underset{#1}{#2}}}

\DeclarePairedDelimiter\abs{\lvert}{\rvert}%
\DeclarePairedDelimiter\norm{\lVert}{\rVert}%

\newcommand{\prob}[1]{\mathbb{P}\left\lbrace{#1}\right\rbrace}
\newcommand{\pr}{\mathbb{P}}
\newcommand{\expect}[1]{\mathbb{E}\left[{#1}\right]}
\newcommand{\expv}{\mathbb{E}}
\newcommand{\expectation}[2]{\mathbb{E}_{#1}\hspace*{-0.02in}\left[{#2}\right]}

\newcommand{\dev}[2]{\mathbb{D}\left({#1},{#2}\right)}

\newcommand{\convinprob}{\xrightarrow{p}}

\newcommand{\D}{\mathcal{D}}
\newcommand{\F}{\mathcal{F}}
\newcommand{\I}{\mathcal{I}}
\newcommand{\J}{\mathcal{J}}

\newcommand{\Q}{\mathcal{Q}}
\newcommand{\R}{\mathbb{R}}

\newcommand{\hS}{\hat{S}}
\newcommand{\X}{\mathcal{X}}
\newcommand{\Y}{\mathcal{Y}}
\newcommand{\Z}{\mathcal{Z}}

\newcommand{\hf}{\hat{f}}
\newcommand{\hg}{\hat{g}}
\newcommand{\hQ}{\hat{Q}}

\newcommand{\heps}{\hat{\varepsilon}}
\newcommand{\teps}{\tilde{\varepsilon}}

\newcommand{\hth}{\hat{\theta}}
\newcommand{\hpi}{\hat{\pi}}
\newcommand{\sth}{\theta^*}
\newcommand{\spi}{\pi^*}
\newcommand{\hv}{\hat{v}}

\newcommand{\bx}{\bar{x}}
\newcommand{\by}{\bar{y}}
\newcommand{\bz}{\bar{z}}
\newcommand{\hz}{\hat{z}}
\newcommand{\tz}{\tilde{z}}

\newcommand{\proj}[2]{\operatorname{proj}_{#1}(#2)}

\providecommand{\keywords}[1]
{
  \small	
  \textbf{Key words:} #1
}

\newtheorem{theorem}{Theorem}[]
\newtheorem{lemma}[theorem]{Lemma}
\newtheorem{assumption}{Assumption}[]
\newtheorem{remark}{Remark}
\newtheorem{definition}{Definition}
\newtheorem{corollary}[theorem]{Corollary}
\newtheorem{proposition}[theorem]{Proposition}
\newtheorem{example}{Example}
\newtheorem{conjecture}[theorem]{Conjecture}

\let\oldtheorem\theorem
\renewcommand{\theorem}{\oldtheorem\normalfont}

\let\oldlemma\lemma
\renewcommand{\lemma}{\oldlemma\normalfont}

\let\oldassumption\assumption
\renewcommand{\assumption}{\oldassumption\normalfont}

\let\oldremark\remark
\renewcommand{\remark}{\oldremark\normalfont}

\let\olddefinition\definition
\renewcommand{\definition}{\olddefinition\normalfont}

\let\oldcorollary\corollary
\renewcommand{\corollary}{\oldcorollary\normalfont}

\let\oldproposition\proposition
\renewcommand{\proposition}{\oldproposition\normalfont}

\let\oldexample\example
\renewcommand{\example}{\oldexample\normalfont}

\let\oldconjecture\conjecture
\renewcommand{\conjecture}{\oldconjecture\normalfont}

\newtheorem{assp}{Assumption}

\newtheorem{assj}{Assumption}
\newenvironment{assjack}[1]
  {\renewcommand{\theassj}{\ref{#1}J}%
   \addtocounter{assj}{-1}%
   \begin{assj}}
  {\end{assj}}
  
\newtheorem{assjp}{Assumption}
\newenvironment{assjackplus}[1]
  {\renewcommand{\theassjp}{\ref{#1}J+}%
   \addtocounter{assjp}{-1}%
   \begin{assjp}}
  {\end{assjp}}
  
\newtheorem{assjcv}{Assumption}

\newtheorem{assjpcv}{Assumption}

\begin{document}

\maketitle

\begin{abstract}
We study optimization for data-driven decision-making when we have observations of the uncertain parameters within the optimization model together with concurrent observations of covariates. 
Given a new covariate observation, the goal is to choose a decision that minimizes the expected cost conditioned on this observation.
We investigate three data-driven frameworks that integrate a machine learning prediction model within a stochastic programming sample average approximation (SAA) for approximating the solution to this problem. 
Two of the SAA frameworks are new and use out-of-sample residuals of leave-one-out prediction models for scenario generation. 
The frameworks we investigate are flexible and accommodate parametric, nonparametric, and semiparametric regression techniques.
We derive conditions on the data generation process, the prediction model, and the stochastic program under which solutions of these data-driven SAAs are consistent and asymptotically optimal, and also derive convergence rates and finite sample guarantees.
Computational experiments validate our theoretical results, demonstrate the potential advantages of our data-driven formulations over existing approaches (even when the prediction model is misspecified), and illustrate the benefits of our new data-driven formulations in the \mbox{limited data regime}. \\[0.1in]
\end{abstract}

\section{Introduction}
\label{sec:intro}

We study data-driven decision-making under uncertainty, where the decision-maker (DM) has access to a finite number of observations of uncertain parameters of an optimization model together with concurrent observations of auxiliary features/covariates.
Stochastic programming~\citep{shapiro2009lectures,birge2011introduction} is a popular modeling framework for decision-making under uncertainty in such applications.
A standard formulation of a stochastic program is
\begin{alignat*}{2}
&\uset{z \in \Z}{\min} \: && \expect{c(z,Y)},
\end{alignat*}
where $z$ denotes the decision variables, $Y$ denotes the uncertain model parameters, $\Z$ denotes the feasible region, $c$ is a cost function, and the expectation is computed with respect to the distribution of~$Y$.
Data-driven solution methods such as sample average approximation (SAA) traditionally assume access to only samples of the random vector~$Y$~\citep{shapiro2009lectures,homem2014monte}.
However, in many real-world applications, values of $Y$ (e.g., demand for water and energy) are predicted using available covariate information (e.g., weather).

Motivated by the developments in~\citet{ban2018big,bertsimas2014predictive}, and~\citet{sen2018learning}, we study the case in which covariate information is available and can be used to inform the distribution of $Y$.
Specifically, given a new random observation $X = x$ of covariates,
the goal of the DM is to solve the conditional stochastic program
\begin{alignat}{2}
\label{eqn:sp}
\uset{z \in \Z}{\min} \: && \expect{c(z,Y) \mid X = x}. \tag{SP}
\end{alignat}
The aim of this paper is to analyze the SAA framework when a prediction model---obtained by statistical or machine learning---is explicitly integrated into the SAA for~\eqref{eqn:sp} to leverage the covariate observation $X = x$.
Here, residuals of the prediction model are {scaled and} added on to a point prediction of~$Y$ at~$X = x$ to construct scenarios of $Y$ for use within the SAA.
We formally define our data-driven approximations to~\eqref{eqn:sp} in Section~\ref{sec:form}.

Applications of this framework include (i) the data-driven newsvendor problem~\citep{ban2018big}, where the product's demand can be predicted using seasonality and location data before making order decisions, (ii) dynamic procurement of a new product~\citep{ban2018dynamic} whose demand can be predicted using historical data for similar past products, (iii) shipment planning under uncertainty~\citep{bertsimas2014predictive}, where historical demands, weather forecasts, and web search results can be used to predict products' demands before making production and inventory decisions, and (iv) grid scheduling under uncertainty~\citep{donti2017task}, where seasonality, weather, and historical demand data can be used to predict the load before creating generator schedules.

Formulation~\eqref{eqn:sp} requires knowledge of the conditional distribution of the random variables given a new realization of the covariates.
Since this distribution is typically unknown, we are interested in using an estimate of it to approximately solve~\eqref{eqn:sp} given access to a finite set of joint observations of $(X,Y)$. 
In this setting, we would like to construct approximations to~\eqref{eqn:sp} that not only have good statistical properties, but are also practically effective in the limited data regime. 
At a minimum, we would like a data-driven approach that is asymptotically optimal in the sense that the objective value of its solutions approaches the optimal value of~\eqref{eqn:sp} as the number of samples increases.
We would also like to determine the rate at which this convergence occurs.

Our first contribution is to generalize and analyze the approach proposed in~\citet{ban2018dynamic} and~\citet{sen2018learning}, in which data-driven approximations to~\eqref{eqn:sp} are constructed using explicit models that predict the random vector~$Y$ using the covariates~$X$.
In this approach, a prediction model is first used to generate a point prediction of~$Y$ at the new observation~$X = x$.
The residuals obtained during the training of the prediction model are then {scaled and} added on to this point prediction to construct scenarios for use within an SAA framework to approximate the solution to~\eqref{eqn:sp}.
We refer to this approximation as the \textit{empirical residuals-based SAA} (ER-SAA). We demonstrate asymptotic optimality, rates of convergence, and finite sample guarantees of solutions obtained from the ER-SAA under mild assumptions.
Inspired by {jackknife}-based methods for constructing prediction intervals~\citep{barber2019predictive}, we also propose two new data-driven SAA frameworks that use \textit{leave-one-out residuals} instead of empirical residuals, and demonstrate how our analysis can be extended to these frameworks.
The motivation for these new data-driven SAA formulations is that using leave-one-out residuals might result in a better approximation of the true conditional distribution of $Y$ given $X = x$, particularly when the sample size is small.

The prediction frameworks we analyze are flexible and accommodate parametric, nonparametric, and semiparametric regression techniques~\citep{van2000asymptotic,gyorfi2006distribution,wainwright2019high}.
While our results imply that using nonparametric regression techniques within our SAA frameworks results in convergent approximations to~\eqref{eqn:sp} under mild assumptions~\cite[cf.][]{bertsimas2014predictive}, the rate at which such approximations converge typically exhibits poor dependence on the dimension of the covariate vector~$X$.
Parametric (and semiparametric) regression approaches, on the other hand, presume some knowledge of the functional dependence of $Y$ on~$X$.
If the assumed functional dependence is a good approximation of the true dependence, they may yield significantly better solutions when the number of samples is limited.
The tradeoff between employing parametric and nonparametric regression techniques within our framework is evident upon looking at the assumptions under which these approaches are guaranteed to yield convergent approximations to~\eqref{eqn:sp}, the rates at which their optimal solutions converge, and numerical experience in Section~\ref{sec:computexp}.
The generality of our framework enables DMs to choose the modeling approach that works best for their application.

\subsection{Relation to existing literature}
\label{subsec:relatedwork}

The papers of \citet{ban2018dynamic} and~\citet{sen2018learning} are most closely related to this work.
Motivated by the application of dynamic procurement of a short-life-cycle product in the presence of demand uncertainty,~\citet{ban2018dynamic} propose a residual tree method for the data-driven solution of multistage stochastic programs (also see references to the operations management literature therein for other data-driven approaches). 
They propose to use ordinary least squares (OLS) or Lasso regression to generate demand forecasts for a new product using historical demand and covariate data for similar products, and establish asymptotic optimality of their data-driven procurement decisions for their particular application.
\citet{sen2018learning} also use predictive models to generate scenarios of random variables in stochastic programs with exogenous and endogenous uncertainty when covariate information is available.
They propose an empirical additive error method that is similar to the residual tree method of~\citet{ban2018dynamic}. 
They also consider estimating distributions of the coefficients and residuals of a linear regression model and propose to subsequently sample from these distributions to generate scenarios of the random variables.
They present model validation and model selection strategies for when the DM has access to several candidate prediction models. 
\citet{kim2015two} use empirical residuals to construct scenarios in a computational study, but conduct no analysis of the approach.

Our work differs from the above in the following respects: we introduce a general framework that applies to a wide range of prediction and optimization models {and allows for the covariance matrix of the errors to depend on the covariates}; we establish asymptotic optimality of the solutions from the ER-SAA under general conditions; we derive results establishing rates of convergence and finite sample guarantees of the solutions from the ER-SAA; we propose two new frameworks that use leave-one-out residuals and extend the asymptotic optimality and rate of convergence analysis to these frameworks; and we present an empirical study demonstrating the potential advantage of using these frameworks.

\citet{bertsimas2014predictive} consider approximating the solution to~\eqref{eqn:sp} by solving a reweighted SAA problem, where the weights are chosen using nonparametric regression methods based on k-nearest neighbors (kNN), kernels, classification and regression trees (CART), or random forests (RF).
They pay particular attention to the setting where the joint observations of $(X,Y)$ may not be i.i.d., but arise from a mixing process.
They also consider the setting where decisions affect the realization of the uncertainty, and establish asymptotic optimality and consistency of their data-driven solutions.
They also consider policy-based empirical risk minimization (ERM) approaches for~\eqref{eqn:sp}, and develop out-of-sample guarantees for costs of decisions constructed using such policies.
\citet{diao2020distribution} develop stochastic quasigradient methods for efficiently solving the kNN and kernel-based reweighted SAA formulations of~\citet{bertsimas2014predictive} without sacrificing theoretical guarantees.
\citet{ban2018big} also propose a policy-based ERM approach and a kernel regression-based nonparametric approach for solving~\eqref{eqn:sp} in the context of the data-driven newsvendor problem.
They derive finite sample guarantees on the out-of-sample costs of order decisions, and quantify the gains from using feature information under different demand models.
{\citet{bazier2020generalization} derive out-of-sample performance guarantees for regularized portfolio selection with side information.}
\citet{bertsimas2019predictions} extend the analysis of~\citet{bertsimas2014predictive} to the multistage setting when the covariates evolve according to a Markov process.
They establish asymptotic optimality and consistency of their data-driven decisions along with finite sample guarantees for the solutions to the kNN-based approach.
{\citet{kallus2022stochastic} propose RF-based decision policies for problem~\eqref{eqn:sp} and demonstrate asymptotic optimality of these policies.}
{Finally, \citet{hu2022fast} show that ``estimate and then optimize'' methods can have faster regret convergence rates than ERM-based approaches for contextual linear optimization when the estimated parameters appear in the objective coefficients.}

Our work differs from the above in the following respects: we propose data-driven approaches to approximate the solution to~\eqref{eqn:sp} that rely on the construction of explicit models to predict the random variables from covariates, allow for both parametric and nonparametric regression models, and derive convergence rates and finite sample guarantees for solutions to our approximations that complement the above analyses.

Another stream of research has been investigating methods that change the training of the prediction model in order to obtain better solutions to~\eqref{eqn:sp}~{\citep[e.g., see][]{donti2017task,elmachtoub2017smart,davarnia2018bayesian,el2019generalization}}.
The philosophy behind these approaches is that, instead of constructing the prediction model purely for high predictive accuracy, the DM should construct a model to predict~$Y$ using~$X$ such that the resulting optimization decisions 
provide the lowest cost solution to the true conditional stochastic program~\eqref{eqn:sp}.
These methods result in harder {joint estimation and optimization} problems that can only be solved to optimality in special settings.
In contrast, we focus on the setting where the prediction framework is independent of the stochastic programming model. This is common in many real-world applications and facilitates easily changing or improving the prediction model.

{Several recent works~\citep[e.g., see][]{bertsimas2019dynamic,dou2019distributionally,esteban2021distributionally,nguyen2020distributionally} use distributionally robust optimization (DRO) in a bid to construct better approximations to~\eqref{eqn:sp} than SAA in the limited data regime.
In follow-up work~\citep{kannan2020residuals}, we study residuals-based DRO formulations that are built around our data-driven SAA formulations, analyze their theoretical guarantees, and illustrate their advantages in the limited data regime through a case study.
Our data-driven SAA formulations in this work are flexible and remain tractable under milder assumptions on~\eqref{eqn:sp} compared to such DRO approaches.}

A `traditional data-driven SAA approach' for the conditional stochastic program~\eqref{eqn:sp} would involve constructing a model to predict the random variables~$Y$ given~$X$, fitting a distribution to the residuals of the prediction model, and using samples from this distribution along with the prediction model to construct scenarios for~$Y$
given $X = x$.
While it is difficult to pin down a reference that is \textit{the first} to adopt this approach, we point to the works of~\citet{schutz2009supply},~\citet{royset2014data}, and the references therein for applications-motivated versions.
{Recent work by~\citet{grigas-ice-21} contributes to this approach by attempting to directly estimate the conditional distribution of $Y$ given $X$ while considering the structure of the optimization
problem (assuming $Y$ has finite support).}
Instead of fitting a distribution to the residuals of the prediction model, we propose and analyze methods that directly use empirical residuals within the SAA framework. These methods avoid the need to fit a  distribution of the residuals, and hence we expect them to be advantageous when the available data is insufficient to provide a good estimate of the residuals distribution.

\subsection{Summary of main contributions}
\label{subsec:summaryofcontrib}

The key contributions of this paper are as follows:
\begin{enumerate}
\itemsep0em
\item We demonstrate asymptotic optimality, rates of convergence, and finite sample guarantees of solutions to the ER-SAA formulation under mild assumptions on the data, the prediction framework, and the stochastic programming formulation.

\item We introduce and analyze two new variants of ER-SAA that use leave-one-out residuals instead of empirical residuals, which may lead to better solutions when data is limited.

\item We verify that the assumptions on the underlying stochastic programming formulation hold for a broad class of two-stage stochastic programs, including two-stage stochastic mixed-integer programming (MIP) with continuous recourse. Additionally, we verify that the assumptions on the prediction step hold for a broad class of M-estimation procedures and nonparametric regression methods, including OLS, Lasso, kNN, and RF regression.

\item Finally, we empirically validate our theoretical results, demonstrate the advantages of our data-driven SAA formulations over existing approaches in the limited data regime, and demonstrate the potential benefit of using a structured prediction model even if it is misspecified.
\end{enumerate}

\section{Data-driven SAA frameworks}
\label{sec:form}

Recall that our goal is to approximate the solution to the conditional stochastic program~\eqref{eqn:sp}:
\begin{alignat*}{2}
\uset{z \in \Z}{\min} \: && \expect{c(z,Y) \mid X = x},
\end{alignat*}
where $X = x$ is a new random observation of the covariates and the expectation is taken with respect to the conditional distribution of $Y$ given $X = x$.
Let $P_X$ and $P_Y$ denote the marginal distributions of the covariates~$X$ and the random vector $Y$, respectively, and $\X \subseteq \R^{d_x}$ and $\Y \subseteq \R^{d_y}$ denote their supports.
{We assume that the support $\Y$ is nonempty and convex and $c : \Z \times \Y \to \R$}.

We assume that the `true relationship' between the random vector~$Y$ and the random covariates~$X$ can be described as $$Y = f^*(X) + {Q^*(X)}\varepsilon,$$ where $f^*(x) := \expect{Y \mid X = x}$ is the regression function, {$Q^*(X)$ is the square root of the conditional covariance matrix of the error term}, and the {zero-mean random errors~$\varepsilon$ are independent of the covariates~$X$. 
Because the the error term $Q^*(X)\varepsilon$ is influenced by the covariate $X$, our model is {\it heteroscedastic}.}
{When $Q^* \equiv I$, as assumed in \citet{ban2018dynamic} and \citet{sen2018learning}, the error distribution is {\it homoscedastic}.}
{Heteroscedasticity arises, for instance, when variability of the random vector $Y$ such as the variability of product
demands or wind power availability depends on the covariates $X$ like location and seasonality.
It can also arise when the DM cannot fully identify all the covariates and the remaining covariates appear in the error term.}

We suppose that $f^*$ {and $Q^*$} belong to known {classes} of functions $\F$ {and $\Q$, respectively}.
The model classes~$\F$ {and~$\Q$ may comprise parametric or nonparametric models}.
Let $\Xi \subseteq \R^{d_y}$ denote the support of~$\varepsilon$ and $P_{\varepsilon}$ denote its distribution.
{We assume $Q^*(\bar{x}) \succ 0$, $\forall \bar{x} \in \X$, where $A \succ 0$ denotes the matrix $A$ is positive definite. We may also assume without loss of generality that the covariance matrix $\expect{\varepsilon \tr{\varepsilon}} = I$ since a general covariance matrix for~$\varepsilon$ can be handled by suitably redefining~$Q^*$.}

Under these structural assumptions, the conditional stochastic program~\eqref{eqn:sp} is equivalent to
\begin{align}
\label{eqn:speq}
v^*(x) &:= \uset{z \in \Z}{\min} \left\lbrace  g(z;x) := \expect{c(z,f^*(x)+{Q^*(x)}\varepsilon)} \right\rbrace,
\end{align}
where the expectation is computed with respect to the distribution $P_{\varepsilon}$ of~$\varepsilon$.
We refer to problem~\eqref{eqn:speq} as the \textit{true problem}, and denote its optimal solution set by $S^*(x)$.
Throughout, we assume that the feasible set $\Z \subset \R^{d_z}$ is nonempty and compact, $\expect{\abs{c(z,f^*(x)+{Q^*(x)}\varepsilon)}} < +\infty$ for each $z \in \Z$ and almost every (a.e.) $x \in \X$, and the function~$g(\cdot;x)$ is lower semicontinuous (lsc) on $\Z$ for a.e.\ $x \in \X$ (see Theorem~7.42 of~\citet{shapiro2009lectures} for conditions that guarantee $g(\cdot;x)$ is lsc).
These assumptions ensure problem~\eqref{eqn:speq} is well defined and the solution set $S^*(x) \neq \emptyset$ for a.e.~$x \in \X$.

Let $\D_n := \{(y^i,x^i)\}_{i=1}^{n}$ denote joint observations of $(Y,X)$.
If the functions~$f^*$ {and $Q^*$ are} known, then the \textit{full-information SAA} {(FI-SAA)} counterpart to the true problem~\eqref{eqn:speq} using data~$\D_n$ is 
\begin{align}
\label{eqn:fullinfsaa}
&\uset{z \in \Z}{\min} \biggl\{ g^*_n(z;x) := \dfrac{1}{n} \displaystyle\sum_{i=1}^{n} c(z,f^*(x)+{Q^*(x)}\varepsilon^i) \biggr\},
\end{align}
where  $\varepsilon^i := {[Q^*(x^i)]^{-1}(y^i - f^*(x^i))}$, $\forall i \in \{1,\dots,n\}$ denote the realizations of the errors~$\varepsilon$ at the given observations.
We cannot solve problems~\eqref{eqn:speq} or~\eqref{eqn:fullinfsaa} directly because the functions~$f^*$ {and $Q^*$ are} unknown.
A practical alternative is to {first} estimate $f^*$ from the data~$\D_n$, for instance by using an M-estimator~\citep{van2000asymptotic,van2000empirical} of the form
\begin{alignat}{2}
\label{eqn:regr}
\hf_n(\cdot) &\in \uset{f(\cdot) \in \F}{\argmin} && \: \dfrac{1}{n}\sum_{i=1}^{n} \ell \left( y^i,f(x^i) \right)
\end{alignat}
with some loss function $\ell: \R^{d_y} \times \R^{d_y} \to \R_+$. 
We sometimes assume that the regression model class~$\F$ is parameterized by~$\theta$ (e.g., the parameters of a linear regression model) and let~$\sth$ denote the true value of~$\theta$ corresponding to the regression function~$f^*$. In this setting, the aim of the regression step~\eqref{eqn:regr} is to estimate~$\sth$, and we denote the estimate corresponding to~$\hf_n$ by~$\hth_n$.

{Given an estimate $\hat{f}_n$ of $f^*$, we use the fact that $\mathbb{E}\bigl[(Y - f^*(X))\tr{(Y - f^*(X))} \mid X = x\bigr] = Q^*(x) \tr{Q^*(x)}$ and plug in $\hf_n$ instead of $f^*$ to determine the best regression estimate $\hat{Q}_n$ of $Q^*$ in the model class~$\Q$~\citep{bauwens2006multivariate}.
We may then update our estimate $\hf_n$ using an estimate $\hat{Q}_n$ of $Q^*$, e.g., using weighted least squares regression~\citep{romano2017resurrecting}, which could yield an improved estimate of $f^*$ with lower variance.
Alternatively, we could estimate $f^*$ and $Q^*$ jointly using M-estimation~\citep{davidian1987variance}.
In the homoscedastic setting, we simply set $\hQ_n := Q^* \equiv I$.}
Throughout, we reference equation~\eqref{eqn:regr} for {both} regression steps ({i.e., for estimating $f^*$ \textit{and} $Q^*$}) with the understanding that our regression setup is not restricted to M-estimation.

Given an estimate {$\hat{f}_n$ of $f^*$ and a nonsingular estimate $\hat{Q}_n$ of $Q^*$}, residuals $\heps^i_{n} := {[\hat{Q}_n(x^i)]^{-1}(y^i - \hat{f}_n(x^i))}$, $i \in \{1,\dots,n\}$, of this estimate can be used as proxy for samples of $\varepsilon$ from~$P_{\varepsilon}$.
{Let $\proj{\Y}{v}$ denote the orthogonal projection of $v \in \R^{d_y}$ onto $\Y$. Then,}
the \textit{empirical residuals-based SAA} (ER-SAA) {corresponding} to problem~\eqref{eqn:speq} is defined as
\begin{align}
\label{eqn:app}
\hv^{ER}_n(x) &:= \uset{z \in \Z}{\min} \biggl\{ \hg^{ER}_n(z;x) := \dfrac{1}{n}\displaystyle\sum_{i=1}^{n} c \bigl( z,{\proj{\Y}{\hf_n(x) + \hat{Q}_n(x)\heps^i_{n}}} \bigr) \biggr\}.
\end{align}
{While the ER-SAA problem~\eqref{eqn:app} is equally tractable and does not lose any theoretical guarantees if the scenarios $\{\hf_n(x) + \hat{Q}_n(x)\heps^i_{n}\}$ are \textit{not} projected onto the support~$\Y$, this projection step may be helpful in situations where (some of) the ER-SAA scenarios $\{\hf_n(x) + \hat{Q}_n(x)\heps^i_{n}\}$ lie outside the support~$\Y$ even though the ``true'' FI-SAA scenarios $\{f^*(x) + Q^*(x)\varepsilon^i\}$ are contained in~$\Y$.}
We let $\hz^{ER}_n(x)$ denote an optimal solution to problem~\eqref{eqn:app}
and~$\hS^{ER}_n(x)$ denote its optimal solution set. 
We assume throughout that the set~$\hS^{ER}_n(x)$ is nonempty for a.e.~$x \in \X$, which holds, for example, if the function $c(\cdot,y)$ is lower semicontinuous on~$\Z$ for each $y \in {\Y}$. 
We stress that problem \eqref{eqn:app} is different from the following \textit{naive SAA} (N-SAA) problem that directly uses the observations~$\{y^i\}_{i=1}^{n}$ of the random vector~$Y$ without using the new {covariate} observation $X=x$:
\begin{alignat}{2}
\label{eqn:nsaa}
{\hv^{\text{NSAA}}_n} &:= \uset{z \in \Z}{\min} \: && \dfrac{1}{n}\displaystyle\sum_{i=1}^{n} c \left( z,y^i \right).
\end{alignat}
The computational complexity of the ER-SAA problem~\eqref{eqn:app} is similar to that of the N-SAA problem~\eqref{eqn:nsaa} with the only additional computation cost being the cost of estimating~$f^*$ {and $Q^*$}.
{Problem~\eqref{eqn:app} also differs from the following point prediction-based deterministic approximation to~\eqref{eqn:speq}}:
\begin{align}
\label{eqn:pointpred}
{\hv^{\text{PP}}_n(x)} &{:= \uset{z \in \Z}{\min} \:  c \bigl( z, \proj{\Y}{\hf_n(x)} \bigr).}
\end{align}
{Problem~\eqref{eqn:app} is a modification of problem~\eqref{eqn:pointpred} that accounts for the uncertainty in the point estimate.}

We also propose two alternatives to the ER-SAA problem~\eqref{eqn:app} that construct scenarios differently. 
{Note that the $n$ observations $\D_n = \{(y^i,x^i)\}_{i=1}^{n}$ that are used to estimate the regression functions $\hf_n$ and  $\hat{Q}_n$ are also used to estimate the errors $\heps^i_{n},  i \in  \{1,\ldots,n\}$. This can cause a bias in the  estimation of the residuals, especially when the sample size $n$ is small, yielding suboptimal solutions for some problems. To alleviate this issue, we propose to use jackknife-based variants of the ER-SAA problem.}
For each $i \in \{1,\ldots,n\}$, let $\hf_{-i}$ {and $\hat{Q}_{-i}$} denote the estimates of $f^*$ {and $Q^*$} obtained by omitting the data point $(y^i,x^i)$ from the training set~$\D_n$ while carrying out the regression step~\eqref{eqn:regr}, and define the residual term $\heps^{i}_{n,J} := {[\hat{Q}_{-i}(x^i)]^{-1} (y^i - \hf_{-i}(x^i))}${, calculated at the omitted point $(y^i,x^i)$}. The alternatives we propose are
\begin{align}
\hv^{J}_n(x) &:= \uset{z \in \Z}{\min} \Biggl\{ \hat{g}^{J}_n(z;x) := \dfrac{1}{n}\displaystyle\sum_{i=1}^{n} c\bigl( z,{\proj{\Y}{\hf_n(x) + \hat{Q}_n(x)\heps^{i}_{n,J}}} \bigr) \Biggr\}, \label{eqn:jackknife} \\
\hv^{J+}_n(x) &:= \uset{z \in \Z}{\min} \Biggl\{ \hat{g}^{J+}_n(z;x) := \dfrac{1}{n}\displaystyle\sum_{i=1}^{n} c \bigl( z,{\proj{\Y}{\hf_{-i}(x) + \hat{Q}_{-i}(x)\heps^{i}_{n,J}}} \bigr) \Biggr\}. \label{eqn:jackknife+}
\end{align}
We call problems~\eqref{eqn:jackknife} and~\eqref{eqn:jackknife+} \textit{{jackknife}-based SAA} (J-SAA) and \textit{{jackknife+}-based SAA} (J+-SAA), respectively~\citep[cf.][]{barber2019predictive}.
These data-driven SAAs are well-motivated when the data $\D_n$ is independent, in which case the leave-one-out residual $\heps^i_{n,J}$ may be a significantly more accurate estimate of the scaled prediction error at the covariate observation $x^i$ than the empirical residual $\heps^i_n$, particularly when $n$ is small {relative to the complexity of the regression step~\eqref{eqn:regr} due to overfitting~\citep{barber2019predictive}}.
When $\D_n$ is not independently generated, omitting blocks of data (instead of individual observations as in the {jackknife}-based methods) during the regression steps~\eqref{eqn:regr} can yield better-motivated variants of the J-SAA and J+-SAA formulations~\citep{lahiri2013resampling}.

Problems~\eqref{eqn:jackknife} and~\eqref{eqn:jackknife+} roughly require the construction of~$n$ regression models, which may be computationally unattractive in some settings. This extra computational burden can be alleviated in some special settings such as OLS regression 
by re-using information from one regression model to the next (see page~13 of~\citet{barber2019predictive} for other regression setups that can re-use information).
We make use of this computational speed-up in our experiments in Section~\ref{sec:computexp}.

We use the following two-stage stochastic linear program (LP) as our running example for problem~\eqref{eqn:speq}. Section~\ref{sec:tssp} in the Appendix includes a discussion of more general forms of problem~\eqref{eqn:speq} that satisfy the assumptions of our framework. {The more general classes of problems discussed in the Appendix subsume the problem class presented in Example 1.}

\begin{example}[Two-stage stochastic LP]
\label{exm:runningexample}
The set~$\Z$ is a nonempty convex polytope and the function~$c(z,Y) := \tr{c}_z z + {V}(z,Y)$, with 
${V}(z,Y) := \min_{v \in \R^{d_v}_+} \Set{ {\tr{c}_v} v}{ Wv = Y - Tz}$.
The quantities~$c_z$, {$c_v$}, $W$, 
and $T$ have commensurate dimensions.
We assume that ${V}(z,y) < +\infty$ 
for each $z \in \Z$ and $y \in \R^{d_y}$,
the matrix~$W$ has full row rank,
and the dual feasible set $\Set{\lambda}{\tr{\lambda} W \leq {\tr{c}_v}}$ is nonempty.
\end{example}

We also use OLS regression {with a structured parametric model for heteroscedasticity} as our running example for the regression step~\eqref{eqn:regr}. 
Section~\ref{sec:regression} in the Appendix includes a detailed discussion of how other prediction models fit within our framework.

\begin{example}[OLS regression]
\label{exm:regrexample}
The model class~$\F := \Set{f(\cdot)}{f(X) = \theta X \text{ for some } \theta \in \R^{d_y \times d_x}}$, with $X_1 \equiv 1$ and loss function is $\ell(y,\hat{y}) := \norm{y - \hat{y}}^2$.
We assume the regression function is $f^*(X) = \sth X$ for some $\sth \in \R^{d_y \times d_x}$, and estimate $\sth$ by $\hth_n \in \argmin_{\theta \in \R^{d_y \times d_x}} \frac{1}{n}\sum_{i=1}^{n} \norm{y^i - \theta x^i}^2$. 
In the homoscedastic setting, the model class is~$\Q := \{Q : \R^{d_x} \to \R^{d_y \times d_y} : Q \equiv I\}$ and we simply set $\hQ_n := Q^* \equiv I$.
In the heteroscedastic setting, following \citet{romano2017resurrecting}, the model class is a set of diagonal matrices \mbox{$\Q := \{Q : \R^{d_x} \to \R^{d_y \times d_y} : Q(X) = \text{diag}(q_1(X), q_2(X), \dots, q_{d_y}(X))\}$}, where $q^2_j(X) := \exp(\sum_{i=1}^{d_x}\pi^j_i \log\abs{X_i})$ for some $\pi^j \in \R^{d_x}$.
We assume that $Q^* \in \Q$ with parameters $\{\pi^{j*}\}_{j=1}^{d_y}$ and $\mathbb{P}_X\{X = 0\} = 0$.
For $j \in \{1,\dots,d_y\}$, we use $\mathbb{E}\bigl[(Y_j - f^*_j(X))^2 \mid X = x\bigr] = q^2_j(x)$ to estimate $\pi^{j*}$ by 
\[
\hat{\pi}^{j}_n \in \uset{\pi^j \in \R^{d_x}}{\argmin} \: \dfrac{1}{n}\displaystyle\sum_{i=1}^{n} \Bigl(\log(\max\{\delta_n,y^i_j - (\hat{\theta}_n x^i)_j\}^2) - \sum_{k=1}^{d_x}\pi^j_k \log\abs{X_k} \Bigr)^2. 
\]
The tolerance sequence $\{\delta_n\} \downarrow 0$ is chosen to avoid ill-conditioning~\citep{romano2017resurrecting}.
\end{example}

There is an inherent tradeoff between using parametric and nonparametric regression techniques for estimating the {functions~$f^*$ and $Q^*$}.
If the function classes~$\F$ {and $\Q$ are} correctly specified, then parametric regression approaches may yield much faster rates of convergence of the data-driven SAA estimators relative to nonparametric approaches (see Section~\ref{subsec:finitesample}).
On the other hand, misspecification of the prediction model can result in our data-driven solutions being asymptotically inconsistent and suboptimal.
Empirical evidence in Section~\ref{sec:computexp} indicates that it \textit{may} still be beneficial to use a misspecified prediction model when we do not have access to an abundance of data.

\begin{remark}
\label{rem:modelclass}
Although we assume that the functions~$f^* \in \F$ and~$Q^* \in \Q$ to establish our theoretical guarantees, our ER-SAA formulation~\eqref{eqn:app} is well defined even when $f^* \not\in \F$ and $Q^* \not\in \Q$, i.e., when the regression models are misspecified.
In this misspecified setting,  the regression estimates satisfy $\hf_n \convinprob \bar{f} \in \F$ and $\hat{Q}_n \convinprob \bar{Q} \in \Q$, where $\convinprob$ denotes convergence in probability, and $\bar{f}$ and $\bar{Q}$ are the best (in terms of prediction error) approximations to $f^*$ and $Q^*$ in $\F$ and $\Q$, respectively, under mild assumptions.
The optimal value and solutions of the ER-SAA formulation then converge in probability to the optimal value and solutions of
\[
\uset{z \in \Z}{\min} \: \dfrac{1}{n} \displaystyle\sum_{i=1}^{n} c\bigl(z,\proj{\Y}{\bar{f}(x)+\bar{Q}(x)\bar{\varepsilon}^i}\bigr),
\]
under mild assumptions, where $\bar{\varepsilon}^i := [\bar{Q}(x^i)]^{-1}(f^*(x^i) - \bar{f}(x^i) + Q^*(x^i)\varepsilon^i)$.
Although the ER-SAA estimators may no longer be consistent, their asymptotic and finite sample properties in this setting may be characterized by replacing $f^*$, $Q^*$, $\{\varepsilon^i\}$ by $\bar{f}$, $\bar{Q}$, and $\{\bar{\varepsilon}^i\}$ in our assumptions and results.
\end{remark}

\paragraph{Notation.} Let~$[n] := \{1,\dots,n\}$, $\abs{S}$ denote the cardinality of a finite set $S$, $\norm{\cdot}$ denote the Euclidean norm {or its induced matrix norm}, $\norm{\cdot}_0$ denote the $\ell_0$ ``norm'', $\mathcal{B}_{\delta}(v)$ denote a Euclidean ball of radius~$\delta > 0$ around a point~$v$, $M_{[j]}$ denote the $j^{\text{th}}$ row of a matrix $M$ and
{$M \succ 0$ denote that it is positive definite.}
For sets $A, B \subseteq \R^{d_z}$, let $\dev{A}{B} := \sup_{v \in A} \text{dist}(v,B)$ denote the deviation of~$A$ from~$B$, where $\text{dist}(v,B) := \inf_{w \in B} \norm{v - w}$.
A random vector $V$ is said to be sub-Gaussian with variance proxy $\sigma^2$ if $\expect{V} = 0$ and $\expect{\exp(s\tr{u} V)} \leq \exp(0.5\sigma^2 s^2)$, $\forall s \in \R$ and $\norm{u} = 1$.
The abbreviations `a.e.', `LLN', and `r.h.s.' are shorthand for `almost everywhere', `law of large numbers', and `right-hand side'. 
By `a.e.~$X$' and `a.e.~$Y$', we mean {$P_X$-a.e.~$x \in \X$} and $P_{Y}$-a.e.~$Y$.
Throughout, `a.s.' is written to mean almost surely with respect to the probability measure by which the data~$\{(x^i,\varepsilon^i)\}$ is generated.
The symbols~$\xrightarrow{p}$,~$\xrightarrow{a.s.}$, and~$\xrightarrow{d}$ are used to denote convergence in probability, almost surely, and in distribution with respect to this probability measure.
For sequences of random variables~$\{V_n\}$ and~$\{W_n\}$, $V_n = o_p(W_n)$ and $V_n = O_p(W_n)$  convey that~$V_n = R_n W_n$ with $\{R_n\}$ converging in probability to zero ($R_n \convinprob 0$), or being bounded in probability, respectively (see Chapter~2 of~\citet{van2000asymptotic} for basic theory). 
We write 
$O(1)$ to denote generic constants.
We assume throughout this work that all functions, sets and selections are measurable (see~\citet{vaart1996weak} and~\citet{shapiro2009lectures} for detailed consideration of these issues).

\section{Analysis of the empirical residuals-based SAA}
\label{sec:ersaa}

We first analyze the theoretical properties of solutions to the ER-SAA problem~\eqref{eqn:app}.
{In particular, after establishing some preliminary results that are useful for the remaining analysis (Section \ref{subsec:prelim}) we investigate conditions under which solutions to problem~\eqref{eqn:app} are asymptotically optimal and consistent (Section \ref{subsec:consistency})
and develop finite sample guarantees for solutions to problem~\eqref{eqn:app} using large deviations theory (Section \ref{subsec:finitesample}).
We outline the modifications required to analyze the J-SAA and J+-SAA methods in Section~\ref{subsec:jackoutline}. 
Omitted proofs are provided in Appendix~\ref{sec:proofs}.
We also present a number of complementary results in the Appendix. 
In Section~\ref{sec:alt_assumptions} of the Appendix, we briefly discuss alternative assumptions under which our theoretical guarantees hold.}
{We analyze the rate of convergence of the optimal value of problem~\eqref{eqn:app} to that of problem~\eqref{eqn:speq} in Section~\ref{sec:ersaa_rate}. In
 Section~\ref{sec:jackknife}, we provide more details about the adaptation of our analysis to the J-SAA and J+-SAA methods.
Finally, in  Sections~\ref{sec:tssp} and~\ref{sec:regression}, we verify that a variety of stochastic optimization and {regression} setups satisfy the assumptions made in our analysis. }

\subsection{Preliminary results}
\label{subsec:prelim}

{The difference between the objective functions of the ER-SAA problem~\eqref{eqn:app} and the true problem~\eqref{eqn:speq} can be bounded uniformly over the decision variables~$z \in \Z$ as follows:
\begin{align}
\label{eqn:ersaa_error_decomp}
\uset{z \in \Z}{\sup}\: \abs*{\hg^{ER}_n(z;x) - g(z;x)} &\leq \uset{z \in \Z}{\sup} \: \abs*{\hg^{ER}_n(z;x) - g^*_n(z;x)} + \uset{z \in \Z}{\sup} \: \abs*{g^*_n(z;x) - g(z;x)}.
\end{align}
The second term on the r.h.s.\ of inequality~\eqref{eqn:ersaa_error_decomp} corresponds to the maximum deviation between the FI-SAA objective function~\eqref{eqn:fullinfsaa} and the objective function of the true problem on~$\Z$.
This term converges to zero in probability whenever a uniform weak LLN result holds~\citep{shapiro2009lectures}.
The first term on the r.h.s.\ of inequality~\eqref{eqn:ersaa_error_decomp} corresponds to the maximum deviation between the ER-SAA and FI-SAA objective functions on~$\Z$.
Let $\teps^i_{n}(x)$ denote the difference between the $i$th ER-SAA scenario $(\hf_n(x) + {\hat{Q}_n(x)}\heps^i_{n})$
and the {$i$th FI-SAA scenario} $(f^*(x) + {Q^*(x)}\varepsilon^i)$, i.e.,
\begin{align*}
\teps^i_{n}(x) :=& \bigl(\hf_n(x) + {\hat{Q}_n(x)}\heps^i_{n}\bigr) - \left( f^*(x) + {Q^*(x)}\varepsilon^i \right), \quad \forall i \in [n].
\end{align*}
Then, the first term on the r.h.s.\ of inequality~\eqref{eqn:ersaa_error_decomp} can be bounded using the \textit{deviation sequence} $\{\teps^i_{n}(x)\}_{i=1}^{n}$ whenever the following Lipschitz assumption holds.}

\begin{assumption}
\label{ass:equilipschitz}
For each $z \in \Z$, the function~$c$ in problem~\eqref{eqn:speq} satisfies the Lipschitz condition
\[
\abs*{c(z,\bar{y}) - c(z,y)} \leq L(z) \norm{\bar{y} - y}, \quad \forall y,\bar{y} \in \Y,
\]
with Lipschitz constant $L$ satisfying $\uset{z \in \Z}{\sup}\: L(z) < +\infty$.
\end{assumption}

Assumption~\ref{ass:equilipschitz} requires the function~$c(z,\cdot)$ to be Lipschitz continuous {on $\Y$} for each $z \in \Z$. 
We show in Section~\ref{sec:tssp} in the Appendix that it is satisfied by Example~\ref{exm:runningexample}. 
{When Assumption~\ref{ass:equilipschitz} holds, the first term on the r.h.s.\ of inequality~\eqref{eqn:ersaa_error_decomp} can be bounded as follows.}

\begin{lemma}
\label{lem:ersaa_error_decomp_term1}
{Suppose Assumption~\ref{ass:equilipschitz} holds. Then}
\begin{align*}
{\uset{z \in \Z}{\sup} \: \abs*{\hg^{ER}_n(z;x) - g^*_n(z;x)}} &{\leq \Bigl(\uset{z \in \Z}{\sup} \: L(z)\Bigr) \biggl(\dfrac{1}{n}\displaystyle\sum_{i=1}^{n} \norm{\teps^i_{n}(x)}\biggr).}
\end{align*}
\end{lemma}

{Therefore, when Assumption~\ref{ass:equilipschitz} holds, convergence of the mean deviation term $\frac{1}{n}\sum_{i=1}^{n} \norm{\teps^i_{n}(x)}$ to zero in probability readily translates to convergence to zero in probability of the first term on the r.h.s.\ of inequality~\eqref{eqn:ersaa_error_decomp}.
Next, we focus on bounding the mean deviation term $\frac{1}{n}\sum_{i=1}^{n} \norm{\teps^i_{n}(x)}$ using the arguments in Section~3.1 of~\citet{kannan2021heteroscedasticity}.}

\begin{lemma}
\label{lem:meandeviation}
{Given regression estimates $\hf_n$ of $f^*$ and $\hQ_n$ of $Q^*$ with $[\hQ_n(\bar{x})]^{-1} \succ 0$ for each $\bar{x} \in \X$:}
\begin{enumerate}
\item {In the homoscedastic setting (i.e., $\hQ_n := Q^* \equiv I$), we have}
\[
{\frac{1}{n} \sum_{i=1}^{n} \norm{\teps^i_n(x)} \leq \norm{\hf_n(x) - f^*(x)} + \frac{1}{n} \sum_{i=1}^{n} \norm{\hf_n(x^i) - f^*(x^i)}}.
\]

\item {In the heteroscedastic setting, we have}
\begin{align}
\label{eqn:meandeviation}
\hspace*{-0.2in}{\frac{1}{n} \sum_{i=1}^{n} \norm{\teps^i_n(x)}} {\leq} &{\norm{\hf_n(x) - f^*(x)} + \norm{\hQ_n(x) - Q^*(x)}\biggl(\frac{1}{n} \sum_{i=1}^{n} \norm{\varepsilon^i}\biggr) +} \nonumber \\
& { \norm{\hQ_n(x)} \biggl(\frac{1}{n} \sum_{i=1}^{n} \bigl\lVert \bigl[\hQ_n(x^i)\bigr]^{-1} - \bigl[Q^*(x^i)\bigr]^{-1}\bigr\rVert^2\biggr)^{1/2} \biggl(\frac{1}{n} \sum_{i=1}^{n} \norm{Q^*(x^i)}^4\biggr)^{1/4} \biggl(\frac{1}{n} \sum_{i=1}^{n} \norm{\varepsilon^i}^4\biggr)^{1/4} +} \nonumber\\ 
& {\norm{\hQ_n(x)} \biggl(\frac{1}{n} \sum_{i=1}^{n} \bigl\lVert \bigl[\hQ_n(x^i)\bigr]^{-1}\bigr\rVert^2\biggr)^{1/2} \biggl(\frac{1}{n} \sum_{i=1}^{n} \norm{f^*(x^i) - \hf_n(x^i)}^2\biggr)^{1/2}.}
\end{align}
\end{enumerate}
\end{lemma}
\begin{proof}
{We begin by noting that}
\begin{align}
\label{eqn:meandeviation_int1}
{\frac{1}{n} \sum_{i=1}^{n} \norm{\teps^i_n(x)}} &{:= \frac{1}{n} \sum_{i=1}^{n} \norm{(\hf_n(x) + \hQ_n(x)\heps^i_{n}) - (f^*(x) + Q^*(x)\varepsilon^i)}} \nonumber\\
&{\leq \norm{\hf_n(x) - f^*(x)} + \frac{1}{n} \sum_{i=1}^{n} \norm{\hQ_n(x)\heps^i_{n} - Q^*(x)\varepsilon^i}.}
\end{align}
{In the homoscedastic case, we can bound the r.h.s.\ of inequality~\eqref{eqn:meandeviation_int1} further as}
\[
{\frac{1}{n} \sum_{i=1}^{n} \norm{\teps^i_n(x)} \leq \norm{\hf_n(x) - f^*(x)} + \frac{1}{n} \sum_{i=1}^{n} \norm{\heps^i_{n} - \varepsilon^i} = \norm{\hf_n(x) - f^*(x)} + \frac{1}{n} \sum_{i=1}^{n} \norm{\hf_n(x^i) - f^*(x^i)}.}
\]
{We now focus on the heteroscedastic setting by bounding the second term on the r.h.s.\ of~\eqref{eqn:meandeviation_int1}.}
\begin{align}
\label{eqn:meandeviation_int2}
&{\frac{1}{n} \sum_{i=1}^{n} \norm{\hQ_n(x)\heps^i_{n} - Q^*(x)\varepsilon^i}} \nonumber\\
{=}& {\frac{1}{n} \sum_{i=1}^{n} \bigl\lVert\hQ_n(x)\bigl[\hQ_n(x^i)\bigr]^{-1} (y^i - \hf_n(x^i)) - Q^*(x)\varepsilon^i\bigr\rVert} \nonumber\\
{=}& {\frac{1}{n} \sum_{i=1}^{n} \bigl\lVert\hQ_n(x)\bigl[\hQ_n(x^i)\bigr]^{-1} \bigl(f^*(x^i) - \hf_n(x^i) + Q^*(x^i)\varepsilon^i\bigr) - Q^*(x)\varepsilon^i\rVert} \nonumber\\
{\leq}& {\frac{1}{n} \sum_{i=1}^{n} \bigl\lVert \hQ_n(x)\bigl[\hQ_n(x^i)\bigr]^{-1} (f^*(x^i) - \hf_n(x^i))\bigr\rVert + \frac{1}{n} \sum_{i=1}^{n} \bigl\lVert \bigl(\hQ_n(x)\bigl[\hQ_n(x^i)\bigr]^{-1} Q^*(x^i) - Q^*(x) \bigr) \varepsilon^i\bigr\rVert}.
\end{align}
{We have for each $i \in [n]$}
\begin{align*}
{\hQ_n(x)\bigl[\hQ_n(x^i)\bigr]^{-1} Q^*(x^i) - Q^*(x)} &{= \bigl(\hQ_n(x)\bigl[\hQ_n(x^i)\bigr]^{-1} - Q^*(x)\bigl[Q^*(x^i)\bigr]^{-1}\bigr) Q^*(x^i)} \\
&{= \hQ_n(x) \bigl(\bigl[\hQ_n(x^i)\bigr]^{-1} - \bigl[Q^*(x^i)\bigr]^{-1}\bigr) Q^*(x^i) + [\hQ_n(x) - Q^*(x)].}
\end{align*}
{Plugging the above equality into inequality~\eqref{eqn:meandeviation_int2}, we get}
\begin{align}
\label{eqn:meandeviation_int3}
&{\frac{1}{n} \sum_{i=1}^{n} \norm{\hQ_n(x)\heps^i_{n} - Q^*(x)\varepsilon^i}} \nonumber\\
{\leq}& {\frac{1}{n} \sum_{i=1}^{n} \Bigl\lVert \Bigl(\hQ_n(x) \bigl(\bigl[\hQ_n(x^i)\bigr]^{-1} - \bigl[Q^*(x^i)\bigr]^{-1}\bigr) Q^*(x^i) + \bigl(\hQ_n(x) - Q^*(x)\bigr) \Bigr) \varepsilon^i\Bigr\rVert +}  \nonumber\\
& \qquad {\frac{1}{n} \sum_{i=1}^{n} \bigl\lVert \hQ_n(x)\bigl[\hQ_n(x^i)\bigr]^{-1} (f^*(x^i) - \hf_n(x^i))\bigr\rVert} \nonumber \\
{\leq}& {\frac{1}{n} \sum_{i=1}^{n} \bigl\lVert \hQ_n(x) \bigl(\bigl[\hQ_n(x^i)\bigr]^{-1} - \bigl[Q^*(x^i)\bigr]^{-1}\bigr) Q^*(x^i) \varepsilon^i \bigr\rVert + \frac{1}{n} \sum_{i=1}^{n} \bigl\lVert \bigl(\hQ_n(x) - Q^*(x)\bigr) \varepsilon^i\bigr\rVert +}  \nonumber\\
& \qquad {\frac{1}{n} \sum_{i=1}^{n} \bigl\lVert \hQ_n(x)\bigl[\hQ_n(x^i)\bigr]^{-1} (f^*(x^i) - \hf_n(x^i))\bigr\rVert} \nonumber \\
{\leq}& { \norm{\hQ_n(x)} \biggl(\frac{1}{n} \sum_{i=1}^{n} \bigl\lVert \bigl[\hQ_n(x^i)\bigr]^{-1} - \bigl[Q^*(x^i)\bigr]^{-1}\bigr\rVert^2\biggr)^{1/2} \biggl(\frac{1}{n} \sum_{i=1}^{n} \norm{Q^*(x^i)}^4\biggr)^{1/4} \biggl(\frac{1}{n} \sum_{i=1}^{n} \norm{\varepsilon^i}^4\biggr)^{1/4} + } \\ 
& {  \norm{\hQ_n(x) - Q^*(x)} \biggl(\frac{1}{n}\sum_{i=1}^{n} \norm{\varepsilon^i}\biggr) + \norm{\hQ_n(x)} \biggl(\frac{1}{n} \sum_{i=1}^{n} \bigl\lVert \bigl[\hQ_n(x^i)\bigr]^{-1}\bigr\rVert^2\biggr)^{1/2} \biggl(\frac{1}{n} \sum_{i=1}^{n} \norm{f^*(x^i) - \hf_n(x^i)}^2\biggr)^{1/2},} \nonumber
\end{align}
{where the last step above follows by the Cauchy-Schwarz inequality.
Finally, using inequality~\eqref{eqn:meandeviation_int3} in inequality~\eqref{eqn:meandeviation_int1}, we get the stated result.}
\end{proof}

{In the homoscedastic setting, the bound on the mean deviation term $\frac{1}{n}\sum_{i=1}^{n} \norm{\teps^i_{n}(x)}$ can be interpreted as the sum of the prediction error $\norm{\hf_n(x) - f^*(x)}$ at the new covariate realization $x \in \X$ and the average estimation error $\frac{1}{n} \sum_{i=1}^{n} \norm{f^*(x^i) - \hf_n(x^i)}$ at the training data points $\{x^i\}$.}

{We remark that Section~\ref{sec:alt_assumptions} of the Appendix presents alternative assumptions
under which the theoretical guarantees studied in the paper continue to hold. These alternative assumptions relax the uniform Lipschitz continuity of Assumption~\ref{ass:equilipschitz} to a weaker {\it local} Lipschitz continuity condition but require more stringent conditions on the regression step. Conditions and examples under which these alternative assumptions hold are discussed in the Appendix.}

\subsection{Consistency and asymptotic optimality}
\label{subsec:consistency}

In this section, we investigate conditions under which the optimal value and optimal solutions to the ER-SAA problem~\eqref{eqn:app} asymptotically converge to those of the true problem~\eqref{eqn:speq}.
We begin by making the following assumption on the uniform convergence of the sequence of objective functions of the FI-SAA problem~\eqref{eqn:fullinfsaa} to the objective function of the true problem~\eqref{eqn:speq} on~$\Z$.

\begin{assumption}
\label{ass:uniflln}
For a.e.~$x \in \X$, the sequence of sample average functions $\left\lbrace g^*_n(\cdot;x) \right\rbrace$ defined in~\eqref{eqn:fullinfsaa} converges in probability to the true function $g(\cdot;x)$ defined in~\eqref{eqn:speq} uniformly on the set~$\Z$.
\end{assumption}

Assumption~\ref{ass:uniflln} is a uniform weak LLN result that is guaranteed to hold if $c(\cdot,y)$ is continuous for a.e.~$y \in \Y$, $c(\cdot,y)$ is dominated by an integrable function for a.e.~$y \in \Y$, and the errors $\{\varepsilon^i\}$ are i.i.d.~\citep[see Theorem~7.48 of][]{shapiro2009lectures}.
Using pointwise LLN results in~\citet{walk2010strong} and~\citet{white2014asymptotic}, we can show that Assumption \ref{ass:uniflln} also holds for some mixing/stationary processes by noting the proof of Theorem~7.48 of~\citet{shapiro2009lectures} also extends to these settings.
{Proposition~\ref{prop:tssp-checkass} in the Appendix shows that Assumption~\ref{ass:uniflln} holds for our running example (Example~\ref{exm:runningexample}) of two-stage stochastic LP whenever $\expect{\norm{Y}} < +\infty$ and the errors $\{\varepsilon^i\}$ are i.i.d.}

{Next, we need the following weak LLN assumptions on the function $Q^*$ and the errors~$\varepsilon$.
These assumptions hold, for instance, when the samples $\{(x^i,\varepsilon^i)\}$ are i.i.d.\ and the quantities $\mathbb{E}[\norm{Q^*(X)}^4]$, $\mathbb{E}[\lVert [Q^*(X)]^{-1} \rVert^2]$, and $\mathbb{E}[\norm{\varepsilon}^4]$ are finite.
In particular, Assumption~\ref{ass:varweaklln} holds for Example~\ref{exm:regrexample} if $\{x^i\}$ is i.i.d.\ and $\mathbb{E}\bigl[\bigl(\exp(\sum_{k=1}^{d_x} \pi^{j*}_k \log(\abs{X_k}))\bigr)^2\bigr] < +\infty$ and $\mathbb{E}\bigl[\exp(-\sum_{k=1}^{d_x} \pi^{j*}_k \log(\abs{X_k}))\bigr] < +\infty$ for each $j \in \{1,\dots,d_y\}$.
Assumptions~\ref{ass:varweaklln} and~\ref{ass:errorsweaklln} also hold for non-i.i.d.\ data arising from mixing/stationary processes that satisfy suitable assumptions (see the discussion above).}

\begin{assumption}
\label{ass:varweaklln}
{The function~$Q^*$ and the covariate samples~$\{x^i\}_{i=1}^{n}$ satisfy the weak LLNs}
\[
{\frac{1}{n} \sum_{i=1}^{n} \norm{Q^*(x^i)}^4 \convinprob \mathbb{E}[\norm{Q^*(X)}^4] \quad \text{and} \quad \frac{1}{n} \sum_{i=1}^{n} \bigl\lVert \bigl[Q^*(x^i)\bigr]^{-1}\bigr\rVert^2 \convinprob \mathbb{E}\bigl[ \bigl\lVert \bigl[Q^*(X)\bigr]^{-1} \bigr\rVert^2 \bigr].}
\]
\end{assumption}

\begin{assumption}
\label{ass:errorsweaklln}
{The error samples~$\{\varepsilon^i\}_{i=1}^{n}$ satisfy the weak LLN $\dfrac{1}{n} \displaystyle\sum_{i=1}^{n} \norm{\varepsilon^i}^4 \convinprob \mathbb{E}[\norm{\varepsilon}^4]$.}
\end{assumption}

Finally, we need the assumption below on the consistency of the regression {estimates $\hf_n$ and~$\hQ_n$}.

\begin{assumption}
\label{ass:regconsist}
\setcounter{mycounter}{\value{assumption}}
The {regression estimates $\hf_n$ and $\hQ_n$ possess} the following consistency properties:
\vspace*{-0.1in}
\begin{multicols}{2}
\begin{enumerate}[label=(\themycounter\alph*),itemsep=0em]
\item \label{ass:regconsist_point} $\hf_n(x) \xrightarrow{p} f^*(x)$ for a.e.~$x \in \X$,

\item \label{ass:regconsist_mse} $\dfrac{1}{n} \displaystyle\sum_{i=1}^{n} \norm{f^*(x^i) - \hf_n(x^i)}^2 \xrightarrow{p} 0$,

\columnbreak

\item \label{ass:regconsist_point2} {$\hQ_n(x) \xrightarrow{p} Q^*(x)$} for a.e.~$x \in \X$,

\item \label{ass:regconsist_mse2} {$\dfrac{1}{n} \displaystyle\sum_{i=1}^{n} \bigl\lVert \bigl[\hQ_n(x^i)\bigr]^{-1} - \bigl[Q^*(x^i)\bigr]^{-1}\bigr\rVert^2 \xrightarrow{p} 0$.}
\end{enumerate}%
\end{multicols}%
\end{assumption}

Assumption~\ref{ass:regconsist_point} holds for our running example of OLS regression (Example \ref{exm:regrexample}) if the parameter estimate $\hth_n$ is weakly consistent (i.e., $\hth_n \convinprob \sth$), and Assumption~\ref{ass:regconsist_mse} holds if, in addition, the weak LLN $\frac{1}{n}\sum_{i=1}^{n} \norm{x^i}^2 \convinprob \expect{\norm{X}^2}$ is satisfied (see Chapter~3 of~\citet{white2014asymptotic} for various assumptions on the data~$\D_n$ and the distributions~$P_X$ and~$P_{\varepsilon}$ under which these conditions hold).
The quantity \mbox{$\frac{1}{n} \sum_{i=1}^{n} \norm{f^*(x^i) - \hf_n(x^i)}^2$} is called the empirical $L^2$ semi-norm in the empirical process theory literature~\citep{van2000empirical}.
{Assumption~\ref{ass:regconsist_point2} holds for our running example of OLS regression with structured heteroscedasticity if the parameter estimate $\hat{\pi}_n$ is weakly consistent (i.e., $\hat{\pi}_n \convinprob \pi^*$, see Appendix~B.2 of~\citet{romano2017resurrecting} for assumptions under which this holds).
Assumption~\ref{ass:regconsist_mse2} holds for our running example if, e.g., we additionally have the support~$\X$ to be compact and bounded away from the origin (i.e., for each $x \in \X$, $b \leq \norm{x} \leq B$ for constants $b, B > 0$) and assume that the estimates~$\{\hat{\pi}_n\}$ lie in a compact set a.s.\ for $n$ large enough.}
Assumption~\ref{ass:regconsist} is implied by the stronger assumption of uniform convergence of the estimates $\hf_n$ and $\hQ_n$ to the functions $f^*$ and $Q^*$, respectively, on the support~$\X$ of the covariates, i.e., when
$\sup_{x \in \X} \norm{f^*(x) - \hf_n(x)} \convinprob 0$, {$\sup_{x \in \X} \norm{Q^*(x) - \hQ_n(x)} \convinprob 0$, and $\sup_{x \in \X} \norm{[Q^*(x)]^{-1} - [\hQ_n(x)]^{-1}} \convinprob 0$}.
Section~\ref{sec:regression} in the Appendix expands on the above arguments and shows that Assumption~\ref{ass:regconsist} also holds when $f^*$ is estimated using Lasso, kNN, and RF regression under certain conditions.

{The following result will prove useful in our analysis of ER-SAA in the heteroscedastic setting.}

\begin{lemma}
\label{lem:varbound}
{We have}
\[
{\biggl(\frac{1}{n} \sum_{i=1}^{n} \bigl\lVert \bigl[\hQ_n(x^i)\bigr]^{-1}\bigr\rVert^2\biggr)^{1/2} \leq \biggl(\frac{1}{n} \sum_{i=1}^{n} \bigl\lVert \bigl[\hQ_n(x^i)\bigr]^{-1} - \bigl[Q^*(x^i)\bigr]^{-1}\bigr\rVert^2\biggr)^{1/2} + \biggl(\frac{1}{n} \sum_{i=1}^{n} \bigl\lVert \bigl[Q^*(x^i)\bigr]^{-1}\bigr\rVert^2\biggr)^{1/2}.}
\]
\end{lemma}
\begin{proof}
{The triangle inequality for the operator norm implies 
\[
\bigl\lVert \bigl[\hQ_n(x^i)\bigr]^{-1}\bigr\rVert \leq \bigl\lVert \bigl[\hQ_n(x^i)\bigr]^{-1} - \bigl[Q^*(x^i)\bigr]^{-1}\bigr\rVert + \bigl\lVert \bigl[Q^*(x^i)\bigr]^{-1}\bigr\rVert, \quad \forall i \in [n].
\]
Therefore, the following component-wise inequality holds:
\[
0 \leq \begin{pmatrix} \bigl\lVert \bigl[\hQ_n(x^1)\bigr]^{-1}\bigr\rVert \\ \vdots \\ \bigl\lVert \bigl[\hQ_n(x^n)\bigr]^{-1}\bigr\rVert \end{pmatrix} \leq \begin{pmatrix} \bigl\lVert \bigl[\hQ_n(x^1)\bigr]^{-1} - \bigl[Q^*(x^1)\bigr]^{-1}\bigr\rVert \\ \vdots \\ \bigl\lVert \bigl[\hQ_n(x^n)\bigr]^{-1} - \bigl[Q^*(x^n)\bigr]^{-1}\bigr\rVert \end{pmatrix} + \begin{pmatrix} \bigl\lVert \bigl[Q^*(x^1)\bigr]^{-1}\bigr\rVert \\ \vdots \\ \bigl\lVert \bigl[Q^*(x^n)\bigr]^{-1}\bigr\rVert \end{pmatrix}.
\]
The stated result then follows as a consequence of the triangle inequality for the $\ell_2$-norm.
}
\end{proof}

Our next result {uses Lemmas~\ref{lem:ersaa_error_decomp_term1} and~\ref{lem:meandeviation}} to establish conditions under which the sequence of objective functions of the ER-SAA problem~\eqref{eqn:app} converges uniformly to the objective function of the true problem~\eqref{eqn:speq} on the feasible region~$\Z$.

\begin{proposition}
\label{prop:uniformconvofobj}
Suppose {Assumptions~\ref{ass:equilipschitz} through~\ref{ass:regconsist}} hold.
Then, for a.e.~$x \in \X$, the sequence of objective functions of the ER-SAA problem~\eqref{eqn:app} converges in probability to the objective function of the true problem~\eqref{eqn:speq} uniformly on the feasible region~$\Z$.
\end{proposition}
\begin{proof}
We wish to show that
$\uset{z \in \Z}{\sup} \: \abs*{\hg^{ER}_n(z;x) - g(z;x)} \xrightarrow{p} 0$ for a.e.~$x \in \X$. Equation~\eqref{eqn:ersaa_error_decomp} yields
{
\begin{align*}
\abs*{\hg^{ER}_n(z;x) - g(z;x)} &\leq {\uset{z \in \Z}{\sup} \: \abs*{\hg^{ER}_n(z;x) - g^*_n(z;x)}} + \uset{z \in \Z}{\sup} \: \abs*{g^*_n(z;x) - g(z;x)}.
\end{align*}
}%
The second term on the r.h.s.\ of the above inequality vanishes in the limit in probability under Assumption~\ref{ass:uniflln}.
If the first term also vanishes in the limit in probability, by $o_p(1) + o_p(1) = o_p(1)$, we obtain the desired result. 
We now show that the first term vanishes in the limit in probability.

{Since Assumption~\ref{ass:equilipschitz} holds, Lemma~\ref{lem:ersaa_error_decomp_term1} implies the first term converges to zero in probability whenever the mean deviation term $\frac{1}{n}\sum_{i=1}^{n} \norm{\teps^i_{n}(x)} \xrightarrow{p} 0$.
Therefore, the desired result holds if each term on the r.h.s.\ of inequality~\eqref{eqn:meandeviation} converges to zero in probability.
The first term on the r.h.s.\ of~\eqref{eqn:meandeviation} converges to zero in probability by Assumption~\ref{ass:regconsist_point}, and the second term converges to zero in probability by Assumptions~\ref{ass:errorsweaklln} and~\ref{ass:regconsist_point2}.
The third term on the r.h.s.\ of~\eqref{eqn:meandeviation} converges to zero in probability by Assumptions~\ref{ass:varweaklln},~\ref{ass:errorsweaklln},~\ref{ass:regconsist_point2}, and~\ref{ass:regconsist_mse2}.
Finally, the last term on the r.h.s.\ of inequality~\eqref{eqn:meandeviation} converges to zero in probability by Assumptions~\ref{ass:varweaklln},~\ref{ass:regconsist_mse},~\ref{ass:regconsist_point2}, and~\ref{ass:regconsist_mse2} and Lemma~\ref{lem:varbound}.
}
\end{proof}

{Fewer assumptions are needed to establish Proposition~\ref{prop:uniformconvofobj} in the homoscedastic case. In that setting, Assumptions~\ref{ass:varweaklln},~\ref{ass:errorsweaklln},~\ref{ass:regconsist_point2}, and~\ref{ass:regconsist_mse2} are not required and Assumption~\ref{ass:regconsist_mse} may be weakened to the assumption $\frac{1}{n} \sum_{i=1}^{n} \norm{f^*(x^i) - \hf_n(x^i)} \xrightarrow{p} 0$ on account of Lemma~\ref{lem:meandeviation}.}
Proposition~\ref{prop:uniformconvofobj} provides the foundation for the following result, which demonstrates that the optimal value and solutions of the ER-SAA problem~\eqref{eqn:app} converge to those of the true problem~\eqref{eqn:speq}.

\begin{restatable}{theorem}{thmapproxconv}
\label{thm:approxconv}
Suppose {Assumptions~\ref{ass:equilipschitz} to~\ref{ass:regconsist}} hold.
Then, we have $\hv^{ER}_n(x) \xrightarrow{p} v^*(x)$, $\mathbb{D}\bigl(\hS^{ER}_n(x),S^*(x)\bigr) \xrightarrow{p} 0$, and $\sup_{z \in \hS^{ER}_n(x)} g(z;x) \xrightarrow{p} v^*(x)$ for a.e.~$x \in \X$.
\end{restatable}

The proof of Theorem~\ref{thm:approxconv} follows a similar outline as the proof of Theorem~5.3 of~\citet{shapiro2009lectures}, {except that we} consider convergence in probability rather than almost sure convergence. 
Under an \textit{inf-compactness} condition on the ER-SAA problem~\eqref{eqn:app}, the conclusions of Theorem~\ref{thm:approxconv} hold even if the set~$\Z$ is unbounded~\citep[see the discussion following Theorem~5.3 of][]{shapiro2009lectures}.
While we consider convergence in probability instead of almost sure convergence (because the statistics literature is typically concerned with conditions under which Assumption~\ref{ass:regconsist} holds rather than its almost sure counterpart), note that our results until this point can be naturally extended to the latter setting by suitably strengthening Assumptions~\ref{ass:uniflln} to~\ref{ass:regconsist}.

{Next, we identify conditions under which the optimal value of the ER-SAA problem~\eqref{eqn:app} converges to the optimal value of the true problem~\eqref{eqn:speq} on average over the space of the covariates.
This can be important in situations where the new covariate observations are random and the DM needs to make decisions facing different observations of the covariates.
Given $q \in [1,+\infty]$, we write $\norm{F}_{L^q}$ to denote the $L^q$-norm of a measurable function $F : \X \to \R^{d_F}$, i.e.,
$\norm{F}_{L^q} := \bigl(\int_S \norm{F}^q dP_X \bigr)^{1/q}$.
We require the following adaptation of Assumptions~\ref{ass:uniflln} and~\ref{ass:regconsist}.}

\begin{assumption}
\label{ass:uniflln_lq}
{The sequence of sample average functions $\left\lbrace g^*_n(\cdot;x) \right\rbrace$ defined in~\eqref{eqn:fullinfsaa} satisfies}
\[
{\Big\lVert\sup_{z \in \Z} \abs*{g^*_{n}(z;X) - g(z;X)} \Big\rVert_{L^q} \convinprob 0.}
\]
\end{assumption}

\begin{assumption}
\label{ass:regconsist_lq}
\setcounter{mycounter}{\value{assumption}}
{The regression estimates $\hf_n$ and $\hQ_n$ possess the following consistency properties:}
\vspace*{-0.1in}
\begin{multicols}{2}
\begin{enumerate}[label=(\themycounter\alph*),itemsep=0em]
\item \label{ass:regconsist_point_lq} {$\norm{\hf_n(X)-f^*(X)}_{L^q} \convinprob 0$,}

\item \label{ass:regconsist_mse_lq} {$\dfrac{1}{n} \displaystyle\sum_{i=1}^{n} \norm{f^*(x^i) - \hf_n(x^i)}^2 \xrightarrow{p} 0$,}

\columnbreak

\item \label{ass:regconsist_point2_lq} {$\norm{\hQ_n(X) - Q^*(X)}_{L^q} \convinprob 0$},

\item \label{ass:regconsist_mse2_lq} {$\dfrac{1}{n} \displaystyle\sum_{i=1}^{n} \bigl\lVert \bigl[\hQ_n(x^i)\bigr]^{-1} - \bigl[Q^*(x^i)\bigr]^{-1}\bigr\rVert^2 \xrightarrow{p} 0$.}
\end{enumerate}
\end{multicols}
\end{assumption}

{Assumption~\ref{ass:uniflln_lq} is implied by the uniform convergence in probability of the FI-SAA objective function $g^*_n$ to $g$ on $\Z \times \X$, i.e., $\sup_{(z,x) \in \Z \times \X} \abs*{g^*_{n}(z;x) - g(z;x)} \convinprob 0$.
We show in Section~\ref{sec:tssp} that this assumption holds for our running Examples~\ref{exm:runningexample} and~\ref{exm:regrexample} whenever $\expect{\norm{\varepsilon}} < \infty$ and the support~$\X$ is compact and bounded away from the origin.
Assumptions~\ref{ass:regconsist_point_lq} and~\ref{ass:regconsist_point2_lq} hold for our running Example~\ref{exm:regrexample} if the estimates $\hth_n$ and $\hat{\pi}_n$ are weakly consistent, $\norm{X}_{L^q} < +\infty$, and if, e.g., we additionally have the support~$\X$ to be compact and bounded away from the origin and assume that the estimates~$\{\hat{\pi}_n\}$ lie in a compact set a.s.\ for $n$ large enough (see Section~\ref{sec:regression} in the Appendix for details).
Unlike Assumptions~\ref{ass:regconsist_point} and~\ref{ass:regconsist_point2} that require $\hf_n$ and $\hQ_n$ to be pointwise consistent, Assumptions~\ref{ass:regconsist_point_lq} and~\ref{ass:regconsist_point2_lq} only require $\hf_n$ and $\hQ_n$ to be consistent on average over the covariates~$X$.
We have the following result.
}

\begin{theorem}
\label{thm:averageconsist}
{Suppose Assumptions~\ref{ass:equilipschitz},~\ref{ass:varweaklln},~\ref{ass:errorsweaklln},~\ref{ass:uniflln_lq}, and~\ref{ass:regconsist_lq} hold and $\big\lVert\norm{Q^*(X)}\big\rVert_{L^q}< +\infty$ for some constant $q \in [1,+\infty]$.
Then, $\norm*{\hv^{ER}_n(X) - v^*(X)}_{L^q} \convinprob 0$ and $\norm*{g(\hz^{ER}_n(X);X) - v^*(X)}_{L^q} \convinprob 0$.}
\end{theorem}

\subsection{Finite sample guarantees}
\label{subsec:finitesample}

We {now} establish a lower bound on the probability that solutions to the ER-SAA problem~\eqref{eqn:app} are nearly optimal to the true problem~\eqref{eqn:speq}.
{Section~\ref{sec:ersaa_rate} in the Appendix investigates the rate of convergence of ER-SAA estimators under weaker assumptions.}
Our next assumption is motivated by the analysis in Section~2 of~\citet{homem2008rates} and Section~7.2.9 of~\citet{shapiro2009lectures}.

\begin{assumption}
\label{ass:tradsaalargedev}
The full-information SAA problem~\eqref{eqn:fullinfsaa} possesses the following uniform exponential bound property: for any constant $\kappa > 0$ and a.e.~$x \in \X$, there exist positive constants $K(\kappa,x)$ and $\beta(\kappa,x)$ such that
$\pr \Bigl\{\uset{z \in \Z}{\sup} \: \abs*{ g^*_n(z;x) - g(z;x)} > \kappa \Bigr\} \leq K(\kappa,x) \exp\left(-n\beta(\kappa,x)\right)$, $\forall n \in \mathbb{N}$.
\end{assumption}

Lemma~2.4 of~\citet{homem2008rates} provides conditions under which Assumption~\ref{ass:tradsaalargedev} holds~\citep[also see Section~7.2.9 of][]{shapiro2009lectures}.
In particular,~\citet{homem2008rates} shows that Assumption~\ref{ass:tradsaalargedev} holds whenever 
the function~$c(\cdot,y)$ is Lipschitz continuous on~$\Z$ for a.e.\ $y \in \Y$ with an integrable Lipschitz constant
and some pointwise exponential bounds hold.
When the errors $\{\varepsilon^i\}$ are i.i.d., Section~7.2.9 of~\citet{shapiro2009lectures} presents conditions under which these pointwise exponential bound conditions are satisfied via Cram\'{e}r's large deviation theorem.
\citet{bryc1996large} present mixing conditions on the errors~$\{\varepsilon^i\}$ under which these assumptions are also satisfied (also see the references therein).
The G{\"a}rtner-Ellis Theorem~\citep[see Section~2.3 of][]{dembo2011large} provides an alternative avenue for verifying Assumption~\ref{ass:tradsaalargedev} for non-i.i.d.\ errors~$\{\varepsilon^i\}$~\citep{dai2000convergence}.
If we also assume that the random variable $c(z,f^*(x)+Q^*(x)\varepsilon) - \expect{c(z,f^*(x)+Q^*(x)\varepsilon)}$
is sub-Gaussian for each~$z \in \Z$ and a.e.\ $x \in \X$, then we can characterize the dependence of~$\beta(\kappa,x)$ on~$\kappa$, see Assumption~(C4) on page~396 and Theorem~7.67 of~\citet{shapiro2009lectures}.

{Proposition~\ref{prop:tssp-checkass} in the Appendix shows that Assumption~\ref{ass:tradsaalargedev} holds for Example~\ref{exm:runningexample} whenever the errors~$\{\varepsilon^i\}$ are i.i.d.\ and sub-Gaussian (which includes zero-mean Gaussian).
Unlike Assumption~\ref{ass:uniflln}, Assumption~\ref{ass:tradsaalargedev} may not hold when the distribution of the errors~$\varepsilon$ is heavy-tailed (heavy-tailed error distributions such as the Pareto and Weibull distributions occur in finance, weather forecasting, and reliability engineering applications).}

{Next, we require the following strengthening of Assumptions~\ref{ass:varweaklln} and~\ref{ass:errorsweaklln}.}

\begin{assumption}
\label{ass:varlargedev}
{For any constant $\kappa > 0$ and $n \in \mathbb{N}$, there exist positive constants $J_Q(\kappa)$, $\gamma_{Q}(n,\kappa)$, $\bar{J}_Q(\kappa)$, and $\bar{\gamma}_{Q}(n,\kappa)$, with $\lim_{n \to \infty} \gamma_{Q}(n,\kappa) = \infty$ and $\lim_{n \to \infty} \bar{\gamma}_{Q}(n,\kappa) = \infty$ for each $\kappa > 0$, such that}
\begin{align*}
{\mathbb{P}\biggl\{ \biggl(\frac{1}{n}\sum_{i=1}^{n} \bigl\lVert \bigl[Q^*(x^i)\bigr]^{-1}\bigr\rVert^2\biggr)^{1/2} > \Bigl(\expect{\bigl\lVert \bigl[Q^*(X)\bigr]^{-1} \bigr\rVert^2}\Bigr)^{1/2} + \kappa \biggr\}} &{\leq J_Q(\kappa)\exp(- \gamma_{Q}(n,\kappa)),} \\
{\mathbb{P}\biggl\{ \biggl(\frac{1}{n} \sum_{i=1}^{n} \norm{Q^*(x^i)}^4\biggr)^{1/4} > \bigl(\expect{\norm{Q^*(X)}^4}\bigr)^{1/4} + \kappa \biggr\}} &{\leq \bar{J}_Q(\kappa)\exp(-\bar{\gamma}_{Q}(n,\kappa)).}
\end{align*}
\end{assumption}

\begin{assumption}
\label{ass:errorlargedev}
{For any constant $\kappa > 0$ and $n \in \mathbb{N}$, there exist positive constants $J_{\varepsilon}(\kappa)$, $\gamma_{\varepsilon}(n,\kappa)$, $\bar{J}_{\varepsilon}(\kappa)$, and $\bar{\gamma}_{\varepsilon}(n,\kappa)$, with $\lim_{n \to \infty} \gamma_{\varepsilon}(n,\kappa) = \infty$ and $\lim_{n \to \infty} \bar{\gamma}_{\varepsilon}(n,\kappa) = \infty$ for each $\kappa > 0$, such that}
\begin{align*}
{\mathbb{P}\biggl\{ \frac{1}{n}\sum_{i=1}^{n} \norm{\varepsilon^i} > \expect{\norm{\varepsilon}} + \kappa \biggr\}} &{\leq J_{\varepsilon}(\kappa)\exp(-\gamma_{\varepsilon}(n,\kappa)),} \\
{\mathbb{P}\biggl\{ \biggl(\frac{1}{n}\sum_{i=1}^{n} \norm{\varepsilon^i}^4\biggr)^{1/4} > (\expect{\norm{\varepsilon}^4})^{1/4} + \kappa \biggr\}} &{\leq \bar{J}_{\varepsilon}(\kappa)\exp(-\bar{\gamma}_{\varepsilon}(n,\kappa)).}
\end{align*}
\end{assumption}

{The first part of Assumption~\ref{ass:varlargedev} holds, e.g., if for each $\kappa > 0$ and $n \in \mathbb{N}$, constants $J_Q(\kappa) > 0$ and $\gamma_Q(n,\kappa) > 0$ exist such that
\[
\mathbb{P}\biggl\{ \frac{1}{n}\sum_{i=1}^{n} \bigl\lVert \bigl[Q^*(x^i)\bigr]^{-1}\bigr\rVert^2 > \expect{\bigl\lVert \bigl[Q^*(X)\bigr]^{-1} \bigr\rVert^2} + \kappa^2 \biggr\} \leq J_Q(\kappa)\exp(- \gamma_{Q}(n,\kappa)).
\]
The function $\gamma_Q$ in the inequality above is related to the so-called rate function in large deviations theory (see~\citet{dembo2011large} and Section~7.2.8 of~\cite{shapiro2009lectures}).
Similar conclusions hold for the probability inequalities involving the terms $\frac{1}{n} \sum_{i=1}^{n} \norm{Q^*(x^i)}^4$ and $\frac{1}{n}\sum_{i=1}^{n} \norm{\varepsilon^i}^4$ in Assumptions~\ref{ass:varlargedev} and~\ref{ass:errorlargedev}.
Using large deviations theory, we can show that the constants $\gamma_Q(n,\kappa)$ and $\bar{\gamma}_Q(n,\kappa)$ in Assumption~\ref{ass:varlargedev} increase linearly with the sample size~$n$ (i.e., $\gamma_Q(n,\kappa) = n \gamma_{Q,1}(\kappa)$ and $\bar{\gamma}_Q(n,\kappa) = n \bar{\gamma}_{Q,1}(\kappa)$) for our running Example~\ref{exm:regrexample} with i.i.d.\ data~$\{(x^i,\varepsilon^i)\}$ whenever the support~$\X$ of the covariates is compact and bounded away from the origin.
The constant $\gamma_{\varepsilon}(n,\kappa)$ in Assumption~\ref{ass:errorlargedev} also increases linearly with~$n$ whenever the errors~$\varepsilon$ are sub-Gaussian (see Chapter~3 of~\citet{vershynin2018high}).
The discussion following Assumption~\ref{ass:tradsaalargedev} provides avenues for verifying Assumptions~\ref{ass:varlargedev} and~\ref{ass:errorlargedev} for non-i.i.d.\ data~$\{(x^i,\varepsilon^i)\}$.}

Next, we make the following large deviation assumption on the regression procedure~\eqref{eqn:regr} that is similar in spirit to Assumption~\ref{ass:tradsaalargedev}.

\begin{assumption}
\label{ass:reglargedev}
\setcounter{mycounter}{\value{assumption}}
{The regression estimates $\hf_n$ and $\hQ_n$ possess the following finite sample properties: for any constant $\kappa > 0$ and $n \in \mathbb{N}$, there exist positive constants $K_f(\kappa,x)$, $\bar{K}_f(\kappa)$, $\beta_f(n,\kappa,x)$, $\bar{\beta}_f(n,\kappa)$, $K_Q(\kappa,x)$, $\bar{K}_Q(\kappa)$, $\beta_Q(n,\kappa,x)$, and $\bar{\beta}_Q(n,\kappa)$, 
with $\lim_{n \to \infty} \beta_f(n,\kappa,x) = \infty$, 
$\lim_{n \to \infty} \bar{\beta}_f(n,\kappa) = \infty$, 
$\lim_{n \to \infty} \beta_Q(n,\kappa,x) = \infty$, and 
$\lim_{n \to \infty} \bar{\beta}_Q(n,\kappa) = \infty$ for each $\kappa > 0$ and a.e.\ $x \in \X$, such that}
\begin{enumerate}[label=(\themycounter\alph*),itemsep=0.2em]
\item \label{ass:reglargedev_point} {$\mathbb{P}\bigl\{\norm{f^*(x) - \hf_n(x)} > \kappa \bigr\} \leq K_f(\kappa,x) \exp\bigl(-\beta_f(n,\kappa,x)\bigr)$ for a.e.\ $x \in \X$,}

\item \label{ass:reglargedev_mse} {$\mathbb{P}\biggl\{\dfrac{1}{n} \displaystyle\sum_{i=1}^{n} \norm{f^*(x^i) - \hf_n(x^i)}^2 > \kappa^2 \biggr\} \leq \bar{K}_f(\kappa) \exp\bigl(-\bar{\beta}_f(n,\kappa)\bigr)$,}

\item \label{ass:reglargedev_point2} {$\mathbb{P}\bigl\{\norm{Q^*(x) - \hQ_n(x)} > \kappa \bigr\} \leq K_Q(\kappa,x) \exp\bigl(-\beta_Q(n,\kappa,x)\bigr)$ for a.e.\ $x \in \X$,}

\item \label{ass:reglargedev_mse2} {$\mathbb{P}\biggl\{\dfrac{1}{n} \displaystyle\sum_{i=1}^{n} \bigl\lVert \bigl[\hQ_n(x^i)\bigr]^{-1} - \bigl[Q^*(x^i)\bigr]^{-1}\bigr\rVert^2 > \kappa^2 \biggr\} \leq \bar{K}_Q(\kappa) \exp\bigl(-\bar{\beta}_Q(n,\kappa)\bigr)$.}
\end{enumerate}
\end{assumption}

We verify in Section~\ref{sec:regression} of the Appendix that Assumptions~\ref{ass:reglargedev_point} and~\ref{ass:reglargedev_mse} hold for OLS regression and the Lasso with constants~$\beta_f(n,\kappa,x)$ and $\bar{\beta}_f(n,\kappa)$ scaling as~$O(n\kappa^2)$ under sub-Gaussian assumptions on the errors~$\varepsilon$.
{The finite sample guarantees on the estimate~$\hQ_n$ in Assumptions~\ref{ass:reglargedev_point2} and~\ref{ass:reglargedev_mse2} are typically harder to verify.}

The next result provides conditions under which the maximum deviation of the ER-SAA objective from the full-information SAA objective on the feasible set~$\Z$ satisfies a qualitatively similar large deviations bound as that in Assumption~\ref{ass:tradsaalargedev}.

\begin{restatable}{lemma}{lemlargedevofdev}
\label{lem:largedevofdev}
Suppose {Assumptions~\ref{ass:equilipschitz},~\ref{ass:tradsaalargedev} to~\ref{ass:reglargedev} hold.}
Then for any constant $\kappa > 0$, $n \in \mathbb{N}$, and a.e.\ $x \in \X$, there exist positive constants $\bar{K}(\kappa,x)$ and $\bar{\beta}(n,\kappa,x)$, 
with $\lim_{n \to \infty} \bar{\beta}(n,\kappa,x) = \infty$ for each $\kappa > 0$ and a.e.\ $x \in \X$, such that
$\pr \Bigl\{\uset{z \in \Z}{\sup} \: \abs*{ \hg^{ER}_n(z;x) - g^*_n(z;x)} > \kappa \Bigr\} \leq \bar{K}(\kappa,x) \exp\bigl(-\bar{\beta}(n,\kappa,x)\bigr)$.
\end{restatable}

We are now ready to present the main result of this section. 
It extends finite sample results that are known for traditional SAA estimators, see, e.g., Theorem~2.3 of~\citet{homem2008rates} and Section~5.3 of~\citet{shapiro2009lectures}, to the ER-SAA setting.

\begin{restatable}{theorem}{thmexponentialconv}
\label{thm:exponentialconv}
Suppose {Assumptions~\ref{ass:equilipschitz},~\ref{ass:tradsaalargedev} to~\ref{ass:reglargedev} hold.}
Then, for each $n \in \mathbb{N}$ and a.e.\ $x \in \X$, given $\eta > 0$, there exist constants $Q(\eta,x) > 0$ and $\gamma(n,\eta,x) > 0$, 
with $\lim_{n \to \infty} \gamma(n,\eta,x) = \infty$ for each $\eta > 0$ and a.e.\ $x \in \X$, such that
$\prob{\textup{dist}(\hz^{ER}_n(x),S^*(x)) \geq \eta} \leq Q(\eta,x) \exp(-\gamma(n,\eta,x))$.
\end{restatable}

{Assumptions~\ref{ass:varlargedev},~\ref{ass:errorlargedev},~\ref{ass:reglargedev_point2}, and~\ref{ass:reglargedev_mse2} are not required to establish Lemma~\ref{lem:largedevofdev} and Theorem~\ref{thm:exponentialconv} in the homoscedastic case {($\hQ_n := Q^* \equiv I$)}. Additionally, Assumption~\ref{ass:reglargedev_mse} can be weakened in this setting to $\mathbb{P}\bigl\{\frac{1}{n} \sum_{i=1}^{n} \norm{f^*(x^i) - \hf_n(x^i)} > \kappa \bigr\} \leq \bar{K}_f(\kappa) \exp\bigl(-\bar{\beta}_f(n,\kappa)\bigr)$ on account of Lemma~\ref{lem:meandeviation}.}

{To give an example of how Theorem \ref{thm:exponentialconv} can be used to provide a qualitative estimate of sample size required to obtain a desired accuracy with high probability,} we now specialize the results in this section to the {homoscedastic} setting when the ER-SAA formulation is applied to two-stage stochastic LP with OLS, Lasso, or kNN regression for estimating $f^*$.
{Given $\kappa > 0$, let $S^{\kappa}(x) := \Set{z \in \Z}{g(z;x) \leq v^*(x) + \kappa}$ denote the set of~$\kappa$-optimal solutions to the true problem~\eqref{eqn:speq}.
Given an unreliability level $\delta \in (0,1)$, our goal is to estimate the sample size~$n$ required for every solution to the ER-SAA problem~\eqref{eqn:app} to be $\kappa$-optimal to the true problem~\eqref{eqn:speq} with probability at least $1-\delta$, i.e., to estimate~$n$ such that $\mathbb{P}\bigl\{\hS^{ER}_n(x) \subseteq S^{\kappa}(x)\bigr\} \geq 1 - \delta$.}
In the following, we make stronger than necessary assumptions on the regression setups for readability.

{Our sample size estimate for ER-SAA proceeds by estimating the sample size required for the full-information SAA problem~\eqref{eqn:fullinfsaa} to be `close to' the true problem~\eqref{eqn:speq} \textit{and} for the ER-SAA problem~\eqref{eqn:app} to be `close to' the FI-SAA problem~\eqref{eqn:fullinfsaa}; see also \eqref{eqn:ersaa_error_decomp}.
From Section~5.3 of~\citet{shapiro2009lectures}, we have the following sample size estimate for every solution to the FI-SAA problem~\eqref{eqn:fullinfsaa} to be $\kappa$-optimal to the true problem~\eqref{eqn:speq} with probability at least $1-\delta$:}
\[
{n \geq n^* := \dfrac{O(1)\sigma^2_c(x)}{\kappa^2} \left[ d_z \log\left(\dfrac{O(1) D}{\kappa}\right) + \log\left(\dfrac{O(1)}{\delta}\right) \right].}
\]
{Learning of the regression function~$f^*$ introduces additional terms in the estimate for ER-SAA that depend on the dimensions $d_y$ and $d_x$ of the random vector~$Y$ and the random covariates~$X$.}

\begin{restatable}{proposition}{propspecfinsamp}
\label{prop:specfinsamp}
Consider Example~\ref{exm:runningexample} {and assume $\hQ_n = Q^* \equiv I$}.
Suppose~$\Z$ is compact with diameter $D$ and for each $z \in \Z$ and a.e.\ $x \in \X$,  the random variable $c(z,f^*(x)+\varepsilon) -
\expect{c(z,f^*(x)+\varepsilon)}$
is sub-Gaussian with variance proxy $\sigma^2_c(x)$.
Let $\{\varepsilon^i\}_{i \in [n]}$ be i.i.d.\ sub-Gaussian random vectors with variance proxy $\sigma^2$, $\kappa > 0$ be the target optimality gap, and $\delta \in (0,1)$ be the desired unreliability level.
\begin{enumerate}
\smallskip

\item Suppose the regression function $f^*$ is linear, the regression step~\eqref{eqn:regr} is OLS regression, 
the covariance matrix $\Sigma_X$ of the covariates is positive definite,
and the random vector $\Sigma_X^{-\frac{1}{2}} X$ is sub-Gaussian.
Then, we have $\mathbb{P}\bigl\{\hS^{ER}_n(x) \subseteq S^{\kappa}(x)\bigr\} \geq 1 - \delta$ for sample size $n$ satisfying
\[
n \geq {n^*} + \dfrac{O(1)\sigma^2 d_y}{\kappa^2} \left[ \log\left(\dfrac{O(1)}{\delta}\right) + d_x \right].
\]

\item Suppose the regression function $f^*$ is linear with $\norm{\theta^*_{[j]}}_0 \leq s$, $\forall j \in [d_y]$, the regression step~\eqref{eqn:regr} is Lasso regression,
the support~$\X$ of the covariates~$X$ is compact, $\expect{\abs{X_j}^2} > 0$, $\forall j \in [d_x]$, and the matrix $\expect{X\tr{X}} - \tau \textup{diag}(\expect{X\tr{X}})$ is positive semidefinite for some constant $\tau \in (0,1]$.
Then, we have $\mathbb{P}\bigl\{\hS^{ER}_n(x) \subseteq S^{\kappa}(x)\bigr\} \geq 1 - \delta$ for sample size $n$ satisfying
\[
n \geq {n^*} + \dfrac{O(1)\sigma^2 s d_y}{\kappa^2} \left[ \log\left(\dfrac{O(1)}{\delta}\right) + \log(d_x) \right].
\]

\item Suppose the regression function $f^*$ is Lipschitz continuous, the regression step~\eqref{eqn:regr} is kNN regression with parameter $k = \ceil{O(1)n^{\gamma}}$ for some constant $\gamma \in (0,1)$, 
the support~$\X$ of the covariates~$X$ is compact, and there exists a constant $\tau > 0$ such that $\prob{X \in \mathcal{B}_{\kappa}(x)} \geq \tau \kappa^{d_x}$, $\forall x \in \X$ and $\kappa > 0$.
Then, we have $\mathbb{P}\bigl\{\hS^{ER}_n(x) \subseteq S^{\kappa}(x)\bigr\} \geq 1 - \delta$ for sample size $n$ satisfying $n \geq O(1)\left(\dfrac{O(1)}{\kappa}\right)^{\frac{d_x}{1-\gamma}}$, $\dfrac{n^{\gamma}}{\log(n)} \geq \dfrac{O(1) d_x d_y \sigma^2}{\kappa^2}$, and
{
\small
\begin{align*}
\hspace*{-0.2in}n &\geq {n^*} + \left(\dfrac{O(1)\sigma^2 d_y}{\kappa^2}\right)^{\frac{1}{\gamma}} \left[ d_x \log\left(\dfrac{O(1)}{d_x}\right) + \log\left(\dfrac{O(1)}{\delta}\right)\right]^{\frac{1}{\gamma}} + \left(\dfrac{O(1)d_y}{\kappa^2}\right)^{d_x} \left[ \dfrac{d_x}{2} \log\left(\dfrac{O(1)d_x d_y}{\kappa^2}\right) + \log\left(\dfrac{O(1)}{\delta}\right)\right].
\end{align*}
}%
\end{enumerate}
\end{restatable}
\medskip

Proposition~\ref{prop:specfinsamp} illustrates the tradeoff between using parametric and nonparametric regression approaches within the ER-SAA framework.
{The sample size estimates in Proposition~\ref{prop:specfinsamp} involve the sum of two contributions, the FI-SAA contribution $n^*$ and additional regression-related terms introduced by the use of estimates of~$f^*$ within the ER-SAA.}
\textit{Assuming that} the regression function $f^*$ satisfies the necessary structural properties, using OLS regression or the Lasso for the regression step~\eqref{eqn:regr} can yield sample size estimates that depend modestly on the accuracy $\kappa$ and the dimensions $d_x$ and $d_y$
compared to kNN regression.
On the other hand, unlike OLS and Lasso regression, the sample size estimates for kNN regression are valid under mild assumptions on the regression function~$f^*$.
Nevertheless, we empirically demonstrate in Section~\ref{sec:computexp} that it may be beneficial to use a structured but misspecified prediction model when we do not have an abundance of data.
Note that the OLS estimate includes a term that depends linearly on the dimension~$d_x$ of the covariates, whereas the corresponding term in the Lasso estimate only depends logarithmically on~$d_x$.

\subsection{Outline of analysis for the {jackknife}-based estimators}
\label{subsec:jackoutline}

The results thus far carry over to the J-SAA and J+-SAA estimators if the assumptions that ensure that the ER-SAA mean deviation term $\frac{1}{n} \sum_{i=1}^{n}\norm{\teps^{i}_{n}(x)} \convinprob 0$ at a certain rate are adapted to ensure that the J-SAA and J+-SAA mean deviation terms $\frac{1}{n} \sum_{i=1}^{n}\norm{\teps^{i,J}_{n}(x)}$ and $\frac{1}{n} \sum_{i=1}^{n}\norm{\teps^{i,J+}_{n}(x)}$ converge to zero in probability at a certain rate, where
\begin{align*}
\teps^{i,J}_{n}(x) &:= \left( \hf_n(x) + {\hQ_n(x)}\heps^{i}_{n,J} \right) - \left( f^*(x) + {Q^*(x)}\varepsilon^i \right), \quad \forall i \in [n], \\
\teps^{i,J+}_{n}(x) &:= \left( \hf_{-i}(x) + {\hQ_{-i}(x)}\heps^{i}_{n,J} \right) - \left( f^*(x) + {Q^*(x)}\varepsilon^i \right), \quad \forall i \in [n].
\end{align*}
{Lemma~\ref{lem:jack_meandeviation} in Section~\ref{sec:jackknife} of the Appendix presents the analogue of Lemma~\ref{lem:meandeviation} for the jackknife-based mean deviation terms.
It also provides guidance for how the assumptions on the quantities appearing in inequality~\eqref{eqn:meandeviation} could be replaced with assumptions on the quantities appearing in the bounds of the jackknife-based mean deviation terms to derive similar results for the J-SAA and J+-SAA estimators as the ER-SAA estimator.}
Since the formal statements of the assumptions and results for the J-SAA and J+-SAA estimators closely mirror those for the ER-SAA given in Sections~\ref{subsec:consistency} and~\ref{subsec:finitesample}, we present these details in Section~\ref{sec:jackknife} of the Appendix.

\section{Computational experiments}
\label{sec:computexp}

We consider instances of the following resource allocation model adapted from~\citet{luedtke2014branch}:
\[
\uset{z \in \mathbb{R}^{\abs{\I}}_+}{\min} \: \tr{c}_z z + \expect{Q(z,Y)} ,
\]
where the second-stage function is defined as
\[
Q(z,y) := \min_{v \in \mathbb{R}^{\abs{\I} \times \abs{\J}}_+, w \in \mathbb{R}^{\abs{\J}}_+} \: \Bigl\{ \tr{q}_w w \, : \, \sum_{j \in \J} v_{ij} \leq \rho_i z_i, \:\: \forall i \in \I, \:\: \sum_{i \in \I} \mu_{ij} v_{ij} + w_{j} \geq y_{j}, \:\: \forall j \in \J \Bigr\}.
\]
The first-stage variables $z_{i}$ denote the order quantities of resources $i \in \I$, and the second-stage variables $v_{ij}$ and $w_j$ denote the amount of resource $i \in \I$ allocated to customer type $j \in \J$ and the unmet demand of customer type $j$, respectively. We consider instances with $\abs{\I} = 20$ and $\abs{\J} = 30$.
The yield and service rate parameters $\rho$ and $\mu$ and the cost coefficients $c_z$ and $q_w$ are assumed to be deterministic. Parameters $c_z$, $\rho$, and $\mu$ are set using the procedure described in~\citet{luedtke2014branch}, and the coefficients $q_w$ are determined by $q_{w} := \tau \norm{c_{z}}_{\infty}$, where each component of the vector $\tau$ is drawn independently from a lognormal $\mathcal{LN}(0.5,0.05)$ distribution.

The demands $y_j$, $j \in \J$, of the customer types are considered to be stochastic {with $\Y = \R^{\abs{\J}}_+$}.  
We assume that some of the variability in the demands can be explained with knowledge of covariates $X_l$, $l \in \mathcal{L}$, where $\abs{\mathcal{L}} = d_x$.
The demands $Y$ are assumed to be related to the covariates through
\[
Y_j = \varphi^*_j + \sum_{l \in \mathcal{L}^*} \zeta^*_{jl} (X_l)^p + {q^*_j(X)}\varepsilon_j, \quad \forall j \in \J,
\]
where $p \in \{0.5,1,2\}$ is a fixed parameter that determines the model class, {$q^*_j : \X \to \R_+$ is defined as 
$q^*_j(X) = \bigl(s_j \exp\bigl(\sum_{l \in \mathcal{L}^*} \pi^*_{jl} \log\abs{1 + X_l}\bigr)\bigr)^{1/2}$
for parameters $\{\pi^*_{jl}\}_{l \in \mathcal{L}^*}$~\citep{romano2017resurrecting}, errors} $\varepsilon_j \sim \mathcal{N}\left(0,\sigma^2_j\right)$ {are independent of $X$}, $\varphi^*$, $\zeta^*$ and $\sigma_j$ are {additional} model parameters, and $\mathcal{L}^* \subseteq \mathcal{L}$ contains the indices of a subset of covariates with predictive power (note that $\mathcal{L}^*$ does not depend on~$j \in \J$). 
{In the definition of $q^*_j(X)$, the scaling factor $s_j$ is a numerical approximation of the median of $\exp\bigl(\sum_{l \in \mathcal{L}^*} \pi^*_{jl} \log\abs{1 + X_l}\bigr)$ so that $\mathbb{P}_X(q^*_j(X) > 1) \approx 0.5$.
The form of the heteroscedasticity functions $q^*_j$ is chosen to simulate increasing error variance with increasing magnitude of the covariates~\citep{romano2017resurrecting}.}
Throughout, we assume $\abs{\mathcal{L}^*} = 3$, i.e., the demands truly depend only on three covariates.
We simulate i.i.d.\ data $\{(x^i,\varepsilon^i)\}$ with $\varphi^*$ and $\zeta^*$ randomly generated, $\sigma_j = \sigma = 5$, $\forall j \in \J$, unless otherwise specified, and draw covariate samples $\{x^i\}_{i=1}^{n}$ from a multivariate folded normal distribution (see Section~\ref{sec:computexp-details} in the Appendix for details).
{We vary the heteroscedasticity level $\omega \in \{1,2,3\}$, where $\omega = 1$, $2$, and $3$ correspond to zero, moderate, and severe heteroscedasiticity, respectively, and sample $\{\pi^*_{jl}\}_{l \in \mathcal{L}^*}$ independently from the uniform distribution $U(0, 2(\omega - 1)^2)$.
Figure~\ref{fig:het_stats} plots statistics of the functions $q^*_j$ for moderate and severe heteroscedasticity.}

\begin{figure}[t!]
    \centering
        \includegraphics[width=0.45\textwidth]{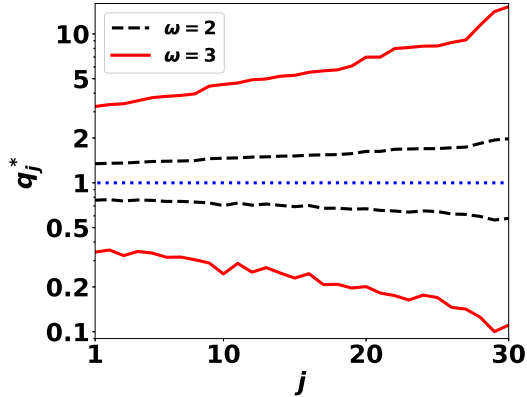}
    \caption{{Median (dotted line) and the $10$th and $90$th quantiles (dashed lines for $\omega = 2$ [moderate heteroscedasticity] and solid lines for $\omega = 3$ [severe heteroscedasticity]) of the heteroscedasticity function $q^*_j$ over $1100$ independent realizations of the covariates. The $y$-axis is in logarithmic scale and the indices $j$ are reordered to generate approximately monotonic plots.}}
    \label{fig:het_stats}
\end{figure}

Given data $\D_n$ on the demands and covariates, we estimate the coefficients of the linear model
\[
Y_j = \varphi_j + \sum_{l \in \mathcal{L}} \zeta_{jl} X_l + \eta_j, \quad \forall j \in \J,
\]
where $\eta_j$ are zero-mean errors, using OLS or Lasso regression and use this prediction model within the ER-SAA, J-SAA, and J+-SAA frameworks.
We use this linear prediction model even when the degree $p \neq 1$, in which case the prediction model is misspecified.
{We also evaluate the performance of the ER-SAA approach when kNN regression is used to predict $Y$ from $X$ (the parameter $k$ in kNN regression is chosen from $\left[ \floor{n^{0.1}},\ceil{n^{0.9}} \right]$ to minimize the $5$-fold CV test error).
We compare these ER-SAA and J-SAA estimators with the point prediction-based (PP) deterministic approximation~\eqref{eqn:pointpred} that uses OLS or Lasso regression (assuming a linear model).
While we do not estimate the heteroscedasticity functions $\{q^*_j\}$ by default in our experiments (equivalent to setting $\hat{Q}_n \equiv I$), we also compare the performance of ER-SAA formulations that ignore heteroscedasticity with heteroscedasticity-aware ER-SAA formulations where each $q^*_j$ is estimated using Lasso regression (cf.\ Example~\ref{exm:regrexample}).
When the function $f^*$ is estimated using OLS regression, we use the estimate of $Q^*$ to update the estimate $\hf_n$ of $f^*$ using weighted least squares~\citep{romano2017resurrecting}.
While estimating~$Q^*$, we assume that its parametric form is known; model misspecification for~$f^*$ also simulates model misspecification for heteroscedasticity in our two-step approach for estimating~$Q^*$.}

We compare our data-driven SAA estimators with  the kNN-based reweighted SAA (kNN-SAA) approach of~\citet{bertsimas2014predictive} on a few test instances by varying the dimensions of the covariates~$d_x$, the sample size~$n$, the degree~$p$, the standard deviation~$\sigma$ of the errors~$\varepsilon$, {and the degree of heteroscedasticity~$\omega$}.
While our case studies illustrate the potential advantages of employing parametric regression models (such as OLS and the Lasso) within our data-driven formulations, we do not claim that this advantage holds for arbitrary model instances. 
We choose the kNN-SAA approach to compare against because it is easy to implement and tune this approach, and the empirical results of~\citet{bertsimas2014predictive} and~\citet{bertsimas2019predictions} show that this approach is one of the better performing reweighted SAA approaches.The parameter $k$ in kNN-SAA is {once again chosen from $\left[ \floor{n^{0.1}},\ceil{n^{0.9}} \right]$ to minimize the $5$-fold CV test error.}

Solutions from the different approaches are compared by estimating a normalized version of the upper bound of a $99\%$ confidence interval (UCB) on their optimality gaps using the multiple replication procedure of~\citet{mak1999monte} (see Section~\ref{sec:computexp-details} for details).
{We choose to compare $99\%$ UCBs of the different estimators as it provides a conservative estimate of their suboptimality and mitigates the variability of a single random evaluation sample}.
Because the data-driven solutions depend on the realization of samples~$\D_n$, we perform $100$ replications per test instance and report our results in the form of box plots of these UCBs (the boxes denote the $25^{\text{th}}$, $50^{\text{th}}$, and $75^{\text{th}}$ percentiles of the $99\%$ UCBs, and the whiskers denote their {$5^{\text{th}}$ and $95^{\text{th}}$} percentiles over the $100$ replicates).

Source code and data for the test instances are available at \url{https://github.com/rohitkannan/DD-SAA}.
Our codes are written in Julia~0.6.4~\citep{bezanson2017julia}, use Gurobi~8.1.0 to solve LPs through the JuMP~0.18.5 interface~\citep{dunning2017jump}, and use \texttt{glmnet}~0.3.0~\citep{friedman2010regularization} for Lasso regression. 
All computational tests were conducted through the UW-Madison high throughput computing software \texttt{HTCondor} (\url{http://chtc.cs.wisc.edu/}).

\paragraph{Effect of varying covariate dimension.} 
Figure~\ref{fig:comp_ols} compares the performance of the kNN-SAA, {ER-SAA+kNN}, ER-SAA+OLS, {and PP+OLS} approaches {for the homoscedastic case ($\omega = 1$)} by varying the model degree $p$, the covariate dimension among $d_x \in \{3,10,100\}$, and the sample size among $n \in \{{5(d_x + 1)}, 20(d_x + 1), 100(d_x + 1)\}$. 
Note that OLS regression estimates $d_x + 1$ parameters for each $j \in \J$.
{The ER-SAA approaches in this case study do not estimate $Q^* \equiv I$, but directly assume $\hat{Q}_n \equiv I$.}
{The kNN-SAA and ER-SAA+kNN approaches exhibit similar performance overall.}
When the prediction model is correctly specified (i.e., $p = 1$), the ER-SAA+OLS approach unsurprisingly dominates the {ER-SAA+kNN and kNN-SAA approaches}.
When $p \neq 1$, as anticipated, the ER-SAA+OLS approach does not yield a consistent estimator, whereas the {ER-SAA+kNN and kNN-SAA approaches} yields consistent estimators, albeit with a slow rate of convergence (cf.\ Proposition~\ref{prop:specfinsamp}).
However, the {ER-SAA+OLS} approach consistently outperforms the {ER-SAA+kNN and kNN-SAA approaches} when $p = 0.5$ even for the largest sample size of $n = 100(d_x + 1)$. 
When the degree $p = 2$, the {ER-SAA+kNN and kNN-SAA approaches} fare better than the ER-SAA+OLS approach only for a sample size of $n \geq 80$ when the covariate dimension is small ($d_x = 3$), and lose this advantage in the larger covariate dimensions.
{Although the PP+OLS does not yield a consistent estimator even when $p = 1$, it performs better than the ER-SAA+kNN and kNN-SAA approaches in many cases even for large sample sizes~$n$.
However, the ER-SAA+OLS approach outperforms the PP+OLS approach across all cases.}
While we do not show results, we mention that the N-SAA estimator is not asymptotically optimal for all three model instances with the median values of the $99\%$ UCBs of its percentage optimality gaps being about {$11\%$, $5\%$, and $26\%$} for the $p = 1$, $p = 0.5$, and $p = 2$ instances, respectively, for {$n = 10100$}. 
This indicates that using covariate information can be advantageous in these instances.

\begin{figure}[t!]
    \centering
    \begin{subfigure}[t]{0.33\textwidth}
        \centering
        \includegraphics[width=\textwidth]{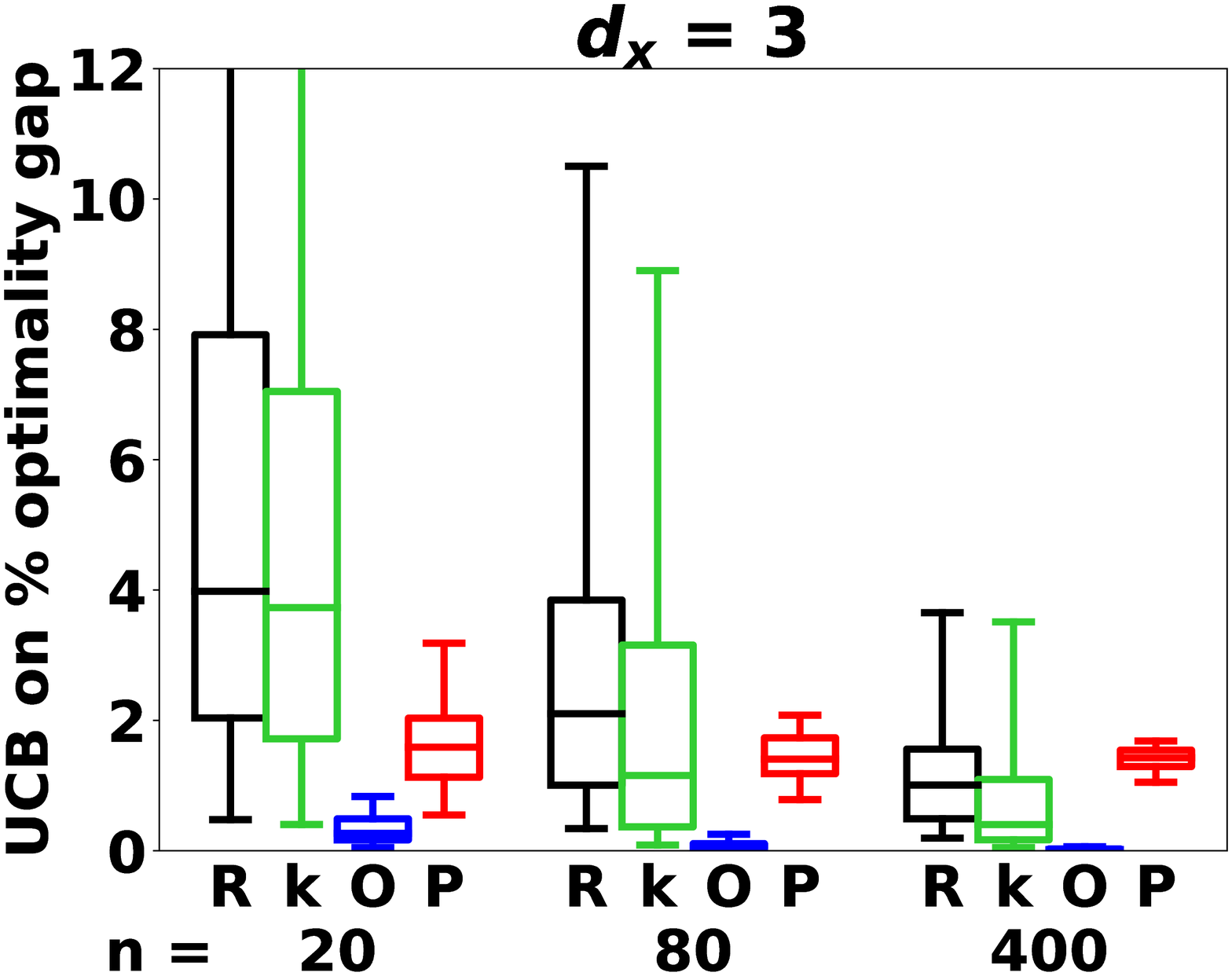}
    \end{subfigure}%
    ~ 
    \begin{subfigure}[t]{0.33\textwidth}
        \centering
        \includegraphics[width=\textwidth]{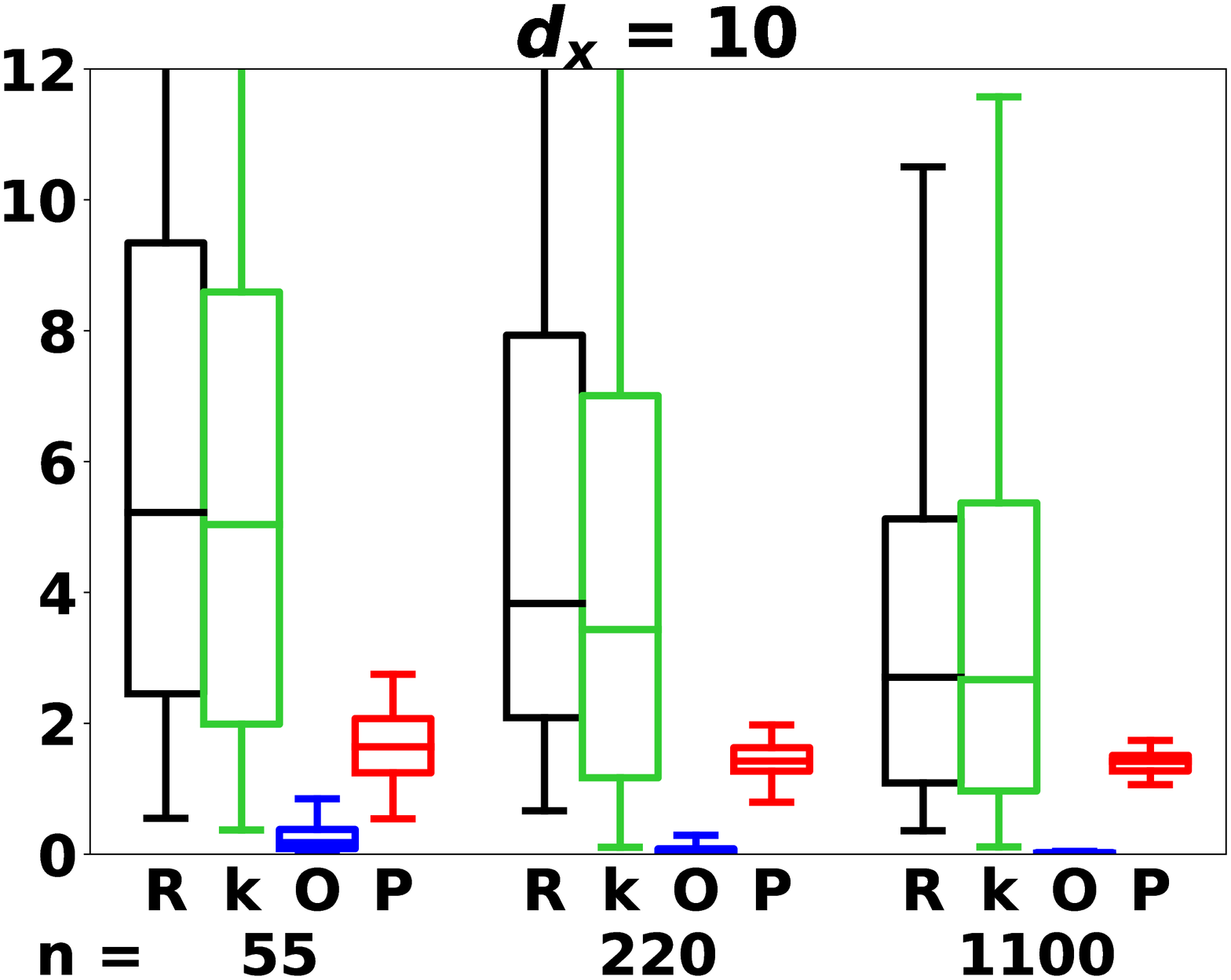}
    \end{subfigure}%
    ~ 
    \begin{subfigure}[t]{0.33\textwidth}
        \centering
        \includegraphics[width=\textwidth]{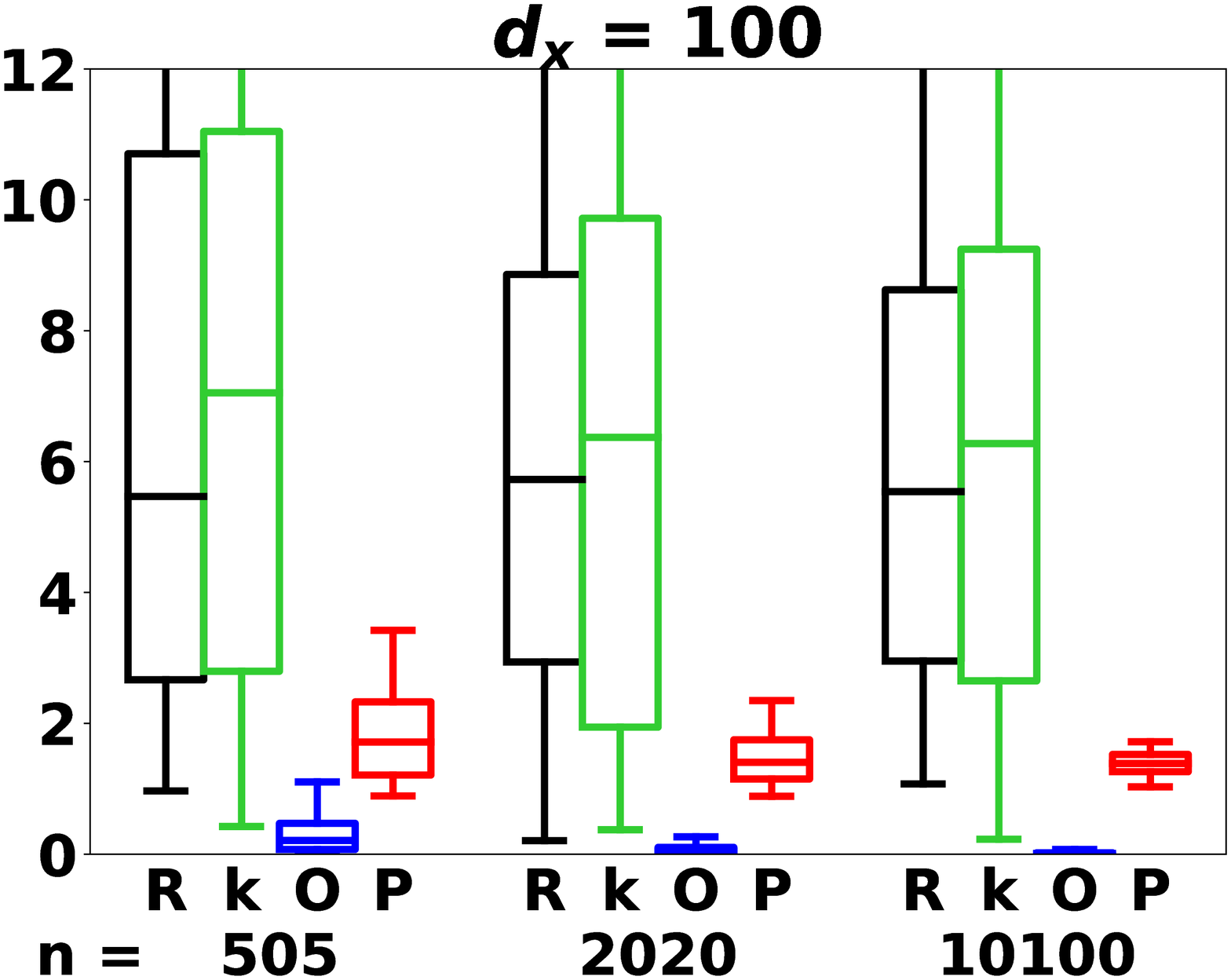}
    \end{subfigure}\\
    \begin{subfigure}[t]{0.33\textwidth}
        \centering
        \includegraphics[width=\textwidth]{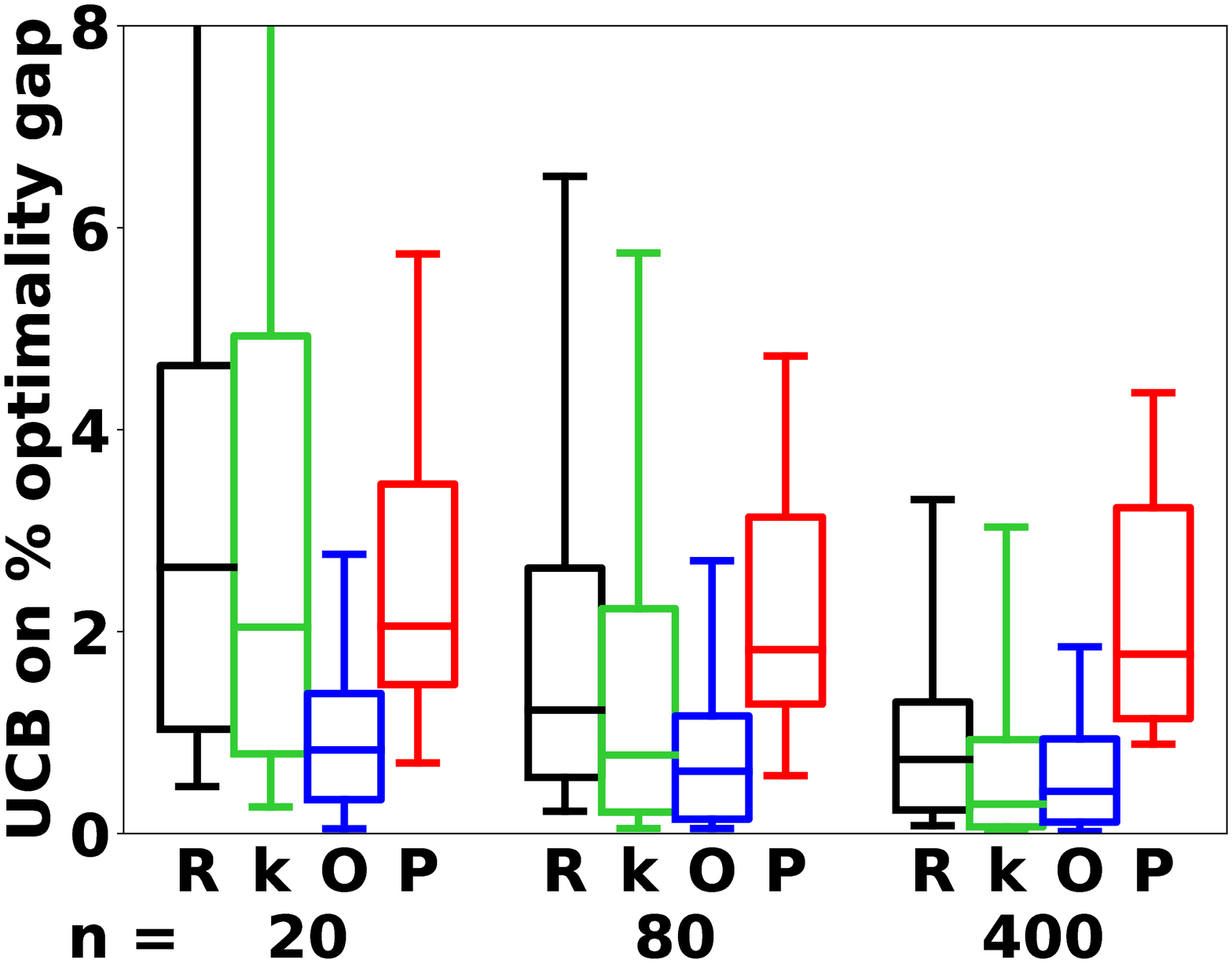}
    \end{subfigure}%
    ~ 
    \begin{subfigure}[t]{0.33\textwidth}
        \centering
        \includegraphics[width=\textwidth]{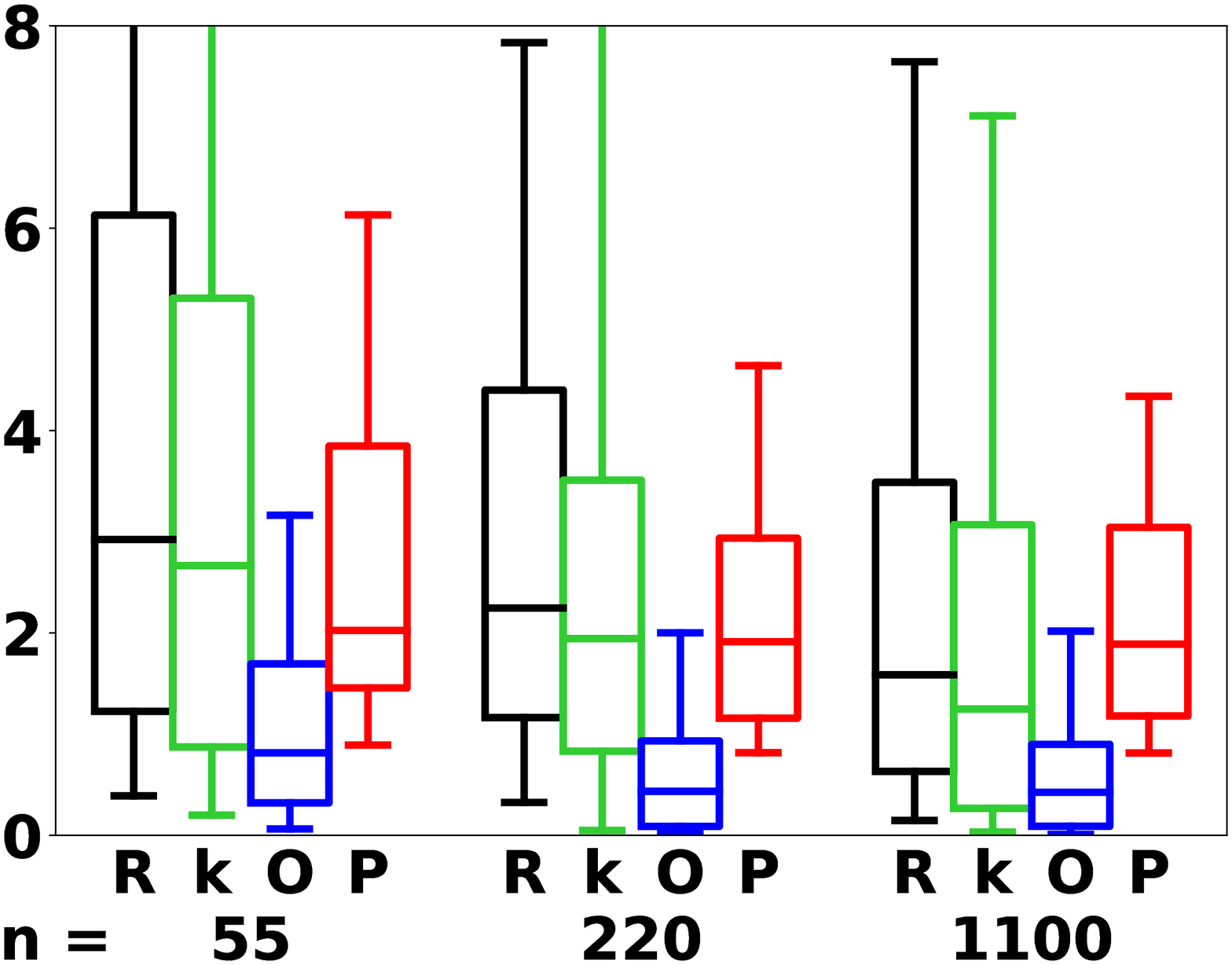}
    \end{subfigure}%
    ~ 
    \begin{subfigure}[t]{0.33\textwidth}
        \centering
        \includegraphics[width=\textwidth]{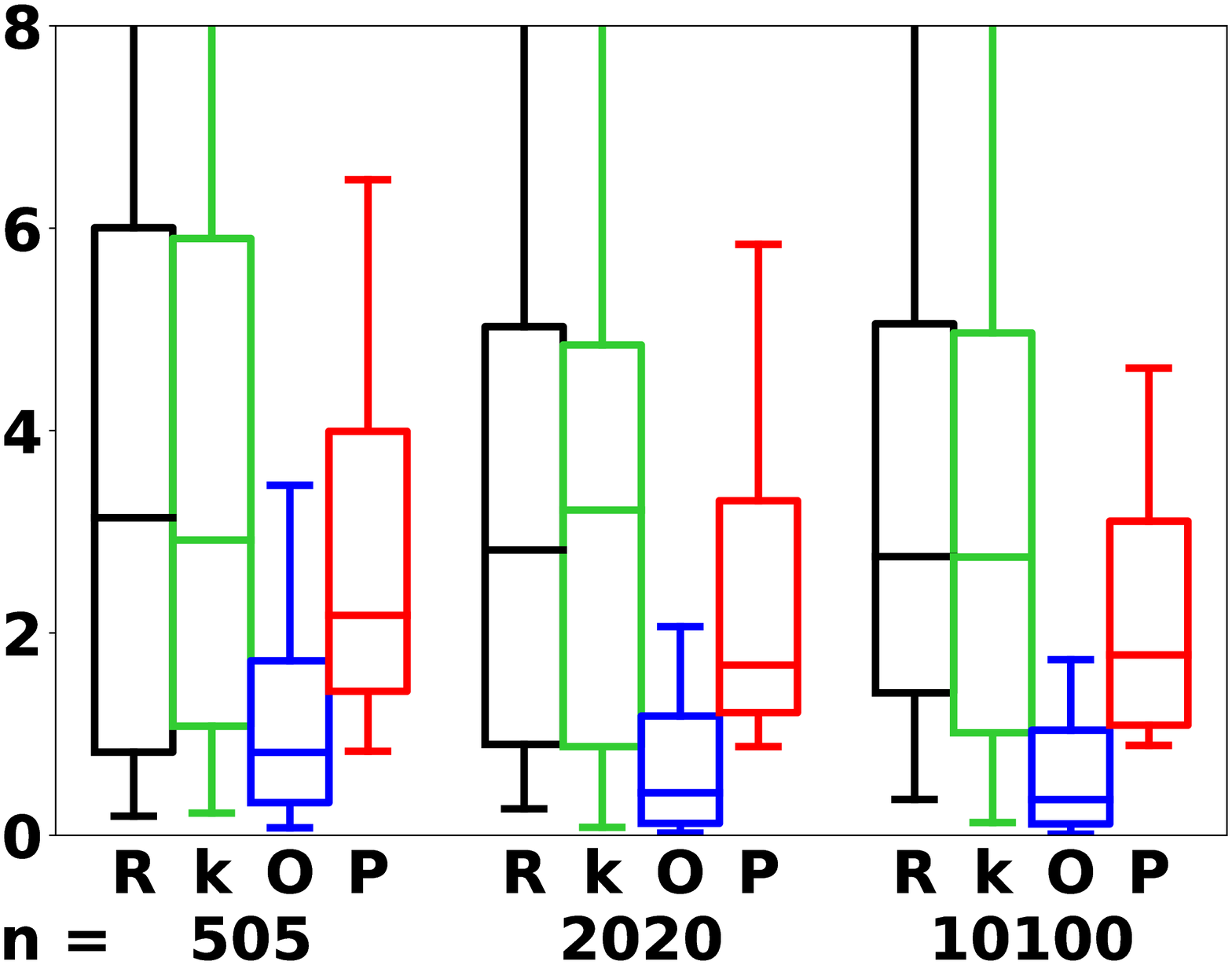}
    \end{subfigure}\\
    \begin{subfigure}[t]{0.33\textwidth}
        \centering
        \includegraphics[width=\textwidth]{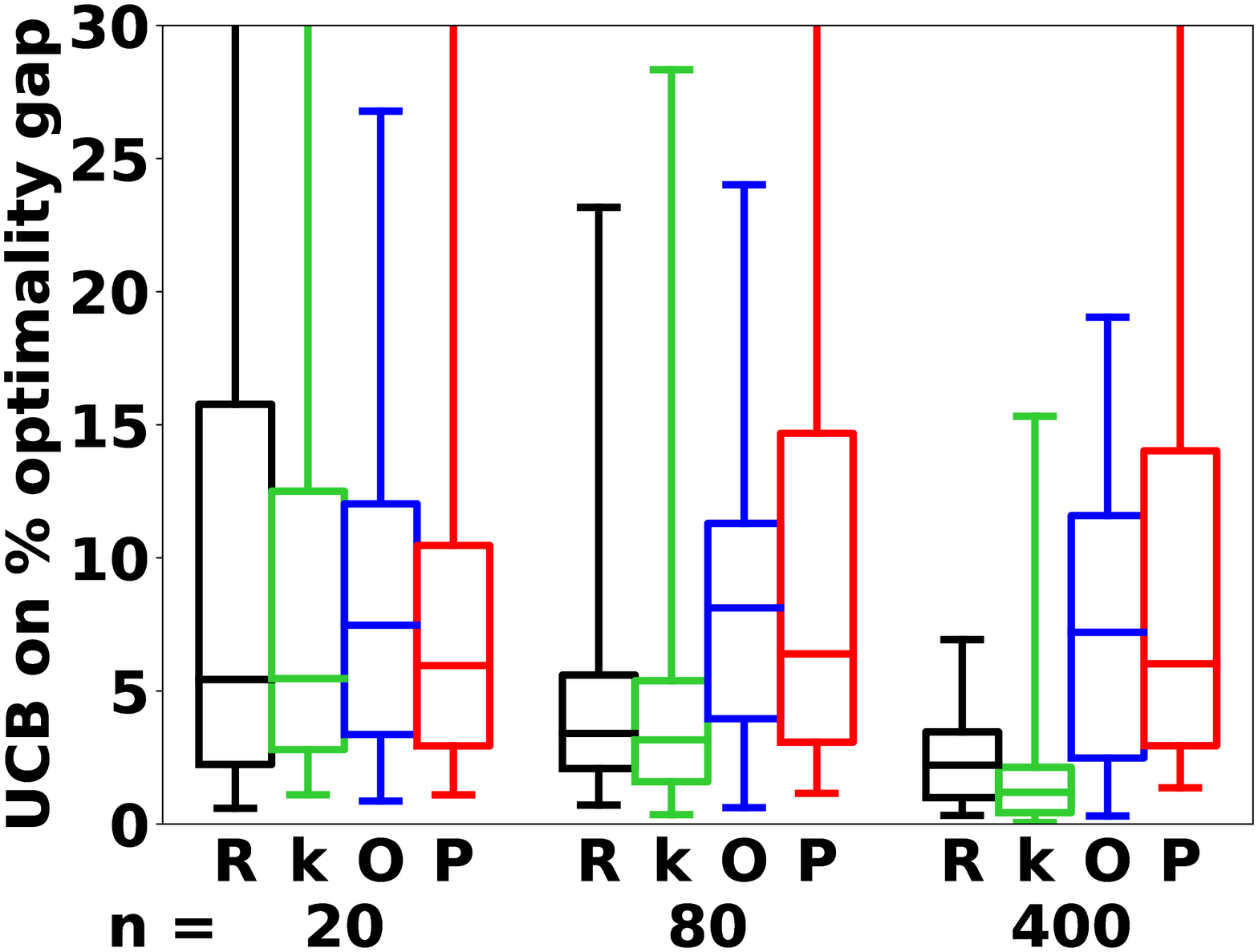}
    \end{subfigure}%
    ~ 
    \begin{subfigure}[t]{0.33\textwidth}
        \centering
        \includegraphics[width=\textwidth]{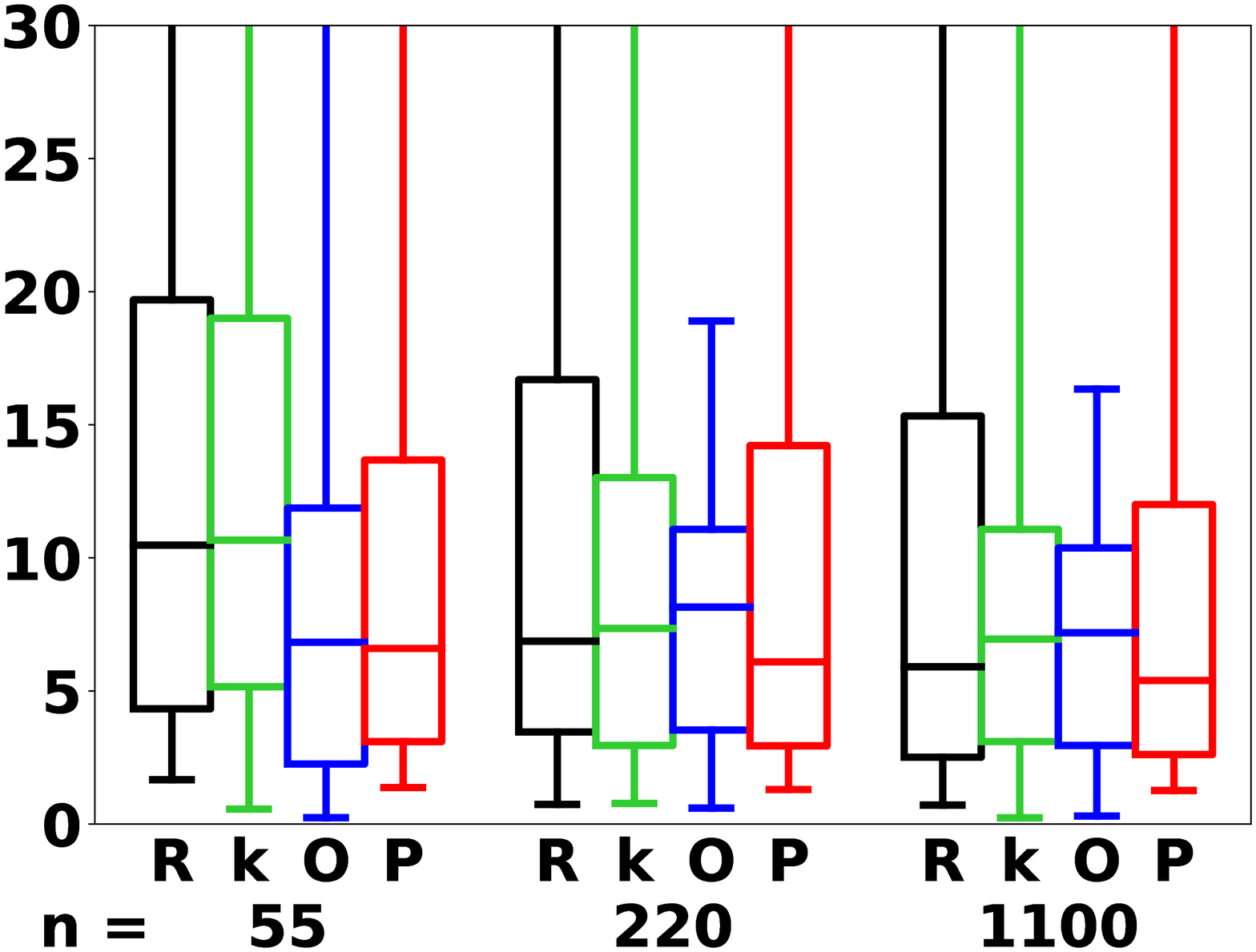}
    \end{subfigure}%
    ~ 
    \begin{subfigure}[t]{0.33\textwidth}
        \centering
        \includegraphics[width=\textwidth]{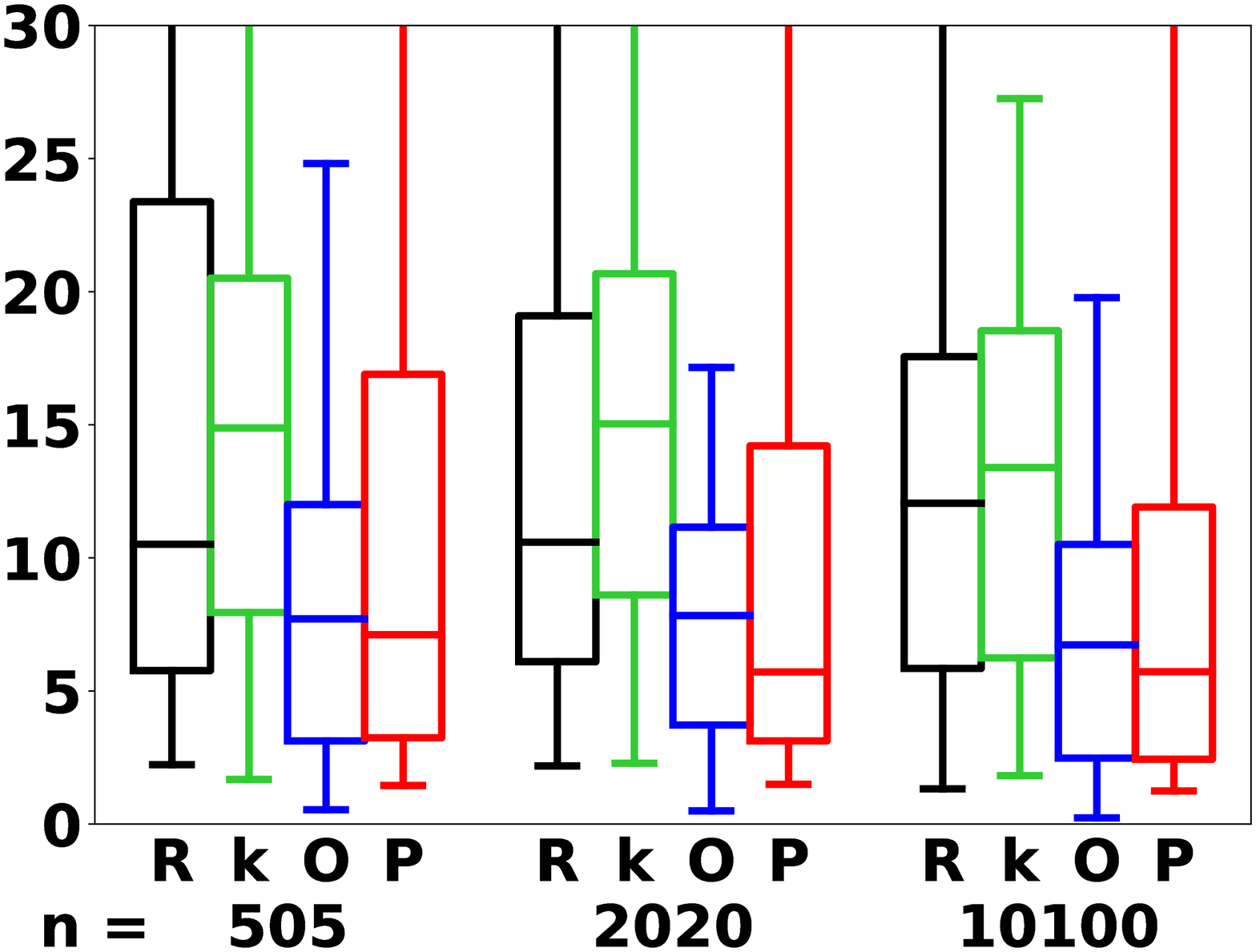}
    \end{subfigure}
    \caption{{Comparison of kNN-SAA (\texttt{R}), ER-SAA+kNN (\texttt{k}), ER-SAA+OLS (\texttt{O}), and PP+OLS (\texttt{P}) approaches for $\omega = 1$. Top row: $p = 1$. Middle row: $p = 0.5$. Bottom row: $p = 2$. Left column: $d_x = 3$. Middle column: $d_x = 10$. Right column: $d_x = 100$.}}
    \label{fig:comp_ols}
\end{figure}

\paragraph{Impact of the {jackknife}-based formulations.} 
Figure~\ref{fig:comp_jack_lasso} compares the performance of the ER-SAA and J-SAA approaches with OLS regression by varying the model degree $p$, the covariate dimension among $d_x \in \{10,100\}$, and the sample size among $n \in \{{1.3(d_x + 1)}, 1.5(d_x + 1), {2(d_x + 1)}\}$ {for $\omega = 1$}.
{Once again, we directly assume $\hat{Q}_n \equiv I$.}
We employ smaller sample sizes in these experiments to see if the {jackknife}-based SAAs perform better in the limited data regime.
We observe that the solutions obtained from the J-SAA formulation typically have smaller $75^{\text{th}}$ and $95^{\text{th}}$ percentiles of the $99\%$ UCBs than those from the ER-SAA formulation, particularly when the sample size $n$ is small.
Performance gains are more pronounced for larger sample sizes when the covariate dimension is larger ($d_x = 100$), possibly because the OLS estimators overfit more.
Note that, as expected, the J-SAA results converge to the ER-SAA results when the sample size increases.
We do not plot the results for the J+-SAA formulation because they are similar to those of the J-SAA formulation.

\begin{figure}[t!]
    \centering
    \begin{subfigure}[t]{0.33\textwidth}
        \centering
        \includegraphics[width=\textwidth,right]{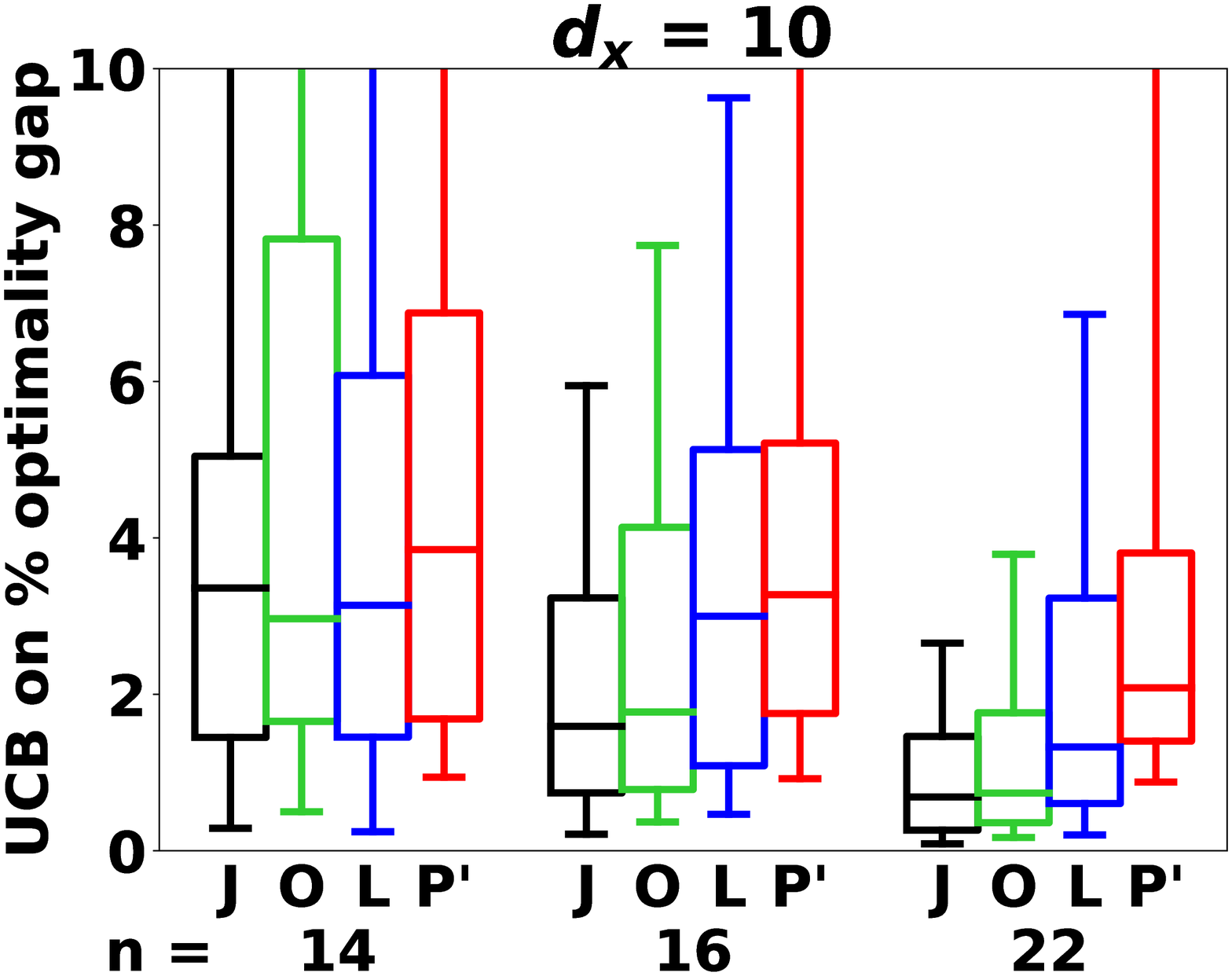}
    \end{subfigure}%
    ~ 
    \begin{subfigure}[t]{0.33\textwidth}
        \centering
        \includegraphics[width=\textwidth,left]{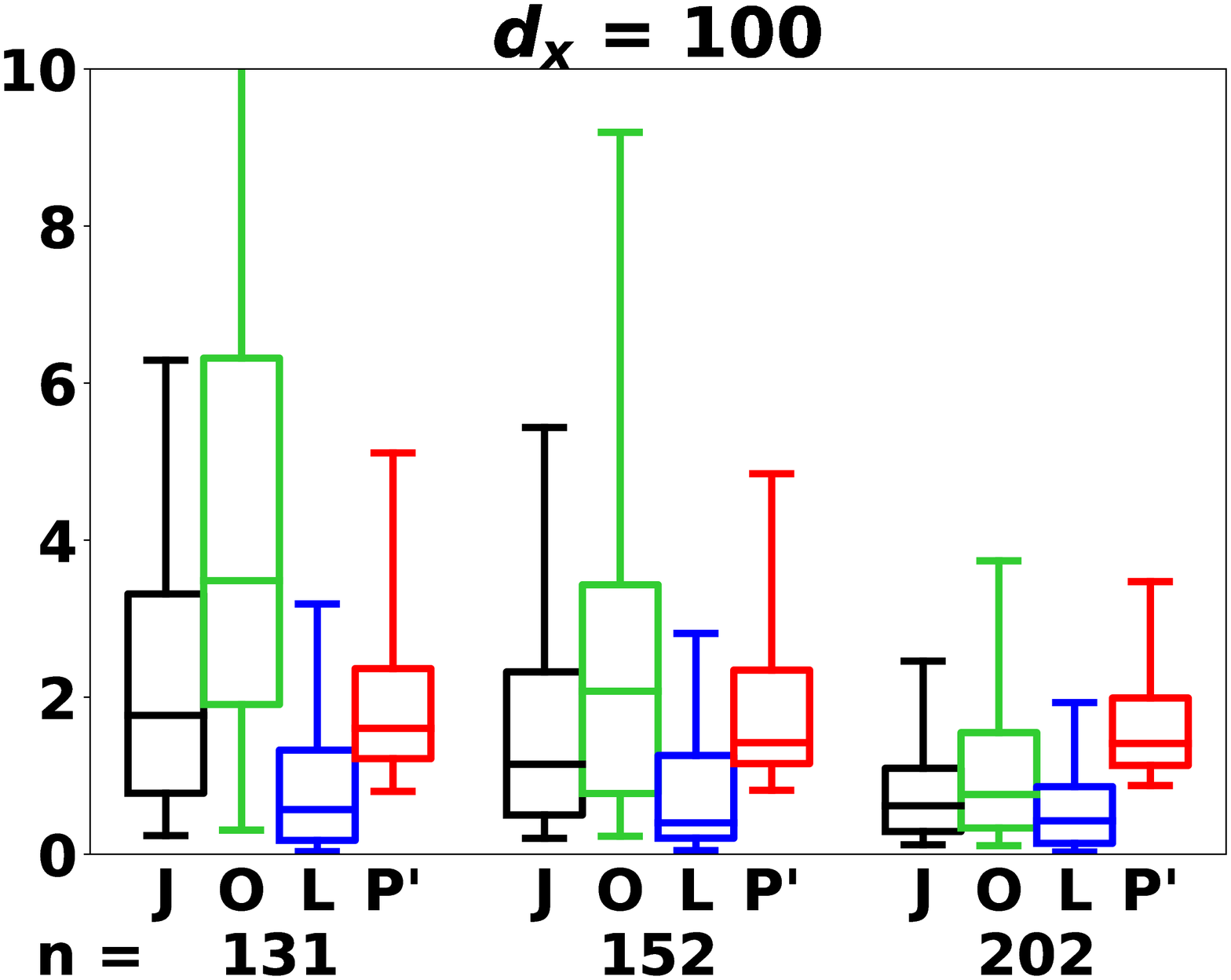}
    \end{subfigure}\\
    \begin{subfigure}[t]{0.33\textwidth}
        \centering
        \includegraphics[width=\textwidth,right]{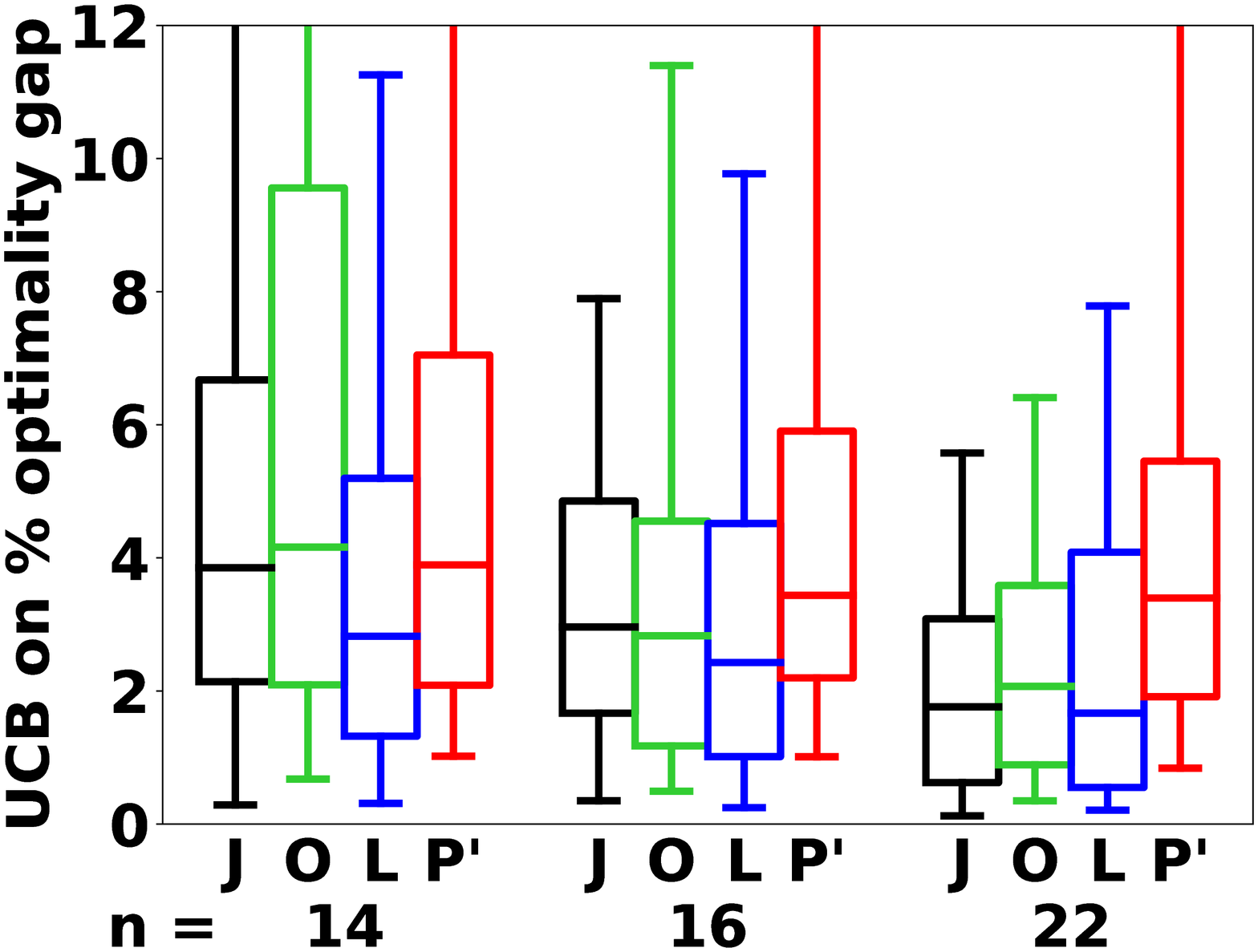}
    \end{subfigure}%
    ~
    \begin{subfigure}[t]{0.33\textwidth}
        \centering
        \includegraphics[width=\textwidth,left]{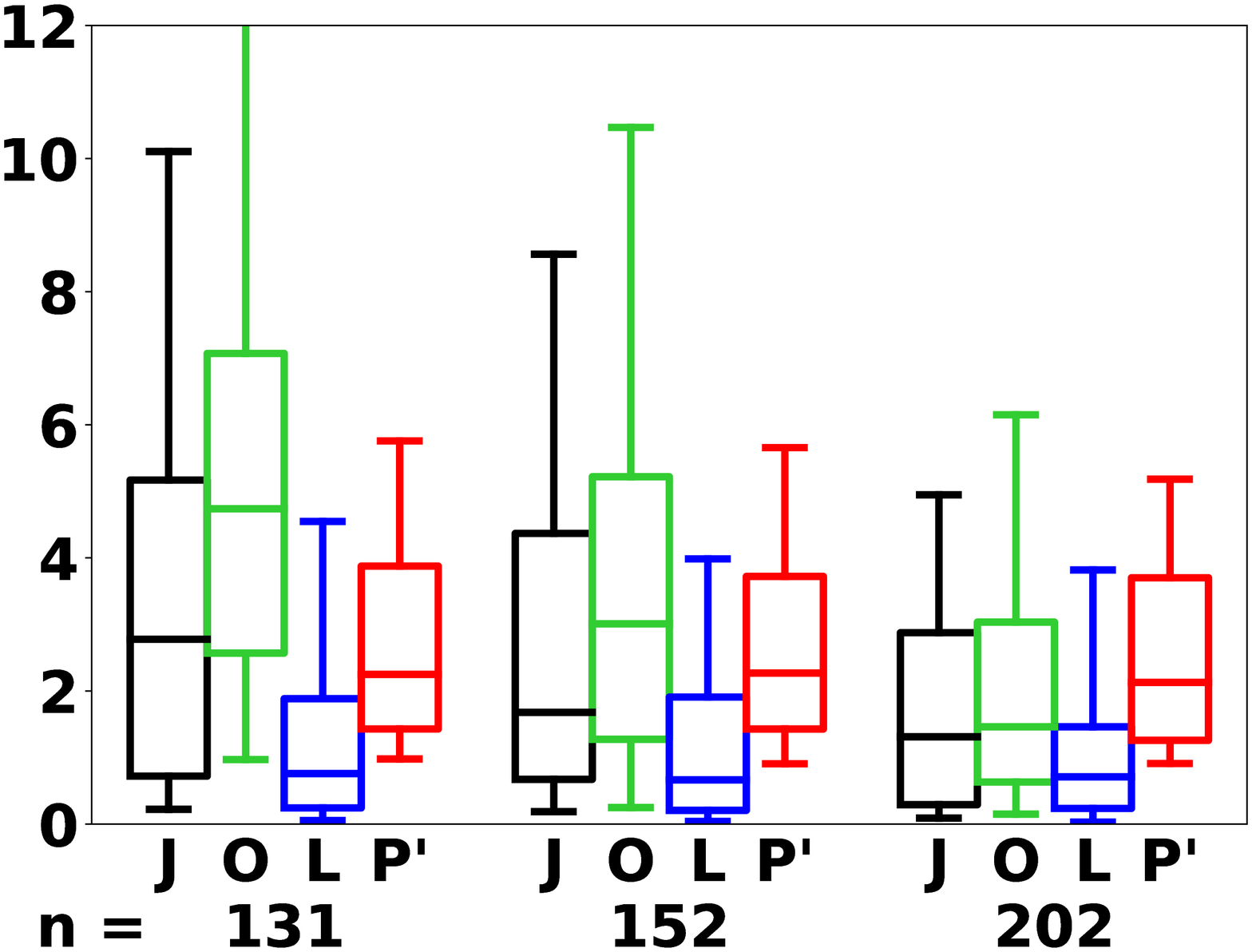}
    \end{subfigure}\\
    \begin{subfigure}[t]{0.33\textwidth}
        \centering
        \includegraphics[width=\textwidth,right]{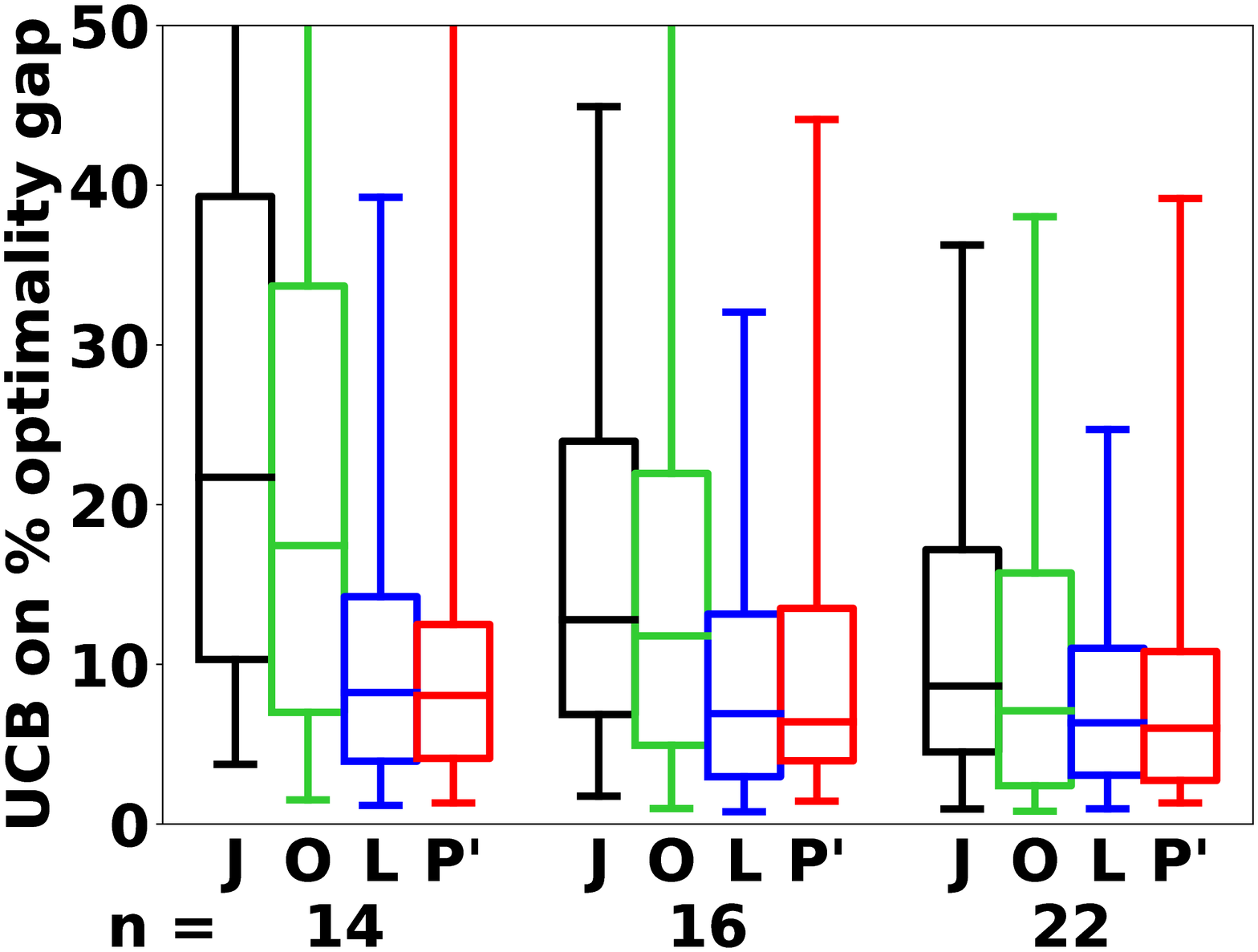}
    \end{subfigure}%
    ~ 
    \begin{subfigure}[t]{0.33\textwidth}
        \centering
        \includegraphics[width=\textwidth,left]{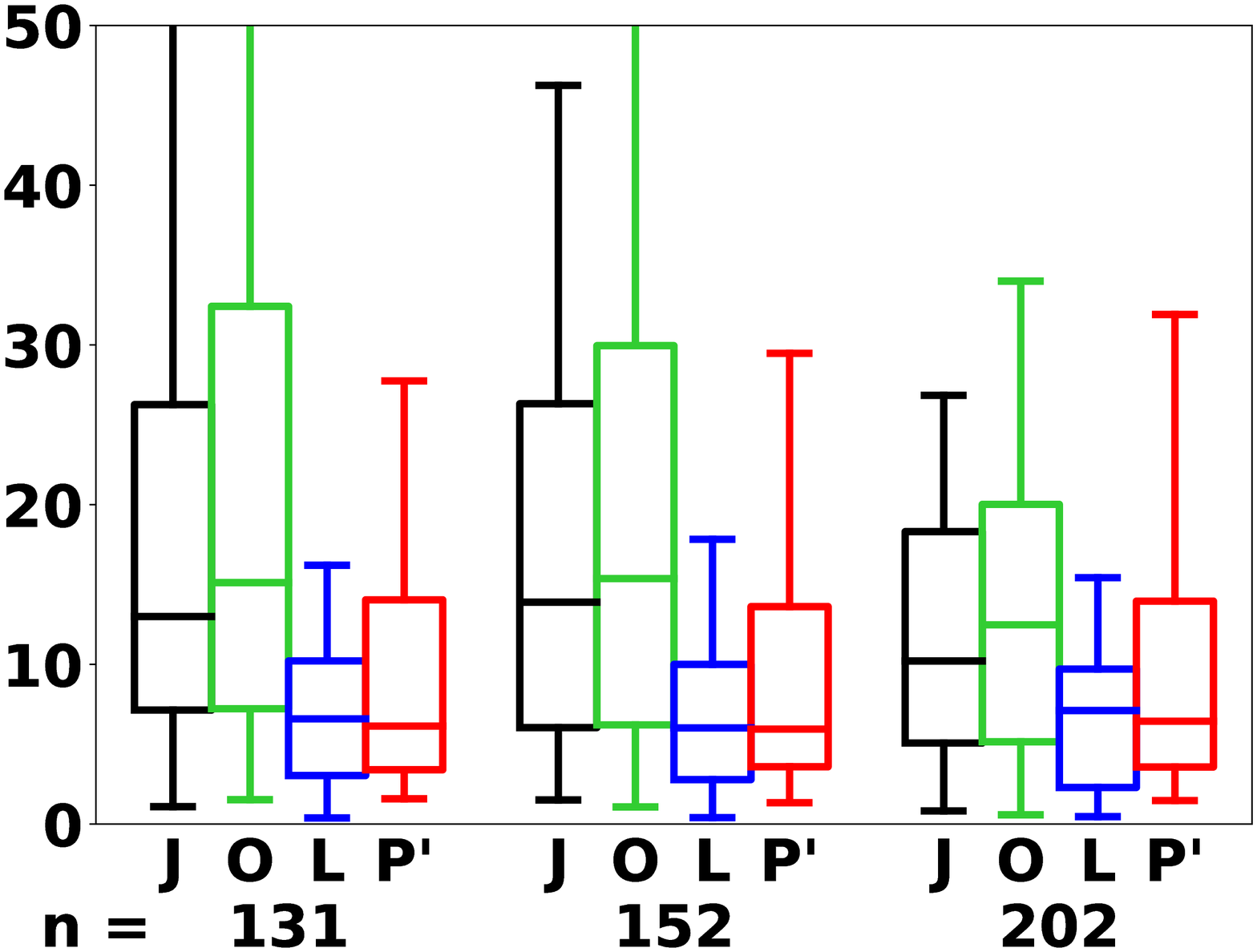}
    \end{subfigure}
    \caption{{Comparison of J-SAA+OLS (\texttt{J}), ER-SAA+OLS (\texttt{O}), ER-SAA+Lasso (\texttt{L}), and PP+Lasso ($\texttt{P}'$) approaches for $\omega = 1$. Top row: $p = 1$. Middle row: $p = 0.5$. Bottom row: $p = 2$. Left column: $d_x = 10$. Right column: $d_x = 100$.}}
    \label{fig:comp_jack_lasso}
\end{figure}

\paragraph{Impact of the prediction setup.} 
Figure~\ref{fig:comp_jack_lasso} also compares the performance of the ER-SAA+OLS, ER-SAA+Lasso, {and PP+Lasso} approaches by varying the model degree $p$, the covariate dimension among $d_x \in \{10,100\}$, and the sample size among $n \in \{{1.3(d_x + 1)}, 1.5(d_x + 1), {2(d_x + 1)}\}$ {\mbox{for $\omega = 1$}}.
{We again assume $\hat{Q}_n \equiv I$ in the ER-SAA approaches.}
We observe that the ER-SAA+Lasso formulation yields better estimators than the ER-SAA+OLS formulation when the sample size $n$ is small relative to the covariate dimension $d_x$.
This effect is accentuated when the covariate dimension is larger ($d_x = 100$), in which case the OLS-based estimators overfit more and there is increased benefit in using the Lasso to fit a sparser model.
The advantage of the Lasso-based estimators shrinks as the sample size increases.
{The ER-SAA+Lasso approach outperforms the PP+Lasso approach with increased gains for larger sample sizes and covariate dimensions.}

\paragraph{Impact of the error variance.} 
Figure~\ref{fig:sigma_plots} {and Figure~\ref{fig:sigma_plots_2} in Section~\ref{sec:computexp-details} of the Appendix} compare the performance of the kNN-SAA, {ER-SAA+kNN}, ER-SAA+OLS, and {PP+OLS} approaches by varying the standard deviation of the errors~$\varepsilon$ among $\sigma \in \{5,10,20\}$ (the case studies thus far used $\sigma = 5$), the model degree $p$, and the sample size among $n \in \{{5(d_x + 1)}, 20(d_x + 1), 100(d_x + 1)\}$ for $d_x = 10$ and {$\omega = 1$}.
We observe that the ER-SAA+OLS formulation needs a larger sample size to yield a similar certificate of optimality as the standard deviation $\sigma$ increases.
On the other hand, the performance of the {ER-SAA+kNN and} kNN-SAA approaches appear to be unaffected (and even slightly improve!) with increasing error variance.
{The performance of the PP+OLS approach deteriorates significantly with increasing $\sigma$, especially for $p \neq 2$, and it no longer dominates the ER-SAA+kNN and kNN-SAA approaches for $\sigma = 20$.}

\begin{figure}[t!]
    \centering
    \begin{subfigure}[t]{0.33\textwidth}
        \centering
        \includegraphics[width=\textwidth]{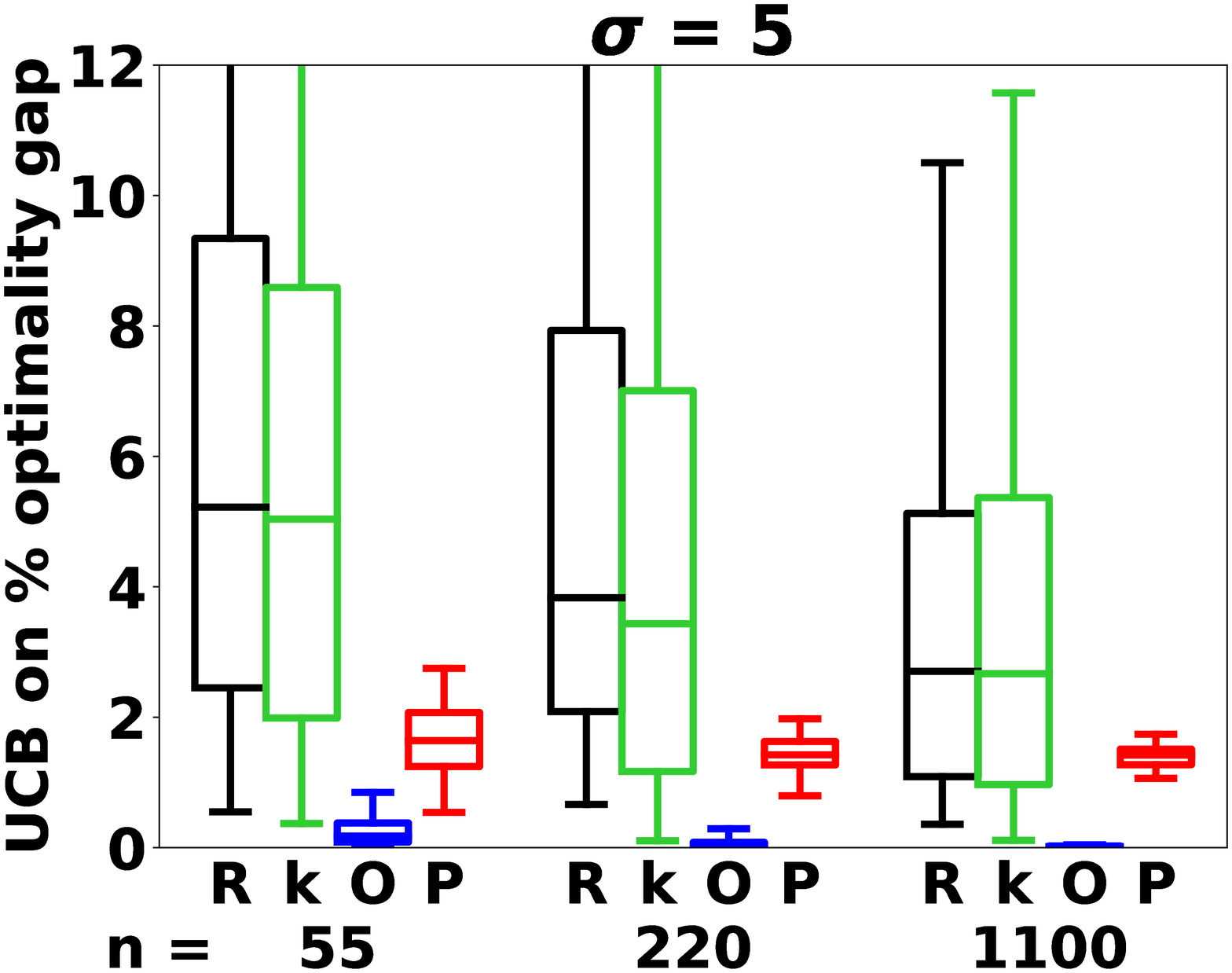}
    \end{subfigure}%
    ~ 
    \begin{subfigure}[t]{0.33\textwidth}
        \centering
        \includegraphics[width=\textwidth]{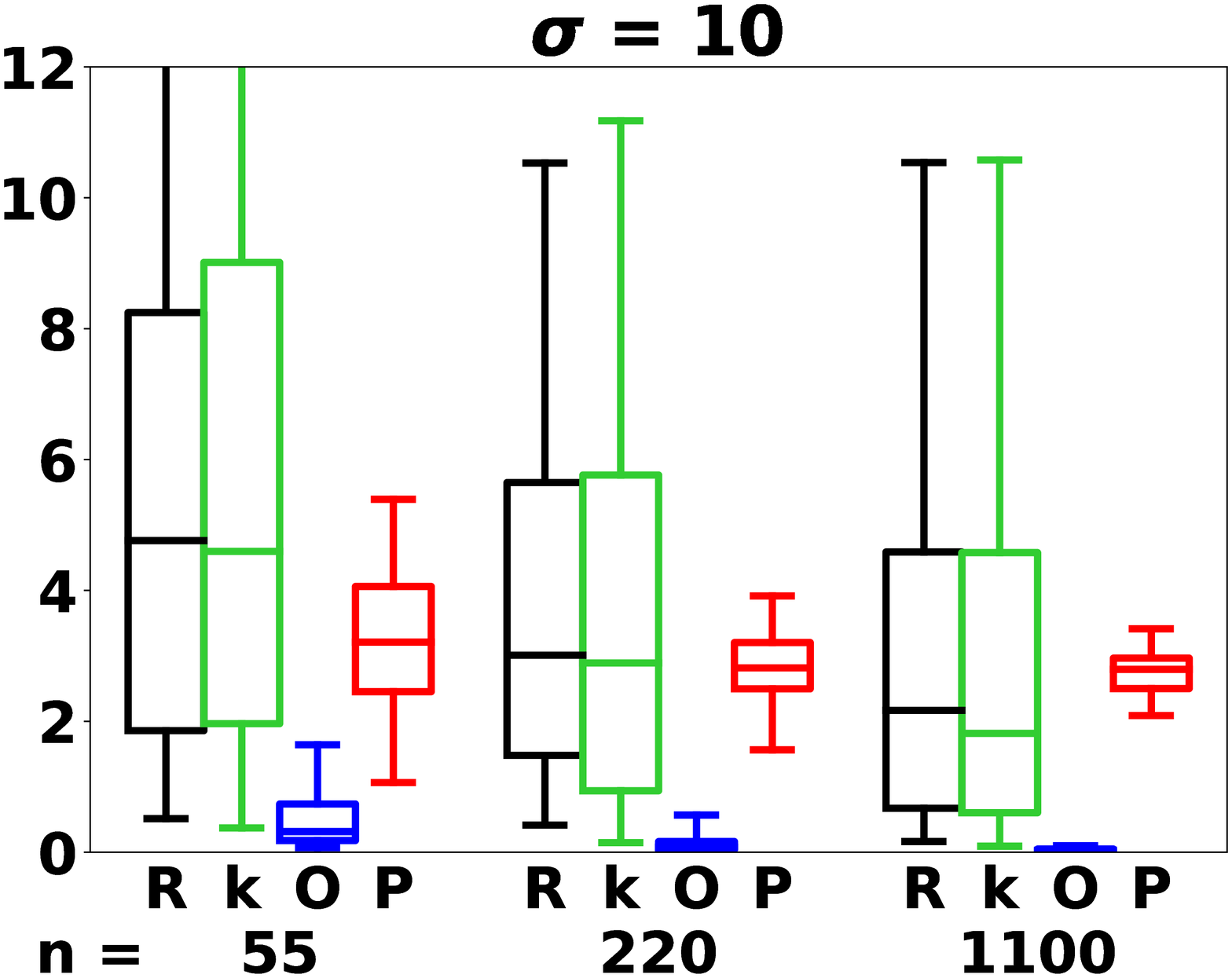}
    \end{subfigure}%
    ~ 
    \begin{subfigure}[t]{0.33\textwidth}
        \centering
        \includegraphics[width=\textwidth]{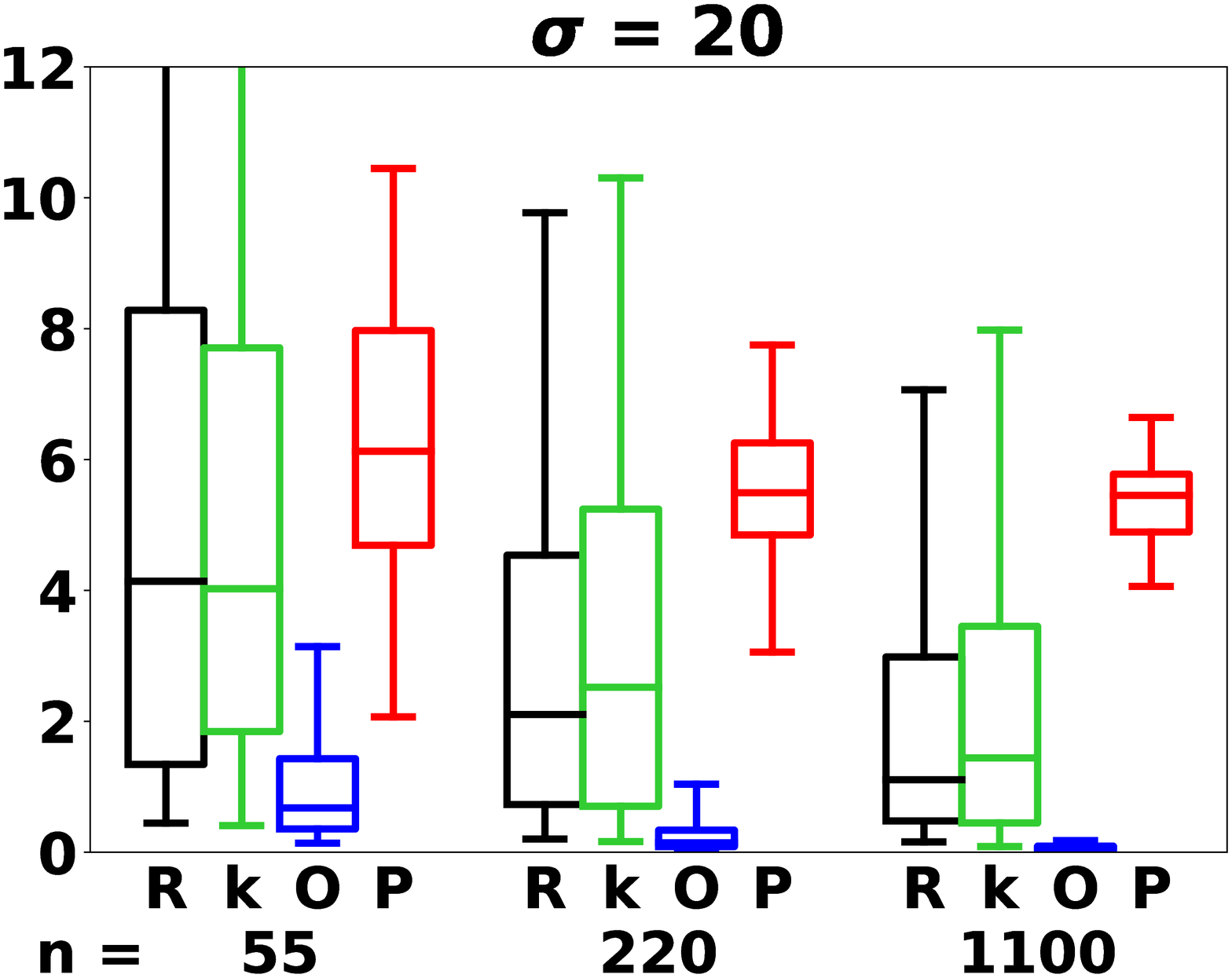}
    \end{subfigure}
    \caption{{Effect of increasing $\sigma$ on kNN-SAA (\texttt{R}), ER-SAA+kNN (\texttt{k}), ER-SAA+OLS (\texttt{O}), and PP+OLS (\texttt{P}) approaches when $d_x = 10$, $p = 1$, and $\omega = 1$.}}
    \label{fig:sigma_plots}
\end{figure}

\paragraph{{Impact of heteroscedasticity.}} 
{Figures~\ref{fig:het_plots} and~\ref{fig:het_plots_2} compare the performance of the kNN-SAA approach, ER-SAA+kNN approaches with and without estimation of $Q^*$, ER-SAA+OLS approaches with and without estimation of $Q^*$, and the PP+OLS approach by varying the heteroscedasticity level $\omega \in \{1,2,3\}$, the model degree $p$, and the sample size among $n \in \{5(d_x + 1), 20(d_x + 1), 100(d_x + 1)\}$ for $d_x = 10$ and $\sigma = 5$.
We assume $\hat{Q}_n \equiv I$ in the ER-SAA approaches that do not estimate $Q^*$.
The kNN-SAA approach and the ER-SAA+kNN approach without heteroscedasticity estimation again exhibit similar performance across the different test instances, with the former performing slightly better on instances with severe heteroscedasticity ($\omega = 3$) and the latter performing slightly better on instances with zero and moderate heteroscedasticity ($\omega \in \{1,2\}$). 
The ER-SAA+kNN approach with heteroscedasticity estimation outperforms the kNN-SAA approach and the ER-SAA+kNN approach without heteroscedasticity estimation, particularly for model degree $p = 2$ and for severe heteroscedasticity ($\omega = 3$).
The ER-SAA+OLS approach without heteroscedasticity estimation outperforms the kNN-SAA approach and the ER-SAA+kNN approaches for $\omega = 1$ and $\omega = 2$ irrespective of the model degree $p$.
However, the kNN-SAA approach outperforms the ER-SAA+OLS approach without heteroscedasticity estimation for $\omega = 3$, $p \in \{0.5,1\}$, and small sample sizes, with the reverse holding true for large sample sizes.
Similar to the observation for kNN regression, the ER-SAA+OLS approach with estimation of $Q^*$ outperforms the ER-SAA+OLS approach without estimation of $Q^*$ in several instances, particularly for large sample sizes and for severe heteroscedasticity ($\omega = 3$).
Finally, the ER-SAA+OLS approaches demonstrate significant gains over the PP+OLS approach, especially when the degree of model misspecification is low ($p = 1$ or $p = 0.5$).
The results for the PP+OLS and ER-SAA+OLS approaches for $\omega = 3$ may be partly explained by an increase in the variance of the error terms (cf.\ Figure~\ref{fig:het_stats}).
Note that the median values of the $99\%$ UCBs of the percentage optimality gaps of the N-SAA estimator (for $n = 10100$) for both $\omega = 2$ and $\omega = 3$ are roughly $11\%$, $5\%$, and $25\%$ for $p = 1$, $0.5$, and $2$, respectively, which underscores the benefit of using covariate information. 
}

\begin{figure}[t!]
    \centering
    \begin{subfigure}[t]{0.33\textwidth}
        \centering
        \includegraphics[width=\textwidth]{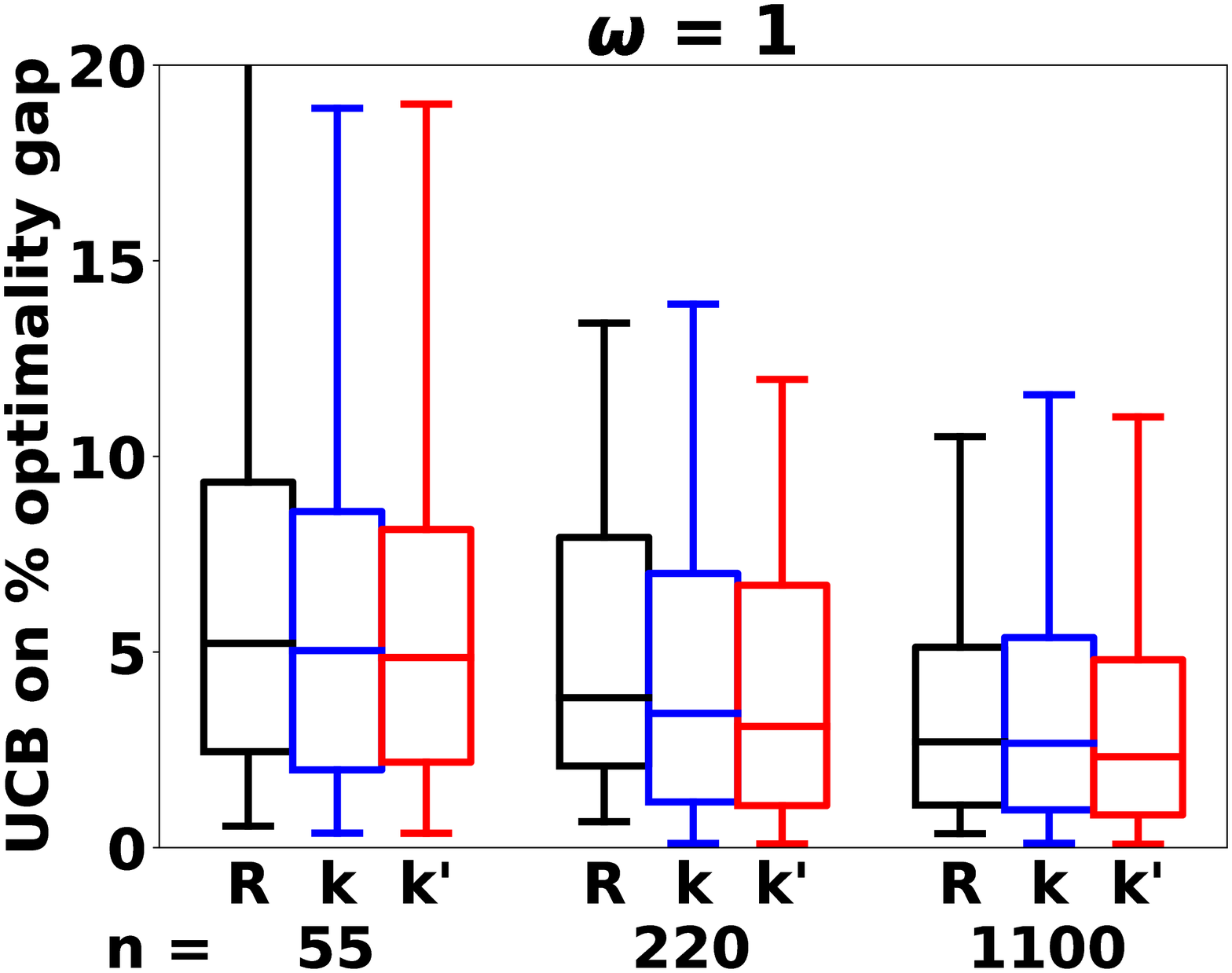}
    \end{subfigure}%
    ~ 
    \begin{subfigure}[t]{0.33\textwidth}
        \centering
        \includegraphics[width=\textwidth]{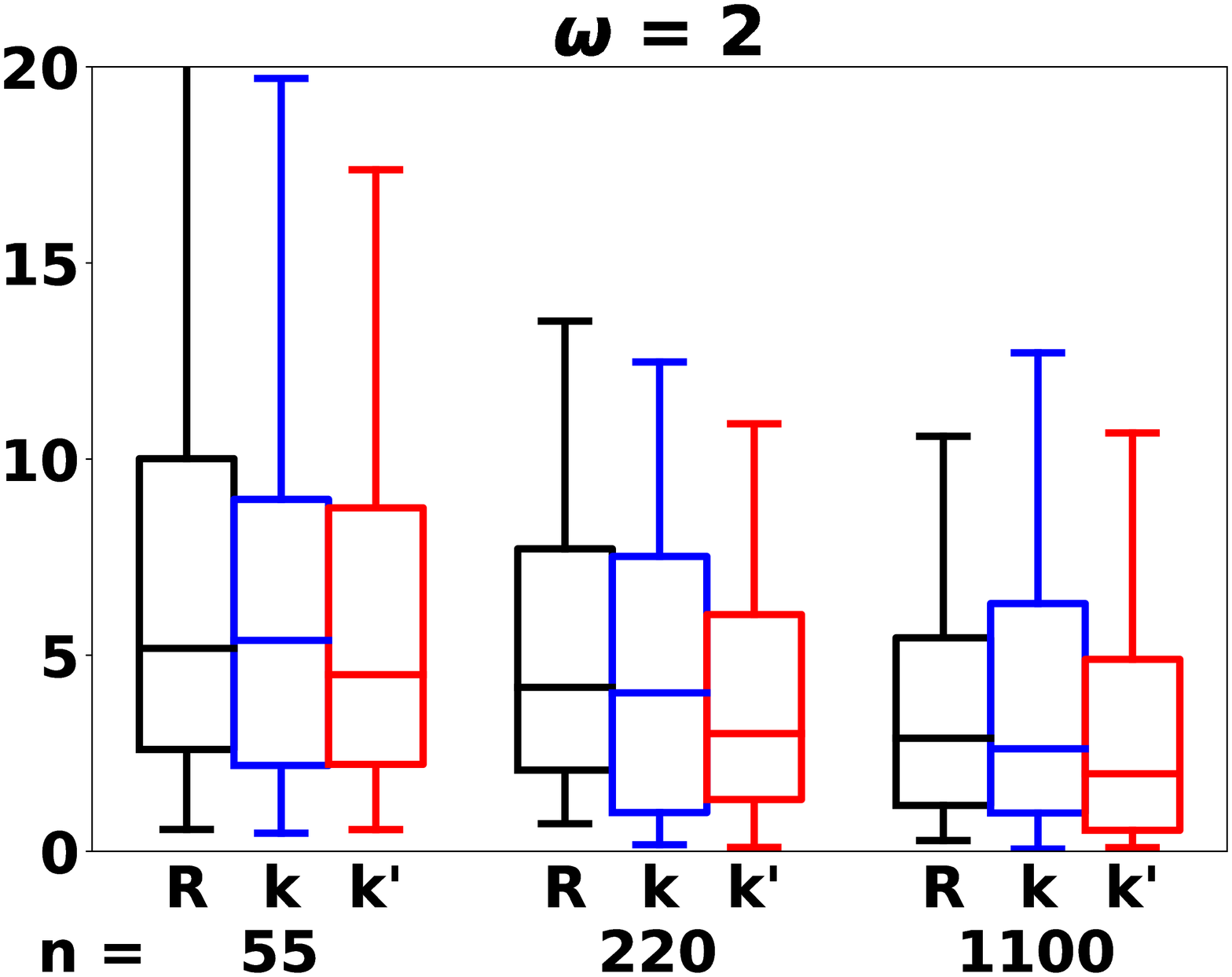}
    \end{subfigure}%
    ~ 
    \begin{subfigure}[t]{0.33\textwidth}
        \centering
        \includegraphics[width=\textwidth]{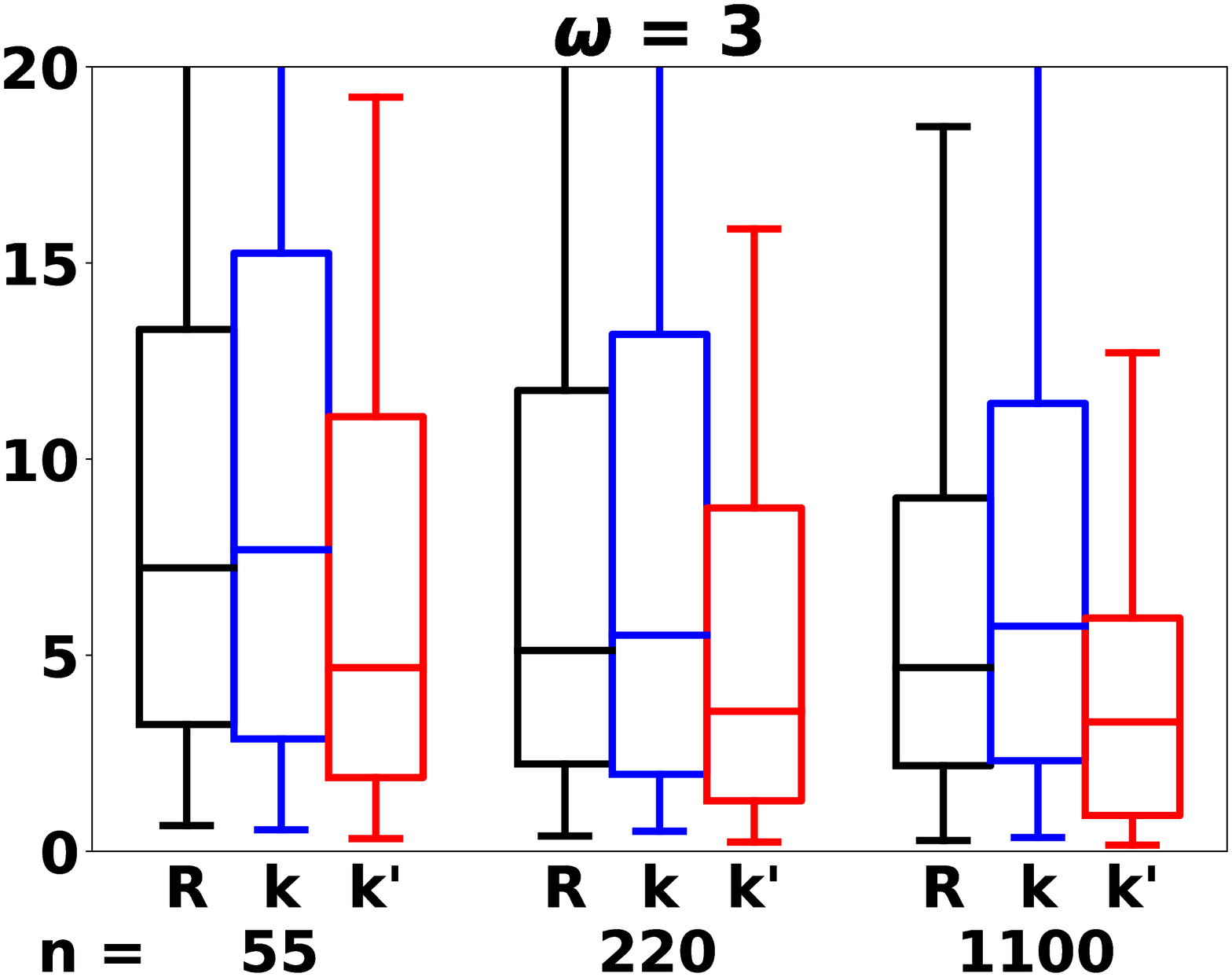}
    \end{subfigure}\\
    \begin{subfigure}[t]{0.33\textwidth}
        \centering
        \includegraphics[width=\textwidth]{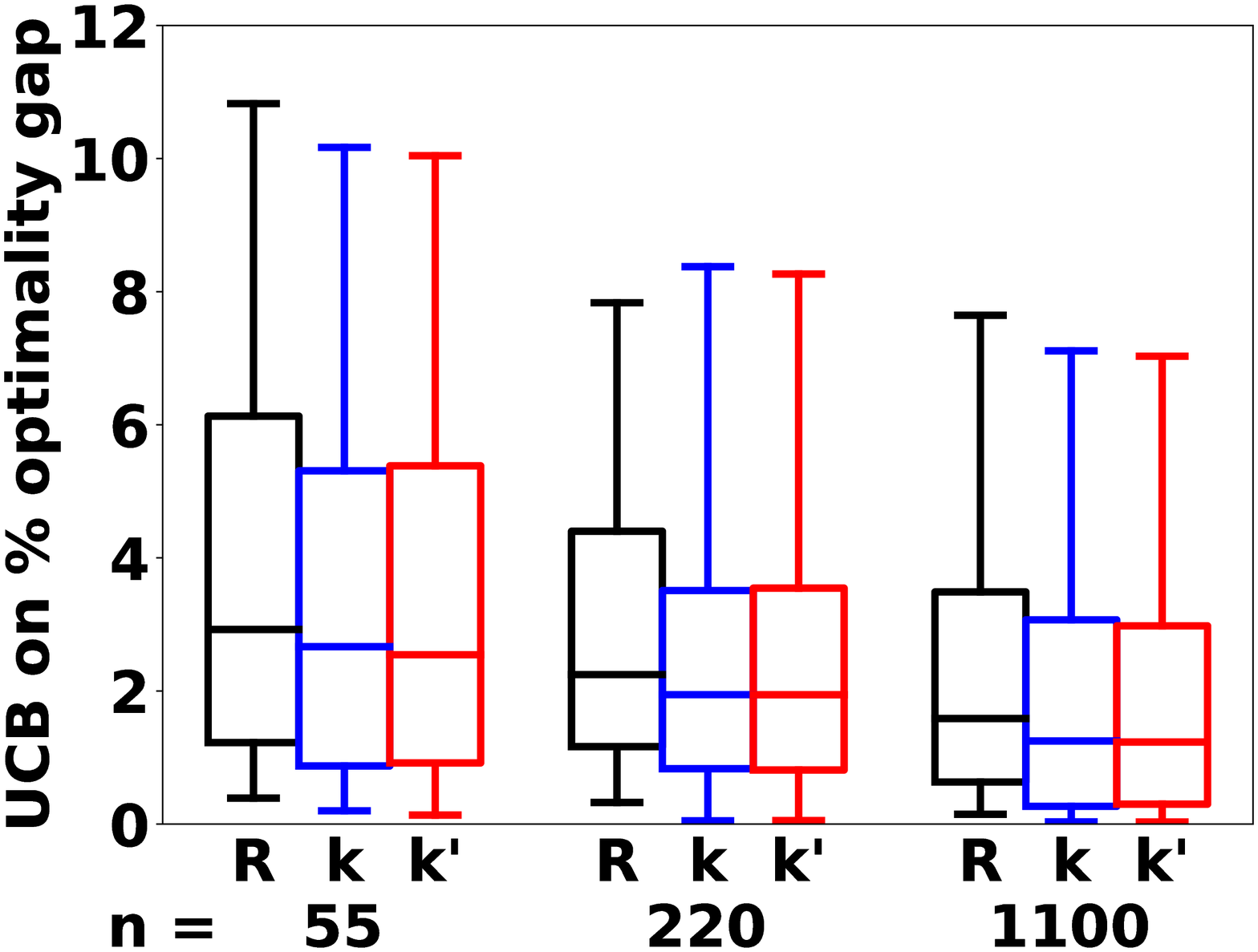}
    \end{subfigure}%
    ~ 
    \begin{subfigure}[t]{0.33\textwidth}
        \centering
        \includegraphics[width=\textwidth]{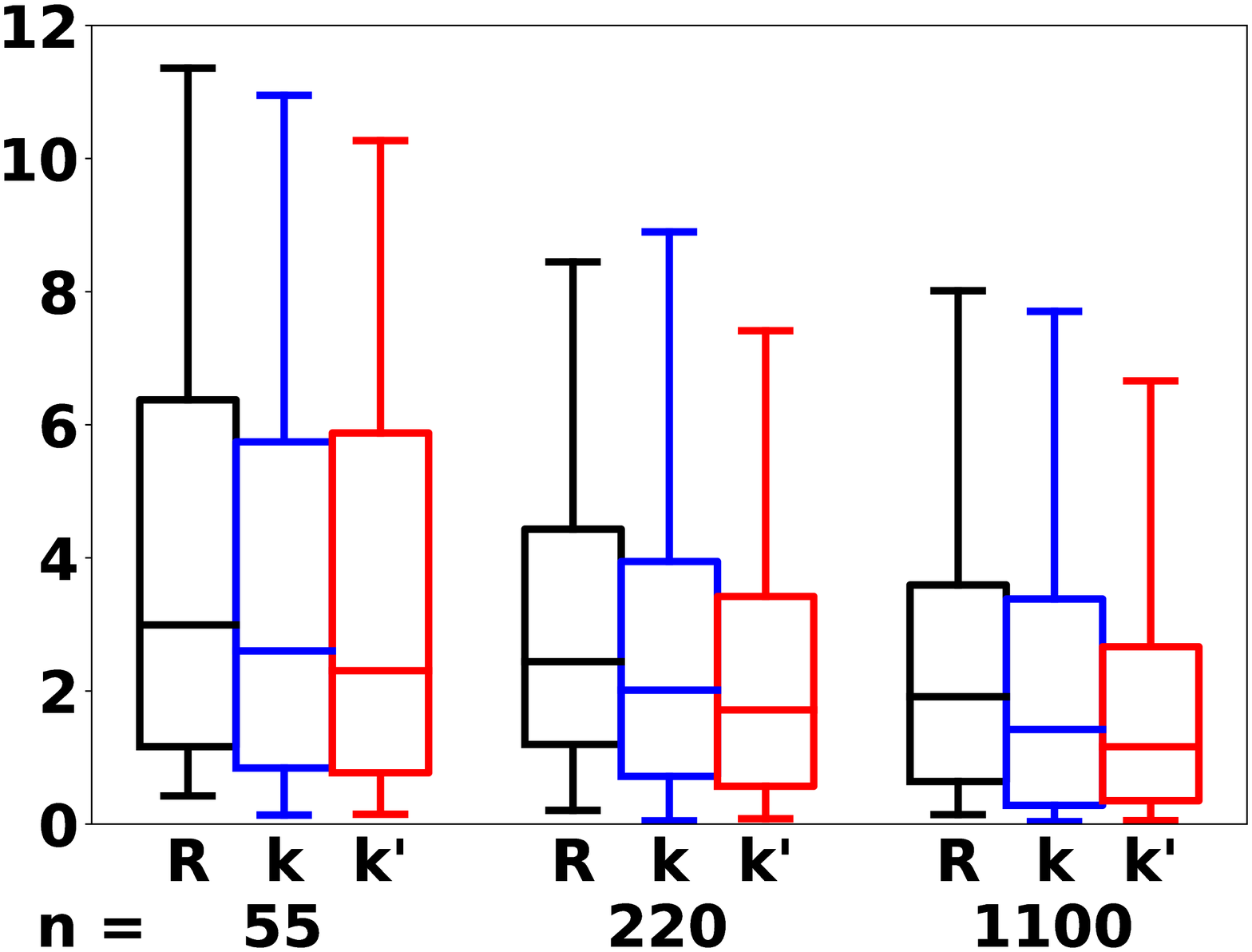}
    \end{subfigure}%
    ~ 
    \begin{subfigure}[t]{0.33\textwidth}
        \centering
        \includegraphics[width=\textwidth]{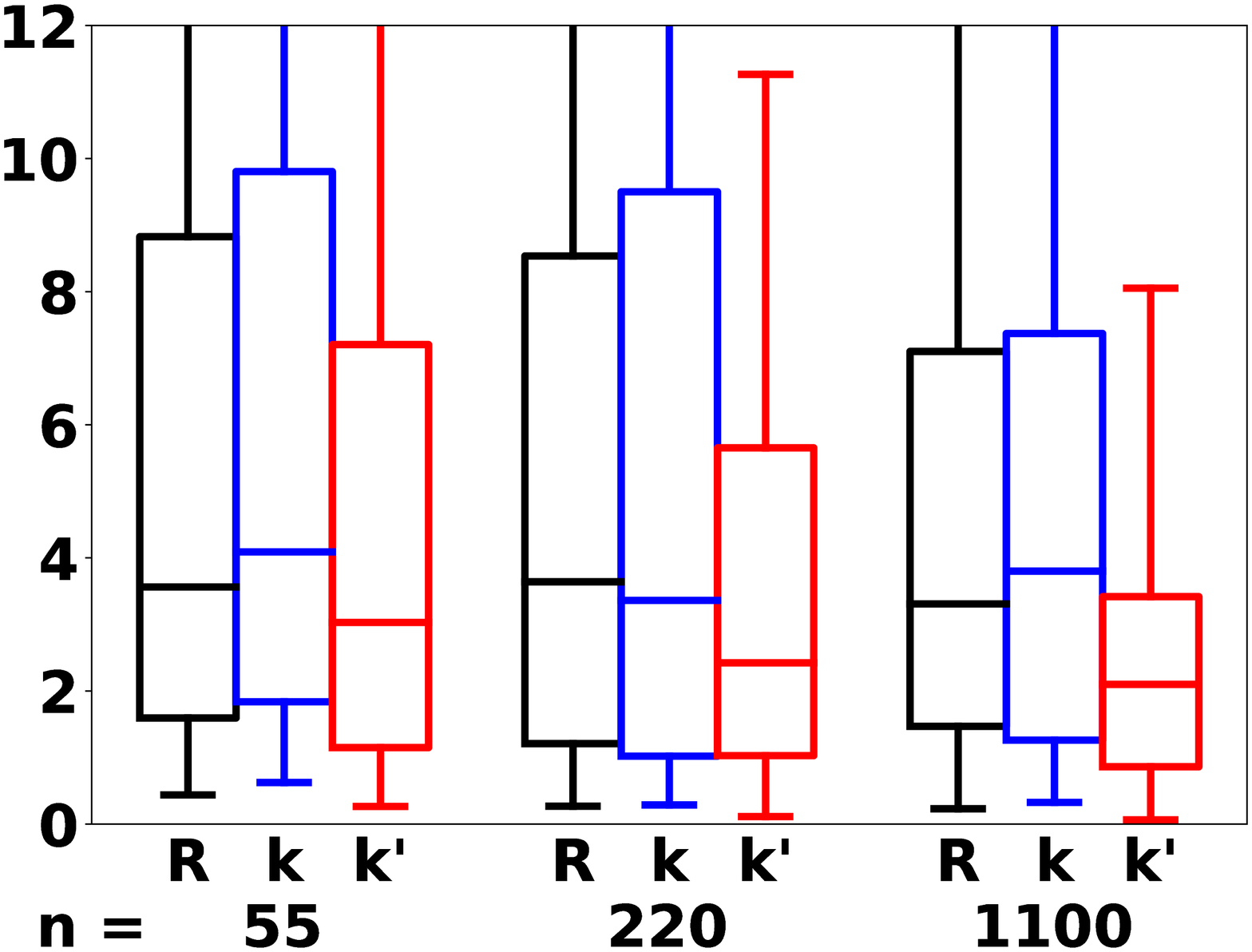}
    \end{subfigure}\\
    \begin{subfigure}[t]{0.33\textwidth}
        \centering
        \includegraphics[width=\textwidth]{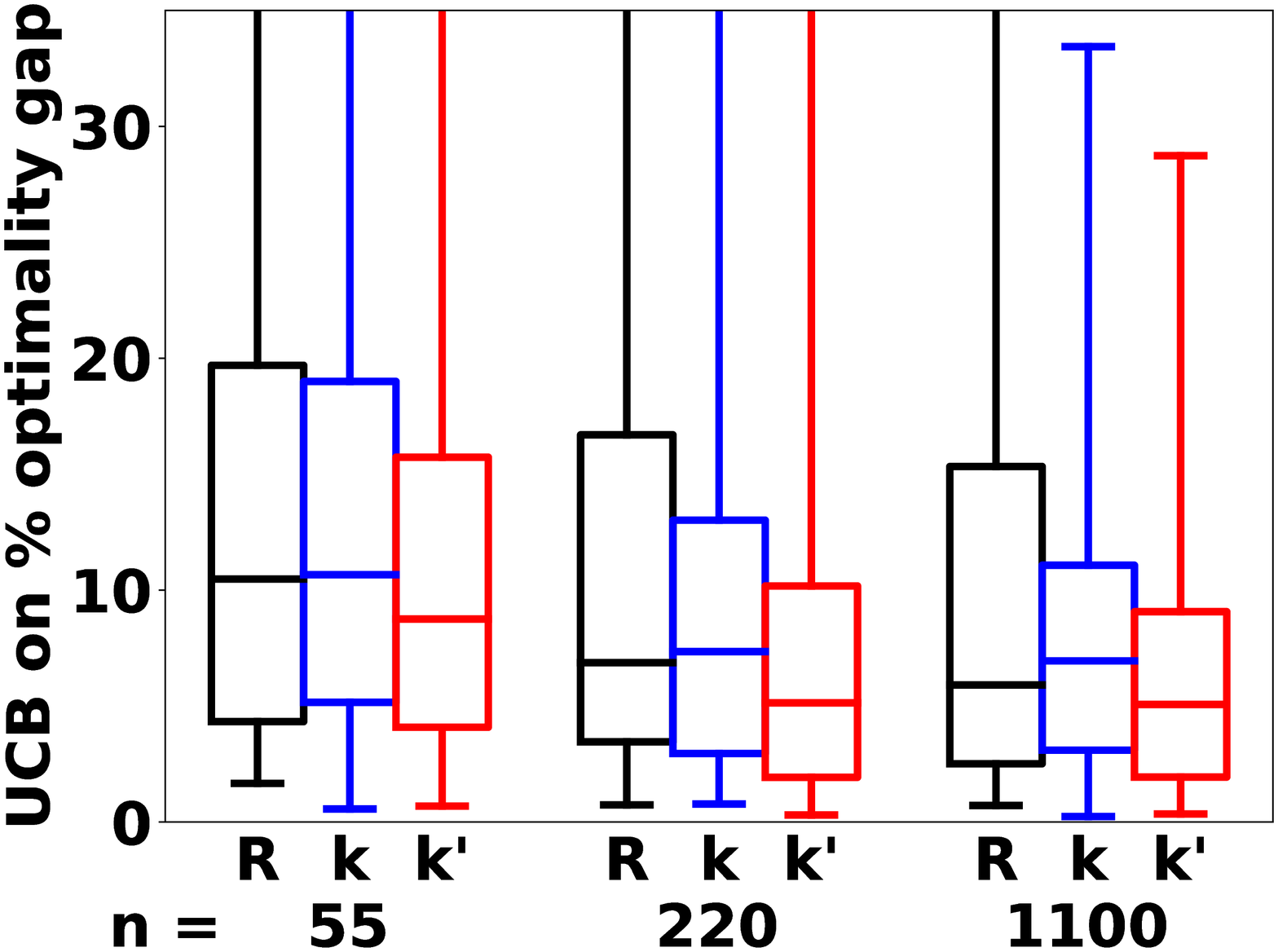}
    \end{subfigure}%
    ~ 
    \begin{subfigure}[t]{0.33\textwidth}
        \centering
        \includegraphics[width=\textwidth]{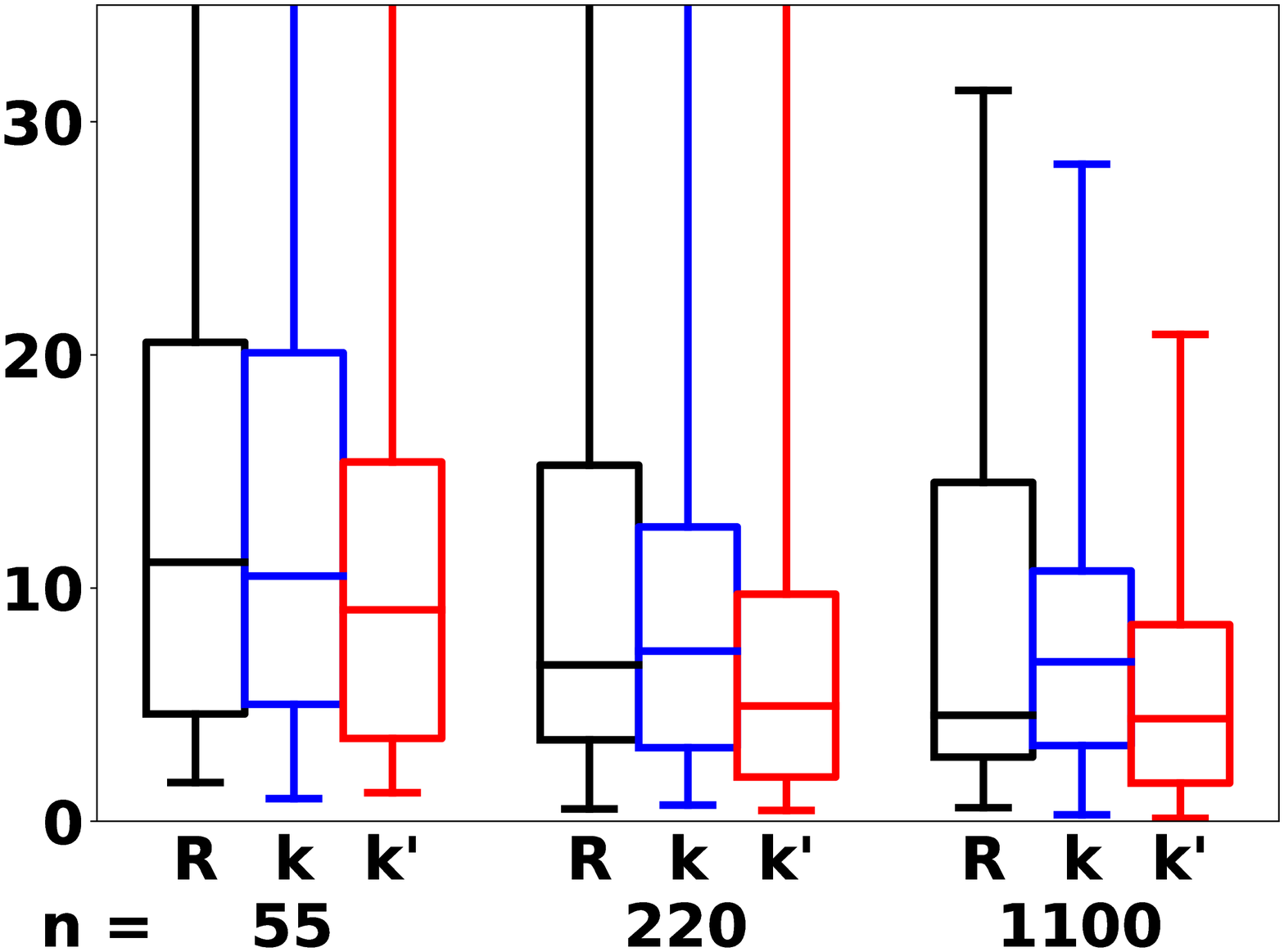}
    \end{subfigure}%
    ~ 
    \begin{subfigure}[t]{0.33\textwidth}
        \centering
        \includegraphics[width=\textwidth]{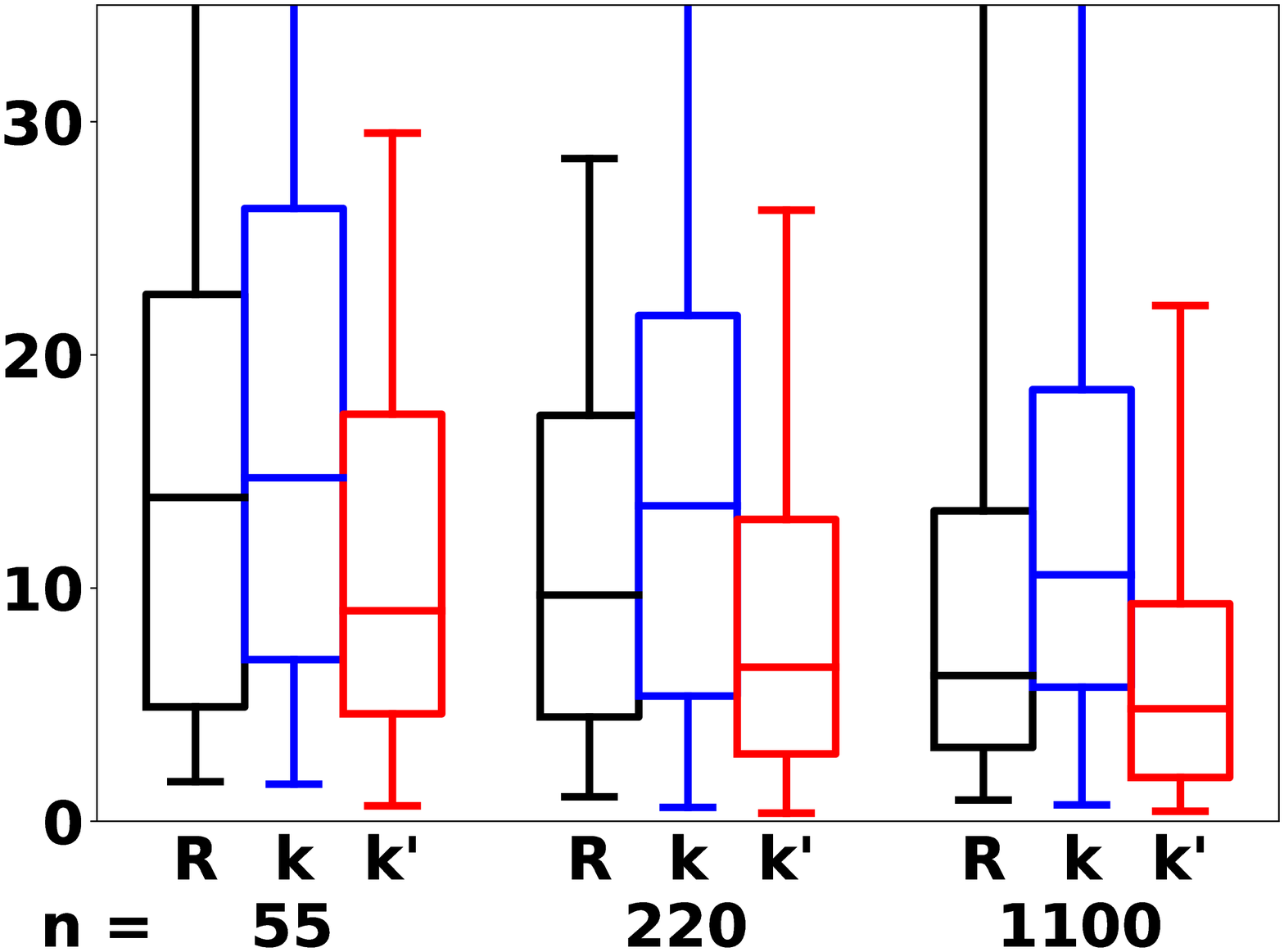}
    \end{subfigure}
    \caption{{Comparison of kNN-SAA (\texttt{R}), ER-SAA+kNN without heteroscedasticity estimation (\texttt{k}) and ER-SAA+kNN with heteroscedasticity estimation ($\texttt{k}'$) approaches for $d_x = 10$. Top row: $p = 1$. Middle row: $p = 0.5$. Bottom row: $p = 2$. Left column: $\omega = 1$. Middle column: $\omega = 2$. Right column: $\omega = 3$.}}
    \label{fig:het_plots}
\end{figure}

\begin{figure}[t!]
    \centering
    \begin{subfigure}[t]{0.33\textwidth}
        \centering
        \includegraphics[width=\textwidth]{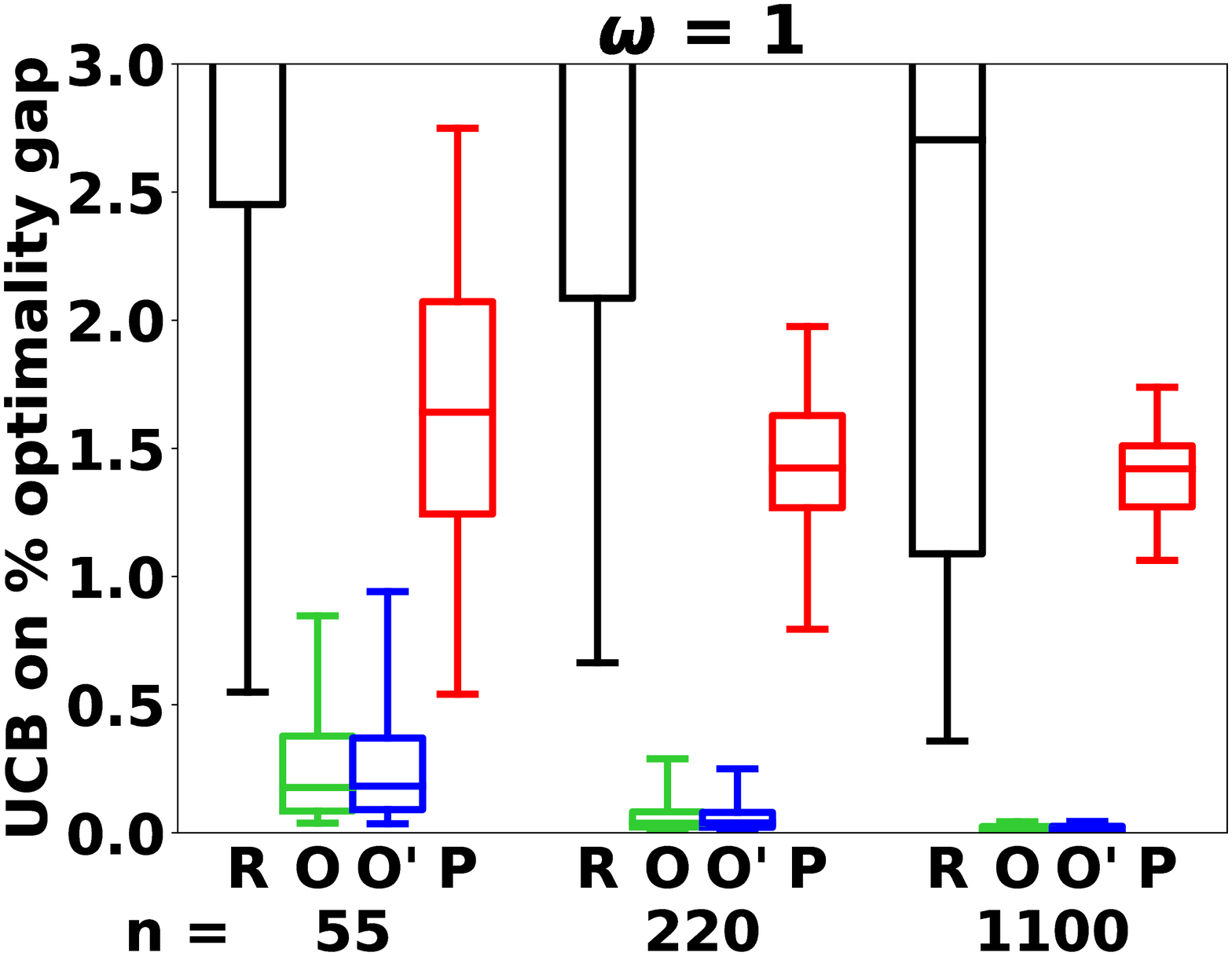}
    \end{subfigure}%
    ~ 
    \begin{subfigure}[t]{0.33\textwidth}
        \centering
        \includegraphics[width=\textwidth]{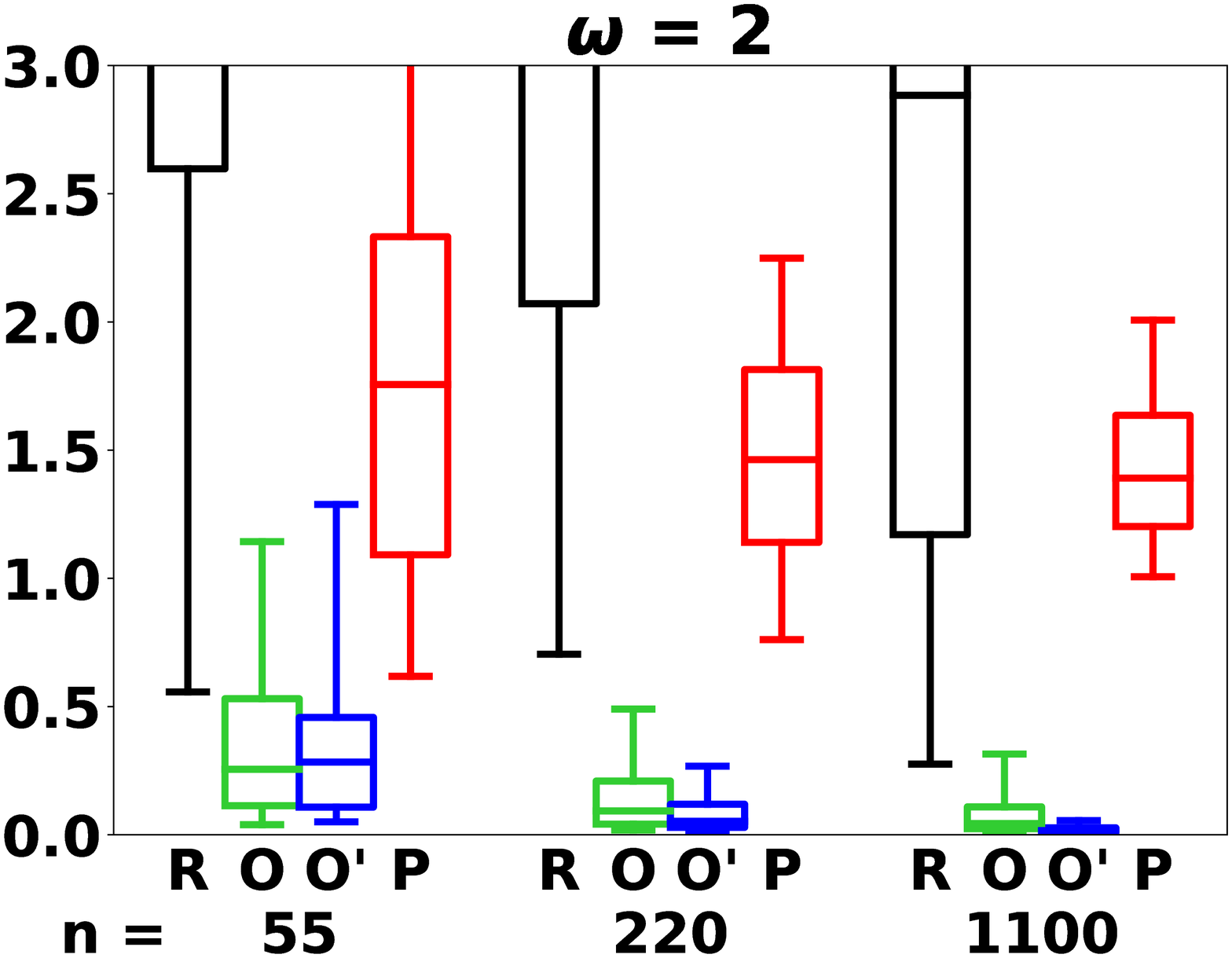}
    \end{subfigure}%
    ~ 
    \begin{subfigure}[t]{0.33\textwidth}
        \centering
        \includegraphics[width=\textwidth]{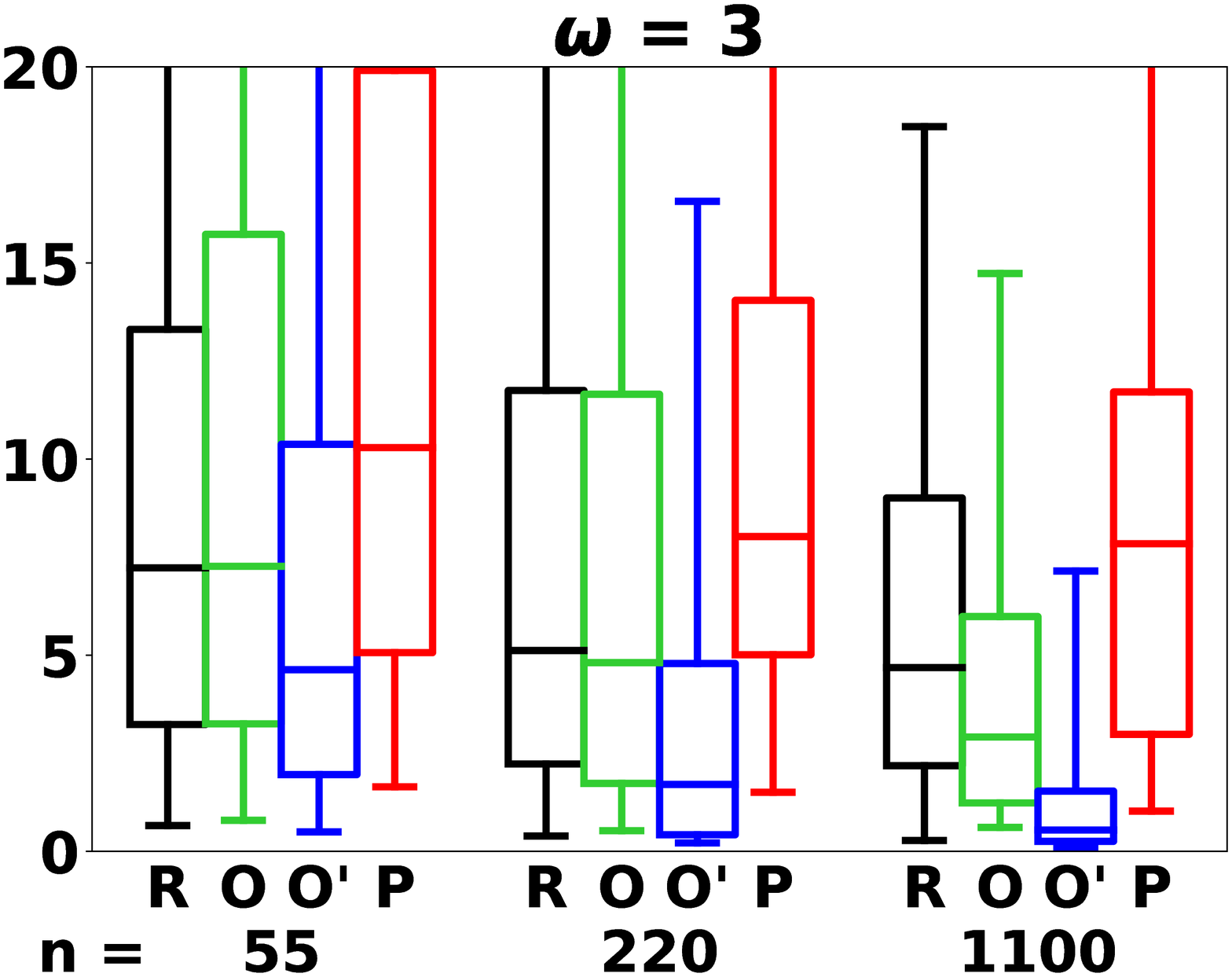}
    \end{subfigure}\\
    \begin{subfigure}[t]{0.33\textwidth}
        \centering
        \includegraphics[width=\textwidth]{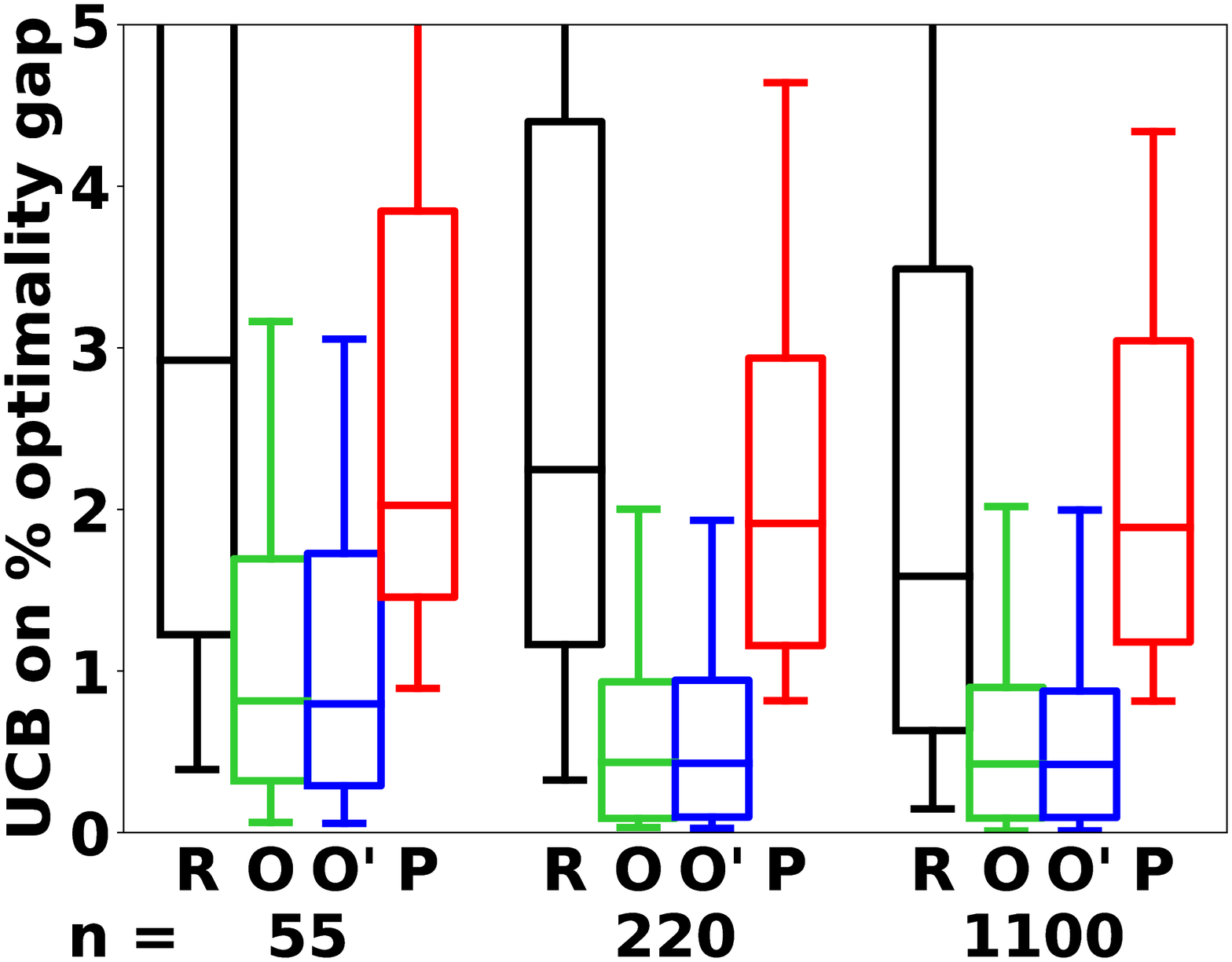}
    \end{subfigure}%
    ~ 
    \begin{subfigure}[t]{0.33\textwidth}
        \centering
        \includegraphics[width=\textwidth]{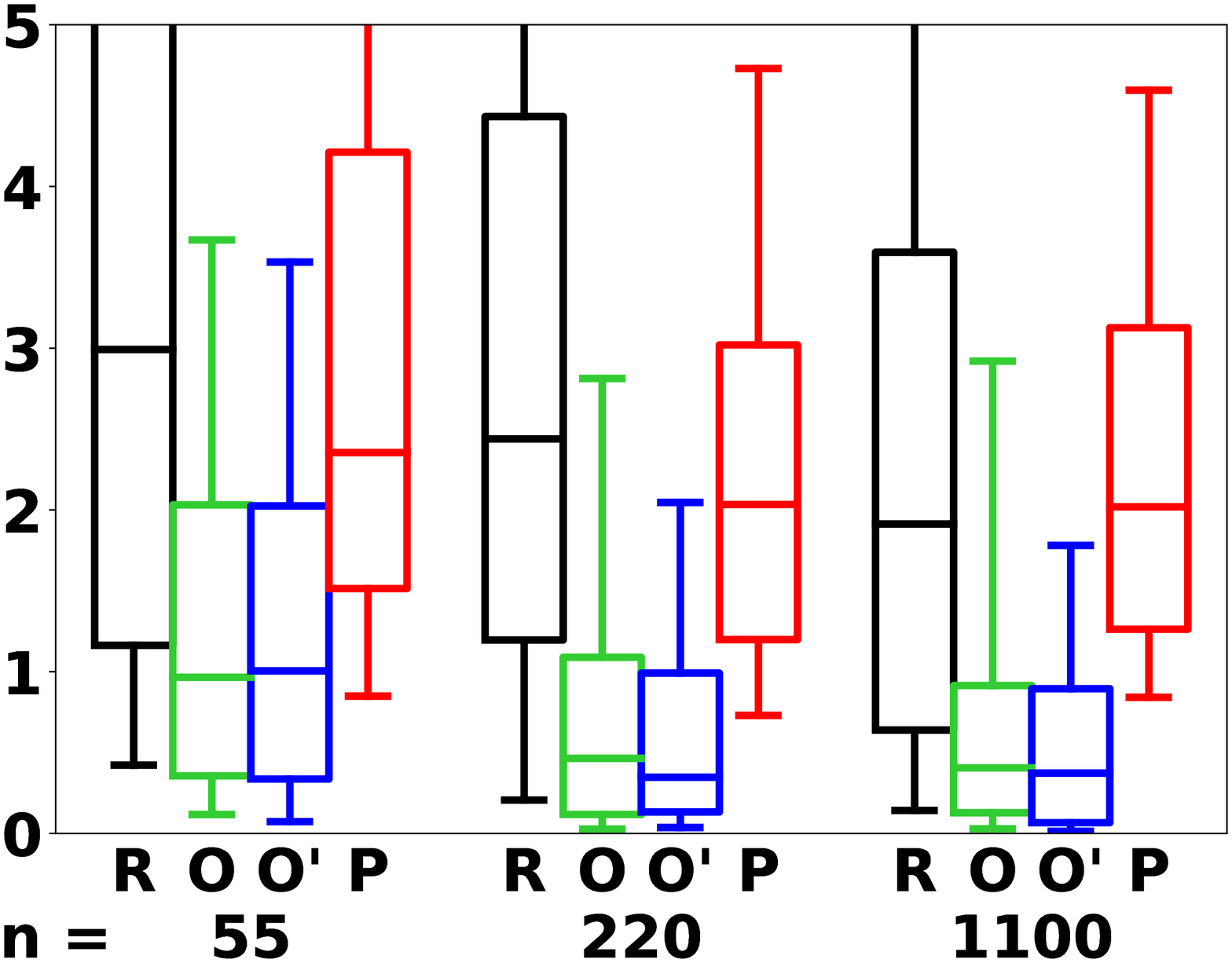}
    \end{subfigure}%
    ~ 
    \begin{subfigure}[t]{0.33\textwidth}
        \centering
        \includegraphics[width=\textwidth]{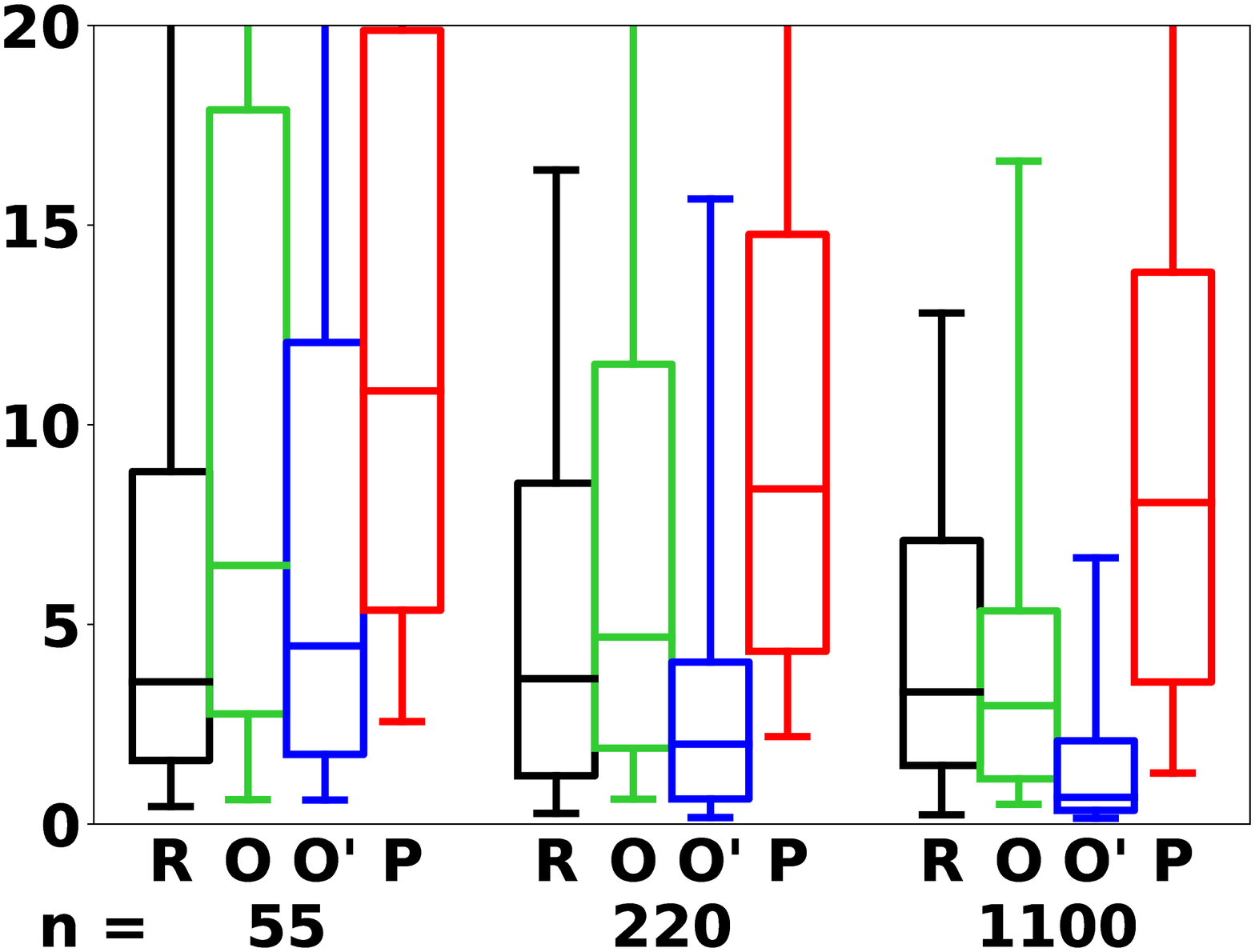}
    \end{subfigure}\\
    \begin{subfigure}[t]{0.33\textwidth}
        \centering
        \includegraphics[width=\textwidth]{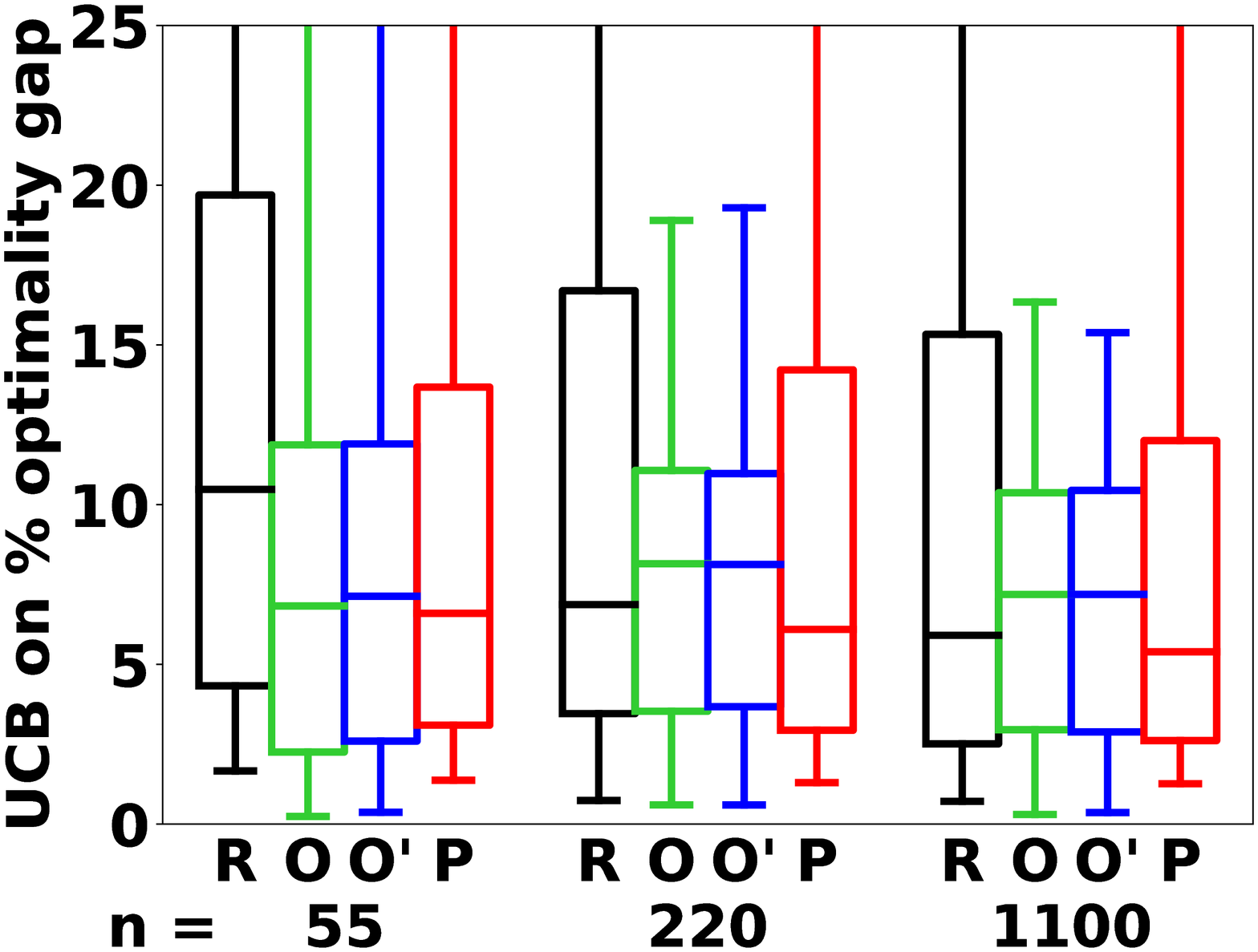}
    \end{subfigure}%
    ~ 
    \begin{subfigure}[t]{0.33\textwidth}
        \centering
        \includegraphics[width=\textwidth]{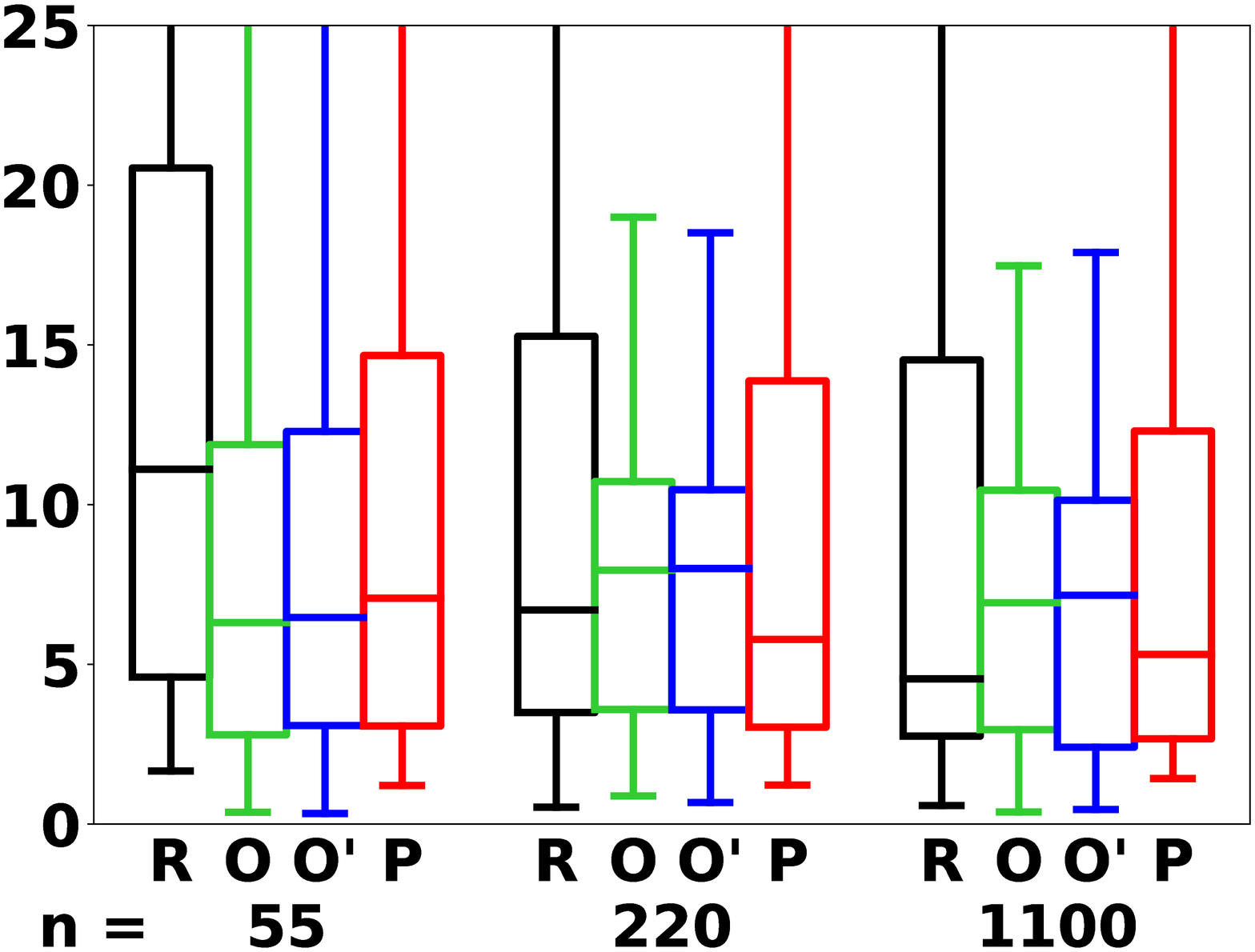}
    \end{subfigure}%
    ~ 
    \begin{subfigure}[t]{0.33\textwidth}
        \centering
        \includegraphics[width=\textwidth]{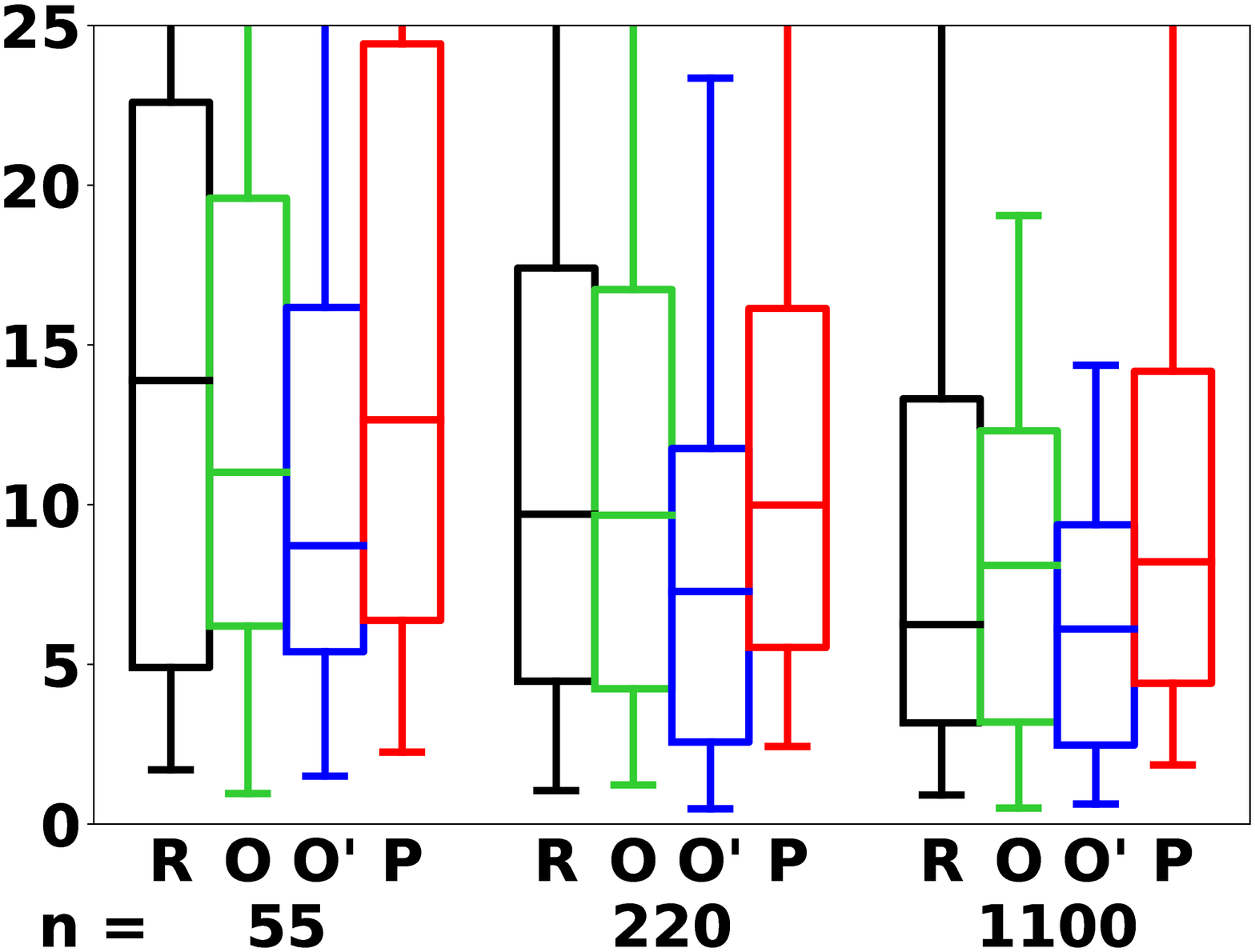}
    \end{subfigure}
    \caption{{Comparison of kNN-SAA (\texttt{R}), ER-SAA+OLS without heteroscedasticity estimation (\texttt{O}), ER-SAA+OLS with heteroscedasticity estimation ($\texttt{O}'$) and PP+OLS ($\texttt{P}$) approaches for $d_x = 10$. Top row: $p = 1$. Middle row: $p = 0.5$. Bottom row: $p = 2$. Left column: $\omega = 1$. Middle column: $\omega = 2$. Right column: $\omega = 3$.}}
    \label{fig:het_plots_2}
\end{figure}

\section{Conclusion and future work}
\label{sec:conclandfw}

We propose three data-driven SAA frameworks for approximating the solution to two-stage stochastic programs when the DM has access to a finite number of samples of random variables and concurrently observed covariates.
These formulations fit a model to predict the random variables given covariate values, and use the prediction model and its (out-of-sample) residuals on the given data to construct scenarios for the original stochastic program at a new covariate realization.
We provide conditions on the prediction and optimization frameworks and the data generation process under which these data-driven estimators are asymptotically optimal, possess a certain rate of convergence, and possess finite sample guarantees.
In particular, we show that our assumptions hold for two-stage stochastic LP in conjunction with popular regression setups such as OLS, Lasso, kNN, and RF regression under various assumptions on the data generation process.
Numerical experiments demonstrate the benefits of our data-driven SAA frameworks, in particular, those of our new data-driven formulations in the limited data regime.

Verifying the assumptions on the prediction setup for other frameworks of interest is an important task to be undertaken by the DM.
Ongoing work includes analysis of an extension of the ER-SAA approach to multistage stochastic programming \citep[cf.][]{ban2018dynamic,bertsimas2019predictions}.
Designing asymptotically optimal estimators for problems with stochastic constraints \citep{hombay_15} 
and problems with decision-dependent uncertainty \citep{bertsimas2014predictive} are interesting avenues for future work.
{Modifications of the ER-SAA that achieve better performance guarantees for a given covariate~$x \in \X$ would also be interesting to explore.}

\section*{Acknowledgments}
This research is supported by the Department of Energy, Office of Science, Office of Advanced Scientific Computing Research, Applied Mathematics program under Contract Number DE-AC02-06CH11357, and was performed using the computing resources of the UW-Madison Center For High Throughput Computing (CHTC) in the Dept.\ of Computer Sciences. The CHTC is supported by UW-Madison, the Advanced Computing Initiative, the Wisconsin Alumni Research Foundation, the Wisconsin Institutes for Discovery, and the National Science Foundation.
{R.K.\ also gratefully acknowledges the support of the U.S.\ Department of Energy through the LANL/LDRD Program and the Center for Nonlinear Studies.}
{The authors thank the three anonymous reviewers, the associate editor, and Prof.\ Erick Delage for suggestions that helped improve this paper.}
R.K.\ also thanks Dr.\ Rui Chen and Prof.\ Garvesh Raskutti for helpful \mbox{discussions}.

{
\footnotesize
\section*{References}
\begingroup
\renewcommand{\section}[2]{}%
\bibliographystyle{abbrvnat}
\bibliography{main}
\endgroup
}

\appendix

\section*{Appendix}

Section~\ref{sec:proofs} provides omitted proofs for results in Section~\ref{sec:ersaa}.
We continue in Section~\ref{sec:alt_assumptions} with a discussion of alternative assumptions under which the results of Section~\ref{sec:ersaa} hold and derive rates of convergence for the ER-SAA in Section~\ref{sec:ersaa_rate}.
We then outline the analysis for the {jackknife}-based SAA variants in Section~\ref{sec:jackknife}.
Then, in Section~\ref{sec:tssp}, we present a class of two-stage stochastic programs that satisfy our assumptions. 
Section~\ref{sec:regression} lists several prediction setups (including M-estimators, OLS, Lasso, kNN, CART, and RF regression) that satisfy the assumptions in our analysis.  Finally, we end with Section~\ref{sec:computexp-details} by providing  omitted details for the computational experiments.

\section{Omitted proofs}
\label{sec:proofs}

\subsection{Proof of Lemma~\ref{lem:ersaa_error_decomp_term1}}

{We have}
\begin{align*}
{\uset{z \in \Z}{\sup} \: \abs*{\hg^{ER}_n(z;x) - g^*_n(z;x)}} &{= \uset{z \in \Z}{\sup} \: \bigg\lvert \dfrac{1}{n}\displaystyle\sum_{i=1}^{n} c \bigl( z,\proj{\Y}{\hf_n(x) + \hat{Q}_n(x)\heps^i_{n}} \bigr) - \dfrac{1}{n}\displaystyle\sum_{i=1}^{n} c \left( z,f^*(x) + Q^*(x)\varepsilon^i \right) \bigg\rvert }\nonumber \\
&{\leq \uset{z \in \Z}{\sup} \: \dfrac{1}{n}\displaystyle\sum_{i=1}^{n} \big\lvert c \bigl( z,\proj{\Y}{\hf_n(x) + \hat{Q}_n(x)\heps^i_{n}} \bigr) - c \left( z,f^*(x) + Q^*(x)\varepsilon^i \right) \big\rvert} \nonumber \\
&{\leq \uset{z \in \Z}{\sup} \: \dfrac{1}{n}\displaystyle\sum_{i=1}^{n} L(z) \norm{\proj{\Y}{\hf_n(x) + \hat{Q}_n(x)\heps^i_{n}} - \left( f^*(x) + Q^*(x)\varepsilon^i \right)}} \nonumber \\
&{\leq \Bigl(\uset{z \in \Z}{\sup} \: L(z)\Bigr) \biggl(\dfrac{1}{n}\displaystyle\sum_{i=1}^{n} \norm{\teps^i_{n}(x)}\biggr),}
\end{align*}
{where the second inequality above follows by Assumption~\ref{ass:equilipschitz} and the final inequality follows by the Lipschitz continuity of orthogonal projections.}
\qed

\subsection{Proof of Theorem~\ref{thm:approxconv}}

Before we prove Theorem~\ref{thm:approxconv}, we present the following lemma that is needed in its proof.

\begin{lemma}
\label{lem:lscsuboptimal}
Let $W \subset \R^{d_w}$ be a nonempty and compact set and $h: W \to \R$ be a lower semicontinuous function. 
Define 
$W^* := \argmin_{w \in W} h(w)$.
Suppose there exists $\delta > 0$ and $\bar{w} \in W$ such that $\textup{dist}(\bar{w},W^*) \geq \delta$.
Then, there exists $\kappa > 0$ such that $h(\bar{w}) \geq \min_{w \in W} h(w) + \kappa$.
\end{lemma}
\begin{proof}

Let $W_{\delta}:= \Set{w \in W}{\text{dist}(w,W^*) \geq \delta}$, and note that $\bar{w} \in W_{\delta}$.
Since $h$ is lower semicontinuous and $W_{\delta}$ is nonempty and compact, $\uset{w \in W_{\delta}}{\inf} \: h(w)$ is attained.
Furthermore, we readily have $\uset{w \in W_{\delta}}{\min} \: h(w) > \uset{w \in W}{\min} \: h(w)$.
Setting $\kappa = \uset{w \in W_{\delta}}{\min} \: h(w) - \uset{w \in W}{\min} \: h(w)$ yields the desired result.
\end{proof}

\bigskip

\noindent{\it Proof of Theorem~\ref{thm:approxconv}.}
Let $z^*(x) \in S^*(x)$ and $\hz^{ER}_n(x) \in \hS^{ER}_n(x)$.
Consider any constant~$\delta > 0$.
From Proposition~\ref{prop:uniformconvofobj}, we have for a.e.~$x \in \X$:
\begin{align*}
\prob{\abs*{\hg^{ER}_n(z^*(x);x) - v^*(x)} > \delta} \to 0 &\implies \prob{\hv^{ER}_n(x) > v^*(x) + \delta} \to 0 \quad \text{and}\\
\prob{\abs*{\hv^{ER}_n(x) - g(\hz^{ER}_n(x);x)} > \delta} \to 0 &\implies \prob{v^*(x) > \hv^{ER}_n(x) + \delta} \to 0. \\ 
\therefore \quad \prob{\abs*{\hv^{ER}_n(x) - v^*(x)} > \delta} \to 0 &\implies \hv^{ER}_n(x) \convinprob v^*(x) \text{ for a.e. } x \in \X.
\end{align*}%
Suppose for contradiction that $\dev{\hS^{ER}_n(x)}{S^*(x)} \not\xrightarrow{p} 0$, $\forall x \in \bar{\X}$, where $\bar{\X} \subseteq \X$ with $P_X(\bar{\X}) > 0$. This implies for any~$\bar{x} \in \bar{\X}$, there exist constants $\delta > 0$ and $\beta > 0$ and a subsequence $\{n_q\}$ of $\mathbb{N}$ such that
$\prob{\dev{\hS^{ER}_{n_q}(\bar{x})}{S^*(\bar{x})} \geq \delta} \geq \beta$, $\forall q \in \mathbb{N}$.
Lemma~\ref{lem:lscsuboptimal} then implies that for a.e.\ $\bx \in \bar{\X}$, there exists $\kappa(\bar{x}) > 0$ such that
\begin{align}
\label{eqn:ersaanotconv}
&\pr \Bigl\{\uset{z \in \hS^{ER}_{n_q}(\bar{x})}{\sup} g(z;\bar{x}) > v^*(\bar{x}) + \kappa(\bar{x}) \Bigr\} \geq \beta, \quad \forall q \in \mathbb{N}.
\end{align}
From Proposition~\ref{prop:uniformconvofobj}, we have for a.e.\ $\bx \in \bar{\X}$:
\begin{align*}
\pr \Bigl\{\uset{z \in \hS^{ER}_n(\bx)}{\sup} \abs*{\hg^{ER}_n(z;\bx) - g(z;\bx)} \leq 0.5\kappa(\bx) \Bigr\} \to 1 &\implies \pr \Bigl\{\uset{z \in \hS^{ER}_n(\bx)}{\sup} g(z;\bx) \leq \hv^{ER}_n(\bx) + 0.5\kappa(\bx) \Bigr\} \to 1, \\
\pr \Bigl\{\uset{z \in S^*(\bx)}{\sup} \: \abs*{\hg^{ER}_n(z,\bx) - g(z,\bx)} \leq 0.5\kappa(\bx) \Bigr\} \to 1 &\implies \pr\Bigl\{\uset{z \in S^*(\bx)}{\sup} \hg^{ER}_n(z;\bx) \leq v^*(\bx) + 0.5\kappa(\bx)\Bigr\} \to 1.
\end{align*}
Since $\hv^{ER}_n(\bx) \leq \sup_{z \in S^*(\bx)} \hg^{ER}_n(z;\bx)$ by definition, the above inequalities in turn imply
\[
\pr \Bigl\{\uset{z \in \hS^{ER}_n(\bx)}{\sup} g(z;\bx) \leq v^*(\bx) + \kappa(\bx) \Bigr\} \to 1,
\]
which contradicts the inequality~\eqref{eqn:ersaanotconv}.
The above arguments also readily imply that the ER-SAA estimators are asymptotically optimal, i.e., $\sup_{z \in \hS^{ER}_n(x)} g(z;x) \xrightarrow{p} v^*(x)$ for a.e.~$x \in \X$.
\qed

\subsection{Proof of Theorem~\ref{thm:averageconsist}}

{Since}
\begin{align*}
&{\norm*{\hv^{ER}_n(X) - v^*(X)}_{L^q} = \big\lVert \min_{z \in \Z} \hg^{ER}_n(z;X) - \min_{z \in \Z} g(z;X)\big\rVert_{L^q} \leq \Big\lVert \uset{z \in \Z}{\sup} \abs*{\hg^{ER}_{n}(z;X) - g(z;X)}\Big\rVert_{L^q},} \\
&{\norm*{g(\hz^{ER}_n(X);X) - v^*(X)}_{L^q}} {\leq \norm*{g(\hz^{ER}_n(X);X) - \hv^{ER}_n(X)}_{L^q} + \norm*{\hv^{ER}_n(X) - v^*(X)}_{L^q},}
\end{align*}
{we focus on establishing convergence of $\hg^{ER}_n(\cdot;X)$ to $g(\cdot;X)$ with respect to the $L^q$-norm on~$\X$.
From inequality~\eqref{eqn:ersaa_error_decomp} and the triangle inequality for the $L^q$-norm, we have}
\[
{\Big\lVert\uset{z \in \Z}{\sup}\: \abs*{\hg^{ER}_n(z;X) - g(z;X)}\Big\rVert_{L^q} \leq \Big\lVert\uset{z \in \Z}{\sup} \: \abs*{\hg^{ER}_n(z;X) - g^*_n(z;X)}\Big\rVert_{L^q} + \Big\lVert\uset{z \in \Z}{\sup} \: \abs*{g^*_n(z;X) - g(z;X)}\Big\rVert_{L^q}.}
\]
{Assumption~\ref{ass:uniflln_lq} implies that the second term on the r.h.s.\ of the above inequality converges to zero in probability.
We now show that the first term also converges to zero in probability. We have}
\begin{align*}
&{\Big\lVert\uset{z \in \Z}{\sup} \: \abs*{\hg^{ER}_n(z;X) - g^*_n(z;X)}\Big\rVert_{L^q}} \\
\leq& {\Bigl(\uset{z \in \Z}{\sup} \: L(z)\Bigr) \Big\lVert\dfrac{1}{n}\displaystyle\sum_{i=1}^{n} \norm{\teps^i_{n}(X)}\Big\rVert_{L^q}} \\
\leq& {O(1)\big\lVert\norm{\hf_n(X) - f^*(X)}\big\rVert_{L^q} + O(1)\big\lVert\norm{\hQ_n(X) - Q^*(X)}\big\rVert_{L^q}\biggl(\frac{1}{n} \sum_{i=1}^{n} \norm{\varepsilon^i}\biggr) +} \\
& { O(1)\big\lVert\norm{\hQ_n(X)}\big\rVert_{L^q} \biggl(\frac{1}{n} \sum_{i=1}^{n} \bigl\lVert \bigl[\hQ_n(x^i)\bigr]^{-1} - \bigl[Q^*(x^i)\bigr]^{-1}\bigr\rVert^2\biggr)^{1/2} \biggl(\frac{1}{n} \sum_{i=1}^{n} \norm{Q^*(x^i)}^4\biggr)^{1/4} \biggl(\frac{1}{n} \sum_{i=1}^{n} \norm{\varepsilon^i}^4\biggr)^{1/4} +} \\ 
& {O(1)\big\lVert\norm{\hQ_n(X)}\big\rVert_{L^q} \biggl(\frac{1}{n} \sum_{i=1}^{n} \bigl\lVert \bigl[\hQ_n(x^i)\bigr]^{-1}\bigr\rVert^2\biggr)^{1/2} \biggl(\frac{1}{n} \sum_{i=1}^{n} \norm{f^*(x^i) - \hf_n(x^i)}^2\biggr)^{1/2},}
\end{align*}
{where the first inequality follows from Lemma~\ref{lem:ersaa_error_decomp_term1} and the second inequality follows from Lemma~\ref{lem:meandeviation} and the triangle inequality.
The rest of the arguments for $\big\lVert\sup_{z \in \Z} \: \abs*{\hg^{ER}_n(z;X) - g^*_n(z;X)}\big\rVert_{L^q} \convinprob 0$ follow by similar arguments as in the proof of Proposition~\ref{prop:uniformconvofobj} upon noting that}
\[
{\big\lVert\norm{\hQ_n(X)}\big\rVert_{L^q} \leq \big\lVert\norm{\hQ_n(X) - Q^*(X)}\big\rVert_{L^q} + \big\lVert\norm{Q^*(X)}\big\rVert_{L^q}}
\]
{and by replacing Assumptions~\ref{ass:uniflln} and~\ref{ass:regconsist} with Assumptions~\ref{ass:uniflln_lq} and~\ref{ass:regconsist_lq}.}
\qed

\subsection{Proofs of Lemma~\ref{lem:largedevofdev},  Theorem~\ref{thm:exponentialconv}, and Proposition~\ref{prop:specfinsamp}}

{\it Proof of Lemma~\ref{lem:largedevofdev}.}
Assumption~\ref{ass:equilipschitz} and Lemma~\ref{lem:ersaa_error_decomp_term1} imply
\begin{align}
\label{eqn:meandeviation_finitesamp_0}
\mathbb{P}\Bigl\{\uset{z \in \Z}{\sup} \abs*{ \hg^{ER}_n(z;x) - g^*_n(z;x)} > \kappa\Bigr\} \leq \: & \pr \biggl\{\Bigl(\uset{z \in \Z}{\sup} \: L(z) \Bigr) \biggl(\dfrac{1}{n}\displaystyle\sum_{i=1}^{n} \norm{\teps^i_{n}(x)}\biggr) > \kappa \biggr\} \nonumber \\
\leq \: & \pr \Biggl\{\biggl(\dfrac{1}{n}\displaystyle\sum_{i=1}^{n} \norm{\teps^i_{n}(x)}\biggr) > \dfrac{\kappa}{\Bigl(\uset{z \in \Z}{\sup} \: L(z) \Bigr)} \Biggr\}.
\end{align}
Therefore, it suffices to bound $\pr \bigl\{\frac{1}{n}\sum_{i=1}^{n} \norm{\teps^i_{n}(x)} > \kappa \bigr\}$ for any $\kappa > 0$.

{From Lemma~\ref{lem:meandeviation} and the inequality $\mathbb{P}\{V + W > c_1 + c_2\} \leq \mathbb{P}\{V > c_1\} + \mathbb{P}\{W > c_2\}$ for any random variables $V$, $W$ and constants $c_1$, $c_2$, we have:}
\begin{align}
\label{eqn:meandeviation_finitesamp}
&{\mathbb{P}\Bigl\{\dfrac{1}{n} \displaystyle\sum_{i=1}^{n} \norm{\teps^i_{n}(x)} > \kappa \Bigr\}} \nonumber\\
\leq& {\mathbb{P}\Bigl\{\norm{\hf_n(x) - f^*(x)} > \frac{\kappa}{4}\Bigr\} + \mathbb{P}\biggl\{\biggl(\frac{1}{n} \sum_{i=1}^{n} \norm{\varepsilon^i}\biggr) \norm{\hQ_n(x) - Q^*(x)} > \frac{\kappa}{4}\biggr\} +} \nonumber \\
&\:\: {\mathbb{P}\biggl\{ \norm{\hQ_n(x)} \biggl(\frac{1}{n} \sum_{i=1}^{n} \bigl\lVert \bigl[\hQ_n(x^i)\bigr]^{-1} - \bigl[Q^*(x^i)\bigr]^{-1}\bigr\rVert^2\biggr)^{1/2} \biggl(\frac{1}{n} \sum_{i=1}^{n} \norm{Q^*(x^i)}^4\biggr)^{1/4} \biggl(\frac{1}{n} \sum_{i=1}^{n} \norm{\varepsilon^i}^4\biggr)^{1/4} > \frac{\kappa}{4} \biggr\} +} \nonumber \\
&\:\: {\mathbb{P}\biggl\{ \norm{\hQ_n(x)} \biggl(\frac{1}{n} \sum_{i=1}^{n} \bigl\lVert \bigl[\hQ_n(x^i)\bigr]^{-1}\bigr\rVert^2\biggr)^{1/2} \biggl(\frac{1}{n} \sum_{i=1}^{n} \norm{f^*(x^i) - \hf_n(x^i)}^2\biggr)^{1/2} > \frac{\kappa}{4}\biggr\}.}
\end{align}
{For a.e.\ $x \in \X$, the first term on the r.h.s.\ of~\eqref{eqn:meandeviation_finitesamp} can be bounded using Assumption~\ref{ass:reglargedev_point} as}
\begin{align*}
&{\mathbb{P}\Bigl\{\norm{\hf_n(x) - f^*(x)} > \frac{\kappa}{4}\Bigr\} \leq K_f(\tfrac{\kappa}{4},x) \exp(-\beta_f(n,\tfrac{\kappa}{4},x)).}
\end{align*}

{Next, consider the second term on the r.h.s.\ of inequality~\eqref{eqn:meandeviation_finitesamp}.
We have for a.e.\ $x \in \X$}
\begin{align*}
&{\mathbb{P}\biggl\{\biggl(\frac{1}{n} \sum_{i=1}^{n} \norm{\varepsilon^i}\biggr) \norm{\hQ_n(x) - Q^*(x)} > \frac{\kappa}{4}\biggr\}} \\
\leq& {\mathbb{P}\biggl\{\frac{1}{n} \sum_{i=1}^{n} \norm{\varepsilon^i} > \mathbb{E}[\norm{\varepsilon}] + \kappa\biggr\} + \mathbb{P}\Bigl\{(\mathbb{E}[\norm{\varepsilon}] + \kappa) \norm{\hQ_n(x) - Q^*(x)} > \frac{\kappa}{4}\Bigr\}} \\
\leq& {J_{\varepsilon}(\kappa)\exp(-\gamma_{\varepsilon}(n,\kappa)) + \mathbb{P}\biggl\{ \norm{\hQ_n(x) - Q^*(x)} > \frac{\kappa}{4(\mathbb{E}[\norm{\varepsilon}] + \kappa)}\biggr\}} \\
\leq& {J_{\varepsilon}(\kappa)\exp(-\gamma_{\varepsilon}(n,\kappa)) + K_Q\bigl(\tfrac{\kappa}{4(\mathbb{E}[\norm{\varepsilon}] + \kappa)},x\bigr) \exp\bigl(-\beta_Q\bigl( n,\tfrac{\kappa}{4(\mathbb{E}[\norm{\varepsilon}] + \kappa)},x\bigr)\bigr),}
\end{align*}
{where the second inequality follows by Assumption~\ref{ass:errorlargedev} and the last step by Assumption~\ref{ass:reglargedev_point2}.}

{The third term on the r.h.s.\ of inequality~\eqref{eqn:meandeviation_finitesamp} can be bounded for a.e.\ $x \in \X$ as}
\begin{align*}
&{\mathbb{P}\biggl\{ \norm{\hQ_n(x)} \biggl(\frac{1}{n} \sum_{i=1}^{n} \bigl\lVert \bigl[\hQ_n(x^i)\bigr]^{-1} - \bigl[Q^*(x^i)\bigr]^{-1}\bigr\rVert^2\biggr)^{1/2} \biggl(\frac{1}{n} \sum_{i=1}^{n} \norm{Q^*(x^i)}^4\biggr)^{1/4} \biggl(\frac{1}{n} \sum_{i=1}^{n} \norm{\varepsilon^i}^4\biggr)^{1/4} > \frac{\kappa}{4} \biggr\}} \\
\leq& {\mathbb{P}\bigl\{ \norm{\hQ_n(x)} > \norm{Q^*(x)} + \kappa \bigr\} + \mathbb{P}\biggl\{ \biggl(\frac{1}{n} \sum_{i=1}^{n} \norm{Q^*(x^i)}^4\biggr)^{1/4} > \bigl(\mathbb{E}[\norm{Q^*(X)}^4]\bigr)^{1/4} + \kappa\biggr\} +} \\
& {\mathbb{P}\biggl\{ \biggl(\frac{1}{n} \sum_{i=1}^{n} \norm{\varepsilon^i}^4\biggr)^{1/4} > \bigl(\mathbb{E}[\norm{\varepsilon}^4]\bigr)^{1/4} + \kappa\biggr\} +} \\
& {\mathbb{P}\biggl\{ \bigl( \norm{Q^*(x)} + \kappa \bigr) \bigl(\bigl(\mathbb{E}[\norm{Q^*(X)}^4]\bigr)^{1/4} + \kappa\bigr) \bigl(\bigl(\mathbb{E}[\norm{\varepsilon}^4]\bigr)^{1/4} + \kappa\bigr) \biggl(\frac{1}{n} \sum_{i=1}^{n} \bigl\lVert \bigl[\hQ_n(x^i)\bigr]^{-1} - \bigl[Q^*(x^i)\bigr]^{-1}\bigr\rVert^2\biggr)^{1/2} > \frac{\kappa}{4}\biggr\}} \\
\leq& {K_Q(\kappa,x) \exp\left(-\beta_Q(n,\kappa,x)\right) + \bar{J}_Q(\kappa)\exp(-\bar{\gamma}_{Q}(n,\kappa)) + \bar{J}_{\varepsilon}(\kappa)\exp(-\bar{\gamma}_{\varepsilon}(n,\kappa)) +} \\
& {\mathbb{P}\biggl\{ \biggl(\frac{1}{n} \sum_{i=1}^{n} \bigl\lVert \bigl[\hQ_n(x^i)\bigr]^{-1} - \bigl[Q^*(x^i)\bigr]^{-1}\bigr\rVert^2\biggr)^{1/2} > h_1(\kappa,x)\biggr\}} \\
\leq& {K_Q(\kappa,x) \exp\left(-\beta_Q(n,\kappa,x)\right) + \bar{J}_Q(\kappa)\exp(-\bar{\gamma}_{Q}(n,\kappa)) + \bar{J}_{\varepsilon}(\kappa)\exp(-\bar{\gamma}_{\varepsilon}(n,\kappa)) +} \\
& {\bar{K}_Q(h_1(\kappa,x)) \exp(-\bar{\beta}_Q(n,h_1(\kappa,x))),}
\end{align*}
{where the second inequality follows from Assumptions~\ref{ass:varlargedev},~\ref{ass:errorlargedev}, and~\ref{ass:reglargedev_point2}, the final inequality follows from Assumption~\ref{ass:reglargedev_mse2}, and}
\[
{h_1(\kappa,x) := \frac{\kappa}{4\bigl( \norm{Q^*(x)} + \kappa \bigr) \bigl(\bigl(\mathbb{E}[\norm{Q^*(X)}^4]\bigr)^{1/4} + \kappa\bigr) \bigl(\bigl(\mathbb{E}[\norm{\varepsilon}^4]\bigr)^{1/4} + \kappa\bigr)}.}
\]

{Finally, the fourth term on the r.h.s.\ of inequality~\eqref{eqn:meandeviation_finitesamp} can be bounded for a.e.\ $x \in \X$ as}
\begin{align*}
&{\mathbb{P}\biggl\{ \norm{\hQ_n(x)} \biggl(\frac{1}{n} \sum_{i=1}^{n} \bigl\lVert \bigl[\hQ_n(x^i)\bigr]^{-1}\bigr\rVert^2\biggr)^{1/2} \biggl(\frac{1}{n} \sum_{i=1}^{n} \norm{f^*(x^i) - \hf_n(x^i)}^2\biggr)^{1/2} > \frac{\kappa}{4}\biggr\}} \\
\leq& {\mathbb{P}\bigl\{ \norm{\hQ_n(x)} > \norm{Q^*(x)} + \kappa \bigr\} + \mathbb{P}\biggl\{ \biggl(\frac{1}{n} \sum_{i=1}^{n} \bigl\lVert \bigl[\hQ_n(x^i)\bigr]^{-1}\bigr\rVert^2\biggr)^{1/2} > \Bigl(\mathbb{E}\bigl[ \bigl\lVert \bigl[Q^*(X)\bigr]^{-1} \bigr\rVert^2 \bigr]\Bigr)^{1/2} + 2\kappa\biggr\} +} \\
& {\mathbb{P}\biggl\{ \bigl( \norm{Q^*(x)} + \kappa \bigr) \Bigl(\Bigl(\mathbb{E}\bigl[ \bigl\lVert \bigl[Q^*(X)\bigr]^{-1} \bigr\rVert^2 \bigr]\Bigr)^{1/2} + 2\kappa\Bigr) \biggl(\frac{1}{n} \sum_{i=1}^{n} \norm{f^*(x^i) - \hf_n(x^i)}^2\biggr)^{1/2} > \frac{\kappa}{4}\biggr\}} \\
\leq& {K_Q(\kappa,x) \exp\left(-\beta_Q(n,\kappa,x)\right) + \mathbb{P}\biggl\{ \biggl(\frac{1}{n} \sum_{i=1}^{n} \norm{f^*(x^i) - \hf_n(x^i)}^2\biggr)^{1/2} > h_2(\kappa,x)\biggr\} +} \\
&{\mathbb{P}\biggl\{ \biggl(\frac{1}{n} \sum_{i=1}^{n} \bigl\lVert \bigl[\hQ_n(x^i)\bigr]^{-1} - \bigl[Q^*(x^i)\bigr]^{-1}\bigr\rVert^2\biggr)^{1/2} + \biggl(\frac{1}{n} \sum_{i=1}^{n} \bigl\lVert \bigl[Q^*(x^i)\bigr]^{-1}\bigr\rVert^2\biggr)^{1/2} > \Bigl(\mathbb{E}\bigl[ \bigl\lVert \bigl[Q^*(X)\bigr]^{-1} \bigr\rVert^2 \bigr]\Bigr)^{1/2} + 2\kappa\biggr\}} \\
\leq& {K_Q(\kappa,x) \exp\left(-\beta_Q(n,\kappa,x)\right) + \bar{K}_f(h_2(\kappa,x)) \exp(-\bar{\beta}_f(n,h_2(\kappa,x))) + \bar{K}_Q(\kappa) \exp(-\bar{\beta}_Q(n,\kappa)) +} \\
&{J_Q(\kappa)\exp(-\gamma_{Q}(n,\kappa)),}
\end{align*}
{where the second inequality follows by Assumption~\ref{ass:reglargedev} and Lemma~\ref{lem:varbound}, the final inequality follows by Assumptions~\ref{ass:reglargedev} and~\ref{ass:varlargedev} and the probability inequality stated at the beginning of this proof, and}
\[
{h_2(\kappa,x) := \frac{\kappa}{4\bigl( \norm{Q^*(x)} + \kappa \bigr) \Bigl(\Bigl(\mathbb{E}\bigl[ \bigl\lVert \bigl[Q^*(X)\bigr]^{-1} \bigr\rVert^2 \bigr]\Bigr)^{1/2} + 2\kappa\Bigr)}.}
\]
{Putting the above bounds together in inequality~\eqref{eqn:meandeviation_finitesamp}, we have for a.e.\ $x \in \X$:}
\begin{align}
\label{eqn:meandeviation_finitesamp2}
{\mathbb{P}\biggl\{\dfrac{1}{n} \displaystyle\sum_{i=1}^{n} \norm{\teps^i_{n}(x)} > \kappa \biggr\}} &{\leq J_{\varepsilon}(\kappa)\exp(-\gamma_{\varepsilon}(n,\kappa)) + \bar{J}_{\varepsilon}(\kappa)\exp(-\bar{\gamma}_{\varepsilon}(n,\kappa)) + } \\
&\quad { J_Q(\kappa)\exp(-\gamma_{Q}(n,\kappa)) + \bar{J}_Q(\kappa)\exp(-\bar{\gamma}_{Q}(n,\kappa)) +} \nonumber \\
&\quad {K_f(\tfrac{\kappa}{4},x) \exp(-\beta_f(n,\tfrac{\kappa}{4},x)) + \bar{K}_f(h_2(\kappa,x)) \exp(-\bar{\beta}_f(n,h_2(\kappa,x))) +} \nonumber \\
&\quad {K_Q\bigl(\tfrac{\kappa}{4(\mathbb{E}[\norm{\varepsilon}] + \kappa)},x\bigr) \exp\bigl(-\beta_Q(n,\tfrac{\kappa}{4(\mathbb{E}[\norm{\varepsilon}] + \kappa)},x)\bigr) + 2K_Q(\kappa,x) \exp\left(-\beta_Q(n,\kappa,x)\right) +} \nonumber\\
&\quad {\bar{K}_Q(\kappa) \exp(-\bar{\beta}_Q(n,\kappa)) + \bar{K}_Q(h_1(\kappa,x)) \exp(-\bar{\beta}_Q(n,h_1(\kappa,x))) \nonumber,}
\end{align}
{which along with inequality~\eqref{eqn:meandeviation_finitesamp_0} implies the desired result.}
\qed

\bigskip

\noindent {\it Proof of Theorem~\ref{thm:exponentialconv}.}
Note that for any $\kappa > 0$:
\begin{align*}
\pr \Bigl\{\uset{z \in \Z}{\sup} \: \abs*{\hg^{ER}_n(z;x) - g(z;x)} > \kappa \Bigr\} \leq & \: \pr \Bigl\{\uset{z \in \Z}{\sup} \: \abs*{ \hg^{ER}_n(z;x) - g^*_n(z;x)} > \dfrac{\kappa}{2} \Bigr\} + \pr \Bigl\{\uset{z \in \Z}{\sup} \: \abs*{g^*_n(z;x) - g(z;x)} > \dfrac{\kappa}{2} \Bigr\}.
\end{align*}
Bounding the two terms on the r.h.s.\ of the above inequality using Assumption~\ref{ass:tradsaalargedev} and Lemma~\ref{lem:largedevofdev} yields for a.e.~$x \in \X$:
\begin{align}
\label{eqn:largedeveqn}
&\pr \Bigl\{\uset{z \in \Z}{\sup} \: \abs*{\hg^{ER}_n(z;x) - g(z;x)} > \kappa \Bigr\} \leq \tilde{K}(\kappa,x) \exp\bigl(- \tilde{\beta}(n,\kappa,x) \bigr),
\end{align}
where $\tilde{K}(\kappa,x) := 2\max\left\lbrace K(0.5\kappa,x), \bar{K}(0.5\kappa,x) \right\rbrace$ and $\tilde{\beta}(n,\kappa,x) := \min\left\lbrace n\beta(0.5\kappa,x), \bar{\beta}(n,0.5\kappa,x)  \right\rbrace$.
{Inequality~\eqref{eqn:largedeveqn} implies for $n \in \mathbb{N}$, a.e.\ $x \in \X$, and any $\kappa := \kappa(x) > 0$:}
\begin{align*}
&{\pr \Bigl\{\uset{z \in \hS^{ER}_n(x)}{\sup} \: \abs*{\hg^{ER}_n(z;x) - g(z;x)} \leq 0.5\kappa \Bigr\} \geq 1-\tilde{K}(0.5\kappa,x) \exp\bigl(- \tilde{\beta}(n,0.5\kappa,x) \bigr)} \\
\implies & {\pr \Bigl\{\uset{z \in \hS^{ER}_n(x)}{\sup} g(z;x) \leq \hv^{ER}_n(x) + 0.5\kappa \Bigr\} \geq 1-\tilde{K}(0.5\kappa,x) \exp\bigl(- \tilde{\beta}(n,0.5\kappa,x) \bigr),}\\
\text{and} \quad &{\pr\Bigl\{\uset{z \in S^*(x)}{\sup} \: \abs*{\hg^{ER}_n(z;x) - g(z;x)} \leq 0.5\kappa \Bigr\} \geq  1-\tilde{K}(0.5\kappa,x) \exp\bigl(- \tilde{\beta}(n,0.5\kappa,x) \bigr)}\\
\implies & {\pr \Bigl\{\uset{z \in S^*(x)}{\sup} \hg^{ER}_n(z;x) \leq v^*(x) + 0.5\kappa \Bigr\} \geq 1-\tilde{K}(0.5\kappa,x) \exp\bigl(- \tilde{\beta}(n,0.5\kappa,x) \bigr).}
\end{align*}
{Since $\hv^{ER}_n(x) \leq \sup_{z \in S^*(x)} \hg^{ER}_n(z;x)$ by the definition of $\hv^{ER}_n(x)$, the above two inequalities imply}
\begin{align}
\label{eqn:largedeveqn2}
\pr \Bigl\{\uset{z \in \hS^{ER}_n(x)}{\sup} g(z;x) \leq v^*(x) + \kappa \Bigr\} &\geq 1-2\tilde{K}(0.5\kappa,x) \exp\bigl(- \tilde{\beta}(n,0.5\kappa,x) \bigr), \\
\implies \prob{g(\hz^{ER}_n(x);x) \leq v^*(x) + \kappa(x)} &\geq 1-2\tilde{K}(0.5\kappa,x) \exp\bigl(- \tilde{\beta}(n,0.5\kappa,x) \bigr). \nonumber
\end{align}

Suppose $\text{dist}(\hz^{ER}_n(x),S^*(x)) \geq \eta$ for some $x \in \X$ and some sample path.
Since $g(\cdot;x)$ is lsc on the compact set~$\Z$ for a.e.\ $x \in \X$,
Lemma~\ref{lem:lscsuboptimal} implies that 
there exists $\kappa(\eta,x) > 0$ such that $g(\hz^{ER}_n(x);x) > v^*(x) + \kappa(\eta,x)$ on that path (except for some paths of measure zero).
We now provide a bound on the probability of this event. 
By the above arguments, we have for a.e.\ $x \in \X$:
\begin{align*}
\prob{\text{dist}(\hz^{ER}_n(x),S^*(x)) \geq \eta} &\leq \prob{g(\hz^{ER}_n(x);x) > v^*(x) + \kappa(\eta,x)} \\
&\leq 2\tilde{K}(0.5\kappa(\eta,x),x) \exp\bigl(- \tilde{\beta}(n,0.5\kappa(\eta,x),x) \bigr). \tag*{\qed}
\end{align*}

\bigskip

\noindent {\it Proof of Proposition~\ref{prop:specfinsamp}.}
{We show that for each $n \in \mathbb{N}$ and a.e.~$x \in \X$, there exist positive constants $\tilde{K}(\kappa,x)$ and $\tilde{\beta}(n,\kappa,x)$ such that inequality~\eqref{eqn:largedeveqn} holds.}
{Following the arguments in the proof of Theorem~\ref{thm:exponentialconv}, inequality~\eqref{eqn:largedeveqn2} then implies}
\[
\pr \Bigl\{\uset{z \in \hS^{ER}_n(x)}{\sup} g(z;x) \leq v^*(x) + \kappa \Bigr\} \geq 1-2\tilde{K}(0.5\kappa,x) \exp\bigl(- \tilde{\beta}(n,0.5\kappa,x) \bigr),
\]
which {along with the definition of $S^{\kappa}(x)$} in turn implies that
\[
\prob{\hS^{ER}_n(x) \subseteq S^{\kappa}(x)} \geq 1-2\tilde{K}(0.5\kappa,x) \exp\bigl(- \tilde{\beta}(n,0.5\kappa,x) \bigr).
\]

We now state results that can be used to bound the constants $\tilde{K}(\kappa,x)$ and $\tilde{\beta}(\kappa,x)$ in inequality~\eqref{eqn:largedeveqn}; we ignore their dependence on $x$ to keep the exposition simple.
Theorems~7.66 and~7.67 of~\citet{shapiro2009lectures} imply for our setting of two-stage stochastic LP the bound 
\begin{align}
\label{eqn:unifexpboundlp}
&\mathbb{P}\Big\lbrace \uset{z \in \Z}{\sup} \: \abs*{g^*_n(z;x) - g(z;x)} > \kappa \Big\rbrace \leq O(1) \left(\dfrac{O(1)D}{\kappa}\right)^{d_z} \exp\left(-\dfrac{n\kappa^2}{O(1)\sigma^2_c(x)}\right)
\end{align}
for a.e.\ $x \in \X$. The following large deviation inequalities for our three different regression setups (see Section~\ref{sec:regression}) can be used to specialize the bound afforded by Lemma~\ref{lem:largedevofdev}:

\begin{enumerate}[itemsep=0.5em]
\item OLS regression: $\pr \Bigl\{\dfrac{1}{n} \displaystyle\sum_{i=1}^{n} \norm{\teps^i_{n}(x)}^2 > \kappa^2 \Bigr\} \leq \exp(d_x) \exp\left( -\dfrac{n\kappa^2}{O(1) \sigma^2 d_y}  \right)$, which follows from Remark~12 of~\citet{hsu2012random}, Theorem~2.2 and Remark~2.3 of~\citet{rigollet2015high}.

\item Lasso regression: $\pr \Bigl\{\dfrac{1}{n} \displaystyle\sum_{i=1}^{n} \norm{\teps^i_{n}(x)}^2 > \kappa^2 \Bigr\} \leq 2 d_x \exp\left( -\dfrac{n\kappa^2}{O(1) \sigma^2 s d_y}  \right)$, which follows from Theorem~2.1 and Corollary~1 of~\citet{bunea2007sparsity}.

\item kNN regression: Whenever $n \geq O(1)\left(\dfrac{O(1)}{\kappa}\right)^{\frac{d_x}{1-\gamma}}$ and $\dfrac{n^{\gamma}}{\log(n)} \geq \dfrac{O(1) d_x d_y \sigma^2}{\kappa^2}$, we have
{
\footnotesize
\[
\prob{\dfrac{1}{n} \displaystyle\sum_{i=1}^{n} \norm{\teps^i_{n}(x)}^2 > \kappa^2d_y} \leq \left(\dfrac{O(1) \sqrt{d_x}}{\kappa}\right)^{d_x} \exp\left( - O(1)n (O(1)\kappa)^{2d_x} \right) + O(1) n^{2d_x} \left(\dfrac{O(1)}{d_x}\right)^{d_x} \exp\left( - \dfrac{n^{\gamma}\kappa^2}{O(1)\sigma^2}  \right)
\]
}%
from Lemma~10 of~\citet{bertsimas2019predictions}.
\end{enumerate}
Suppose the regression step~\eqref{eqn:regr} is Lasso regression.
We have from Lemma~\ref{lem:largedevofdev} that
\[
\prob{\uset{z \in \Z}{\sup} \: \abs*{ \hg^{ER}_n(z;x) - g^*_n(z;x)} > \kappa} \leq O(1) d_x \exp\left( -\dfrac{n\kappa^2}{O(1) \sigma^2 s d_y}  \right).
\]
Along with the uniform exponential bound inequality~\eqref{eqn:unifexpboundlp}, this yields for a.e.\ $x \in \X$:
{
\small
\begin{align*}
\prob{\uset{z \in \Z}{\sup} \: \abs*{\hg^{ER}_n(z;x) - g(z;x)} > \kappa} &\leq O(1) \left(\dfrac{O(1)D}{\kappa}\right)^{d_z} \exp\left(-\dfrac{n\kappa^2}{O(1)\sigma^2_c(x)}\right) + O(1) d_x \exp\left( -\dfrac{n\kappa^2}{O(1) \sigma^2 s d_y}  \right).
\end{align*}
}%
Requiring each term in the r.h.s.\ of the above inequality to be $\leq 0.5\delta$ and using the union bound yields the stated conservative sample size estimates.
Sample complexities for OLS and kNN regression can be similarly derived.
\qed

\section{Alternative assumptions for analysis of the ER-SAA}
\label{sec:alt_assumptions}

{
In this section, we present alternative sets of assumptions under which theoretical guarantees similar to those in Section~\ref{sec:ersaa} hold for the ER-SAA problem.
The following assumption provides an alternative to the Lipschitz continuity Assumption~\ref{ass:equilipschitz}, and can be used to derive asymptotic guarantees for the ER-SAA.}

\begin{assumption}
\label{ass:equilipschitz_alt}
\setcounter{mycounter}{\value{assumption}}

Problem~\eqref{eqn:speq}, the regression step~\eqref{eqn:regr}, and the data~$\D_n$ satisfy:

\begin{enumerate}[label=(\themycounter\alph*),itemsep=0em]
\item \label{ass:boundederrorsa}  there exists a function $\delta: \X \to \R_+$ such that {for a.e.\ $x \in \X$}, $\norm{\teps_n^i(x)} \leq \delta(x)$, $\forall i \in [n]$, a.s.\ for $n$ large enough,

\item \label{ass:boundederrorsb} 
{for a.e.\ $x \in \X$ and for each $(z,\varepsilon) \in \Z \times \Xi$, the function~$c$ in problem~\eqref{eqn:speq} satisfies the following local Lipschitz inequality for each $\bar{y} \in \mathcal{B}_{\delta(x)}(f^*(x)+Q^*(x)\varepsilon) \cap \Y$:}
\[
{\abs*{c(z,\bar{y}) - c(z,f^*(x)+Q^*(x)\varepsilon)} \leq L_{\delta(x)}(z,f^*(x)+Q^*(x)\varepsilon) \norm{\bar{y} - (f^*(x)+Q^*(x)\varepsilon)},}
\]
{with the `local Lipschitz constant' $L_{\delta(x)}(z,f^*(x)+Q^*(x)\varepsilon)$ satisfying }
\[
{\uset{z \in \Z}{\sup} \: \dfrac{1}{n}\displaystyle\sum_{i=1}^{n} L^2_{\delta(x)}(z,f^*(x)+Q^*(x)\varepsilon^i) = O_p(1).}
\]
\end{enumerate}

\end{assumption}

Unlike Assumption~\ref{ass:equilipschitz},
Assumption~\ref{ass:boundederrorsb} only requires the function~$c(z,\cdot)$ to be locally Lipschitz continuous for each $z \in \Z$ with the local Lipschitz constant satisfying a uniform stochastic boundedness condition, 
but using this weaker assumption necessitates the stronger Assumption~\ref{ass:boundederrorsa} on the regression step~\eqref{eqn:regr}.
{Consider the homoscedastic setting (i.e., $\hQ_n := Q^* \equiv I$).}
Because the deviation terms satisfy $\norm{\teps^i_{n}(x)} \leq \norm{f^*(x) - \hf_n(x)} + \norm{f^*(x^i) - \hf_n(x^i)}$, 
$\forall i \in [n]$,
Assumption~\ref{ass:boundederrorsa} is satisfied {in this setting} for our running example of OLS regression, e.g., if the support~$\X$ is compact, the population regression problem has a unique solution~$\sth$, the strong pointwise LLN holds for the objective function (i.e., the empirical loss) of the regression problem~\eqref{eqn:regr},
and $\expect{\norm{\varepsilon}^2} < +\infty$
(see Theorem~5.4 of~\citet{shapiro2009lectures} for details).
{The support~$\X$ of the covariates being compact may not be an overly restrictive assumption because in many applications the covariates (e.g., temperature, precipitation, wind, or location) can be assumed to be bounded for all practical purposes.
For the heteroscedastic setting in Example~\ref{exm:regrexample}, following the proof of Lemma~\ref{lem:meandeviation}, we can show that Assumption~\ref{ass:boundederrorsa} holds if in addition the supports~$\X$ and~$\Xi$ of the covariates~$X$ and the errors~$\varepsilon$ are compact and a.s.\ for $n$ large enough the parameter estimates~$\hpi_n$ of~$\spi$ lie in a compact set.}
We present conditions under which Assumption~\ref{ass:boundederrorsb} holds in Section~\ref{sec:tssp} of the Appendix {(note that it readily holds for our running example of two-stage stochastic LP).}

{Replacing Assumption~\ref{ass:equilipschitz} with Assumption~\ref{ass:equilipschitz_alt} yields the following analogue of Lemma~\ref{lem:ersaa_error_decomp_term1}.}

\begin{lemma}
\label{lem:ersaa_error_decomp_term1_alt}
{Let Assumption~\ref{ass:equilipschitz_alt} hold. Then, for a.e.\ $x \in \X$, we a.s.\ have for $n$ large enough:}
\begin{align*}
{\uset{z \in \Z}{\sup} \: \abs*{\hg^{ER}_n(z;x) - g^*_n(z;x)}} &{\leq \uset{z \in \Z}{\sup} \biggl(\dfrac{1}{n}\displaystyle\sum_{i=1}^{n} L^2_{\delta(x)}(z,f^*(x)+Q^*(x)\varepsilon^i)\biggr)^{1/2} \biggl(\dfrac{1}{n}\displaystyle\sum_{i=1}^{n} \norm{\teps^i_{n}(x)}^2\biggr)^{1/2}.}
\end{align*}
\end{lemma}
\begin{proof}
{We have}
\begin{align*}
&{\uset{z \in \Z}{\sup} \: \abs*{\hg^{ER}_n(z;x) - g^*_n(z;x)}} \\
=&\: { \uset{z \in \Z}{\sup} \: \bigg\lvert \dfrac{1}{n}\displaystyle\sum_{i=1}^{n} c \bigl( z,\proj{\Y}{\hf_n(x) + \hat{Q}_n(x)\heps^i_{n}} \bigr) - \dfrac{1}{n}\displaystyle\sum_{i=1}^{n} c \left( z,f^*(x) + Q^*(x)\varepsilon^i \right) \bigg\rvert }\nonumber \\
\leq&\: { \uset{z \in \Z}{\sup} \: \dfrac{1}{n}\displaystyle\sum_{i=1}^{n} \big\lvert c \bigl( z,\proj{\Y}{\hf_n(x) + \hat{Q}_n(x)\heps^i_{n}} \bigr) - c \left( z,f^*(x) + Q^*(x)\varepsilon^i \right) \big\rvert} \nonumber \\
\leq&\: { \uset{z \in \Z}{\sup} \: \dfrac{1}{n}\displaystyle\sum_{i=1}^{n} L_{\delta(x)}(z,f^*(x)+Q^*(x)\varepsilon^i) \norm{\proj{\Y}{\hf_n(x) + \hat{Q}_n(x)\heps^i_{n}} - \left( f^*(x) + Q^*(x)\varepsilon^i \right)}} \nonumber \\
\leq&\: { \uset{z \in \Z}{\sup} \biggl(\dfrac{1}{n}\displaystyle\sum_{i=1}^{n} L^2_{\delta(x)}(z,f^*(x)+Q^*(x)\varepsilon^i)\biggr)^{1/2} \biggl(\dfrac{1}{n}\displaystyle\sum_{i=1}^{n} \norm{\teps^i_{n}(x)}^2\biggr)^{1/2},}
\end{align*}
{where the second inequality follows by Assumption~\ref{ass:equilipschitz_alt} and the final inequality follows by the Cauchy-Schwarz inequality and the Lipschitz continuity of orthogonal projections.}
\end{proof}

{Comparing with Lemma~\ref{lem:ersaa_error_decomp_term1}, Proposition~\ref{prop:uniformconvofobj}, and Theorem~\ref{thm:approxconv}, we see that convergence in probability of the root mean square deviation term $\bigl(\frac{1}{n}\sum_{i=1}^{n} \norm{\teps^i_{n}(x)}^2\bigr)^{1/2} \convinprob 0$ directly translates to the asymptotic guarantees for ER-SAA in Section~\ref{subsec:consistency} when Assumption~\ref{ass:equilipschitz} is replaced with Assumption~\ref{ass:equilipschitz_alt}.
The following analogue of Lemma~\ref{lem:meandeviation}, which we state without proof, bounds the root mean square deviation term $\bigl(\frac{1}{n}\sum_{i=1}^{n} \norm{\teps^i_{n}(x)}^2\bigr)^{1/2}$.}

\begin{lemma}
\label{lem:meandeviation_alt}
{Given estimates $\hf_n$ of $f^*$ and $\hQ_n$ of $Q^*$ with $[\hQ_n(\bar{x})]^{-1} \succ 0$ for each $\bar{x} \in \X$:}
\begin{enumerate}
\item {In the homoscedastic setting (i.e., $\hQ_n = Q^* \equiv I$), we have}
\[
{\biggl(\frac{1}{n} \sum_{i=1}^{n} \norm{\teps^i_n(x)}^2\biggr)^{1/2} \leq \sqrt{2}\norm{\hf_n(x) - f^*(x)} + \biggl(\frac{2}{n} \sum_{i=1}^{n} \norm{\hf_n(x^i) - f^*(x^i)}^2\biggr)^{1/2}}.
\]

\item {In the heteroscedastic setting, we have}
\begin{align*}
&{\biggl(\frac{1}{n} \sum_{i=1}^{n} \norm{\teps^i_n(x)}^2\biggr)^{1/2}} \nonumber \\
{\leq} &{O(1)\norm{\hf_n(x) - f^*(x)} + O(1)\norm{\hQ_n(x) - Q^*(x)}\biggl(\frac{1}{n} \sum_{i=1}^{n} \norm{\varepsilon^i}^2\biggr)^{1/2} +} \nonumber \\
& { O(1)\norm{\hQ_n(x)} \biggl(\frac{1}{n} \sum_{i=1}^{n} \bigl\lVert \bigl[\hQ_n(x^i)\bigr]^{-1} - \bigl[Q^*(x^i)\bigr]^{-1}\bigr\rVert^4\biggr)^{1/4} \biggl(\frac{1}{n} \sum_{i=1}^{n} \norm{Q^*(x^i)}^8\biggr)^{1/8} \biggl(\frac{1}{n} \sum_{i=1}^{n} \norm{\varepsilon^i}^8\biggr)^{1/8} +} \nonumber\\ 
& {O(1)\norm{\hQ_n(x)} \biggl(\frac{1}{n} \sum_{i=1}^{n} \bigl\lVert \bigl[\hQ_n(x^i)\bigr]^{-1}\bigr\rVert^4\biggr)^{1/4} \biggl(\frac{1}{n} \sum_{i=1}^{n} \norm{f^*(x^i) - \hf_n(x^i)}^4\biggr)^{1/4}.}
\end{align*}
\end{enumerate}
\end{lemma}
\begin{proof}
{Use the AM-QM inequality and follow the arguments in the proof of Lemma~\ref{lem:meandeviation}.}
\end{proof}

{Lemma~\ref{lem:meandeviation_alt} informs how Assumptions~\ref{ass:varweaklln},~\ref{ass:errorsweaklln}, and~\ref{ass:regconsist} may to be adapted to derive similar theoretical guarantees for ER-SAA as in Section~\ref{subsec:consistency}.} 
{Next, we consider the following strengthening of Assumption~\ref{ass:equilipschitz_alt} under which similar finite guarantees as in Section~\ref{subsec:finitesample} hold for the ER-SAA}.

\begin{assumption}
\label{ass:equilipschitz_v2}
There exists a function $\delta:\X \to \R_+$ such that for a.e.\ $x \in \X$, the regression step and the data~$\D_n$ satisfy $\norm{\teps^i_n(x)} \leq \delta(x)$, $\forall i \in [n]$ and $n \in \mathbb{N}$.
Furthermore, Assumption~\ref{ass:boundederrorsb} holds with {the local Lipschitz constant satisfying for a.e.\ $x \in \X$, $\kappa > 0$, and $n \in \mathbb{N}$:}
{
\footnotesize
\[
\pr \biggl\{\biggl(\dfrac{1}{n}\displaystyle\sum_{i=1}^{n} \uset{z \in \Z}{\sup} \: L^2_{\delta(x)}(z,f^*(x)+Q^*(x)\varepsilon^i)\biggr)^{1/2} > \biggl(\mathbb{E}\biggl[\uset{z \in \Z}{\sup} \: L^2_{\delta(x)}(z,f^*(x)+Q^*(x)\varepsilon)\biggr]\biggr)^{1/2} + \kappa \biggr\} \leq J_L(\kappa)\exp(-\gamma_L(n,\kappa;x)),
\]%
}%
where $J_L(\kappa), \gamma_L(n,\kappa;x) > 0$ with $\lim_{n \to \infty} \gamma_L(n,\kappa;x) = \infty$ for each $\kappa > 0$ and a.e.\ $x \in \X$.
\end{assumption}

{
The discussion following Assumption~\ref{ass:equilipschitz_alt} provides conditions under which $\norm{\teps^i_n(x)} \leq \delta(x)$, $\forall i \in [n]$, $n \in \mathbb{N}$, if we further assume the estimates $\hth_n$ and $\hpi_n$ lie in compact sets.
Section~\ref{sec:tssp} identifies conditions under which the uniform Lipschitz condition in Assumption~\ref{ass:equilipschitz_v2} holds (again, it readily holds for our running example of two-stage stochastic LP).}

{When Assumption~\ref{ass:equilipschitz_v2} is used in place of Assumption~\ref{ass:equilipschitz}, we can derive the following analogue of inequality~\eqref{eqn:meandeviation_finitesamp_0} in the proof of Lemma~\ref{lem:largedevofdev}:}
\begin{align*}
&{\mathbb{P}\Bigl\{\uset{z \in \Z}{\sup} \abs*{ \hg^{ER}_n(z;x) - g^*_n(z;x)} > \kappa\Bigr\}} \\
\leq \: & {\pr \biggl\{ \biggl(\dfrac{1}{n}\displaystyle\sum_{i=1}^{n} \uset{z \in \Z}{\sup} \: L^2_{\delta(x)}(z,f^*(x)+Q^*(x)\varepsilon^i)\biggr)^{1/2} \biggl(\dfrac{1}{n}\displaystyle\sum_{i=1}^{n} \norm{\teps^i_{n}(x)}^2\biggr)^{1/2} > \kappa \biggr\}} \\
\leq \: &  {\pr \biggl\{\biggl(\dfrac{1}{n}\displaystyle\sum_{i=1}^{n} \uset{z \in \Z}{\sup} \: L^2_{\delta(x)}(z,f^*(x)+Q^*(x)\varepsilon^i)\biggr)^{1/2} > \biggl(\mathbb{E}\biggl[\uset{z \in \Z}{\sup} \: L^2_{\delta(x)}(z,f^*(x)+Q^*(x)\varepsilon)\biggr]\biggr)^{1/2} + \kappa \biggr\} \: +} \\
& \qquad {\pr \biggl\{ \biggl[\biggl(\mathbb{E}\biggl[\uset{z \in \Z}{\sup} \: L^2_{\delta(x)}(z,f^*(x)+Q^*(x)\varepsilon)\biggr]\biggr)^{1/2} + \kappa\biggr] \biggl(\dfrac{1}{n}\displaystyle\sum_{i=1}^{n} \norm{\teps^i_{n}(x)}^2\biggr)^{1/2} > \kappa \biggr\}} \\
\leq \: & { J_L(\kappa)\exp(-\gamma_L(n,\kappa;x)) + \pr\biggl\{ \biggl(\dfrac{1}{n}\displaystyle\sum_{i=1}^{n} \norm{\teps^i_{n}(x)}^2\biggr)^{1/2} > \dfrac{\kappa}{B_L(\kappa;x)} \biggr\},}
\end{align*}
{where $B_L(\kappa;x) := \Bigl(\mathbb{E}\bigl[\sup_{z \in \Z} L^2_{\delta(x)}(z,f^*(x)+Q^*(x)\varepsilon)\bigr]\Bigr)^{1/2} + \kappa$ and the last inequality follows by Assumption~\ref{ass:equilipschitz_v2}.
}
{Therefore, finite sample guarantees for the root mean square deviation term $\bigl(\frac{1}{n}\sum_{i=1}^{n} \norm{\teps^i_{n}(x)}^2\bigr)^{1/2}$ directly translate to finite sample guarantees for ER-SAA of the form in Lemma~\ref{lem:largedevofdev} and Theorem~\ref{thm:exponentialconv}.
Once again, Lemma~\ref{lem:meandeviation_alt} provides guidance for how Assumptions~\ref{ass:varlargedev},~\ref{ass:errorlargedev}, and~\ref{ass:reglargedev} may to be adapted to derive similar theoretical guarantees for ER-SAA as in Section~\ref{subsec:finitesample}.}

\section{Rate of convergence of the ER-SAA estimator}
\label{sec:ersaa_rate}

We investigate the rate of convergence of the optimal objective value of the sequence of ER-SAA problems~\eqref{eqn:app} to that of the true problem~\eqref{eqn:speq}.
This analysis requires the following additional assumptions on the true problem~\eqref{eqn:speq} and the regression step~\eqref{eqn:regr}.

\begin{assumption}
\label{ass:functionalclt}
The function~$c$ in problem~\eqref{eqn:speq} and the errors~$\{\varepsilon^i\}$ satisfy the following functional central limit theorem (CLT) for the full-information SAA objective~\eqref{eqn:fullinfsaa}:
\[
\sqrt{n} \left( g^*_n(\cdot;x) - g(\cdot;x) \right) \xrightarrow{d} V(\cdot;x), \quad \text{for a.e. } x \in \X,
\]
{where $g^*_n(\cdot;x)$, $g(\cdot;x)$, and $V(\cdot;x)$ are (random) elements of $\mathcal{C}(\Z)$, the Banach space of real-valued continuous functions on $\Z$ equipped with the supremum norm.}
\end{assumption}

Assumption~\ref{ass:functionalclt} holds, for instance, when the errors $\{\varepsilon^i\}$ are i.i.d.,
the function~$c(\cdot,y)$ is Lipschitz continuous on~$\Z$ for a.e.\ $y \in \Y$ with an $L^2(\Y)$ Lipschitz constant,
and, for a.e.~$x \in \X$, there exists \mbox{$\tz \in \Z$} such that $\expv \bigl[\left(c(\tz,f^*(x)+Q^*(x)\varepsilon)\right)^2 \bigr] < +\infty$ (see page~164 of~\citet{shapiro2009lectures} for details).
Theorem~1 of~\citet{doukhan1995invariance}, Theorem~2.1 of~\citet{arcones1994central}, Theorem~9 of~\citet{arcones1994limit}, and Corollary~2.3 of~\citet{andrews1994introduction} provide conditions under which the functional CLT holds 
under mixing assumptions on $\{\varepsilon^i\}$.
Theorems~1.5.4 and~1.5.6 of~\citet{vaart1996weak} present a general set of conditions under which the functional CLT holds.

The next assumption, which strengthens Assumption~\ref{ass:regconsist}, ensures that the deviation of the ER-SAA problem~\eqref{eqn:app} from the full-information SAA problem~\eqref{eqn:fullinfsaa} converges at a certain rate.

\begin{assumption}
\label{ass:regconvrate}
\setcounter{mycounter}{\value{assumption}}
There is a constant $0 < \alpha \leq 1$ (that is independent of the number of samples~$n$, but could depend on the dimension~$d_x$ of the covariates~$X$) such that the {regression estimates~$\hf_n$ and~$\hQ_n$} satisfy the following convergence rate criteria:
\begin{enumerate}[label=(\themycounter\alph*),itemsep=0.2em]
\item \label{ass:regconvrate_point} $\norm{\hf_n(x) - f^*(x)} = O_p(n^{-\alpha/2})$ for a.e.\ $x \in \X$,

\item \label{ass:regconvrate_mse} $\dfrac{1}{n} \displaystyle\sum_{i=1}^{n} \norm{\hf_n(x^i) - f^*(x^i)}^2 = O_p(n^{-\alpha})$,

\item \label{ass:regconvrate_point2} ${\norm{\hQ_n(x) - Q^*(x)} = O_p(n^{-\alpha/2})}$ for a.e.\ $x \in \X$,

\item \label{ass:regconvrate_mse2} {$\dfrac{1}{n} \displaystyle\sum_{i=1}^{n} \bigl\lVert \bigl[\hQ_n(x^i)\bigr]^{-1} - \bigl[Q^*(x^i)\bigr]^{-1}\bigr\rVert^2 = O_p(n^{-\alpha})$.}
\end{enumerate}
\end{assumption}

Note that the $O_p(\cdot)$ terms in Assumption~\ref{ass:regconvrate} hide factors proportional to the dimension~$d_y$ of the random vector~$Y$.
For our running example of OLS regression, Assumptions~\ref{ass:regconvrate_point} and~\ref{ass:regconvrate_mse} hold with~$\alpha = 1$ 
under mild assumptions on the data~$\D_n$ and the distribution~$P_X$ of the covariates~\citep[see Chapter~5 of][]{white2014asymptotic}.
A similar rate holds for Lasso, best subset selection, and many other parametric regression procedures under mild assumptions.
Nonparametric regression procedures such as kNN and RF regression, on the other hand, typically only satisfy these assumptions with constant~$\alpha = \frac{O(1)}{d_x}$. This rate cannot be improved upon in general, and is commonly referred to as the curse of dimensionality.
Structured nonparametric regression methods such as sparse additive models~\citep{raskutti2012minimax} can hope to break the curse of dimensionality and achieve rates with $\alpha = 1$.
{Assumptions~\ref{ass:regconvrate_point2} and~\ref{ass:regconvrate_mse2} also hold for our running Example~\ref{exm:regrexample} if the parameter estimates $\hpi_n$ satisfy $\norm{\hpi_n - \spi} = O_p(n^{-\alpha/2})$.}
Section~\ref{sec:regression} of the Appendix verifies that Assumption~\ref{ass:regconvrate} holds for these prediction setups with the stated constants~$\alpha$.

Our main result of this section extends 
Theorem~5.7 of~\citet{shapiro2009lectures} to establish a rate at which the optimal objective value of the ER-SAA problem~\eqref{eqn:app} converges to that of the true problem~\eqref{eqn:speq}. 
We hide the dependence of the convergence rate on the dimensions~$d_x$ and~$d_y$ of the covariates~$X$ and random vector~$Y$. We discuss how these dimensions affect the rate of convergence via a non-asymptotic/finite sample analysis in Section~\ref{subsec:finitesample}.
Note that the convergence rate analysis in Theorem~5.7 of~\citet{shapiro2009lectures} for the full-information SAA problem~\eqref{eqn:fullinfsaa} is sharper than Theorem~\ref{thm:rateofconv} in the sense that it also characterizes the asymptotic distribution of the optimal objective value, see equations~(5.25) and~(5.26) therein.
{Deriving the asymptotic distribution of the optimal value of the ER-SAA is an interesting direction for future work.}

\begin{restatable}{theorem}{thmrateofconv}
\label{thm:rateofconv}
{Suppose Assumptions~\ref{ass:equilipschitz},~\ref{ass:varweaklln},~\ref{ass:errorsweaklln},~\ref{ass:functionalclt} and~\ref{ass:regconvrate} hold.
Then, we have $\abs{\hv^{ER}_n(x) - v^*(x)} = O_p(n^{-\frac{\alpha}{2}})$ and $\abs{g(\hz^{ER}_n(x);x) - \hv^{ER}_n(x)} = O_p(n^{-\frac{\alpha}{2}})$ for a.e.\ $x \in \X$.}
\end{restatable}
\begin{proof}
{We begin by showing that $\sup_{z \in \Z} \abs*{\hg^{ER}_{n}(z;x) - g(z;x)} = O_p(n^{-\alpha/2})$ for a.e.\ $x \in \X$.}

{Assumption~\ref{ass:functionalclt} implies $\sqrt{n} \sup_{z \in \Z} \abs*{g^*_n(z;x) - g(z;x)} = O_p(1)$ for a.e.\ $x \in \X$, which in turn implies $\sup_{z \in \Z} \abs*{g^*_n(z;x) - g(z;x)} = O_p(n^{-1/2})$ for a.e.\ $x \in \X$.}

{We now bound $\sup_{z \in \Z} \: \abs*{\hg^{ER}_{n}(z;x) - g(z;x)}$ from above using Lemmas~\ref{lem:ersaa_error_decomp_term1} and~\ref{lem:meandeviation}.
Assumption~\ref{ass:equilipschitz} and Lemma~\ref{lem:ersaa_error_decomp_term1} imply}
\[
{\uset{z \in \Z}{\sup} \: \abs*{\hg^{ER}_n(z;x) - g^*_n(z;x)} \leq \Bigl(\uset{z \in \Z}{\sup} \: L(z)\Bigr) \biggl(\dfrac{1}{n}\displaystyle\sum_{i=1}^{n} \norm{\teps^i_{n}(x)}\biggr) = O(1) \biggl(\dfrac{1}{n}\displaystyle\sum_{i=1}^{n} \norm{\teps^i_{n}(x)}\biggr).}
\]
{Lemma~\ref{lem:meandeviation} along with Assumptions~\ref{ass:varweaklln},~\ref{ass:errorsweaklln}, and~\ref{ass:regconvrate} and the continuous mapping theorem imply $\frac{1}{n}\sum_{i=1}^{n} \norm{\teps^i_{n}(x)} = O_p(n^{-\alpha/2})$ for a.e.\ $x \in \X$.
Consequently, $\sup_{z \in \Z} \: \abs*{\hg^{ER}_{n}(z;x) - g^*_n(z;x)} = O_p(n^{-\alpha/2})$ for a.e.\ $x \in \X$.
Since $\alpha \in (0,1]$, the above two probability inequalities yield}
\[
{\sup_{z \in \Z} \abs*{\hg^{ER}_{n}(z;x) - g(z;x)} = O_p(n^{-\alpha/2}), \quad \text{for a.e. } x \in \X.}
\]
{This implies that for a.e.\ $x \in \X$ and any $\beta > 0$, there exists $M_{\beta} > 0$ such that
\[
\mathbb{P}\biggl\{\uset{z \in \Z}{\sup} \abs*{\hg^{ER}_{n}(z;x) - g(z;x)} > M_{\beta}n^{-\alpha/2}\biggr\} < \beta.
\]
Consequently, we have for a.e.\ $x \in \X$:}
\begin{align*}
{\prob{\hv^{ER}_n(x) > v^*(x) + M_{\beta}n^{-\alpha/2}}} &{\leq \prob{\hg^{ER}_{n}(z^*(x);x) > v^*(x) + M_{\beta}n^{-\alpha/2}}} \\
&{\leq \prob{\abs*{\hg^{ER}_{n}(z^*(x);x) - v^*(x)} > M_{\beta}n^{-\alpha/2}} \leq \beta,}\\
{\prob{v^*(x) > \hv^{ER}_n(x) + M_{\beta}n^{-\alpha/2}}} &{\leq \prob{g(\hz^{ER}_n(x);x) > \hv^{ER}_n(x) + M_{\beta}n^{-\alpha/2}}} \\
&{\leq \prob{\abs*{\hv^{ER}_n(x) - g(\hz^{ER}_n(x);x)} > M_{\beta}n^{-\alpha/2}} \leq \beta.}
\end{align*}
{Therefore, both $\abs{\hv^{ER}_n(x) - v^*(x)}$ and $\abs*{g(\hz^{ER}_n(x);x) - v^*(x)}$ are $O_p(n^{-\alpha/2})$.}
\end{proof}

{Note that Assumptions~\ref{ass:varweaklln},~\ref{ass:errorsweaklln},~\ref{ass:regconvrate_point2}, and~\ref{ass:regconvrate_mse2} are not required to establish Theorem~\ref{thm:rateofconv} in the homoscedastic case {($\hQ_n := Q^* \equiv I$)}. Additionally, Assumption~\ref{ass:regconvrate_mse} may be weakened in this setting to $\frac{1}{n} \sum_{i=1}^{n} \norm{f^*(x^i) - \hf_n(x^i)} = O_p(n^{-\alpha/2})$ on account of Lemma~\ref{lem:meandeviation}.
When Assumption~\ref{ass:equilipschitz} is replaced with Assumption~\ref{ass:equilipschitz_alt},
rates of convergence for the root mean square deviation term $\bigl(\frac{1}{n}\sum_{i=1}^{n} \norm{\teps^i_{n}(x)}^2\bigr)^{1/2} \convinprob 0$ directly translate to rates of convergence of the ER-SAA estimators similar to Theorem~\ref{thm:rateofconv} (cf.\ Section~\ref{sec:alt_assumptions} of the Appendix).}

\section{Analysis for the {jackknife} and {jackknife+} estimators}
\label{sec:jackknife}

In this section, we analyze the consistency, rate of convergence, and finite sample guarantees of the J-SAA and J+-SAA estimators obtained by solving problems~\eqref{eqn:jackknife} and~\eqref{eqn:jackknife+}, respectively, under certain assumptions on the true problem~\eqref{eqn:speq} and the prediction step~\eqref{eqn:regr}.
We omit proofs because they are similar to the proofs of results in Section~\ref{sec:ersaa}.
In place of the sequence of deviation terms~$\{\teps^i_{n}(x)\}$ considered in Section~\ref{sec:ersaa}, we consider the following deviation sequences~$\{\teps^{i,J}_{n}(x)\}$ and~$\{\teps^{i,J+}_{n}(x)\}$:
\begin{align*}
\teps^{i,J}_{n}(x) &:= \left( \hf_n(x) + {\hQ_n(x)}\heps^{i}_{n,J} \right) - \left( f^*(x) + {Q^*(x)}\varepsilon^i \right), \quad \forall i \in [n], \\
\teps^{i,J+}_{n}(x) &:= \left( \hf_{-i}(x) + {\hQ_{-i}(x)}\heps^{i}_{n,J} \right) - \left( f^*(x) + {Q^*(x)}\varepsilon^i \right), \quad \forall i \in [n].
\end{align*}
We let~$\hz^{J}_n(x)$ and~$\hz^{J+}_n(x)$ denote an optimal solution to problem~\eqref{eqn:jackknife} and~\eqref{eqn:jackknife+}, respectively, and $\hS^{J}_n(x)$ and~$\hS^{J+}_n(x)$ denote the corresponding sets of optimal solutions.
We assume throughout that the sets~$\hS^{J}_n(x)$ and~$\hS^{J+}_n(x)$ are nonempty for a.e.~$x \in \X$.

{We have the following analogue of Lemma~\ref{lem:meandeviation} for the jackknife-based mean deviation terms.}

\begin{lemma}
\label{lem:jack_meandeviation}
{Given regression estimates $\hf_n$, $\{\hf_{-i}\}$ of $f^*$ and $\hQ_n$, $\{\hQ_{-i}\}$ of $Q^*$ with $[\hQ_n(\bar{x})]^{-1} \succ 0$ and $[\hQ_{-i}(\bar{x})]^{-1} \succ 0$ for each $i \in [n]$ and $\bar{x} \in \X$:}
\begin{enumerate}
\item {In the homoscedastic setting (i.e., $\hQ_n = \hQ_{-i} = Q^* \equiv I$), we have}
\begin{align*}
{\frac{1}{n} \sum_{i=1}^{n} \norm{\teps^{i,J}_n(x)}} &{\leq \norm{\hf_n(x) - f^*(x)} + \frac{1}{n} \sum_{i=1}^{n} \norm{\hf_{-i}(x^i) - f^*(x^i)}}, \\
{\frac{1}{n} \sum_{i=1}^{n} \norm{\teps^{i,J+}_n(x)}} &{\leq  \frac{1}{n} \sum_{i=1}^{n} \norm{\hf_{-i}(x) - f^*(x)} + \frac{1}{n} \sum_{i=1}^{n} \norm{\hf_{-i}(x^i) - f^*(x^i)}}.
\end{align*}

\item {In the heteroscedastic setting, we have}
\small
\begin{align*}
{ \frac{1}{n}} &{\sum_{i=1}^{n} \norm{\teps^{i,J}_n(x)} \leq} \\
& {\norm{\hf_n(x) - f^*(x)} + \norm{\hQ_n(x) - Q^*(x)}\biggl(\frac{1}{n} \sum_{i=1}^{n} \norm{\varepsilon^i}\biggr) +}  \\
& { \norm{\hQ_n(x)} \biggl(\frac{1}{n} \sum_{i=1}^{n} \bigl\lVert \bigl[\hQ_{-i}(x^i)\bigr]^{-1} - \bigl[Q^*(x^i)\bigr]^{-1}\bigr\rVert^2\biggr)^{1/2} \biggl(\frac{1}{n} \sum_{i=1}^{n} \norm{Q^*(x^i)}^4\biggr)^{1/4} \biggl(\frac{1}{n} \sum_{i=1}^{n} \norm{\varepsilon^i}^4\biggr)^{1/4} +} \\ 
& {\norm{\hQ_n(x)} \biggl(\frac{1}{n} \sum_{i=1}^{n} \bigl\lVert \bigl[\hQ_{-i}(x^i)\bigr]^{-1}\bigr\rVert^2\biggr)^{1/2} \biggl(\frac{1}{n} \sum_{i=1}^{n} \norm{f^*(x^i) - \hf_{-i}(x^i)}^2\biggr)^{1/2}}
\end{align*}
{and} 
\begin{align*} 
{\frac{1}{n}} &{\sum_{i=1}^{n} \norm{\teps^{i,J+}_n(x)} \leq}  \\
 &{\frac{1}{n} \sum_{i=1}^{n}\norm{\hf_{-i}(x) - f^*(x)} + \biggl(\frac{1}{n} \sum_{i=1}^{n} \norm{\hQ_{-i}(x) - Q^*(x)}^2\biggr)^{1/2} \biggl(\frac{1}{n} \sum_{i=1}^{n} \norm{\varepsilon^i}^2\biggr)^{1/2} +}  \\
& { \biggl(\frac{1}{n} \sum_{i=1}^{n} \norm{\hQ_{-i}(x)}^2\biggr)^{1/2}  \biggl(\frac{1}{n} \sum_{i=1}^{n} \bigl\lVert \bigl[\hQ_{-i}(x^i)\bigr]^{-1} - \bigl[Q^*(x^i)\bigr]^{-1}\bigr\rVert^4\biggr)^{1/4} \biggl(\frac{1}{n} \sum_{i=1}^{n} \norm{Q^*(x^i)}^8\biggr)^{1/8} \biggl(\frac{1}{n} \sum_{i=1}^{n} \norm{\varepsilon^i}^8\biggr)^{1/8} +} \\ 
& { \biggl(\frac{1}{n} \sum_{i=1}^{n} \norm{\hQ_{-i}(x)}^2\biggr)^{1/2} \biggl(\frac{1}{n} \sum_{i=1}^{n} \bigl\lVert \bigl[\hQ_{-i}(x^i)\bigr]^{-1}\bigr\rVert^4\biggr)^{1/4} \biggl(\frac{1}{n} \sum_{i=1}^{n} \norm{f^*(x^i) - \hf_{-i}(x^i)}^4\biggr)^{1/4}.}
\end{align*}
\normalsize
\end{enumerate}
\end{lemma}

Therefore, assumptions on the quantities appearing in Lemma~\ref{lem:meandeviation} for the mean deviation term of the ER-SAA may be replaced with assumptions on the quantities appearing in the above inequalities to derive similar results for the jackknife-based estimators as the ER-SAA estimator.
{While we focus on analyzing the jackknife-based SAAs when Assumption~\ref{ass:equilipschitz} holds, note that a similar analysis can be carried out when Assumptions~\ref{ass:equilipschitz_alt} and~\ref{ass:equilipschitz_v2} are adapted for the jackknife-based SAAs by following the arguments in Section~\ref{sec:alt_assumptions}. We omit these details for brevity.}

\subsection{Consistency and asymptotic optimality}
\label{subsec:jack_consistency}

We present conditions under which the optimal value and optimal solutions to the J-SAA and J+-SAA problems~\eqref{eqn:jackknife} and~\eqref{eqn:jackknife+} asymptotically converge to those of the true problem~\eqref{eqn:speq}.
We make the following assumptions on the consistency of the (leave-one-out version of the) regression procedure~\eqref{eqn:regr} that adapts Assumption~\ref{ass:regconsist} for the J-SAA and J+-SAA approaches.

\begin{assjack}{ass:regconsist}
\label{ass:jack_regconsist}
The regression estimates $\hf_n$ and $\{\hf_{-i}\}$ of $f^*$ and the regression estimates $\hQ_n$ and $\{\hQ_{-i}\}$ of $Q^*$ satisfy the following consistency properties:
\vspace*{-0.1in}
\begin{multicols}{2}
\begin{enumerate}[label=(\ref{ass:regconsist}J\alph*),itemsep=0em]
\item \label{ass:jack_regconsist_point} $\hf_n(x) \xrightarrow{p} f^*(x)$ for a.e.\ $x \in \X$,

\item \label{ass:jack_regconsist_mse} $\dfrac{1}{n} \displaystyle\sum_{i=1}^{n} \norm{f^*(x^i) - \hf_{-i}(x^i)}^2 \xrightarrow{p} 0$,

\columnbreak

\item \label{ass:jack_regconsist_point2} {$\hQ_n(x) \xrightarrow{p} Q^*(x)$} for a.e.\ $x \in \X$,

\item \label{ass:jack_regconsist_mse2} {$\dfrac{1}{n} \displaystyle\sum_{i=1}^{n} \bigl\lVert \bigl[\hQ_{-i}(x^i)\bigr]^{-1} - \bigl[Q^*(x^i)\bigr]^{-1}\bigr\rVert^2 \xrightarrow{p} 0$.}
\end{enumerate}%
\end{multicols}%
\end{assjack}

\begin{assjackplus}{ass:regconsist}
\label{ass:jack+_regconsist}
The regression estimates $\hf_n$ and $\{\hf_{-i}\}$ of $f^*$ and the regression estimates $\hQ_n$ and $\{\hQ_{-i}\}$ of $Q^*$ satisfy the following consistency properties:
\vspace*{-0.1in}
\begin{multicols}{2}
\begin{enumerate}[label=(\ref{ass:regconsist}J+\alph*),itemsep=0em]
\item \label{ass:jack+_regconsist_point} $\dfrac{1}{n} \displaystyle\sum_{i=1}^{n} \norm{f^*(x) - \hf_{-i}(x)} \convinprob 0$ for a.e.\ $x \in \X$,

\item \label{ass:jack+_regconsist_mse} {$\dfrac{1}{n} \displaystyle\sum_{i=1}^{n} \norm{f^*(x^i) - \hf_{-i}(x^i)}^4 \xrightarrow{p} 0$,}

\columnbreak

\item \label{ass:jack+_regconsist_point2} \mbox{{$\dfrac{1}{n} \displaystyle\sum_{i=1}^{n} \norm{Q^*(x) - \hQ_{-i}(x)}^2 \convinprob 0$} for a.e.\ $x \in \X$,}

\item \label{ass:jack+_regconsist_mse2} {$\dfrac{1}{n} \displaystyle\sum_{i=1}^{n} \bigl\lVert \bigl[\hQ_{-i}(x^i)\bigr]^{-1} - \bigl[Q^*(x^i)\bigr]^{-1}\bigr\rVert^4 \xrightarrow{p} 0$.}
\end{enumerate}%
\end{multicols}%
\end{assjackplus}

Section~\ref{sec:regression} identifies conditions under which Assumptions~\ref{ass:jack_regconsist} and~\ref{ass:jack+_regconsist} hold for OLS, Lasso, kNN, and RF regression with i.i.d.\ data~$\{(x^i,\varepsilon^i)\}$.
We also require the following strengthening of Assumptions~\ref{ass:varweaklln} and~\ref{ass:errorsweaklln} for the J+-SAA problem.

\begin{assjackplus}{ass:varweaklln}
\label{ass:jack+_varweaklln}
{The function~$Q^*$ and the data~$\{x^i\}$ satisfy the weak LLNs}
\[
{\frac{1}{n} \sum_{i=1}^{n} \norm{Q^*(x^i)}^8 \convinprob \mathbb{E}[\norm{Q^*(X)}^8] \quad \text{and} \quad \frac{1}{n} \sum_{i=1}^{n} \bigl\lVert \bigl[Q^*(x^i)\bigr]^{-1}\bigr\rVert^4 \convinprob \mathbb{E}\bigl[ \bigl\lVert \bigl[Q^*(X)\bigr]^{-1} \bigr\rVert^4 \bigr].}
\]
\end{assjackplus}

\begin{assjackplus}{ass:errorsweaklln}
\label{ass:jack+_errorsweaklln}
{The error samples~$\{\varepsilon^i\}_{i=1}^{n}$ satisfy the weak LLN $\dfrac{1}{n} \displaystyle\sum_{i=1}^{n} \norm{\varepsilon^i}^8 \convinprob \mathbb{E}[\norm{\varepsilon}^8]$.}
\end{assjackplus}

We now state conditions under which the sequence of objective functions of problems~\eqref{eqn:jackknife} and~\eqref{eqn:jackknife+} converge uniformly to the objective function of the true problem~\eqref{eqn:speq} on the set~$\Z$.
We group the results for the J-SAA and J+-SAA problems for brevity (the individual results are apparent).

\begin{proposition}
\label{prop:jack_uniformconvofobj}
Suppose Assumptions~\ref{ass:equilipschitz} through~\ref{ass:errorsweaklln} and Assumptions~\ref{ass:jack+_varweaklln},~\ref{ass:jack+_errorsweaklln},~\ref{ass:jack_regconsist}, and~\ref{ass:jack+_regconsist} hold. 
Then, for a.e.~$x \in \X$, the sequences of objective functions of J-SAA and J+-SAA problems~\eqref{eqn:jackknife} and~\eqref{eqn:jackknife+} converge uniformly in probability to the objective function of the true problem~\eqref{eqn:speq} on the feasible region~$\Z$.
\end{proposition}

Proposition~\ref{prop:jack_uniformconvofobj} helps us establish conditions under which the optimal objective values and solutions of the J-SAA and J+-SAA problems~\eqref{eqn:jackknife} and~\eqref{eqn:jackknife+} converge to those of the true problem~\eqref{eqn:speq}.

\begin{theorem}
\label{thm:jack_approxconv}
Suppose Assumptions~\ref{ass:equilipschitz} through~\ref{ass:errorsweaklln} and Assumptions~\ref{ass:jack+_varweaklln},~\ref{ass:jack+_errorsweaklln},~\ref{ass:jack_regconsist}, and~\ref{ass:jack+_regconsist} hold. 
Then, we have $\hv^{J}_n(x) \xrightarrow{p} v^*(x)$, $\hv^{J+}_n(x) \xrightarrow{p} v^*(x)$, $\dev{\hS^{J}_n(x)}{S^*(x)} \xrightarrow{p} 0$, $\dev{\hS^{J+}_n(x)}{S^*(x)} \xrightarrow{p} 0$, $\uset{z \in \hS^{J}_n(x)}{\sup} g(z;x) \xrightarrow{p} v^*(x)$, and $\uset{z \in \hS^{J+}_n(x)}{\sup} g(z;x) \xrightarrow{p} v^*(x)$ for a.e.~$x \in \X$.
\end{theorem}

{Assumptions~\ref{ass:varweaklln},~\ref{ass:errorsweaklln},~\ref{ass:jack+_varweaklln},~\ref{ass:jack+_errorsweaklln},~\ref{ass:jack_regconsist_point2},~\ref{ass:jack+_regconsist_point2},~\ref{ass:jack_regconsist_mse2}, and~\ref{ass:jack+_regconsist_mse2} are not required to establish Proposition~\ref{prop:jack_uniformconvofobj} and Theorem~\ref{thm:jack_approxconv} in the homoscedastic case (assumptions that only involve $\hQ_n$ and $\{\hQ_{-i}\}$ may be omitted in this setting). Additionally, Assumptions~\ref{ass:jack_regconsist_mse} and~\ref{ass:jack+_regconsist_mse} may be weakened in this setting to $\frac{1}{n} \sum_{i=1}^{n} \norm{f^*(x^i) - \hf_{-i}(x^i)} \xrightarrow{p} 0$ on account of Lemma~\ref{lem:jack_meandeviation}.}

\subsection{Rates of convergence}
\label{subsec:jack_rateofconv}

We derive rates of convergence of the optimal objective value of the sequence of J-SAA and \mbox{J+-SAA} problems~\eqref{eqn:jackknife} and~\eqref{eqn:jackknife+} to the optimal objective value of the true problem~\eqref{eqn:speq}.
To enable this, we make the following assumptions on the regression procedure~\eqref{eqn:regr} that adapt Assumption~\ref{ass:regconvrate} to strengthen Assumptions~\ref{ass:jack_regconsist} and~\ref{ass:jack+_regconsist}. 
Assumptions~\ref{ass:jack_regconvrate} and~\ref{ass:jack+_regconvrate} ensure that the deviations of the J-SAA and J+-SAA problems from the FI-SAA problem~\eqref{eqn:fullinfsaa} converge at a certain rate.

\begin{assjack}{ass:regconvrate}
\label{ass:jack_regconvrate}
There is a constant $0 < \alpha \leq 1$ (that is independent of the number of samples~$n$, but could depend on the dimension~$d_x$ of the covariates~$X$) such that the regression procedure~\eqref{eqn:regr} satisfies the following asymptotic convergence rate criterion:
\begin{enumerate}[label=(\ref{ass:regconvrate}J\alph*),itemsep=0em]
\item \label{ass:jack_regconvrate_point} $\norm{\hf_n(x) - f^*(x)} = O_p(n^{-\alpha/2})$ for a.e.\ $x \in \X$,

\item \label{ass:jack_regconvrate_mse} $\dfrac{1}{n} \displaystyle\sum_{i=1}^{n} \norm{f^*(x^i) - \hf_{-i}(x^i)}^2 = O_p(n^{-\alpha})$,

\item \label{ass:jack_regconvrate_point2} {$\norm{\hQ_n(x) - Q^*(x)} = O_p(n^{-\alpha/2})$} for a.e.\ $x \in \X$,

\item \label{ass:jack_regconvrate_mse2} {$\dfrac{1}{n} \displaystyle\sum_{i=1}^{n} \bigl\lVert \bigl[\hQ_{-i}(x^i)\bigr]^{-1} - \bigl[Q^*(x^i)\bigr]^{-1}\bigr\rVert^2 = O_p(n^{-\alpha})$.}
\end{enumerate}
\end{assjack}

\begin{assjackplus}{ass:regconvrate}
\label{ass:jack+_regconvrate}
There is a constant $0 < \alpha \leq 1$ (that is independent of the number of samples~$n$, but could depend on the dimension~$d_x$ of the covariates~$X$) such that the regression procedure~\eqref{eqn:regr} satisfies the following asymptotic convergence rate criterion:
\begin{enumerate}[label=(\ref{ass:regconvrate}J+\alph*),itemsep=0em]
\item \label{ass:jack+_regconvrate_point} $\dfrac{1}{n} \displaystyle\sum_{i=1}^{n} \norm{f^*(x) - \hf_{-i}(x)} = O_p(n^{-\alpha/2})$ for a.e.\ $x \in \X$,

\item \label{ass:jack+_regconvrate_mse} $\dfrac{1}{n} \displaystyle\sum_{i=1}^{n} \norm{f^*(x^i) - \hf_{-i}(x^i)}^4 = O_p(n^{-2\alpha})$,

\item \label{ass:jack+_regconvrate_point2} {$\dfrac{1}{n} \displaystyle\sum_{i=1}^{n} \norm{Q^*(x) - \hQ_{-i}(x)}^2 = O_p(n^{-\alpha})$} for a.e.\ $x \in \X$,

\item \label{ass:jack+_regconvrate_mse2} {$\dfrac{1}{n} \displaystyle\sum_{i=1}^{n} \bigl\lVert \bigl[\hQ_{-i}(x^i)\bigr]^{-1} - \bigl[Q^*(x^i)\bigr]^{-1}\bigr\rVert^4 = O_p(n^{-2\alpha})$.}
\end{enumerate}
\end{assjackplus}

Section~\ref{sec:regression} demonstrates that Assumptions~\ref{ass:jack_regconvrate} and~\ref{ass:jack+_regconvrate} hold with rates similar to those in Assumption~\ref{ass:regconvrate} when the data~$\{(x^i,\varepsilon^i)\}$ is i.i.d.
Along with Lemma~\ref{lem:jack_meandeviation}, Assumptions~\ref{ass:jack_regconvrate} and~\ref{ass:jack+_regconvrate} imply that the mean deviation terms for the J-SAA and J+-SAA approaches can be bounded as $\frac{1}{n}\sum_{i=1}^{n} \norm{\teps^{i,J}_{n}(x)} = O_p(n^{-\alpha})$ and~$\frac{1}{n}\sum_{i=1}^{n} \norm{\teps^{i,J+}_{n}(x)} = O_p(n^{-\alpha})$ for a.e.~$x \in \X$.

We now establish rates at which the optimal objective value of the J-SAA and J+-SAA problems converge to the optimal objective value of the true problem~\eqref{eqn:speq}. 
We hide the dependence of the convergence rate on the dimensions~$d_x$ and~$d_y$ of the covariates~$X$ and random vector~$Y$. The analysis in the next section can account for how these dimensions affect the rate of convergence.

\begin{theorem}
\label{thm:jack_rateofconv}
Suppose Assumptions~\ref{ass:equilipschitz},~\ref{ass:varweaklln},~\ref{ass:errorsweaklln},~\ref{ass:jack+_varweaklln},~\ref{ass:jack+_errorsweaklln},~\ref{ass:jack_regconvrate}, and~\ref{ass:jack+_regconvrate} hold.
Then, for a.e.~$x \in \X$, we have $\abs{\hv^{J}_n(x) - v^*(x)} = O_p(n^{-\frac{\alpha}{2}})$, $\abs{\hv^{J+}_n(x) - v^*(x)} = O_p(n^{-\frac{\alpha}{2}})$, $\abs*{g(\hz^{J}_n(x);x) - v^*(x)} = O_p(n^{-\frac{\alpha}{2}})$, and $\abs*{g(\hz^{J+}_n(x);x) - v^*(x)} = O_p(n^{-\frac{\alpha}{2}})$.
\end{theorem}

{Assumptions~\ref{ass:varweaklln},~\ref{ass:errorsweaklln},~\ref{ass:jack+_varweaklln},~\ref{ass:jack+_errorsweaklln},~\ref{ass:jack_regconvrate_point2},~\ref{ass:jack+_regconvrate_point2},~\ref{ass:jack_regconvrate_mse2}, and~\ref{ass:jack+_regconvrate_mse2} are not required to establish Theorem~\ref{thm:jack_rateofconv} in the homoscedastic case {($\hQ_n = \hQ_{-i}= Q^* \equiv I$)}. Additionally, Assumptions~\ref{ass:jack_regconvrate_mse} and~\ref{ass:jack+_regconvrate_mse} may be weakened in this setting to $\frac{1}{n} \sum_{i=1}^{n} \norm{f^*(x^i) - \hf_{-i}(x^i)} = O_p(n^{-\alpha/2})$ on account of Lemma~\ref{lem:jack_meandeviation}.}

\subsection{Finite sample guarantees}
\label{subsec:jack_finitesample}

We now establish exponential convergence of solutions to the J-SAA and J+-SAA problems to solutions to the true problem~\eqref{eqn:speq} under additional assumptions.
We begin by adapting Assumption~\ref{ass:reglargedev} to assume that the regression procedure~\eqref{eqn:regr} satisfies the following large deviation properties.

\begin{assjack}{ass:reglargedev}
\label{ass:jack_reglargedev}
{The regression estimates $\hf_n$, $\{\hf_{-i}\}$, $\hQ_n$, and $\{\hQ_{-i}\}$ possess the following large deviation properties: for any constant $\kappa > 0$ and $n \in \mathbb{N}$, there exist positive constants $K^J_f(\kappa,x)$, $\bar{K}^J_f(\kappa)$, $\beta^J_f(n,\kappa,x)$, $\bar{\beta}^J_f(n,\kappa)$, $K^J_Q(\kappa,x)$, $\bar{K}^J_Q(\kappa)$, $\beta^J_Q(n,\kappa,x)$, and $\bar{\beta}^J_Q(n,\kappa)$, 
with $\lim_{n \to \infty} \beta^J_f(n,\kappa,x) = \infty$, 
$\lim_{n \to \infty} \bar{\beta}^J_f(n,\kappa) = \infty$, 
$\lim_{n \to \infty} \beta^J_Q(n,\kappa,x) = \infty$, and 
$\lim_{n \to \infty} \bar{\beta}^J_Q(n,\kappa) = \infty$ for each $\kappa > 0$ and a.e.\ $x \in \X$, satisfying}
\begin{enumerate}[label=(\ref{ass:reglargedev}J\alph*),itemsep=0em]
\item \label{ass:jack_reglargedev_point} $\mathbb{P}\bigl\{\norm{f^*(x) - \hf_n(x)} > \kappa \bigr\} \leq K^{J}_f(\kappa,x) \exp\bigl(-\beta^{J}_f(n,\kappa,x)\bigr)$ for a.e.\ $x \in \X$,

\item \label{ass:jack_reglargedev_mse} $\mathbb{P}\biggl\{\dfrac{1}{n} \displaystyle\sum_{i=1}^{n} \norm{f^*(x^i) - \hf_{-i}(x^i)}^2 > \kappa^2 \biggr\} \leq \bar{K}^{J}_f(\kappa) \exp\bigl(-\bar{\beta}^{J}_f(n,\kappa)\bigr)$,

\item \label{ass:jack_reglargedev_point2} {$\mathbb{P}\bigl\{\norm{Q^*(x) - \hQ_n(x)} > \kappa \bigr\} \leq K^J_Q(\kappa,x) \exp\bigl(-\beta^J_Q(n,\kappa,x)\bigr)$ for a.e.\ $x \in \X$,}

\item \label{ass:jack_reglargedev_mse2} {$\mathbb{P}\biggl\{\dfrac{1}{n} \displaystyle\sum_{i=1}^{n} \bigl\lVert \bigl[\hQ_{-i}(x^i)\bigr]^{-1} - \bigl[Q^*(x^i)\bigr]^{-1}\bigr\rVert^2 > \kappa^2 \biggr\} \leq \bar{K}^J_Q(\kappa) \exp\bigl(-\bar{\beta}^J_Q(n,\kappa)\bigr)$.}
\end{enumerate}
\end{assjack}

\begin{assjackplus}{ass:reglargedev}
\label{ass:jack+_reglargedev}
{The regression estimates $\hf_n$, $\{\hf_{-i}\}$, $\hQ_n$, and $\{\hQ_{-i}\}$ possess the following large deviation properties: for any constant $\kappa > 0$ and $n \in \mathbb{N}$, there exist positive constants $K^{J+}_f(\kappa,x)$, $\bar{K}^{J+}_f(\kappa)$, $\beta^{J+}_f(n,\kappa,x)$, $\bar{\beta}^{J+}_f(n,\kappa)$, $K^{J+}_Q(\kappa,x)$, $\bar{K}^{J+}_Q(\kappa)$, $\beta^{J+}_Q(n,\kappa,x)$, and $\bar{\beta}^{J+}_Q(n,\kappa)$, 
with $\lim_{n \to \infty} \beta^{J+}_f(n,\kappa,x) = \infty$, 
$\lim_{n \to \infty} \bar{\beta}^{J+}_f(n,\kappa) = \infty$, 
$\lim_{n \to \infty} \beta^{J+}_Q(n,\kappa,x) = \infty$, and 
$\lim_{n \to \infty} \bar{\beta}^{J+}_Q(n,\kappa) = \infty$ for each $\kappa > 0$ and a.e.\ $x \in \X$, satisfying}
\begin{enumerate}[label=(\ref{ass:reglargedev}J+\alph*),itemsep=0em]
\item \label{ass:jack+_reglargedev_point} $\mathbb{P}\biggl\{\dfrac{1}{n} \displaystyle\sum_{i=1}^{n} \norm{f^*(x) - \hf_{-i}(x)} > \kappa \biggr\} \leq K^{J+}_f(\kappa,x) \exp\bigl(-\beta^{J+}_f(n,\kappa,x)\bigr)$ for a.e.\ $x \in \X$,

\item \label{ass:jack+_reglargedev_mse} $\mathbb{P}\biggl\{\dfrac{1}{n} \displaystyle\sum_{i=1}^{n} \norm{f^*(x^i) - \hf_{-i}(x^i)}^4 > \kappa^4 \biggr\} \leq \bar{K}^{J+}_f(\kappa) \exp\bigl(-\bar{\beta}^{J+}_f(n,\kappa)\bigr)$,

\item \label{ass:jack+_reglargedev_point2} {$\mathbb{P}\biggl\{\dfrac{1}{n} \displaystyle\sum_{i=1}^{n} \norm{Q^*(x) - \hQ_{-i}(x)}^2 > \kappa^2 \biggr\} \leq K^{J+}_Q(\kappa,x) \exp\bigl(-\beta^{J+}_Q(n,\kappa,x)\bigr)$ for a.e.\ $x \in \X$,}

\item \label{ass:jack+_reglargedev_mse2} {$\mathbb{P}\biggl\{\dfrac{1}{n} \displaystyle\sum_{i=1}^{n} \bigl\lVert \bigl[\hQ_{-i}(x^i)\bigr]^{-1} - \bigl[Q^*(x^i)\bigr]^{-1}\bigr\rVert^4 > \kappa^4 \biggr\} \leq \bar{K}^{J+}_Q(\kappa) \exp\bigl(-\bar{\beta}^{J+}_Q(n,\kappa)\bigr)$.}
\end{enumerate}
\end{assjackplus}

Assumptions~\ref{ass:jack_reglargedev} and~\ref{ass:jack+_reglargedev} strengthen Assumptions~\ref{ass:jack_regconvrate} and~\ref{ass:jack+_regconvrate} by imposing restrictions on the tails of the regression estimators.
Please see the discussion in Section~\ref{sec:regression} for when these strengthened assumptions are satisfied with i.i.d.\ data~$\{(x^i,\varepsilon^i)\}$.
We also require the following strengthening of Assumptions~\ref{ass:varlargedev} and~\ref{ass:errorlargedev} for the J+-SAA problem.

\begin{assjackplus}{ass:varlargedev}
\label{ass:jack+_varlargedev}
For any $\kappa > 0$ and $n \in \mathbb{N}$, there exist positive constants $L^{J+}_Q(\kappa)$, $\gamma^{J+}_{Q}(n,\kappa)$, $\bar{L}^{J+}_Q(\kappa)$, and $\bar{\gamma}^{J+}_{Q}(n,\kappa)$, with $\uset{n \to \infty}{\lim} \gamma^{J+}_{Q}(n,\kappa) = \infty$ and $\uset{n \to \infty}{\lim} \bar{\gamma}^{J+}_{Q}(n,\kappa) = \infty$ for each $\kappa > 0$, such that
\begin{align*}
{\mathbb{P}\biggl\{ \biggl(\frac{1}{n}\sum_{i=1}^{n} \bigl\lVert \bigl[Q^*(x^i)\bigr]^{-1}\bigr\rVert^4\biggr)^{1/4} > \Bigl(\expect{\bigl\lVert \bigl[Q^*(X)\bigr]^{-1} \bigr\rVert^4}\Bigr)^{1/4} + \kappa \biggr\}} &{\leq L^{J+}_Q(\kappa)\exp(- \gamma^{J+}_{Q}(n,\kappa)),} \\
{\mathbb{P}\biggl\{ \biggl(\frac{1}{n} \sum_{i=1}^{n} \norm{Q^*(x^i)}^8\biggr)^{1/8} > \bigl(\expect{\norm{Q^*(X)}^8}\bigr)^{1/8} + \kappa \biggr\}} &{\leq \bar{L}^{J+}_Q(\kappa)\exp(- \bar{\gamma}^{J+}_{Q}(n,\kappa)).}
\end{align*}
\end{assjackplus}

\begin{assjackplus}{ass:errorlargedev}
\label{ass:jack+_errorlargedev}
For any $\kappa > 0$ and $n \in \mathbb{N}$, there exist positive constants $L^{J+}_{\varepsilon}(\kappa)$, $\gamma^{J+}_{\varepsilon}(n,\kappa)$, $\bar{L}^{J+}_{\varepsilon}(\kappa)$, and $\bar{\gamma}^{J+}_{\varepsilon}(n,\kappa)$, with $\uset{n \to \infty}{\lim} \gamma^{J+}_{\varepsilon}(n,\kappa) = \infty$ and $\uset{n \to \infty}{\lim} \bar{\gamma}^{J+}_{\varepsilon}(n,\kappa) = \infty$ for each $\kappa > 0$, such that
\begin{align*}
{\mathbb{P}\biggl\{ \biggl(\frac{1}{n}\sum_{i=1}^{n} \norm{\varepsilon^i}^2\biggr)^{1/2} > \Bigl(\expect{\norm{\varepsilon}^2}\Bigr)^{1/2} + \kappa \biggr\}} &{\leq L^{J+}_{\varepsilon}(\kappa)\exp(- \gamma^{J+}_{\varepsilon}(n,\kappa)),} \\
{\mathbb{P}\biggl\{ \biggl(\frac{1}{n}\sum_{i=1}^{n} \norm{\varepsilon^i}^8\biggr)^{1/8} > (\expect{\norm{\varepsilon}^8})^{1/8} + \kappa \biggr\}} &{\leq \bar{L}^{J+}_{\varepsilon}(\kappa)\exp(-\bar{\gamma}^{J+}_{\varepsilon}(n,\kappa)).}
\end{align*}
\end{assjackplus}

The next result presents conditions under which the maximum deviations of the J-SAA and J+-SAA objectives from the FI-SAA objective satisfy qualitatively similar large deviations bounds as that in Assumption~\ref{ass:tradsaalargedev}.

\begin{lemma}
\label{lem:jack_largedevofdev}
{Suppose Assumptions,~\ref{ass:equilipschitz},~\ref{ass:tradsaalargedev},~\ref{ass:varlargedev},~\ref{ass:errorlargedev},~\ref{ass:jack+_varlargedev},~\ref{ass:jack+_errorlargedev},~\ref{ass:jack_reglargedev} and~\ref{ass:jack+_reglargedev} hold.
Then for any constant $\kappa > 0$, $n \in \mathbb{N}$, and a.e.~$x \in \X$, there exist positive constants $\bar{K}^{J}(\kappa,x)$, $\bar{K}^{J+}(\kappa,x)$, $\bar{\beta}^{J}(n,\kappa,x)$, and $\bar{\beta}^{J+}(n,\kappa,x)$, with $\lim_{n \to \infty} \bar{\beta}^{J}(n,\kappa,x) = \infty$ and $\lim_{n \to \infty} \bar{\beta}^{J+}(n,\kappa,x) = \infty$ for each $\kappa > 0$ and a.e.\ $x \in \X$, satisfying}
\begin{align*}
\prob{\uset{z \in \Z}{\sup} \: \abs*{ \hg^{J}_n(z;x) - g^*_n(z;x)} > \kappa} &\leq \bar{K}^{J}(\kappa,x) \exp\left(-\bar{\beta}^{J}(n,\kappa,x)\right), \:\: \text{and} \\
\prob{\uset{z \in \Z}{\sup} \: \abs*{ \hg^{J+}_n(z;x) - g^*_n(z;x)} > \kappa} &\leq \bar{K}^{J+}(\kappa,x) \exp\left(-\bar{\beta}^{J+}(n,\kappa,x)\right).
\end{align*}
\end{lemma}

We now finite sample guarantees for the distances between solutions to the J-SAA and J+-SAA problems~\eqref{eqn:jackknife} and~\eqref{eqn:jackknife+} and the set of optimal solutions to the true problem~\eqref{eqn:speq}.

\begin{theorem}
\label{thm:jack_exponentialconv}
{Suppose Assumptions,~\ref{ass:equilipschitz},~\ref{ass:tradsaalargedev},~\ref{ass:varlargedev},~\ref{ass:errorlargedev},~\ref{ass:jack+_varlargedev},~\ref{ass:jack+_errorlargedev},~\ref{ass:jack_reglargedev} and~\ref{ass:jack+_reglargedev} hold.
Then, for each $n \in \mathbb{N}$ and a.e.\ $x \in \X$, given $\eta > 0$, there exist positive constants $Q^{J}(\eta,x)$, $Q^{J+}(\eta,x)$, $\gamma^{J}(n,\eta,x)$, and $\gamma^{J+}(n,\eta,x)$, with $\lim_{n \to \infty} \gamma^{J}(n,\eta,x) = \infty$ and $\lim_{n \to \infty} \gamma^{J+}(n,\eta,x) = \infty$ for each $\eta > 0$ and a.e.\ $x \in \X$, such that}
\begin{align*}
\prob{\textup{dist}(\hz^{J}_n(x),S^*(x)) \geq \eta} &\leq Q^{J}(\eta,x) \exp(-\gamma^{J}(n,\eta,x)), \qquad \text{and} \\
\prob{\textup{dist}(\hz^{J+}_n(x),S^*(x)) \geq \eta} &\leq Q^{J+}(\eta,x) \exp(-\gamma^{J+}(n,\eta,x)).
\end{align*}
\end{theorem}

{Assumptions~\ref{ass:varlargedev},~\ref{ass:errorlargedev},~\ref{ass:jack+_varlargedev},~\ref{ass:jack+_errorlargedev},~\ref{ass:jack_reglargedev_point2},~\ref{ass:jack+_reglargedev_point2},~\ref{ass:jack_reglargedev_mse2}, and~\ref{ass:jack+_reglargedev_mse2} are not required to establish Theorem~\ref{thm:jack_exponentialconv} in the homoscedastic case. Additionally, Assumptions~\ref{ass:jack_reglargedev_mse} and~\ref{ass:jack+_reglargedev_mse} may be weakened in this setting to $\mathbb{P}\bigl\{\frac{1}{n} \sum_{i=1}^{n} \norm{f^*(x^i) - \hf_{-i}(x^i)} > \kappa \bigr\} \leq \bar{K}^J_f(\kappa) \exp\bigl(-\bar{\beta}^J_f(n,\kappa)\bigr)$ and $\mathbb{P}\bigl\{\frac{1}{n} \sum_{i=1}^{n} \norm{f^*(x^i) - \hf_{-i}(x^i)} > \kappa \bigr\} \leq \bar{K}^{J+}_f(\kappa) \exp\bigl(-\bar{\beta}^{J+}_f(n,\kappa)\bigr)$ on account of Lemma~\ref{lem:jack_meandeviation}.}

\section{Application to two-stage stochastic programming problems}
\label{sec:tssp}
We present a class of stochastic programs that satisfy Assumptions~\ref{ass:equilipschitz},~\ref{ass:uniflln},~{\ref{ass:uniflln_lq}},~\ref{ass:tradsaalargedev},~\ref{ass:equilipschitz_alt},~\ref{ass:equilipschitz_v2}, and~\ref{ass:functionalclt}.
We first consider a class of two-stage stochastic programs with continuous recourse decisions that subsumes Example~\ref{exm:runningexample}, our running example of two-stage stochastic LP.
We then briefly outline the verification of these assumptions for a broader class of stochastic programs.

Consider first the two-stage stochastic program
\begin{align}
\label{eqn:tssp-app}
&\uset{z \in \Z}{\min} \: \expect{c(z,Y)} := p(z) + \expect{V(z,Y)},
\end{align}
where the second-stage function~$V$ is defined by the optimal value of the following LP:
\begin{align*}
V(z,y) := \: &\uset{v \in \R^{d_v}_+}{\min} \Set{\tr{c}_v v}{Wv = h(y) - T(y,z)}.
\end{align*}
We make the following assumptions on problem~\eqref{eqn:tssp-app}.

\begin{assumption}
\label{ass:tssp-welldef}
The set~$\Z$ is nonempty and compact, the matrix~$W$ has full row rank, the set $\Lambda := \Set{\lambda}{\tr{\lambda}W \leq \tr{c}_v}$ is nonempty, and the value function $V(z,y) < +\infty$ for each $(z,y) \in \Z \times \Y$.
\end{assumption}

\begin{assumption}
\label{ass:tssp-lipschitzh}
The functions~$h$ and $T(\cdot,z)$ are Lipschitz continuous on $\Y$ for each $z \in \Z$ with Lipschitz constants $L_h$ and $L_{T,y}(z)$, and the functions~$p$ and $T(y,\cdot)$ are Lipschitz continuous on $\Z$ for each $y \in \Y$ with Lipschitz constants $L_p$ and $L_{T,z}(y)$.
Additionally, the Lipschitz constants for the function~$T$ satisfy $\uset{z \in \Z}{\sup} \: L_{T,y}(z) < +\infty$ and $\uset{y \in \Y}{\sup} \: L_{T,z}(y) < +\infty$.
\end{assumption}

\begin{assumption}
\label{ass:tssp-cov}
The support~$\X$ of the covariates~$X$ is compact and the functions $f^*$ and $Q^*$ are continuous on~$\X$.
Additionally, $\mathbb{E}\Bigl[\uset{(z,x) \in \Z \times \X}{\sup} \: \norm{h(f^*(x)+Q^*(x)\varepsilon) - T(f^*(x)+Q^*(x)\varepsilon,z)}\Bigr] < +\infty$.
\end{assumption}

Let $\text{vert}(\Lambda)$ denote the finite set of extreme points of the dual feasible region~$\Lambda$, and define $\mathcal{V}(Y,\lambda,z) := \tr{\lambda} \left( h(Y) - T(Y,z) \right)$ for each $Y \in \Y$, $\lambda \in \textup{vert}(\Lambda)$, and $z \in \Z$.

\begin{assumption}
\label{ass:tssp-largedev}
We have $\mathbb{E}\Bigl[\uset{z \in \Z}{\sup} \: \norm{h(Y) - T(Y,z)}^2\Bigr] < +\infty$.
Additionally, the random variable $\mathcal{V}(f^*(x)+Q^*(x)\varepsilon,\lambda,z) - \expectation{\bar{\varepsilon} \sim P_{\varepsilon}}{\mathcal{V}(f^*(x)+Q^*(x)\bar{\varepsilon},\lambda,z)}$ is sub-Gaussian with variance proxy $\sigma^2_c(x)$ for each $\lambda \in \textup{vert}(\Lambda)$, $z \in \Z$ and a.e.\ $x \in \X$.
\end{assumption}

Note that the first-stage feasible set~$\Z$ can include integrality constraints.
Our running example of two-stage stochastic LP with OLS regression fits within the above setup and readily satisfies Assumptions~\ref{ass:tssp-welldef} and~\ref{ass:tssp-lipschitzh}.
It also satisfies Assumption~\ref{ass:tssp-cov} when $\expect{\norm{\varepsilon}} < +\infty$ and Assumption~\ref{ass:tssp-largedev} when the error~$\varepsilon$ is sub-Gaussian.
Additionally, under Assumption~\ref{ass:tssp-welldef}, we have by LP duality that for each $y \in \Y$:
\begin{alignat}{2}
\label{eqn:tssp-dual}
V(z,y) = \: &\uset{\lambda \in \Lambda}{\max} \: && \tr{\lambda} \left( h(y) - T(y,z) \right) = \uset{\lambda \in \text{vert}(\Lambda)}{\max} \tr{\lambda} \left( h(y) - T(y,z) \right).
\end{alignat}

\begin{proposition}
\label{prop:tssp-checkass}
Suppose Assumptions~\ref{ass:tssp-welldef},~\ref{ass:tssp-lipschitzh},~\ref{ass:tssp-cov}, and~\ref{ass:tssp-largedev} hold and the data~$\{\varepsilon^i\}$ is i.i.d.
Then, problem~\eqref{eqn:tssp-app} satisfies Assumptions~\ref{ass:equilipschitz},~\ref{ass:uniflln},~\ref{ass:uniflln_lq},~\ref{ass:tradsaalargedev}, and~\ref{ass:functionalclt}.
\end{proposition}
\begin{proof}
We have by Assumptions~\ref{ass:tssp-welldef} and~\ref{ass:tssp-lipschitzh} that for any $y,\by \in \Y$ and $z \in \Z$:
\begin{align*}
\abs{c(z,y) - c(z,\by)} &= \Big\lvert \uset{\lambda \in \text{vert}(\Lambda)}{\max} \: \tr{\lambda} \left( h(y) - T(y,z) \right) - \uset{\lambda \in \text{vert}(\Lambda)}{\max} \: \tr{\lambda} \left( h(\by) - T(\by,z) \right)\Big\rvert \\
&\leq \uset{\lambda \in \text{vert}(\Lambda)}{\max} \: \abs*{\tr{\lambda} \left(h(y) - h(\by) \right) + \tr{\lambda} \left( T(\by,z) - T(y,z) \right)} \\
&\leq \uset{\lambda \in \text{vert}(\Lambda)}{\max} \: \norm{\lambda} \norm{h(y) - h(\by)} + \uset{\lambda \in \text{vert}(\Lambda)}{\max} \: \norm{\lambda} \norm{T(\by,z) - T(y,z)} \\
&\leq \left[ L_h + L_{T,y}(z) \right] \Bigl(\uset{\lambda \in \text{vert}(\Lambda)}{\max} \: \norm{\lambda}\Bigr) \norm{y - \by}.
\end{align*}
Therefore, the Lipschitz continuity Assumption~\ref{ass:equilipschitz} holds since $\sup_{z \in \Z} L_{T,y}(z) < +\infty$ and $\max_{\lambda \in \text{vert}(\Lambda)} \norm{\lambda} < +\infty$. 

The function~$c(\cdot,y)$ is continuous on~$\Z$ for each~$y \in \Y$ by virtue of Assumptions~\ref{ass:tssp-welldef} and~\ref{ass:tssp-lipschitzh} since $p$ is Lipschitz continuous on~$\Z$ and equation~\eqref{eqn:tssp-dual} implies that $V(\cdot,y)$ is a finite maximum of continuous functions for each $y \in \Y$.
{Furthermore, for any $\bz \in \Z$:}
\begin{align*}
{\abs{c(\bz,Y)}} &{= \Big\lvert p(\bz) + \uset{\lambda \in \text{vert}(\Lambda)}{\max} \: \tr{\lambda} \left( h(Y) - T(Y,\bz) \right) \Big\rvert} \\
&{\leq \max_{z \in \Z}\abs{p(z)} + \Bigl(\uset{\lambda \in \text{vert}(\Lambda)}{\max} \: \norm{\lambda}\Bigr) \Bigl(\max_{z \in \Z} \norm{h(Y) - T(Y,z)}\Bigr).}
\end{align*}%
The uniform weak LLN Assumption~\ref{ass:uniflln} then holds by virtue of Assumption~\ref{ass:tssp-welldef}, the first part of Assumption~\ref{ass:tssp-largedev}, and Theorem~7.48 of~\citet{shapiro2009lectures},
which also implies that the objective function of the true problem~\eqref{eqn:speq} is continuous on~$\Z$.

{The function~$c(\cdot,f^*(\cdot)+Q^*(\cdot)\varepsilon)$ is continuous on~$\Z \times \X$ for each $\varepsilon \in \Xi$ by virtue of Assumptions~\ref{ass:tssp-welldef},~\ref{ass:tssp-lipschitzh}, and~\ref{ass:tssp-cov} since $p$ is Lipschitz continuous on~$\Z$, $c(\cdot,y)$ is Lipschitz continuous on~$\Z$ with a Lipschitz constant independent of~$y \in \Y$ (see below), and equation~\eqref{eqn:tssp-dual} implies that $V(\cdot,f^*(\cdot)+Q^*(\cdot)\varepsilon)$ is a finite maximum of continuous functions for each $\varepsilon \in \Xi$.
Additionally, for any $(\bz,\bx) \in \Z \times \X$:}
\begin{align*}
&{\abs{c(\bz,f^*(\bx) + Q^*(\bx)\varepsilon)}} \\
&{= \Big\lvert p(\bz) + \uset{\lambda \in \text{vert}(\Lambda)}{\max} \: \tr{\lambda} \left( h(f^*(\bx)+Q^*(\bx)\varepsilon) - T(f^*(\bx)+Q^*(\bx)\varepsilon,\bz) \right) \Big\rvert} \\
&{\leq \max_{z \in \Z}\abs{p(z)} + \Bigl(\uset{\lambda \in \text{vert}(\Lambda)}{\max} \: \norm{\lambda}\Bigr) \Bigl(\max_{(z,x) \in \Z \times \X} \norm{h(f^*(x)+Q^*(x)\varepsilon) - T(f^*(x)+Q^*(x)\varepsilon,z)}\Bigr).}
\end{align*}%
{Therefore, Assumptions~\ref{ass:tssp-welldef},~\ref{ass:tssp-lipschitzh}, and~\ref{ass:tssp-cov} along with Theorem~7.48 of~\citet{shapiro2009lectures} together implies that $\sup_{(z,x) \in \Z \times \X} \abs*{g^*_{n}(z;x) - g(z;x)} \convinprob 0$, which in turn implies Assumption~\ref{ass:uniflln_lq}.}

Next, note that for any $z, \bz \in \Z$:
\begin{align*}
\abs{c(z,Y) - c(\bz,Y)} &= \Big\lvert p(z) + \uset{\lambda \in \text{vert}(\Lambda)}{\max} \: \tr{\lambda} \left( h(Y) - T(Y,z) \right) - p(\bz) - \uset{\lambda \in \text{vert}(\Lambda)}{\max} \: \tr{\lambda} \left( h(Y) - T(Y,\bz) \right) \Big\rvert \\
&\leq \abs*{p(z) - p(\bz)} + \uset{\lambda \in \text{vert}(\Lambda)}{\max} \: \abs*{\tr{\lambda} \left( T(Y,\bz) - T(Y,z) \right)} \\
&\leq \Bigl[ L_p + L_{T,z}(Y) \Bigl(\uset{\lambda \in \text{vert}(\Lambda)}{\max} \: \norm{\lambda}\Bigr) \Bigr] \norm{z - \bz}.
\end{align*}
{Additionally, for any $\bz \in \Z$:}
\begin{align*}
{\expect{(c(\bz,Y))^2}} &{= \mathbb{E}\Bigl[\Bigl( p(\bz) + \uset{\lambda \in \text{vert}(\Lambda)}{\max} \: \tr{\lambda} \left( h(Y) - T(Y,\bz) \right) \Bigr)^2\Bigr]} \\
&{\leq 2(p(\bz))^2 + 2\Bigl(\uset{\lambda \in \text{vert}(\Lambda)}{\max} \: \norm{\lambda}^2\Bigr)\mathbb{E}\bigl[ \norm{h(Y) - T(Y,\bz)}^2\bigr]} \\
&{\leq 2(p(\bz))^2 + 2\Bigl(\uset{\lambda \in \text{vert}(\Lambda)}{\max} \: \norm{\lambda}^2\Bigr)\mathbb{E}\Bigl[ \sup_{z \in \Z} \norm{h(Y) - T(Y,z)}^2\Bigr].}
\end{align*}
Consequently, the functional CLT Assumption~\ref{ass:functionalclt} holds by virtue of Assumptions~\ref{ass:tssp-welldef} and~\ref{ass:tssp-lipschitzh}, and the first part of Assumption~\ref{ass:tssp-largedev}, see page~164 of~\citet{shapiro2009lectures} for details.

Finally, note that for any $z \in \Z$ and a.e.\ $x \in \X$:
\begin{align*}
&c(z,f^*(x)+Q^*(x)\varepsilon) - g(z;x) \\
=& V(z,f^*(x)+Q^*(x)\varepsilon) - \expectation{\bar{\varepsilon} \sim P_{\varepsilon}}{V(z,f^*(x)+Q^*(x)\bar{\varepsilon})} \\
=& \uset{\lambda \in \text{vert}(\Lambda)}{\max} \mathcal{V}(f^*(x)+Q^*(x)\varepsilon,\lambda,z) - \expectation{\bar{\varepsilon} \sim P_{\varepsilon}}{\uset{\lambda \in \text{vert}(\Lambda)}{\max} \mathcal{V}(f^*(x)+Q^*(x)\bar{\varepsilon},\lambda,z)}\\
\leq& \uset{\lambda \in \text{vert}(\Lambda)}{\max} \bigl\{\mathcal{V}(f^*(x)+Q^*(x)\varepsilon,\lambda,z) - \expectation{\bar{\varepsilon} \sim P_{\varepsilon}}{\mathcal{V}(f^*(x)+Q^*(x)\bar{\varepsilon},\lambda,z)}\bigr\}
\end{align*}
Consequently, for any $z \in \Z$ and a.e.\ $x \in \X$:
{
\small
\begin{align*}
&\expect{\exp\big(t\big(c(z,f^*(x)+Q^*(x)\varepsilon) - g(z;x) \big)\big)} \\
\leq & \expect{\exp\bigg(t\uset{\lambda \in \text{vert}(\Lambda)}{\max} \bigl\{ \mathcal{V}(f^*(x)+Q^*(x)\varepsilon,\lambda,z) - \expectation{\bar{\varepsilon} \sim P_{\varepsilon}}{\mathcal{V}(f^*(x)+Q^*(x)\bar{\varepsilon},\lambda,z)}\bigr\} \bigg)}\\
\leq & \expect{\uset{\lambda \in \text{vert}(\Lambda)}{\max} \exp\Big(t\big( \mathcal{V}(f^*(x)+Q^*(x)\varepsilon,\lambda,z) - \expectation{\bar{\varepsilon} \sim P_{\varepsilon}}{\mathcal{V}(f^*(x)+Q^*(x)\bar{\varepsilon},\lambda,z)} \big)\Big)} \\
\leq & \sum_{\lambda \in \text{vert}(\Lambda)} \expect{\exp\Big(t\big( \mathcal{V}(f^*(x)+Q^*(x)\varepsilon,\lambda,z) - \expectation{\bar{\varepsilon} \sim P_{\varepsilon}}{\mathcal{V}(f^*(x)+Q^*(x)\bar{\varepsilon},\lambda,z)} \big)\Big)} \\
\leq & \abs{\text{vert}(\Lambda)} \exp\left(\frac{\sigma^2_c(x) t^2}{2}\right)
\end{align*}
}%
where the last inequality follows by Assumption~\ref{ass:tssp-largedev}.
Therefore, the large deviation Assumption~\ref{ass:tradsaalargedev} follows by Assumptions~\ref{ass:tssp-welldef},~\ref{ass:tssp-lipschitzh}, and~\ref{ass:tssp-largedev} and Theorem~7.65 of~\citet{shapiro2009lectures}.
\end{proof}

The assumption $\sup_{y \in \Y} L_{T,z}(y) < +\infty$ can be relaxed to assume that the moment generating function (mgf) of $L_{T,z}(Y)$ is finite valued in a neighborhood of zero, see Assumption~(C3) and Theorem~7.65 in Section~7.2.9 of~\citet{shapiro2009lectures}.
The discussion in Section~\ref{sec:ersaa} following Assumptions~\ref{ass:uniflln} and~\ref{ass:tradsaalargedev} and the discussion following Assumption~\ref{ass:functionalclt} provide avenues for relaxing the i.i.d.\ assumption on the errors~$\{\varepsilon^i\}$.
The conclusions of Proposition~\ref{prop:tssp-checkass} can also be established for the case of objective uncertainty (i.e., only the objective coefficients $q$ depend on $Y$) if Assumptions~\ref{ass:tssp-welldef},~\ref{ass:tssp-lipschitzh},~\ref{ass:tssp-cov}, and~\ref{ass:tssp-largedev} are suitably modified.
{Note that the second parts of Assumptions~\ref{ass:equilipschitz_alt} and~\ref{ass:equilipschitz_v2} also readily hold for problem~\eqref{eqn:tssp-app} whenever Assumption~\ref{ass:equilipschitz} holds.}

\paragraph{Generalization to a broader class of stochastic programs.} 
{Suppose the feasible region~$\Z$ and the support~$\X$ are nonempty and compact, functions~$f^*$ and~$Q^*$ are continuous on~$\X$, and the function~$c$ in problem~\eqref{eqn:speq} is of the form:
\[
c(z,Y) := \sum_{i=1}^{\abs{\I}} c_i(z) \phi_i(Y)
\]
for functions $\{c_i : \Z \to \R\}$ and $\{\phi_i : \Y \to \R\}$, where $\abs{\I} < +\infty$.
Suppose for each $i \in \I$ the function~$c_i$ is Lipschitz continuous on~$\Z$ with Lipschitz constant~$L_{c,i}$ and the function~$\phi_i$ is continuously differentiable on~$\Y$.
Additionally, assume for each $i \in \I$ that $\expect{(\phi_i(Y))^2} < +\infty$, $\expect{\max_{x \in \X}\abs{\phi_i(f^*(x)+Q^*(x)\varepsilon)}} < +\infty$, and $\expect{\max_{v \in \mathcal{B}_{\delta(x)}(0)}\norm{\nabla \phi_i(f^*(x)+Q^*(x)\varepsilon + v)}^2} < +\infty$ for a.e.\ $x \in \X$, where $\delta$ is defined in Assumption~\ref{ass:equilipschitz_alt}.
The above conditions hold, e.g., if each function~$\phi_i$ is polynomial and the errors~$\varepsilon$ are sub-exponential.
We argue that the above class of stochastic programs satisfies Assumptions~\ref{ass:uniflln},~\ref{ass:uniflln_lq},~\ref{ass:equilipschitz_alt}, and~\ref{ass:functionalclt} whenever the errors~$\{\varepsilon^i\}$ are i.i.d.
Additionally, we argue that if finite sample guarantees of the form}
{
\small
\begin{align*}
&{\pr \biggl\{ \biggl(\dfrac{1}{n}\sum_{j=1}^{n}\max_{v \in \mathcal{B}_{\delta(x)}(0)}\norm{\nabla \phi_i(f^*(x)+Q^*(x)\varepsilon^j + v)}^2\biggr)^{1/2} > \biggl(\expect{\max_{v \in \mathcal{B}_{\delta(x)}(0)}\norm{\nabla \phi_i(f^*(x)+Q^*(x)\varepsilon + v)}^2}\biggr)^{1/2} + \kappa \biggr\}} \\
&{\leq J_{\phi,i}(\kappa;x)\exp(-\gamma_{\phi,i}(n,\kappa;x))}
\end{align*}%
}%
{hold for each $\kappa > 0$, $i \in \I$, $n \in \mathbb{N}$, and a.e.\ $x \in \X$, then the second part of Assumption~\ref{ass:equilipschitz_v2} holds.
Finally, we argue that if the moment generating functions of $\sum_{i = 1}^{\abs{\I}} c_i(z) \bigl[ \phi_i(Y) - \expect{\phi_i(Y)} \bigr]$ and $\sum_{i = 1}^{\abs{\I}} L_{c,i} \bigl[ \abs{\phi_i(Y)} - \expect{\abs{\phi_i(Y)}} \bigr]$ are finite-valued for all $z \in \Z$ in a neighborhood of the origin~\citep[cf.\ Section~7.2.9]{shapiro2009lectures}, then Assumption~\ref{ass:tradsaalargedev} also holds.
The above two conditions hold, e.g., if each function~$\phi_i$ is polynomial and the distribution of the errors~$\varepsilon$ is sufficiently light-tail.}

{The function~$c(\cdot,y)$ is continuous on~$\Z$ for each~$y \in \Y$ since each $c_i$ is assumed to be Lipschitz continuous on~$\Z$.
Furthermore, for any $\bz \in \Z$:
\[
\abs{c(\bz,Y)} = \bigg\lvert \sum_{i=1}^{\abs{\I}} c_i(\bz) \phi_i(Y) \bigg\rvert \leq \sum_{i=1}^{\abs{\I}} \Bigl( \sup_{z \in \Z} \abs{c_i(z)} \Bigr) \abs{\phi_i(Y)}.
\]
Therefore, the fact that $\expect{(\phi_i(Y))^2} < +\infty$, $\forall i \in \I$, along with Theorem~7.48 of~\citet{shapiro2009lectures} implies that the uniform weak LLN Assumption~\ref{ass:uniflln} holds.}

{The function~$c(\cdot,f^*(\cdot)+Q^*(\cdot)\varepsilon)$ is continuous on~$\Z \times \X$ for each~$\varepsilon \in \Xi$ since each $c_i$ and $\phi_i$ are assumed to be Lipschitz continuous and $f^*$ and $Q^*$ are assumed to be continuous.
Furthermore, for any $(\bz,\bx) \in \Z \times \X$:
\[
\abs{c(\bz,f^*(\bx)+Q^*(\bx)\varepsilon)} = \bigg\lvert \sum_{i=1}^{\abs{\I}} c_i(\bz) \phi_i(f^*(\bx)+Q^*(\bx)\varepsilon) \bigg\rvert \leq \sum_{i=1}^{\abs{\I}} \Bigl( \sup_{z \in \Z} \abs{c_i(z)} \Bigr) \Bigl(\sup_{x \in \X}\abs{\phi_i(f^*(x)+Q^*(x)\varepsilon)}\Bigr).
\]
Therefore, the fact that $\expect{\sup_{x \in \X}\abs{\phi_i(f^*(x)+Q^*(x)\varepsilon)}} < +\infty$, $\forall i \in \I$, along with Theorem~7.48 of~\citet{shapiro2009lectures} implies that $\sup_{(z,x) \in \Z \times \X} \abs*{g^*_{n}(z;x) - g(z;x)} \convinprob 0$, which in turn implies Assumption~\ref{ass:uniflln_lq}.}

{Next, note that for any $z, \bz \in \Z$:
\[
\abs{c(z,Y) - c(\bz,Y)} = \bigg\lvert \sum_{i=1}^{\abs{\I}} \big[ c_i(z) - c_i(\bz)\big] \phi_i(Y) \bigg\rvert \leq \sum_{i=1}^{\abs{\I}} \abs{c_i(z) - c_i(\bz)} \abs{\phi_i(Y)} \leq \biggl( \sum_{i=1}^{\abs{\I}} L_{c,i} \abs{\phi_i(Y)} \biggr) \norm{z - \bz}.
\]
with $\mathbb{E}\bigl[\bigl( \sum_{i=1}^{\abs{\I}} L_{c,i} \abs{\phi_i(Y)} \bigr)^2\bigr] \leq O(1) \sum_{i=1}^{\abs{\I}} \expect{(\phi_i(Y))^2} < +\infty$.
Additionally, for any $\bz \in \Z$:
\[
\expect{(c(\bz,Y))^2} = \mathbb{E}\biggl[\biggl( \sum_{i=1}^{\abs{\I}} c_i(\bz) \phi_i(Y) \biggr)^2\biggr] \leq \mathbb{E}\biggl[\abs{\I} \sum_{i=1}^{\abs{\I}} (c_i(\bz))^2 (\phi_i(Y))^2\biggr] \leq \abs{\I} \sum_{i=1}^{\abs{\I}} (c_i(\bz))^2 \expect{(\phi_i(Y))^2} < +\infty.
\]
Consequently, the CLT Assumption~\ref{ass:functionalclt} holds by arguments in page~164 of~\citet{shapiro2009lectures}.
Since $\sum_{i=1}^{\abs{\I}} L_{c,i} \abs{\phi_i(Y)}$ is a Lipschitz constant for $c(\cdot,Y)$ on~$\Z$, Assumption~\ref{ass:tradsaalargedev} holds if the stated conditions on the mgfs of $\sum_{i = 1}^{\abs{\I}} c_i(z) \bigl[ \phi_i(Y) - \expect{\phi_i(Y)} \bigr]$ and $\sum_{i = 1}^{\abs{\I}} L_{c,i} \bigl[ \abs{\phi_i(Y)} - \expect{\abs{\phi_i(Y)}} \bigr]$ hold.}

{We focus on verifying the second parts of Assumptions~\ref{ass:equilipschitz_alt} and~\ref{ass:equilipschitz_v2} next. We have for any $z \in \Z$, $y\in \Y$, and $\by \in \mathcal{B}_{\delta(x)}(y) \cap \Y$:}
\begin{align*}
{\abs{c(z,\by) - c(z,y)} = \bigg\lvert \sum_{i=1}^{\abs{\I}} c_i(z) \bigl[\phi_i(\by) - \phi_i(y)\bigr] \bigg\rvert} &{\leq \sum_{i=1}^{\abs{\I}} \abs{c_i(z)} \abs{\phi_i(\by) - \phi_i(y)}} \\
&{\leq \biggl( \sum_{i=1}^{\abs{\I}} \abs{c_i(z)} \max_{v \in \mathcal{B}_{\delta(x)}(y)} \norm{\nabla \phi_i(v)} \biggr) \norm{\by - y},}
\end{align*}
{where the last inequality follows by the mean-value theorem.
Therefore, we can choose 
\[
L_{\delta(x)}(z,y) := \biggl( \sum_{i=1}^{\abs{\I}} \abs{c_i(z)} \max_{v \in \mathcal{B}_{\delta(x)}(y)} \norm{\nabla \phi_i(v)} \biggr)
\]
in Assumptions~\ref{ass:equilipschitz_alt} and~\ref{ass:equilipschitz_v2}.
Noting that}
\begin{align*}
{\uset{z \in \Z}{\sup} \: \dfrac{1}{n}\displaystyle\sum_{j=1}^{n} L^2_{\delta(x)}(z,f^*(x)+Q^*(x)\varepsilon^j)} &{\leq \dfrac{1}{n}\displaystyle\sum_{j=1}^{n} \uset{z \in \Z}{\sup} \biggl( \sum_{i=1}^{\abs{\I}} \abs{c_i(z)} \max_{v \in \mathcal{B}_{\delta(x)}(0)} \norm{\nabla \phi_i(f^*(x)+Q^*(x)\varepsilon^j + v)} \biggr)^2} \\
&{\leq O(1) \sum_{i=1}^{\abs{\I}} \Bigl(\uset{z \in \Z}{\sup} \: \abs{c_i(z)}^2\Bigr) \dfrac{1}{n} \displaystyle\sum_{j=1}^{n} \max_{v \in \mathcal{B}_{\delta(x)}(0)} \norm{\nabla \phi_i(f^*(x)+Q^*(x)\varepsilon^j + v)}^2,}
\end{align*}
{we see that $\sup_{z \in \Z} \frac{1}{n}\sum_{j=1}^{n} L^2_{\delta(x)}(z,f^*(x)+Q^*(x)\varepsilon^j) = O_p(1)$ by the weak LLN, which implies that Assumption~\ref{ass:boundederrorsb} holds.
Additionally, we have for each $\varepsilon \in \Xi$ and a.e.\ $x \in \X$:
\[
\sup_{z \in \Z} L^2_{\delta(x)}(z,f^*(x)+Q^*(x)\varepsilon) \leq \abs{\I} \sum_{i=1}^{\abs{\I}} \Bigl(\sup_{z \in \Z}\abs{c_i(z)}^2\Bigr) \max_{v \in \mathcal{B}_{\delta(x)}(0)} \norm{\nabla \phi_i(f^*(x)+Q^*(x)\varepsilon + v)}^2.
\]
Therefore, finite sample guarantees for $\max_{v \in \mathcal{B}_{\delta(x)}(0)} \norm{\nabla \phi_i(f^*(x)+Q^*(x)\varepsilon + v)}^2$, $i \in \I$, directly translate to finite sample guarantees of the form in Assumption~\ref{ass:equilipschitz_v2}.}

\section{Some prediction setups that satisfy our assumptions}
\label{sec:regression}

We verify that Assumptions~\ref{ass:regconsist},~\ref{ass:reglargedev}, and~\ref{ass:regconvrate} and the corresponding assumptions for the J-SAA and J+-SAA problems hold for specific regression procedures, and point to resources within the literature for verifying these assumptions more broadly.
We do not attempt to be exhaustive and, for the most part, restrict our attention to M-estimators~\citep{van2000asymptotic,van2000empirical}, which encapsulate a rich class of prediction techniques.
We often also consider the special case where the true model can be written as $Y = f(X;\sth) + {Q(X;\pi^*)}\varepsilon$ and the goal of the regression procedure~\eqref{eqn:regr} is to estimate the finite-dimensional parameters~$\sth \in \Theta$ and~{$\pi^* \in \Pi$} using the data~$\D_n$. 
To summarize, we largely consider the regression setup (possibly with a regularization~term)
\[
\hth_n \in \uset{\theta \in \Theta}{\argmin} \dfrac{1}{n}\sum_{i=1}^{n} \ell \left( y^i,f(x^i;\theta) \right)
\]
for estimating $\sth$ with a particular emphasis on the squared loss~$\ell(y,\hat{y}) := \norm{y - \hat{y}}^2$.
We call the optimization problem
\[
\uset{\theta \in \Theta}{\min} \: \expectation{(Y,X)}{\ell \left( Y,f(X;\theta) \right)} 
\]
the population regression problem, where the above expectation is taken with respect to the joint distribution of $(Y,X)$.
We mostly assume that the solution set of the population regression problem is the singleton~$\{\sth\}$, {and assume throughout this section that $X_1 \equiv 1$}.

{We consider two distinct approaches for estimating the heteroscedasticity function~$Q^*$.
The first approach uses the fact that $\mathbb{E}\bigl[(Y - f^*(X))\tr{(Y - f^*(X))} \mid X = x\bigr] = Q^*(x) \tr{Q^*(x)}$ and plugs in $\hf_n$ instead of $f^*$ to determine the best regression estimate $\hat{Q}_n$ of $Q^*$ in the model class~$\Q$.
For the parametric case, once we obtain an estimate~$\hth_n$ of~$\sth$, we can estimate $\spi$ using}
\[
{\hpi_n \in \uset{\pi \in \Pi}{\argmin} \dfrac{1}{n}\sum_{i=1}^{n} \bar{\ell} \bigl( (y^i - f(x^i;\hth_n))\tr{(y^i - f(x^i;\hth_n))},Q(x^i;\pi) \bigr),}
\]
{where $\bar{\ell}(V,\hat{V}) := \norm{V - \hat{V}}^2_F$ is the squared Frobenius norm.
In special settings (cf.\ our running Example~\ref{exm:regrexample}), the above problem for estimating $\spi$ can be transformed into a linear regression problem.
An alternative is to estimate~$f^*$ and~$Q^*$ jointly using M-estimation:}
\[
{(\hth_n,\hpi_n) \in \uset{(\theta,\pi) \in \Theta \times \Pi}{\argmin} \dfrac{1}{n}\sum_{i=1}^{n} \bar{\ell} \bigl( (y^i - f(x^i;\theta))\tr{(y^i - f(x^i;\theta))},Q(x^i;\pi) \bigr)}
\]
{for some loss function $\bar{\ell}$~\citep{davidian1987variance}.}

Finally, we emphasize that we only deal with the random design case (where the covariates~$X$ are considered to be random) in this work. Much of the statistics literature presents results for the fixed design setting in which the covariate observations~$\{x^i\}_{i=1}^{n}$ are deterministic and designed by the DM. These results readily carry over to the random design setting \textit{if} $Q^* \equiv I$, the errors $\varepsilon$ are independent of $X$, and no restriction is made on the design points~$\{x^i\}_{i=1}^{n}$.

\subsection{Parametric regression techniques for estimating \texorpdfstring{$f^*$}{fstar}}
\label{subsec:parametricreg}

We verify that {the parts of} Assumptions~\ref{ass:regconsist},~\ref{ass:reglargedev}, and~\ref{ass:regconvrate} {relating to the estimate~$\hf_n$ of~$f^*$} hold for OLS regression, the Lasso and generalized linear regression models under suitable assumptions.
We also verify their counterparts for the J-SAA and J+-SAA problems in Section~\ref{sec:jackknife}.
{Note that Assumption~\ref{ass:regconsist_point_lq} holds for the regression techniques considered in this section whenever $\hth_n \convinprob \sth$ and $\norm{X}_{L^q} < +\infty$.}
Theorem~2.6 and Corollary~2.8 of~\citet{rigollet2015high} present conditions under which these assumptions hold for best subset selection regression {in the homoscedastic setting}, and Theorem~2.14 therein presents similar guarantees for the Bayes Information Criterion estimator.
\citet{koltchinskii2009dantzig} verifies these assumptions for the Dantzig selector under certain conditions ({including homoscedasticity of the error distribution}).
\citet{hsu2012random} verifies these conditions for ridge regression.
\citet{negahban2012unified} provides results for regularized M-estimators in the high-dimensional setting.

\subsubsection{Ordinary least squares regression}
\label{subsubsec:olsreg}

We present sufficient conditions from~\citet{white2014asymptotic},~\citet{hsu2012random}, and~\citet{rigollet2015high} under which {the parts of} Assumptions~\ref{ass:regconsist},~\ref{ass:reglargedev}, and~\ref{ass:regconvrate} {relating to the estimate~$\hf_n$ of~$f^*$} hold. 
Note that Theorems~2.31 and~4.25 of~\citet{white2014asymptotic} present a general set of sufficient conditions for $\hth_n \xrightarrow{p} \sth$ and for $\sqrt{n} (\hth_n - \sth)$ to be asymptotically normally distributed.
Chapters~3 to~5 of~\citet{white2014asymptotic} also present analyses that can handle \textit{instrumental variables}, which can be used to verify Assumptions~\ref{ass:regconsist},~\ref{ass:regconvrate}, and~\ref{ass:reglargedev} when the errors~$\varepsilon$ are correlated with the features~$X$.
We have the following result:

\begin{proposition}
\label{prop:ols-ass}
Suppose $f^*(X) = \sth X$ and we use OLS regression to estimate $\sth$. {Define $\bar{\varepsilon} := Q^*(X)\varepsilon$.}
\begin{enumerate}
\item Suppose for some $\gamma > 0$, $\expect{\abs{X_i \bar{\varepsilon}_j}^{1+\gamma}} < +\infty$, $\forall i \in [d_x]$ and $j \in [d_y]$, $\expect{\norm{X}^{2(1+\gamma)}} < +\infty$, and $\expect{X\tr{X}}$ is positive definite.
If $\{(x^i,\varepsilon^i)\}_{i=1}^{n}$ is i.i.d., then $\hth_n$ a.s.\ exists for $n$ large enough and $\hth_n \xrightarrow{a.s.} \sth$. Consequently, Assumptions~\ref{ass:regconsist_point} and~\ref{ass:regconsist_mse} hold.

\item Suppose for some $\gamma > 0$, $\expect{\abs{X_i \bar{\varepsilon}_j}^{2+\gamma}} < +\infty$, $\forall i \in [d_x]$ and $j \in [d_y]$, $\expect{\norm{X}^{2(1+\gamma)}} < +\infty$, $\expect{X\tr{X}}$ is positive definite, the covariance matrix of the random variable $\sum_{j=1}^{d_y} X \bar{\varepsilon}_j$ is positive definite, and $\{(x^i,\varepsilon^i)\}_{i=1}^{n}$ is i.i.d.
Then Assumptions~\ref{ass:regconvrate_point} and~\ref{ass:regconvrate_mse} hold with $\alpha = 1$.

\item Suppose $\{(x^i,\varepsilon^i)\}_{i=1}^{n}$ is i.i.d.,~$\bar{\varepsilon}$ is uniformly sub-Gaussian with variance proxy $\sigma^2$, i.e.,
\[
{\expect{\exp(s\tr{u}\bar{\varepsilon}) \mid X = x} \leq \exp(0.5s^2\sigma^2), \quad \forall s \in \R, \:\: \norm{u} = 1, \:\: \text{a.e. } x \in \X,}
\]
the covariance matrix $\Sigma_X$ of the covariates is positive definite,
and the random vector $\Sigma_X^{-\frac{1}{2}} X$ is sub-Gaussian.
Then Assumptions~\ref{ass:reglargedev_point} and~\ref{ass:reglargedev_mse} hold with constants $K_f(\kappa,x) = O(\exp(d_x))$, $\beta_f(\kappa,x) = O\left(\frac{\kappa^2}{\sigma^2 d_y \norm{x}^2}\right)$, $\bar{K}_f(\kappa) = O(\exp(d_x))$, and $\bar{\beta}_f(\kappa) = O\left(\frac{\kappa^2}{\sigma^2 d_y}\right)$.
\end{enumerate}
\end{proposition}
\begin{proof}
Part 1 follows from Theorems~3.15 of~\citet{white2014asymptotic}.
Part 2 follows from Theorem~5.13 of~\citet{white2014asymptotic}.
Part 3 follows from Remark~12 of~\citet{hsu2012random}.
For the homoscedastic case, part 3 also follows from
Theorem~2.2 and Remark~2.3 of~\citet{rigollet2015high}.
Although~\citet{rigollet2015high} consider the fixed design case, their proof readily extends to the random design setting since no restrictions were placed on the design.
\end{proof}

{For the homoscedastic setting, part 1 of Proposition~\ref{prop:ols-ass} follows from Theorems~3.5 and~3.37 of~\citet{white2014asymptotic} and part 2 follows from Theorem~5.3 of~\citet{white2014asymptotic}.}
Theorems~3.49 and~3.78 of~\citet{white2014asymptotic} present sufficient conditions under which Assumptions~\ref{ass:regconsist_point} and~\ref{ass:regconsist_mse} hold under mixing and martingale conditions on the data~$\{(x^i,\varepsilon^i)\}$.
Theorem~5.17 and Exercise~5.21 of~\citet{white2014asymptotic} present sufficient conditions under which Assumptions~\ref{ass:regconvrate_point} and~\ref{ass:regconvrate_mse} hold with $\alpha = 1$ for ergodic and mixing data~$\{(x^i,\varepsilon^i)\}$, respectively.
Note that results in~\citet{bryc1996large} and~\cite{dembo2011large} can be used to establish Assumptions~\ref{ass:reglargedev_point} and~\ref{ass:reglargedev_mse} for the non-i.i.d.\ setting.
{Note that once we have an estimate~$\hQ_n$ of~$Q^*$, we can re-estimate~$\sth$ using feasible weighted least squares regression~\citep{romano2017resurrecting}. 
This yields an estimate~$\hth_n$ that is asymptotically more efficient than the OLS estimator of~$\sth$ whenever $\hQ_n$ is a consistent estimate of~$Q^*$~\citep{romano2017resurrecting,robinson1987asymptotically}.
Even if the estimate~$\hQ_n$ of~$Q^*$ is inconsistent, the weighted least squares estimator~$\hth_n$ remains consistent but may no longer be asymptotically efficient (see, e.g., Section~3.3 of~\cite{romano2017resurrecting}).}

The above results can be used in conjunction with the techniques in Section~\ref{subsec:mestimators} to verify Assumptions~\ref{ass:jack_regconsist},~\ref{ass:jack+_regconsist},~\ref{ass:jack_regconvrate},~\ref{ass:jack+_regconvrate},~\ref{ass:jack_reglargedev}, and~\ref{ass:jack+_reglargedev} for i.i.d.\ data~$\{(x^i,\varepsilon^i)\}$.
In the remainder of this section, we specialize the verification of these assumptions for OLS regression {in the homoscedastic setting} when problem~\eqref{eqn:speq} is a two-stage stochastic LP (see Example~\ref{exm:runningexample}).
We assume that $d_y = 1$ for ease of exposition.

Following Lemma~\ref{lem:jack_meandeviation} in Section~\ref{sec:jackknife}, it suffices to establish rates and finite sample guarantees for the terms $\frac{1}{n} \sum_{i=1}^{n} \norm{\hf_n(x^i) - \hf_{-i}(x^i)}$ and $\frac{1}{n} \sum_{i=1}^{n} \norm{\hf_n(x) - \hf_{-i}(x)}$ when the assumptions for the ER-SAA problem hold.
Let $\bar{X}$ denote the $\R^{n \times d_x}$ \textit{design matrix} with $\bar{X}_{[i]} = \tr{(x^i)}$, $h^i := (\bar{X} (\tr{\bar{X}}\bar{X})^{-1} \tr{\bar{X}})_{ii}$ denote the $i$th \textit{leverage score}, and $e^i := y^i - \tr{\hth_n} x^i$ denote the residual of the model $\hf_n$ at the $i$th data point.
From Section~10.6.3 of~\citet{seber2012linear}, we have
{
\small
\begin{align*}
\frac{1}{n} \sum_{i=1}^{n} \norm{\hf_n(x^i) - \hf_{-i}(x^i)} &= \frac{1}{n} \sum_{i=1}^{n} \frac{h^i \abs{e^i}}{1-h^i} \leq \biggl(\frac{1}{n} \sum_{i=1}^{n} (h^i)^2\biggr)^{1/2} \biggl(\frac{1}{n} \sum_{i=1}^{n} \left(\frac{e^i}{1-h^i}\right)^2\biggr)^{1/2} \leq \frac{d_x}{\sqrt{n}} \biggl(\frac{1}{n} \sum_{i=1}^{n} \left(\frac{e^i}{1-h^i}\right)^2\biggr)^{1/2},
\end{align*}%
}%
{
\small
\begin{align*}
\frac{1}{n} \sum_{i=1}^{n} \norm{\hf_n(x) - \hf_{-i}(x)} \leq \frac{\norm{x}}{n} \sum_{i=1}^{n} \norm*{\frac{(\tr{\bar{X}}\bar{X})^{-1} x^i e^i}{1-h^i}} &\leq \norm{x} \biggl(\frac{1}{n} \sum_{i=1}^{n} \norm*{(\tr{\bar{X}}\bar{X})^{-1} x^i}^2\biggr)^{1/2} \biggl(\frac{1}{n} \sum_{i=1}^{n} \left(\frac{e^i}{1-h^i}\right)^2\biggr)^{1/2} \\
&\leq \norm{x} \biggl(\frac{1}{n} \text{Tr}((\tr{\bar{X}}\bar{X})^{-1})\biggr)^{1/2} \biggl(\frac{1}{n} \sum_{i=1}^{n} \left(\frac{e^i}{1-h^i}\right)^2\biggr)^{1/2},
\end{align*}%
}%
where $\text{Tr}$ denotes the trace operator.
The quantity $\Bigl(\frac{1}{n} \sum_{i=1}^{n} \left(\frac{e^i}{1-h^i}\right)^2\Bigr)^{1/2}$ is called the \textit{prediction residual sum of squares} statistic and is bounded in probability under mild assumptions.
The above inequalities can be used to verify the assumptions for the {jackknife}-based estimators for Example~\ref{exm:runningexample}.

\subsubsection{The Lasso and high-dimensional generalized linear models}
\label{subsubsec:lassoreg}

Following~\citet{van2008high} and~\citet{bunea2007sparsity}, we consider generalized linear models with an~$\ell_1$-penalty.
We assume $d_y = 1$ for ease of exposition.
The setup is: 
the model class~$\F := \Set{f}{f(\cdot;\theta) := \sum_{k=1}^{m} \theta_{k} \psi_{k}(\cdot), \theta \in \Theta}$, where $\{\psi_{k}(\cdot)\}_{k=1}^{m}$ is a sequence of real-valued basis functions with domain~$\X$, the data~$\{(x^i,\varepsilon^i)\}$ is assumed to be i.i.d., the number of basis functions~$m$ grows subexponentially with the number of data samples~$n$, the set~$\Theta$ is convex, the loss function~$\ell$ satisfies some Lipschitz assumptions~\citep[see Assumption~L and Example~4 of][]{van2008high}, and the estimate~$\hth_{n}$ of~$\sth$ is obtained as
\[
\hth_{n} \in \uset{\theta \in \Theta}{\argmin} \: \dfrac{1}{n}\sum_{i=1}^{n} \ell \left( y^i,f(x^i;\theta) \right) + \lambda_n \displaystyle\sum_{k=1}^{m} \left( \dfrac{1}{n} \displaystyle\sum_{i=1}^{n} \psi^2_k(x^i) \right)^{\frac{1}{2}} \abs{\theta_k}
\]
for some penalty parameter~$\lambda_n= O\left( \sqrt{\frac{\log m}{n}} \right)$ that is chosen large enough.
The above setup captures both parametric and nonparametric regression models.
Theorem~2.2 of~\citet{van2008high} and Theorems~2.1,~2.2, and~2.3 of~\citet{bunea2007sparsity} present conditions under which Assumptions~\ref{ass:regconsist},~\ref{ass:reglargedev}, and~\ref{ass:regconvrate} hold for the above setting.

In the remainder of this section, we specialize the results of~\citet{bunea2007sparsity} to the traditional Lasso setup~\citep{tibshirani1996regression}. In this setup, $m = d_x$, $\psi_k(x) = x_k$, $\Theta = \R^{d_x}$, and $\ell(y,\hat{y}) = \norm{y - \hat{y}}^2$.
{Once again, we define $\bar{\varepsilon} := Q^*(X)\varepsilon$.}

\begin{proposition}
\label{prop:lasso-ass}
Suppose $f^*(X) = \sth X$ with $\norm{\sth_{[j]}}_0 \leq s$, $\forall j \in [d_y]$, the sequences $\{x^i\}_{i=1}^{n}$ and $\{\bar{\varepsilon}^i\}_{i=1}^{n}$ are i.i.d., and the error~$\bar{\varepsilon}$ is uniformly sub-Gaussian with variance proxy $\sigma^2$, i.e.,
\[
{\expect{\exp(s\tr{u}\bar{\varepsilon}) \mid X = x} \leq \exp(0.5s^2\sigma^2), \quad \forall s \in \R, \:\: \norm{u} = 1, \:\: \text{a.e. } x \in \X.}
\]
Additionally, suppose the support~$\X$ of the covariates~$X$ is compact, $\expect{\abs{X_j}^2} > 0$, $\forall j \in [d_x]$, and the matrix $\expect{X\tr{X}} - \tau \textup{diag}(\expect{X\tr{X}})$ is positive semidefinite for some constant $\tau \in (0,1]$.
If we use the Lasso to estimate $\sth$,
then Assumptions~\ref{ass:regconsist_point} and~\ref{ass:regconsist_mse} hold, Assumptions~\ref{ass:regconvrate_point} and~\ref{ass:regconvrate_mse} hold with $\alpha = 1$, and Assumption~\ref{ass:reglargedev_point} and~\ref{ass:reglargedev_mse} hold with $\bar{K}_f(\kappa) = K_f(\kappa,x) = O(d_x)$, $\bar{\beta}_f(\kappa) = O\left(\frac{\kappa^2}{\sigma^2 s d_y}\right)$, and $\beta_f(\kappa,x) = O\left(\frac{\kappa^2}{\sigma^2 s d_y \norm{x}^2}\right)$.
\end{proposition}
\begin{proof}
Follows from Theorem~2.1 and Corollary~1 of~\citet{bunea2007sparsity}.
\end{proof}

\citet{chatterjee2013assumptionless} establishes consistency of the Lasso {in the homoscedastic setting} under the following weaker assumptions: the data~$\{(x^i,\varepsilon^i)\}$ is i.i.d., 
the error~$\varepsilon$ is sub-Gaussian with variance proxy $\sigma^2$ and is independent of the covariates $X$,
the support $\X$ of the covariates is compact, and the covariance matrix of the covariates is positive definite.
Theorems~1 and~2 therein present conditions under which Assumption~\ref{ass:regconvrate} holds at a slower rate with $\alpha = 0.5$. 
Theorem~2.15 of~\citet{rigollet2015high} can then be used to show that Assumptions~\ref{ass:reglargedev_point} and~\ref{ass:reglargedev_mse} hold {in the homoscedastic setting} with $\bar{K}_f(\kappa) = K_f(\kappa,x) = O(d_x)$, $\bar{\beta}_f(\kappa) = O\left(\frac{\kappa^4}{\sigma^2 s^2 d^2_y}\right)$, and $\beta_f(\kappa,x) = O\left(\frac{\kappa^4}{\sigma^2 s^2 d^2_y \norm{x}^2}\right)$. 
\citet{basu2015regularized} present conditions under which Assumptions~\ref{ass:regconsist},~\ref{ass:reglargedev}, and~\ref{ass:regconvrate} can be verified for time series data~$\{(x^i,\varepsilon^i)\}$ {in the homoscedastic case}.
The above results can be used in conjunction with the discussion in Section~\ref{subsec:mestimators} to derive rates of convergence and finite sample guarantees for the {jackknife}-based estimators for i.i.d.\ data~$\{(x^i,\varepsilon^i)\}$.

{
Theorem~1 of~\citet{belloni2012sparse} outlines conditions under which Assumptions~\ref{ass:regconsist} and~\ref{ass:regconvrate} hold for the heteroscedasticity-adapted Lasso with $\alpha = 1$.
\citet{medeiros2016l1} and~\citet{ziel2016iteratively} present asymptotic analyses of the adaptive Lasso for time series data, including GARCH-type processes.
Theorems~2 and~3 of~\cite{medeiros2016l1} and Theorem~1 of~\cite{ziel2016iteratively} present conditions under which Assumptions~\ref{ass:regconsist} and~\ref{ass:regconvrate} hold with $\alpha = 1$.
\citet{belloni2014pivotal} study asymptotic and finite sample guarantees for the heteroscedasticity-adapted square-root Lasso.
Finally, Theorem~5.2 of~\citet{dalalyan2013learning} introduces a scaled heteroscedastic Dantzig selector and presents large deviation bounds for both~$\hf_n$ and~$\hQ_n$ under certain sparsity assumptions.}

\subsection{General M-estimation procedures for estimating \texorpdfstring{$f^*$}{fstar}}
\label{subsec:mestimators}

We use results from Chapter~5 of~\citet{van2000asymptotic}, Chapter~3 of~\citet{vaart1996weak}, and~\citet{shapiro2009lectures} to verify {the parts of} Assumptions~\ref{ass:regconsist},~\ref{ass:reglargedev}, and~\ref{ass:regconvrate} {relating to the estimate~$\hf_n$ of~$f^*$} for general M-estimators.
To begin, we suppose that the regression function~$f(x;\cdot)$ is Lipschitz continuous at~$\sth$ for a.e.~$x \in \X$ with Lipschitz constant~$L_f(x)$, i.e., we a.s.\ have $\norm{f(x;\sth) - f(x;\hth_n)} \leq L_f(x) \norm{\sth - \hth_n}$.
To establish Assumptions~\ref{ass:regconsist_point},~\ref{ass:regconsist_mse},~\ref{ass:regconvrate_point}, and~\ref{ass:regconvrate_mse}, it suffices to assume that the function~$f(x,\cdot)$ is locally Lipschitz continuous at~$\sth$ and a.s.\ for $n$ large enough, the estimates~$\hth_n$ of~$\sth$ lie in some compact subset of $\Theta$.
Note that
\[
\dfrac{1}{n} \displaystyle\sum_{i=1}^{n}\norm{f(x^i;\sth) - f(x^i;\hth_n)}^2 \leq \biggl(\dfrac{1}{n} \displaystyle\sum_{i=1}^{n} L^2_f(x^i)\biggr) \norm{\sth - \hth_n}^2,
\]
with the first term in the r.h.s.\ of the above inequality bounded in probability under a suitable weak LLN assumption.
Therefore, our main focus is presenting rates at which $\norm{\sth - \hth_n} \convinprob 0$.
{Note that Assumption~\ref{ass:regconsist_point_lq} also holds whenever $\hth_n \convinprob \sth$ and $\norm{L_f(X)}_{L^q} < +\infty$.}

\paragraph{Verifying Assumption~\ref{ass:regconsist}.} 
Theorem~5.7 of~\citet{van2000asymptotic} presents conditions under which $\hth_n \convinprob \sth$ for i.i.d.\ data~$\{(x^i,\varepsilon^i)\}$~\citep[cf.\ Theorems~5.3 and~5.4 of][]{shapiro2009lectures}.
Similar to the discussion following Assumption~\ref{ass:uniflln}, this result also holds when $\D_n$ satisfies certain mixing/stationarity assumptions.
Section~5.2 of~\citet{van2000asymptotic} also presents alternative conditions for $\hth_n \xrightarrow{p} \sth$.

\paragraph{Verifying Assumption~\ref{ass:regconvrate}.} 
We discuss conditions under which $\norm{\hth_n - \sth} \convinprob 0$ at certain rates.
Theorem~5.23 of~\citet{van2000asymptotic} presents regularity conditions under which this convergence holds at the conventional~$n^{-1/2}$ rate, in which case Assumption~\ref{ass:regconvrate} holds with~$\alpha = 1$~\citep[cf.\ Theorem~5.8 of][]{shapiro2009lectures}.
Once again, the above conclusion holds when the observations~$\D_n$ satisfy certain mixing/stationarity assumptions.
Chapter~5 of~\citet{van2000asymptotic} and Chapter~3.2 of~\citet{vaart1996weak} provide some examples of M-estimators that possess this rate of convergence.
Theorem~5.52 and Chapter~25 of~\citet{van2000asymptotic} present conditions under which Assumption~\ref{ass:regconvrate} holds with constant~$\alpha < 1$ (including for semiparametric regression). 
{Similar to the special case of OLS regression, vanilla M-estimators that do not account for heteroscedasticity may no longer be efficient. Feasible weighted M-estimation may provide an asymptotically efficient alternative in the heteroscedastic setting.}

\paragraph{Verifying Assumption~\ref{ass:reglargedev}.} We verify this assumption by establishing finite sample guarantees for $\hth_n$ when the M-estimation problem satisfies uniform exponential bounds similar to Assumption~\ref{ass:tradsaalargedev}.
Specifically, suppose for any constant $\kappa > 0$, there exist positive constants $\hat{K}(\kappa)$ and $\hat{\beta}(\kappa)$ such that
\[
\mathbb{P}\biggl\{\uset{\theta \in \Theta}{\sup} \: \bigg\lvert \dfrac{1}{n}\sum_{i=1}^{n} \ell \left( y^i,f(x^i;\theta) \right) - \expectation{(Y,X)}{\ell \left( Y,f(X;\theta) \right)}\bigg\rvert > \kappa\biggr\} \leq \hat{K}(\kappa) \exp\bigl(-n\hat{\beta}(\kappa)\bigr),
\]
see the discussion surrounding Assumption~\ref{ass:tradsaalargedev} for conditions under which such a uniform exponential bound holds~\citep[the main restriction there is that~$\Theta$ is compact, but this can be relaxed by assuming that the estimates $\hth_n$ lie in a compact subset of $\Theta$, see the discussion following Theorem~5.3 of][]{shapiro2009lectures}.
Theorem~2.3 of~\citet{homem2008rates} then implies that Assumption~\ref{ass:reglargedev} holds whenever the sample average term $\frac{1}{n} \sum_{i=1}^{n} L^2_f(x^i)$ satisfies a large deviation property (i.e., it is concentrated around $\mathbb{E}\bigl[L^2_f(X)\bigr]$).
We note that results in~\citet{bryc1996large,dembo2011large} can be used to establish such uniform exponential bounds for mixing data~$\{(x^i,\varepsilon^i)\}$ by adapting Lemma~2.4 of~\citet{homem2008rates}.
{Theorems~1,~3, and~5 of~\citet{sun2020adaptive} and Theorem~2.1 of~\citet{zhou2018new} present large deviation results in the form of Assumptions~\ref{ass:reglargedev_point} and~\ref{ass:reglargedev_mse} for adaptive Huber regression when~$f^*$ is linear.}

\paragraph{Verifying the assumptions for the {jackknife}-based methods.} 
We now present techniques for verifying {the parts of} Assumptions~\ref{ass:jack_regconsist},~\ref{ass:jack+_regconsist},~\ref{ass:jack_regconvrate},~\ref{ass:jack+_regconvrate},~\ref{ass:jack_reglargedev}, and~\ref{ass:jack+_reglargedev} {relating to the estimate~$\hf_n$ of~$f^*$} when the data~$\{(x^i,\varepsilon^i)\}$ is i.i.d.
Noting from Markov's inequality that
{
\small
\begin{align*}
&\pr\biggl\{\dfrac{1}{n} \displaystyle\sum_{i=1}^{n} \norm{f^*(x^i) - \hf_{-i}(x^i)}^4 > \kappa^4\biggr\} \leq \dfrac{1}{n\kappa^4} \displaystyle\sum_{i=1}^{n} \expectation{\D_n}{\norm{f^*(x^i) - \hf_{-i}(x^i)}^4} = \dfrac{1}{\kappa^4} \expectation{\D_{n-1},X}{\norm{f^*(X) - \hf_{n-1}(X)}^4}, \\
&\pr\biggl\{\dfrac{1}{n} \displaystyle\sum_{i=1}^{n} \norm{f^*(x) - \hf_{-i}(x)} > \kappa\biggr\} \leq \dfrac{1}{n\kappa} \displaystyle\sum_{i=1}^{n} \expectation{\D_n}{\norm{f^*(x) - \hf_{-i}(x)}} = \dfrac{1}{\kappa} \expectation{\D_{n-1}}{\norm{f^*(x) - \hf_{n-1}(x)}},
\end{align*}
}%
when the data~$\D_n$ is i.i.d., we have that Assumptions~\ref{ass:jack_regconvrate} and~\ref{ass:jack+_regconvrate} on the {jackknife}-based methods hold if, for a.e.~$x \in \X$, the expectations $\expectation{\D_{n-1},X}{\norm{f^*(X) - \hf_{n-1}(X)}^4}$ and $\expectation{\D_{n-1}}{\norm{f^*(x) - \hf_{n-1}(x)}}$ converge to zero at suitable rates.
Under the aforementioned Lipshitz continuity assumption on the function~$f(x;\cdot)$ at~$\sth$ and the assumption that $\expect{L^4_f(X)} < +\infty$, it suffices to establish rates of convergence for the expectation term $\expectation{\D_{n-1}}{\norm{\sth - \hth_{n-1}}^4}$.
These results can be readily obtained under assumptions on the curvature of the loss function of the M-estimation problem (e.g., restricted strong convexity) around the true parameter $\sth$, see~\citet{negahban2012unified} for instance.
Chapter~14 of~\citet{biau2015lectures} provides similar rate results for kNN regression.
Alternatively, we can also bound the terms appearing in the assumptions for the {jackknife}-based formulations as
\begin{align*}
&\dfrac{1}{n} \displaystyle\sum_{i=1}^{n}\norm{f(x;\sth) - f(x;\hth_{-i})} \leq \dfrac{1}{n} \displaystyle\sum_{i=1}^{n} L_f(x) \norm{\sth - \hth_{-i}}, \\
&\dfrac{1}{n} \displaystyle\sum_{i=1}^{n}\norm{f(x^i;\sth) - f(x^i;\hth_{-i})}^4 \leq \dfrac{1}{n} \displaystyle\sum_{i=1}^{n} L^4_f(x^i) \norm{\sth - \hth_{-i}}^4 \leq \biggl(\dfrac{1}{n} \displaystyle\sum_{i=1}^{n} L^8_f(x^i)\biggr)^{1/2} \biggl(\dfrac{1}{n} \displaystyle\sum_{i=1}^{n} \norm{\sth - \hth_{-i}}^8\biggr)^{1/2},
\end{align*}
with the first term in the r.h.s.\ of the last inequality bounded under appropriate LLN assumptions.
Therefore, an alternative is to establish rates and finite sample guarantees for the two terms $\frac{1}{n}\sum_{i=1}^{n} \norm{\sth - \hth_{-i}}^8$ and $\frac{1}{n} \sum_{i=1}^{n} L^8_f(x^i)$.
A third direct approach is to use the weaker bounds
{
\small
\begin{align*}
&\pr\biggl\{\dfrac{1}{n} \displaystyle\sum_{i=1}^{n} \norm{f^*(x^i) - \hf_{-i}(x^i)}^4 > \kappa^4\biggr\} \leq \sum_{i=1}^{n} \prob{\norm{f^*(x^i) - \hf_{-i}(x^i)}^4 > \kappa^4} = n \prob{\norm{f^*(X) - \hf_{n-1}(X)}^4 > \kappa^4}, \\
&\pr\bigg\{\dfrac{1}{n} \displaystyle\sum_{i=1}^{n} \norm{f^*(x) - \hf_{-i}(x)} > \kappa\biggr\} \leq \sum_{i=1}^{n} \prob{\norm{f^*(x) - \hf_{-i}(x)} > \kappa} = n \prob{\norm{f^*(x) - \hf_{n-1}(x)} > \kappa}.
\end{align*}
}%
Finally, note that 
it is sufficient to establish rates and finite sample guarantees for $\frac{1}{n} \sum_{i=1}^{n} \norm{\hf_n(x^i) - \hf_{-i}(x^i)}^4$ and $\frac{1}{n} \sum_{i=1}^{n} \norm{\hf_n(x) - \hf_{-i}(x)}$ when Assumptions~\ref{ass:regconsist},~\ref{ass:regconvrate}, and~\ref{ass:reglargedev} hold.

\subsection{Nonparametric regression techniques for estimating \texorpdfstring{$f^*$}{fstar}}
\label{subsec:nonparamreg}

We verify that Assumptions~\ref{ass:regconsist_point},~\ref{ass:regconsist_mse},~\ref{ass:regconvrate_point}, and~\ref{ass:regconvrate_mse} hold for kNN regression, CART, and RF regression, and state a large deviation result similar to Assumptions~\ref{ass:reglargedev_point} and~\ref{ass:reglargedev_mse} for kNN regression.
The discussion in Section~\ref{subsec:mestimators} then provides an avenue for verifying the corresponding assumptions for the J-SAA and J+-SAA problems for i.i.d.\ data~$\{(x^i,\varepsilon^i)\}$.
{Theorem~14.5 in~\citet{biau2015lectures} presents conditions under which Assumption~\ref{ass:regconsist_point_lq} holds for $q \leq 2$.}
Note that results in~\citet{walk2010strong},~\citet{gyorfi2006distribution}, and~\citet{chen2018explaining} can be used to verify some of these assumptions for kernel regression and semi-recursive Devroye-Wagner estimates for mixing data~$\D_n$, results in~\citet{raskutti2012minimax} can be used to verify these assumptions for sparse additive nonparametric regression, Chapter~13 of~\citet{wainwright2019high} can be used to verify these assumptions for (regularized) nonparametric least squares regression, and results in~\citet{seijo2011nonparametric} and~\citet{mazumder2019computational} can be used to verify these assumptions for convex regression.
In what follows, we only consider the setting where the data~$\{(x^i,\varepsilon^i)\}$ is i.i.d.

We assume that the kNN regression estimate is computed as follows: given parameter $k \in [n]$ and $x \in \X$, define $\hf_n(x) := \frac{1}{k} \sum_{i=1}^{k} y^{(i)}(x)$, where $\{(y^{(i)}(x),x^{(i)}(x))\}_{i=1}^{n}$ is a reordering of the data $\{(y^i,x^i)\}_{i=1}^{n}$ such that $\norm{x^{(j)}(x)-x} \leq \norm{x^{(k)}(x)-x}$ whenever $j \leq k$ 
(if $\norm{x^{(j)}(x)-x} = \norm{x^{(k)}(x)-x}$ for some $j < k$, then we assume that $(y^{(j)}(x),x^{(j)}(x))$ appears first in the reordering).

\begin{proposition}
\label{prop:knn-ass}
Suppose the data~$\{(x^i,\varepsilon^i)\}$ is i.i.d.\ and the support $\X$ of the covariates is compact. 
{Define $\bar{\varepsilon} := Q^*(X)\varepsilon$, and suppose the distribution of the errors~$\bar{\varepsilon}$ satisfies $\sup_{x \in \R^{d_x}} \: \expect{\exp(\lambda \abs{\bar{\varepsilon}_j}) \mid X = x} < +\infty$ for each $j \in [d_y]$ and some $\lambda > 0$.}
Consider the setting where we use kNN regression to estimate the regression function $f^*$ with the parameter `$k$' satisfying $\uset{n \to \infty}{\lim} \frac{k}{\log(n)} = \infty$ and $k = o(n)$.
\begin{enumerate}
\item Suppose the function $f^*$ is continuous on $\X$.
Then Assumptions~\ref{ass:regconsist_point} and~\ref{ass:regconsist_mse} hold.

\item Suppose the function $f^*$ is twice continuously differentiable on $\X$ and the random vector $X$ has a density that is twice continuously differentiable.
Then, there exists a choice of the parameter `$k$' such that Assumptions~\ref{ass:regconvrate_point} and~\ref{ass:regconvrate_mse} hold with $\alpha = \frac{O(1)}{d_x}$.

\item Suppose the function $f^*$ is Lipschitz continuous on $\X$ and there exists a constant $\tau > 0$ such that the distribution $P_X$ of the covariates satisfies $\prob{X \in \mathcal{B}_{\kappa}(x)} \geq \tau \kappa^{d_x}$, $\forall x \in \X$ and $\kappa > 0$.
Then, for sample size $n$ satisfying $n \geq O(1)k\left(\frac{O(1)}{\kappa}\right)^{d_x}$ and $\frac{n^{\gamma}}{\log(n)} \geq \frac{O(1) d_x d_y \sigma^2}{\kappa^2}$, we have
{
\small
\begin{align*}
\hspace*{-0.4in}\prob{\uset{x \in \X}{\sup} \: \norm{f^*(x) - \hf_n(x)} > \kappa\sqrt{d_y}} &\leq \left(\dfrac{O(1) \sqrt{d_x}}{\kappa}\right)^{d_x} \exp\left( - O(1)n (O(1)\kappa)^{2d_x} \right) + O(1) n^{2d_x} \left(\dfrac{O(1)}{d_x}\right)^{d_x} \exp\left( - \dfrac{k\kappa^2}{O(1)\sigma^2}  \right).
\end{align*}
}%
\end{enumerate}
\end{proposition}
\begin{proof}
The first part follows from Theorem~12.1 of~\citet{biau2015lectures}.
The second part follows from Theorems~14.3 and~14.5 of~\citet{biau2015lectures} and Markov's inequality.
The last part follows from Lemma~10 of~\citet{bertsimas2019predictions}.
\end{proof}

{\citet{jiang2019non} presents improved rates of convergence and finite sample guarantees for kNN regression in the homoscedastic setting when the data~$\D_n$ lies on a low-dimensional manifold.}
Lemma~7 of~\citet{bertsimas2019predictions} presents conditions under which CART regression satisfies Assumption~\ref{ass:regconsist}.
Along with Theorem~8 of~\citet{wager2018estimation}, the above result can be used to show that Assumption~\ref{ass:regconvrate} holds for CART regression with $\alpha = \frac{O(1)}{d_x}$.
Lemma~9 of~\citet{bertsimas2019predictions} presents conditions under which RF regression satisfies Assumption~\ref{ass:regconsist}.
Once again, we can use this result along with Theorem~8 of~\citet{wager2018estimation} to show that Assumption~\ref{ass:regconvrate} holds for RF regression with $\alpha = \frac{O(1)}{d_x}$.

\subsection{Verifying assumptions on the estimation of \texorpdfstring{$Q^*$}{Qstar}}
\label{subsec:verify_Qstar_ass}

{We verify that the parts of Assumptions~\ref{ass:regconsist},~\ref{ass:reglargedev}, and~\ref{ass:regconvrate} relating to the estimate~$\hQ_n$ of~$Q^*$ hold for some setups with structured heteroscedasticity.
The techniques in Section~\ref{subsec:mestimators} may then be used to verify the assumptions for the jackknife-based estimators in Section~\ref{sec:jackknife} when the data~$\{(x^i,\varepsilon^i)\}$ is i.i.d.
These assumptions for $\hQ_n$---in particular, Assumption~\ref{ass:reglargedev}---are not as well-studied in the literature as those for~$\hf_n$ and are typically harder to verify.
Because deriving theoretical properties of estimators in the heteroscedastic setting and deriving finite sample properties of estimators in general are areas of topical interest, we envision that future research will enable easier verification of these assumptions.
For simplicity, we only consider function classes~$\Q$ that comprise diagonal covariance matrices~\citep{zhou2018new}, i.e., 
\[
\Q := \{Q : \R^{d_x} \to \R^{d_y \times d_y} : Q(X) = \text{diag}(q_1(X), q_2(X), \dots, q_{d_y}(X))\},
\]
although the ER-SAA approach is more generally applicable.
\citet{bauwens2006multivariate} review some model classes~$\Q$ with non-diagonal covariance matrices that are popular in time series modeling.}

\subsubsection{Parametric models for heteroscedasticity}
{We assume that $Q^*(X) \equiv Q(X;\spi)$ for some finite-dimensional parameter~$\spi \in \Pi$ and the goal is to estimate~$\spi$. 
Forms of the functions~$q_j$ of interest include~\citep{powell2010models,romano2017resurrecting}:}
\begin{enumerate}[label=\roman*.]
\item {$(q_j(X;\pi))^2 = (\tr{(\pi^j)} X)^2$,}

\item {$(q_j(X;\pi))^2 = \exp(\tr{(\pi^j)} X)$,}

\item {$(q_j(X;\pi))^2 = \exp\bigl(\tr{(\pi^j)} \log(\abs{X})\bigr)$, where $X_1 \equiv 2$ (instead of $X_1 \equiv 1$).}
\end{enumerate}
{The above setup can also accommodate cases where the parameters of the function~$Q^*$ include some of the parameters of the function~$f^*$.}

{Let~$\hpi_n$ denote the estimate of~$\spi$ corresponding to the regression estimate~$\hQ_n$.
Suppose for a.e.\ realization~$x \in \X$, the function~$Q(x;\cdot)$ is Lipschitz continuous with Lipschitz constant~$L_Q(x)$ and its inverse~$[Q(x;\cdot)]^{-1}$ is also Lipschitz continuous with Lipschitz constant~$\bar{L}_Q(x)$.
These assumptions hold for the above parametric models if the parameters~$\pi$ therein are restricted to lie in a compact sets (similar to Section~\ref{subsec:mestimators}, it suffices to assume that the above Lipschitz continuity assumptions hold locally for the asymptotic results).
Because
\[
\norm{\hQ_n(x) - Q^*(x)} \leq L_Q(x)\norm{\hpi_n - \spi}, \:\: \frac{1}{n} \sum_{i=1}^{n} \bigl\lVert \bigl[\hQ_n(x^i)\bigr]^{-1} - \bigl[Q^*(x^i)\bigr]^{-1}\bigr\rVert^2 \leq \Bigl(\frac{1}{n} \sum_{i=1}^{n}\bar{L}^2_Q(x^i)\Bigr) \norm{\hpi_n - \spi}^2,
\]
asymptotic and finite sample guarantees on the estimator~$\hpi_n$ of~$\spi$ directly translate to the asymptotic and finite sample guarantees on the estimate~$\hQ_n$ in Assumptions~\ref{ass:regconsist_point2},~\ref{ass:regconsist_mse2},~\ref{ass:regconvrate_point2},~\ref{ass:regconvrate_mse2},~\ref{ass:reglargedev_point2}, and~\ref{ass:reglargedev_mse2}.
When the functions~$f^*$ and~$Q^*$ are jointly estimated using an M-estimation procedure, the results in Section~\ref{subsec:mestimators} provide conditions under which the estimator~$\hpi_n$ of~$\spi$ is consistent and Assumptions~\ref{ass:regconsist_point2},~\ref{ass:regconsist_mse2},~\ref{ass:regconvrate_point2}, and~\ref{ass:regconvrate_mse2} hold with $\alpha = 1$.
Section~\ref{subsec:mestimators} also presents a hard-to-verify uniform exponential bound condition under which~$\hpi_n$ possesses a finite sample guarantee.
\citet{carroll1982robust} consider robust M-estimators for~$\spi$ that possess a similar rate of convergence when~$f^*$ is linear.
\citet{dalalyan2013learning} present asymptotic and finite sample guarantees for a scaled Dantzig estimator of~$\spi$ under some sparsity assumptions.
Finally,~\citet{fan2014quasi} present a quasi-maximum likelihood approach for estimating the parameters of GARCH models and investigate their asymptotic properties.}

{In the remainder of this section, we specialize the verification of Assumptions \ref{ass:regconsist_point2}, \ref{ass:regconsist_mse2}, \ref{ass:regconvrate_point2}, and~\ref{ass:regconvrate_mse2}, to Example~\ref{exm:regrexample}.
We are unable to verify Assumptions~\ref{ass:reglargedev_point2} and~\ref{ass:reglargedev_mse2} because the literature lacks suitable finite-sample guarantees for heteroscedasticity estimation.}

\paragraph{{Verifying Assumptions~\ref{ass:regconsist} and~\ref{ass:regconsist_lq}}.}
{We verify Assumptions~\ref{ass:regconsist_point2},~\ref{ass:regconsist_mse2}, and~\ref{ass:regconsist_point2_lq} for Example~\ref{exm:regrexample}.
Since}
\begin{align*}
{\norm{\hQ_n(x) - Q^*(x)}} &{\leq \max_{j \in [d_y]} \abs{\hat{q}_{jn}(x) - q^*_j(x)}} \\
&{= \max_{j \in [d_y]} \bigg\lvert \biggl(\exp\biggl(\sum_{k=1}^{d_x}\hat{\pi}^j_{kn} \log(\abs{x_k})\biggr)\biggr)^{1/2} - \biggl(\exp\biggl(\sum_{k=1}^{d_x}\pi^{j*}_{k} \log(\abs{x_k})\biggr)\biggr)^{1/2}\bigg\rvert,}
\end{align*}
{we have that Assumption~\ref{ass:regconsist_point2} holds whenever $\hpi_n \convinprob \spi$ on account of Slutsky's lemma and the continuous mapping theorem.
Appendix~B.2 of~\citet{romano2017resurrecting} identifies conditions under which $\hpi_n \convinprob \spi$ for Example~\ref{exm:regrexample}.
To verify Assumption~\ref{ass:regconsist_mse2}, assume for simplicity that the support~$\X$ is compact and bounded away from the origin and the estimates~$\{\hpi_n\}$ lie in a compact set a.s.\ for $n$ large enough. 
Since}
{
\small
\begin{align*}
{\frac{1}{n}\sum_{i=1}^{n} \bigl\lVert \bigl[\hQ_n(x^i)\bigr]^{-1} - \bigl[Q^*(x^i)\bigr]^{-1}\bigr\rVert^2} &{\leq \frac{1}{n}\sum_{i=1}^{n} \max_{j \in [d_y]}\abs{(\hat{q}_{jn}(x^i))^{-1} - (q^*_j(x^i))^{-1}}^2} \\
&{\leq \frac{1}{n}\sum_{i=1}^{n} \max_{j \in [d_y]}\abs{(\hat{q}_{jn}(x^i))^{-2} - (q^*_j(x^i))^{-2}}} \\
&{= \frac{1}{n}\sum_{i=1}^{n} \max_{j \in [d_y]}\bigg\lvert \exp\biggl(-\sum_{k=1}^{d_x}\hat{\pi}^j_{kn} \log(\abs{x^i_k})\biggr) - \exp\biggl(-\sum_{k=1}^{d_x}\pi^{j*}_{k} \log(\abs{x^i_k})\biggr)\bigg\rvert} \\
&{\leq \frac{1}{n}\sum_{i=1}^{n} O(1) \max_{j \in [d_y]}\bigg\lvert \sum_{k=1}^{d_x}(\pi^{j*}_k-\hat{\pi}^j_{kn}) \log(\abs{x^i_k})\bigg\rvert} \\
&{\leq O(1)\max_{j \in [d_y]}\sum_{k=1}^{d_x}\abs{\pi^{j*}_k-\hat{\pi}^j_{kn}},}
\end{align*}%
}%
{where the last two steps follow by the mean-value theorem, the assumption that $\X$ is compact and bounded away from the origin, and the assumption that the sequence $\{\hpi_n\}$ lies in a compact set a.s.\ for $n$ large enough (note that the compactness assumption on~$\X$ can be relaxed, e.g., if the sequence $\{\hpi_n\}$ lies in the nonnegative orthant and $\abs{\bx} \geq 1$, $\forall \bx \in \X$). 
Consequently, Assumption~\ref{ass:regconsist_mse2} also holds whenever $\hpi_n \convinprob \spi$.
The above arguments also imply $\sup_{x \in \X} \norm{[Q^*(x)]^{-1} - [\hQ_n(x)]^{-1}} \convinprob 0$.
These conditions also guarantee that Assumption~\ref{ass:regconsist_point2_lq} holds since}
{
\small
\begin{align*}
{\norm{\hQ_n(X) - Q^*(X)}_{L^q}} &{\leq \biggl(\mathbb{E}\biggl[\sum_{j=1}^{d_y} \abs{\hat{q}_{jn}(X) - q^*_j(X)}^q\biggr]\biggr)^{1/q}} \\
&{\leq \biggl(\mathbb{E}\biggl[\sum_{j=1}^{d_y} \bigg\lvert \biggl(\exp\biggl(\sum_{k=1}^{d_x}\hat{\pi}^j_{kn} \log(\abs{X_k})\biggr)\biggr)^{1/2} - \biggl(\exp\biggl(\sum_{k=1}^{d_x}\pi^{j*}_{k} \log(\abs{X_k})\biggr)\biggr)^{1/2}\bigg\rvert^q\biggr]\biggr)^{1/q}} \\
&{\leq \biggl(\sum_{j=1}^{d_y}\mathbb{E}\biggl[ O(1)\bigg\lvert \sum_{k=1}^{d_x}(\hat{\pi}^j_{kn} - \pi^{j*}_k) \log(\abs{X_k})\bigg\rvert^{q}\biggr]\biggr)^{1/q}} \\
&{\leq \biggl(\sum_{j=1}^{d_y} O(1)\biggl\lvert\sum_{k=1}^{d_x}\abs{\hat{\pi}^j_{kn} - \pi^{j*}_k} \biggr\rvert^{q}\biggr)^{1/q}}  \\
&{\leq \sum_{j=1}^{d_y} O(1)\sum_{k=1}^{d_x}\abs{\hat{\pi}^j_{kn} - \pi^{j*}_k},} 
\end{align*}%
}%
{where the third and fourth steps again follow by the mean-value theorem, the compactness of $\X$ and the fact that it is bounded away from the origin, and the a.s.\ compactness of the sequence $\{\hpi_n\}$ for $n$ large enough.
The above arguments also imply $\sup_{x \in \X} \norm{Q^*(x) - \hQ_n(x)} \convinprob 0$.}

\paragraph{{Verifying Assumption~\ref{ass:regconvrate}}.}
{We show that Assumptions~\ref{ass:regconvrate_point2} and~\ref{ass:regconvrate_mse2} hold whenever $\norm{\hpi_n - \spi} = O_p(n^{-\alpha/2})$, the support~$\X$ is compact and bounded away from the origin, and the estimates~$\{\hpi_n\}$ lie in a compact set a.s.\ for $n$ large enough. Note that}
\begin{align*}
{\sup_{x \in \X} \norm{Q^*(x) - \hQ_n(x)}} &{\leq \max_{j \in [d_y]} \sup_{x \in \X} \abs{\hat{q}_{jn}(x) - q^*_j(x)}} \\
&{= \max_{j \in [d_y]} \sup_{x \in \X} \bigg\lvert \biggl(\exp\biggl(\sum_{k=1}^{d_x}\hat{\pi}^j_{kn} \log(\abs{x_k})\biggr)\biggr)^{1/2} - \biggl(\exp\biggl(\sum_{k=1}^{d_x}\pi^{j*}_{k} \log(\abs{x_k})\biggr)\biggr)^{1/2}\bigg\rvert} \\
&{\leq \max_{j \in [d_y]} O(1) \sum_{k=1}^{d_x}\abs{\pi^{j*}_k-\hat{\pi}^j_{kn}},}
\end{align*}
{where the last step follows by arguments similar to the derivation of Assumption~\ref{ass:regconsist_point2_lq} using the mean-value theorem, the compactness of $\X$ and the fact that it is bounded away from the origin, and the a.s.\ compactness of the sequence $\{\hpi_n\}$ for $n$ large enough.
Therefore, Assumption~\ref{ass:regconvrate_point2} holds whenever $\norm{\hpi_n - \spi} = O_p(n^{-\alpha/2})$. 
Similarly, we have}
{
\small
\begin{align*}
{\sup_{x \in \X} \norm{[Q^*(x)]^{-1} - [\hQ_n(x)]^{-1}}} &{\leq \max_{j \in [d_y]}  \sup_{x \in \X} \abs{(\hat{q}_{jn}(x))^{-1} - (q^*_j(x))^{-1}}} \\
&{\leq \max_{j \in [d_y]} \sup_{x \in \X} \bigg\lvert \biggl(\exp\biggl(-\sum_{k=1}^{d_x}\hat{\pi}^j_{kn} \log(\abs{x_k})\biggr)\biggr)^{1/2} - \biggl(\exp\biggl(-\sum_{k=1}^{d_x}\pi^{j*}_{k} \log(\abs{x_k})\biggr)\biggr)^{1/2}\bigg\rvert} \\
&{\leq \max_{j \in [d_y]}  O(1) \sum_{k=1}^{d_x}\abs{\pi^{j*}_k-\hat{\pi}^j_{kn}},}
\end{align*}%
}%
{by arguments similar to the derivation of Assumption~\ref{ass:regconsist_mse2}. Therefore, Assumption~\ref{ass:regconvrate_mse2} holds whenever $\norm{\hpi_n - \spi} = O_p(n^{-\alpha/2})$.
Appendix~B.2 of~\citet{romano2017resurrecting} identifies conditions under which $\hpi_n \convinprob \spi$ for Example~\ref{exm:regrexample} at a certain rate.}

\subsubsection{Nonparametric models for heteroscedasticity}
{We assume that each function $q_j : \R^{d_x} \to \R_+$ is `sufficiently smooth'.
Chapter~8 of~\citet{fan2008nonlinear} presents some popular models for the functions~$q_j$ in a time series context.}

{Suppose the function~$Q^*$ and its regression estimate~$\hQ_n$ are (asymptotically) a.s.\ uniformly invertible, i.e., $\sup_{\bar{x} \in \X} \bigl\lVert \bigl[Q^*(\bar{x})\bigr]^{-1} \bigl\rVert < +\infty$ and $\sup_{\bar{x} \in \X} \bigl\lVert \bigl[\hQ_n(\bar{x})\bigr]^{-1} \bigl\rVert < +\infty$.
We have}
{
\small
\begin{align*}
{\frac{1}{n} \sum_{i=1}^{n} \bigl\lVert \bigl[\hQ_n(x^i)\bigr]^{-1} - \bigl[Q^*(x^i)\bigr]^{-1}\bigr\rVert^2} \leq& {\frac{1}{n} \sum_{i=1}^{n} \bigl\lVert \bigl[Q^*(x^i)\bigr]^{-1}\bigr\rVert^2 \bigl\lVert \bigl[\hQ_n(x^i)\bigr]^{-1}\bigr\rVert^2 \norm{\hQ_n(x^i) - Q^*(x^i)}^2} \\
\leq& {\biggl(\sup_{\bar{x} \in \X} \bigl\lVert \bigl[Q^*(\bar{x})\bigr]^{-1}\bigr\rVert^2\biggr) \biggl(\sup_{\bar{x} \in \X} \bigl\lVert \bigl[\hQ_n(\bar{x})\bigr]^{-1}\bigr\rVert^2\biggr) \biggl( \frac{1}{n} \sum_{i=1}^{n} \norm{\hQ_n(x^i) - Q^*(x^i)}^2 \biggr)}
\end{align*}%
}%
{Therefore, asymptotic and finite sample guarantees for $\norm{\hQ_n(x) - Q^*(x)}$ and $\frac{1}{n} \sum_{i=1}^{n} \norm{\hQ_n(x^i) - Q^*(x^i)}^2$ are sufficient for verifying Assumptions \ref{ass:regconsist_point2}, \ref{ass:regconsist_mse2}, \ref{ass:regconvrate_point2}, \ref{ass:regconvrate_mse2}, \ref{ass:reglargedev_point2}, and \ref{ass:reglargedev_mse2}.
Theorem~8.5 of~\citet{fan2008nonlinear} can be used to identify conditions under which these asymptotic guarantees hold for local linear estimators on time series data when the dimension of the covariates~$d_x = 1$.
They also mention approaches for estimating~$Q^*$ when $d_x > 1$.
Theorem~2 of~\citet{ruppert1997local} can be used to verify Assumptions~\ref{ass:regconsist} and~\ref{ass:regconvrate} for local polynomial smoothers.
Proposition~2.1 and Theorem~3.1 of~\citet{jin2015adaptive} identify conditions under which Assumptions~\ref{ass:regconsist} and~\ref{ass:regconvrate} hold for a local likelihood estimator.
\citet{van2010semiparametric} consider semiparametric models for both~$f^*$ and~$Q^*$.
Theorems~3.1 and~3.2 therein can be used to verify Assumptions~\ref{ass:regconsist} and~\ref{ass:regconvrate} for the estimates~$\hQ_n$. Section~3 of~\citet{zhou2018new} presents robust estimators of~$Q^*$ when~$f^*$ is linear and notes that these estimators~$\hQ_n$ possess asymptotic and finite sample guarantees in the form of Assumptions~\ref{ass:regconsist},~\ref{ass:regconvrate}, and~\ref{ass:reglargedev}.
Finally, Theorem~3.1 of~\citet{chesneau2020nonparametric} can be used to derive asymptotic guarantees for wavelet estimators of~$Q^*$.}

\section{Omitted details for the computational experiments}
\label{sec:computexp-details}

The parameters $\varphi^*$ and $\zeta^*$ in the true demand model are specified as:
\begin{align*}
\varphi^*_j &= 50 + 5\delta_{j0}, \quad \zeta^*_{j1} = 10 + \delta_{j1}, \quad \zeta^*_{j2} = 5 + \delta_{j2}, \quad \text{and} \quad \zeta^*_{j3} = 2 + \delta_{j3}, \quad \forall j \in \J,
\end{align*}
where $\{\delta_{j0}\}_{j \in \J}$ are i.i.d.\ samples from the standard normal distribution $\mathcal{N}(0,1)$, and $\{\delta_{j1}\}_{j \in \J}$, $\{\delta_{j2}\}_{j \in \J}$, and $\{\delta_{j3}\}_{j \in \J}$ are i.i.d.\ samples from the uniform distribution $U(-4,4)$.
We generate i.i.d.\ samples of the covariates~$X$ from a \textit{multivariate folded/half-normal distribution}. 
We specify the underlying normal distribution to have mean $\mu_X = 0$ and set its covariance matrix $\Sigma_X$ to be a random correlation matrix that is generated using the `vine method' of~\citet{lewandowski2009generating} (each partial correlation is sampled from the $\text{Beta}(2,2)$ distribution and rescaled to $[-1,1]$).
Finally, Algorithm~\ref{alg:99ucbs} describes our procedure for estimating the normalized $99$\% UCB on the optimality gap of our data-driven solutions using the multiple replication procedure~\citep{mak1999monte}.

\begin{algorithm}
\caption{Estimating the normalized $99$\% UCB on the optimality gap of a given solution.}
\label{alg:99ucbs}
{
\begin{algorithmic}[1]
\State \textbf{Input}: Covariate realization $X = x$ and data-driven solution $\hz_n(x)$ 
for a particular realization of the data~$\D_n$.

\State \mbox{\textbf{Output}: $\hat{B}_{99}(x)$, which is a normalized estimate of the $99$\% UCB on the optimality gap of $\hz_n(x)$.}

\For{$k = 1, \dots, 30$}

\State Draw $1000$ i.i.d.\ samples $\bar{\D}^k := \{\bar{\varepsilon}^{k,i}\}_{i=1}^{1000}$ of $\varepsilon$ according to the distribution $P_{\varepsilon}$.

\State Estimate the optimal value $v^*(x)$ by solving the full-information SAA problem~\eqref{eqn:fullinfsaa} using 

\Statex \hspace*{0.2in} the data $\bar{\D}^k$:
\[
\bar{v}^k(x) := \uset{z \in \Z}{\min} \:\: \dfrac{1}{1000} \displaystyle\sum_{i=1}^{1000} c(z,f^*(x)+{Q^*(x)}\bar{\varepsilon}^{k,i}).
\]

\State Estimate the out-of-sample cost of the solution $\hz_n(x)$ using the data~$\bar{\D}^k$:
\[
\hat{v}^k(x) := \dfrac{1}{1000} \displaystyle\sum_{i=1}^{1000} c(\hz_n(x),f^*(x)+{Q^*(x)}\bar{\varepsilon}^{k,i}).
\]

\State Estimate the optimality gap of the solution $\hz_n(x)$ as $\hat{G}^k(x) = \hat{v}^k(x) - \bar{v}^k(x)$.

\EndFor

\State Construct the normalized estimate of the $99$\% UCB on the optimality gap of $\hz_n(x)$ as 
\[
\hat{B}_{99}(x) := \dfrac{100}{\abs{\bar{v}(x)}} \left(\dfrac{1}{30} \displaystyle\sum_{k=1}^{30} \hat{G}^k(x) + 2.462\displaystyle\sqrt{\dfrac{\text{var}(\{\hat{G}^k(x)\})}{30}}\right),
\]
where $\bar{v}(x) := \dfrac{1}{30} \displaystyle\sum_{k=1}^{30} \bar{v}^k(x)$ and $\text{var}(\{\hat{G}^k(x)\})$ denotes the variance of the gaps $\{\hat{G}^k(x)\}_{k=1}^{30}$.

\end{algorithmic}
}
\end{algorithm}

\begin{figure}[t!]
    \centering
    \begin{subfigure}[t]{0.33\textwidth}
        \centering
        \includegraphics[width=\textwidth]{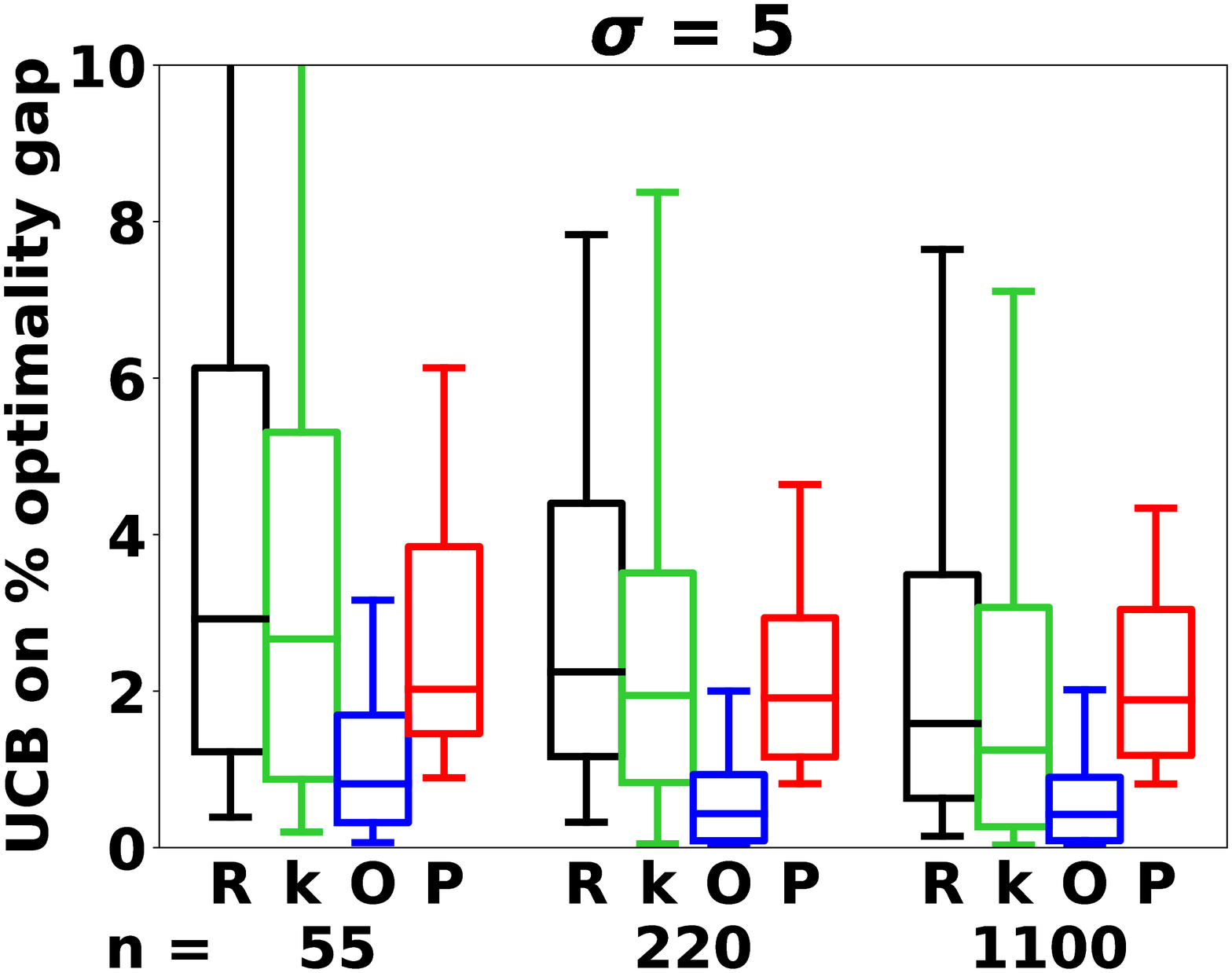}
    \end{subfigure}%
    ~ 
    \begin{subfigure}[t]{0.33\textwidth}
        \centering
        \includegraphics[width=\textwidth]{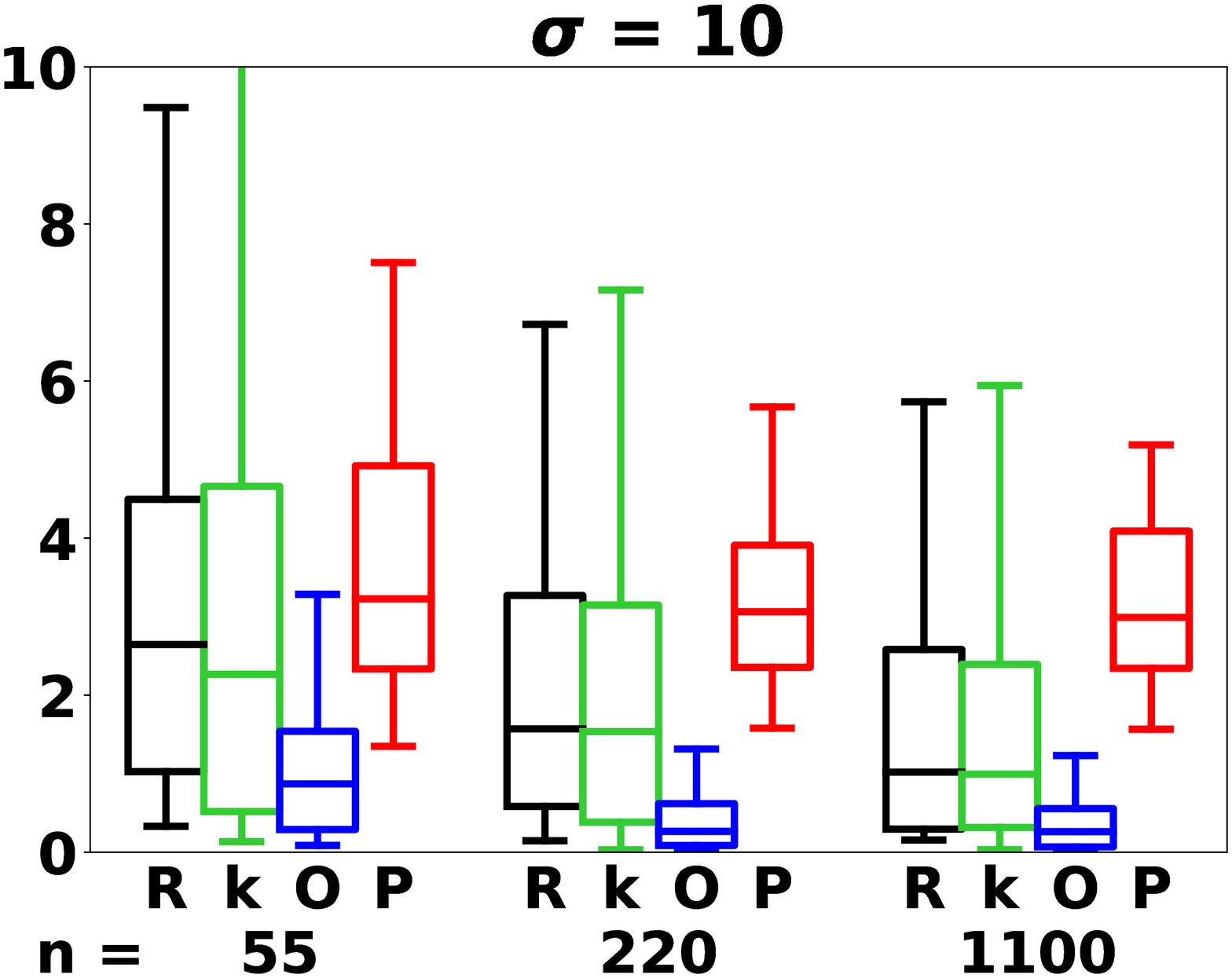}
    \end{subfigure}%
    ~ 
    \begin{subfigure}[t]{0.33\textwidth}
        \centering
        \includegraphics[width=\textwidth]{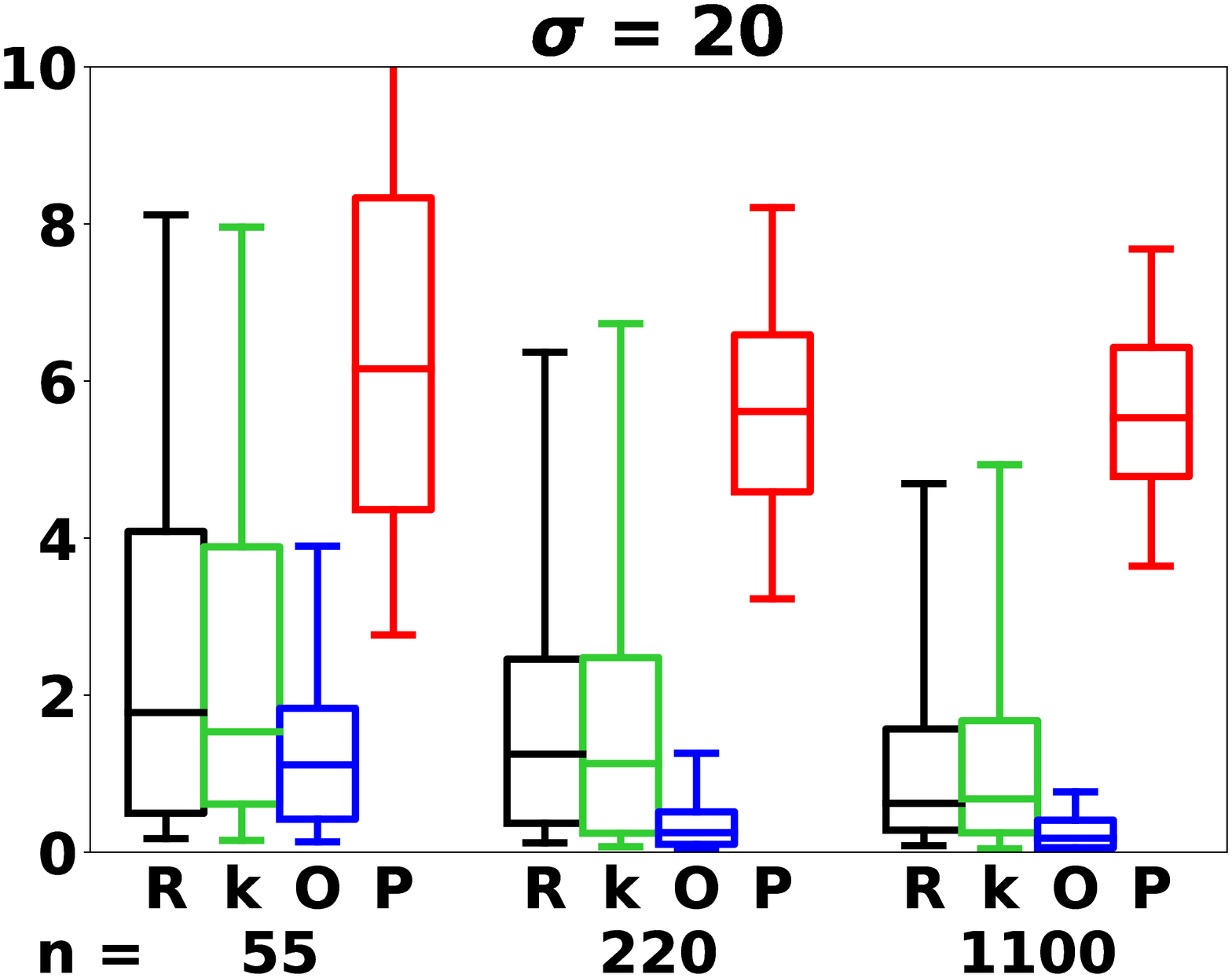}
    \end{subfigure}\\
    \begin{subfigure}[t]{0.33\textwidth}
        \centering
        \includegraphics[width=\textwidth]{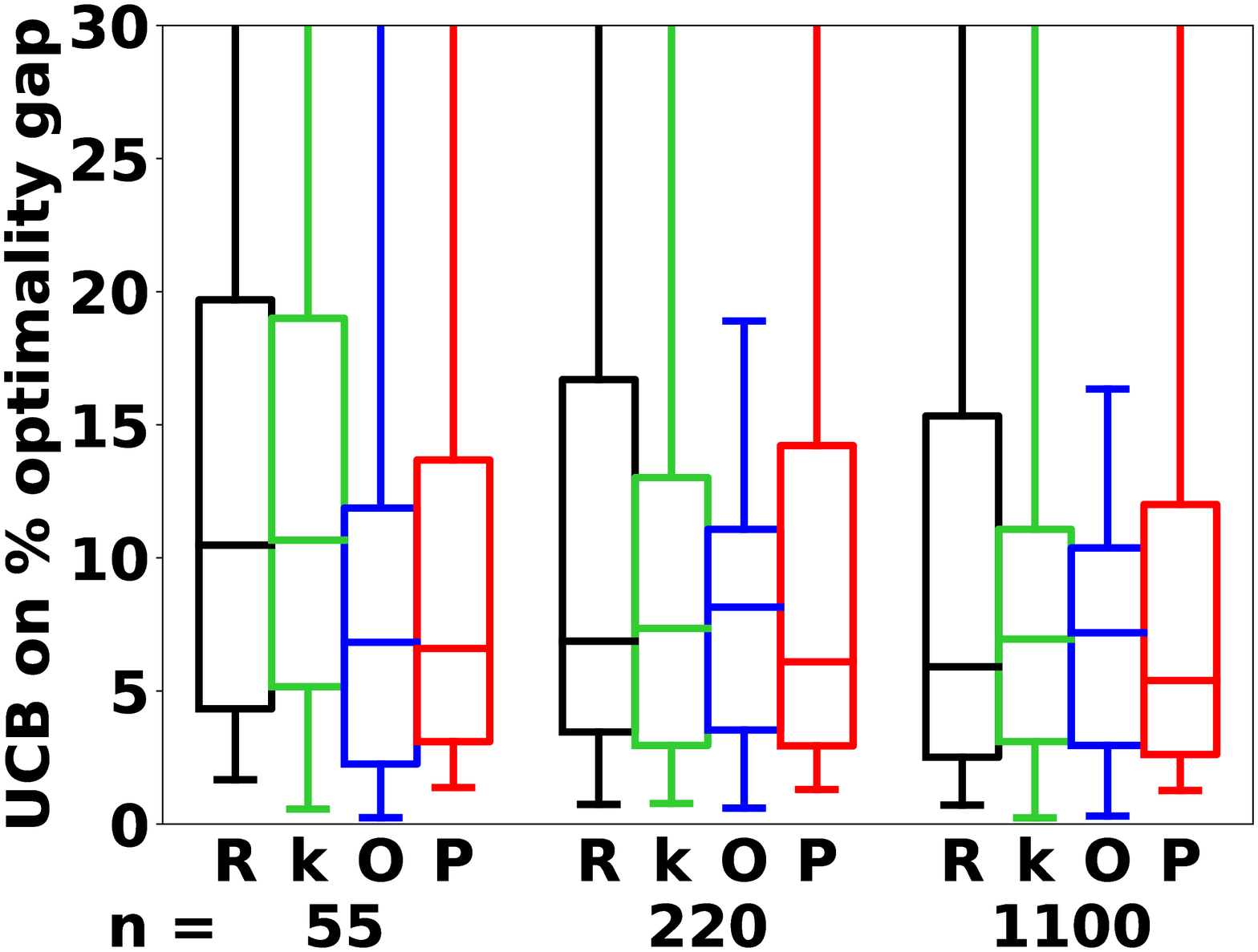}
    \end{subfigure}%
    ~ 
    \begin{subfigure}[t]{0.33\textwidth}
        \centering
        \includegraphics[width=\textwidth]{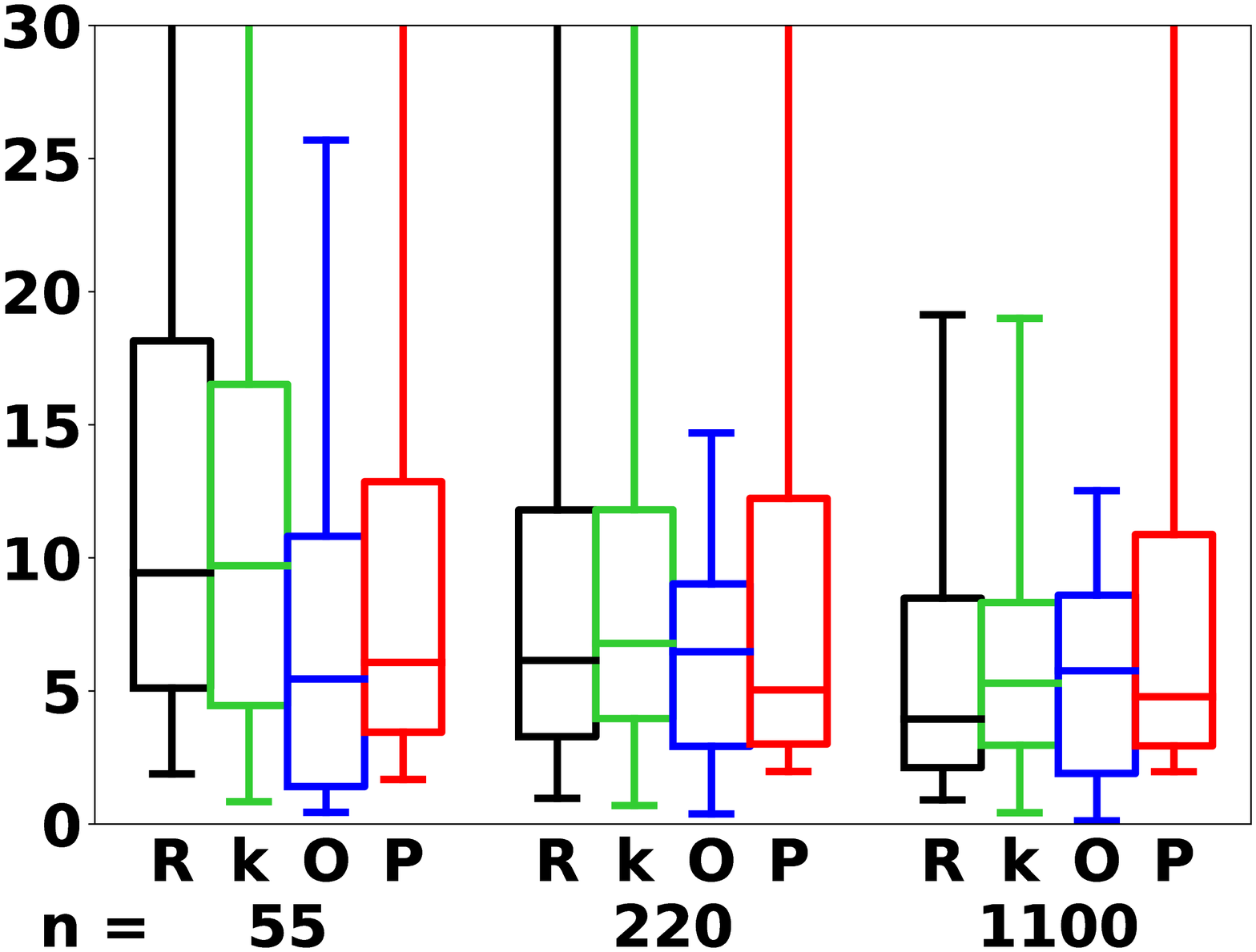}
    \end{subfigure}%
    ~ 
    \begin{subfigure}[t]{0.33\textwidth}
        \centering
        \includegraphics[width=\textwidth]{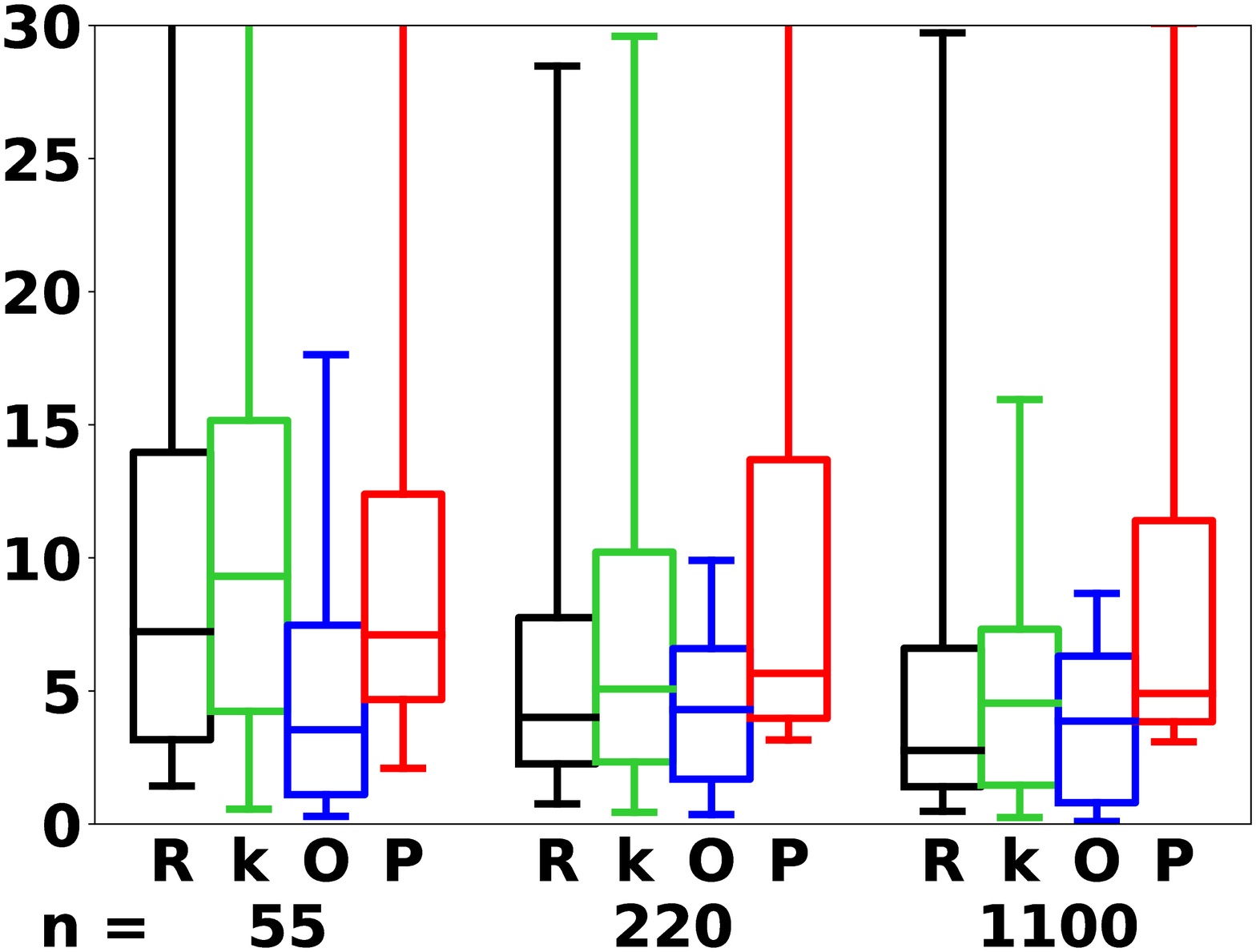}
    \end{subfigure}
    \caption{{Effect of increasing $\sigma$ on re-weighted kNN-SAA (\texttt{R}), ER-SAA+kNN (\texttt{k}), ER-SAA+OLS (\texttt{O}), and PP+OLS (\texttt{P}) approaches when $\omega = 1$. Top row: $p = 0.5$. Bottom row: $p = 2$. Left column: $\sigma = 5$. Middle column: $\sigma = 10$. Right column: $\sigma = 20$.}}
    \label{fig:sigma_plots_2}
\end{figure}

\end{document}